\documentclass[12pt]{amsart}
\usepackage{hyperref}
\usepackage{amsfonts}
\usepackage{bbm}
\usepackage{esint}
\usepackage{mathrsfs}
\usepackage{amssymb,amsmath,amsfonts,amsthm,color}
\usepackage{setspace}
\numberwithin{equation}{section}

\DeclareMathOperator{\supp}{supp}

\allowdisplaybreaks

\theoremstyle{plain}
\newtheorem{theorem}{Theorem}[section]
\newtheorem{prop}[theorem]{Proposition}
\newtheorem{lemma}[theorem]{Lemma}
\newtheorem{corollary}[theorem]{Corollary}

\theoremstyle{definition}
\newtheorem{definition}[theorem]{Definition}

\newtheorem{proposition}[theorem]{Proposition}

\allowdisplaybreaks

\newcommand{\R}{\mathbb{R}}
\newcommand{\C}{\mathbb{C}}

\newtheorem{rem}{\it Remark}[section]

\def\R{\mathbb R}
\def\C{\mathbb C}
\def\12{\frac{1}{2}}

\def\b1 {\dot{B}_1^{-\alpha,1}}

\begin{document}

\setcounter{tocdepth}{3} \allowdisplaybreaks

\title[singular integral operators, Littlewood-Paley theory and Hardy spaces]
{Singular integral operators, $T1$  theorem,  Littlewood-Paley theory and Hardy spaces\\ in Dunkl Setting
}

\author[CQ. Tan, Ya Han, Yo. Han, M.-Y. Lee, and J. Li]{Chaoqian Tan, Yanchang Han, Yongsheng Han, Ming-Yi Lee and  Ji Li }

\subjclass[2010]{Primary 42B35; Secondary 43A85, 42B25, 42B30}

\keywords{Dunkl-Calder\'on-Zygmund singular integral, $T1$ theorem, Littlewood-Paley theory,
Hardy space}
\begin{abstract}
    The purpose of this paper is to introduce a new class of singular integral operators in the Dunkl setting involving both the Euclidean metric and the Dunkl metric. Then we provide the $T1$ theorem, the criterion for the boundedness on $L^2$ for these operators. Applying this  singular integral operator theory, we establish the Littlewood-Paley theory and the Dunkl-Hardy spaces. As applications, the boundedness of singular integral operators, particularly, the Dunkl-Rieze transforms, on the Dunkl-Hardy spaces is given. The $L^2$ theory and the singular integral operator theory play crucial roles. New tools developed in this paper include the weak-type discrete Calder\'on reproducing formulae, new test functions and distributions, the Littlewood-Paley, the wavelet-type  decomposition, and molecule characterizations of the Dunkl-Hardy space, Coifman's approximation to the identity and the decomposition of the identity operator on $L^2$, Meyer's commutation Lemma, and new almost orthogonal estimates in the Dunkl setting.
\end{abstract}

\maketitle
\begin{spacing}{1.5}
    \tableofcontents
\end{spacing}
\section{Introduction and statement of main results}

It is well known that group structures enter in a decisive way in
harmonic analysis. In this paper, we develop harmonic analysis in
the Dunkl setting which is associated with finite reflection groups
on the Euclidean space. This particular group structure is
conducting the analysis. Indeed, in the Dunkl setting, there are
corresponding Dunkl transform, the Dunkl translation and convolution
operators. Therefore, the Dunkl setting does not fall in the scope
of spaces of homogeneous type in the sense of Coifman and Weiss.
More precisely, in the Dunkl setting, we introduce the
Dunkl-Calder\'on-Zygmund singular integral operators with operator
kernels involving both the Euclidean metric and the so called Dunkl
"metric" deduced by the finite reflection groups. This new class of
singular integrals covers the well-known Dunkl Riesz transforms and
generalizes the classical Calder\'on-Zygmund singular integrals on
space of homogeneous type. Moreover, we establish the  $T1$ theorem
to provide the criterion for the boundedness on $L^2$ for the
Dunkl-Calder\'on-Zygmund singular integral operators. Further, we
also establish the Littlewood-Paley theory and the Dunkl-Hardy
spaces $H_d^p$ for $p\leqslant 1$ and close to $1$, in which $H_d^p$
when $p<1$ was missing in the Dunkl literature. We characterize
$H_d^p$ via Littlewood--Paley theory, the wavelet-type and atomic decomposition, and molecule theory in the Dunkl setting. As applications, we obtain the
boundedness of this new singular integrals on the Dunkl-Hardy
spaces. The crucial ideas used in this paper are the $L^2$ theory
and the Dunkl-Calder\'on-Zygmund singular integral operator theory.
New tools developed in this paper are new test functions and
distributions, the weak-type discrete Calder\'on reproducing
formula, and several almost orthogonal estimates in the Dunkl
setting.

We now state the background and main results in more details.

\subsection{{Preliminaries and questions in the Dunkl setting}}
\ \\

The classical Fourier transform, initially defined on
$L^1(\mathbb{R}^{N}),$ extends to an isometry of
$L^2(\mathbb{R}^{N})$ and satisfies certain properties with
translation, dilation and rotation groups. Dunkl introduced a
similar transform, the Dunkl transform, which enjoys properties
similar to the classical Fourier transform. This transform is
defined by
$$ \hat{f}(x)=c_h\int_{\mathbb{R}^{N}} E(x, -iy)f(y)h^2_{\kappa}(y)dy,$$
where the usual character $e^{-i\langle x,y\rangle}$ is replaced by
$E(x, -iy)=V_{\kappa}(e^{-i\langle .,y\rangle})(x)$ for some
positive linear operator $V_{\kappa}$ and the weight functions
$h_{\kappa}$ are invariant under a finite reflection group $G$ on
$\mathbb{R}^{N},$ see \cite{D2}. Particularly, the Dunkl transform
satisfies the Plancheral identity, namely, $\|\hat{f}\|_2=\|f\|_2$
and  if the parameter $\kappa=0,$ then $h_{\kappa}(x)=1$ and
$V_{\kappa}=id,$ thus the Dunkl transform becomes the classical
Fourier transform.

The classical Fourier transform behaves well with the translation
operator. However, the measure $h_{\kappa}^2(x)dx$ is no longer
invariant under the usual translation. In \cite{TX1} the translation
operator related to Dunkl transform then is defined on the Dunkl
transform side by
$$ {\widehat{\tau_y f}}(x)=E(y, -ix) \hat{f}(x)$$
for all $x\in \mathbb{R}^{N}.$

When the function $f$ is in the Schwartz class, the above equality
holds pointwise. As an operator on $L^2(\mathbb R^N,h^2_{\kappa}),$
$\tau_y$ is bounded. However, it is not at all clear whether the
translation operator can be defined for $L^p$ functions with
$p\not=2.$ Even the $L^p$ boundedness of $\tau_y$ on the dense
subspace of Schwartz class for $p\not=2$ is still open. So far, an
explicit formula for $\tau_y$ is known only in the cases: when $f$
is a radial function or when $G={\mathbb{Z}}^{N}_2.$ To be precise,
in \cite{R4} it was proved that if $f$ is a radial Schwartz function
and $f(x)=f_0(\|x\|),$ then
$$ \tau_y f(x)=V_{\kappa}[f_0((\|x\|^2+\|y\|^2-2\|x\|\|y\|\langle x',  {\color{red}\cdot}\rangle)^{\frac{1}{2}})](y'),$$
where $x'=\frac{x}{\|x\|}$ for non-zero $x\in \mathbb R^N.$

For $f, g\in L^2(\mathbb{R}^{N},h^2_{\kappa}),$ their convolution
can be defined in terms of the translation operator by
$$f{\ast}_{\kappa} g(x)=\int_{\mathbb{R}^{N}}f(y)\tau_{x}g^{\vee}(y)h^2_{\kappa}(y)dy,$$
where $g^{\vee}(y)=g(-y).$

In \cite{D2}, Dunkl also introduced so-called Dunkl operators, that
is a family of first order differential-difference operators which
play the role similar to the usual partial differentiation for the
reflection group. To be precise, denote the standard inner product
in the Euclidean space $\mathbb{R}^{N}$  by $\langle
x,y\rangle=\sum\limits_{j=1}^Nx_jy_j$ and the corresponding norm by
$\|x\|=\{\sum\limits_{j=1}^N|x_j|^2\}^\frac{1}{2}$. Let
$B(x,r):=\{y\in \mathbb{R}^{N}:\|x-y\|<r\}$ stand for the ball with
center  $ x\in\mathbb{R}^{N}$  and radius $r>0$. Let $R$ be a root
system in $\mathbb{R}^{N}$ normalized so that
$\langle\alpha,\alpha\rangle=2$ for $\alpha\in R$ with $R_+$ a
fixed positive subsystem, and $G$ be the finite reflection group
generated by the reflections $\sigma_\alpha$ ($\alpha\in R$), where
$\sigma_{\alpha}x=x-\langle\alpha,x\rangle\alpha$ for
$x\in\mathbb{R}^{N}$. We set
$$V(x,y,r):=\max\{\omega(B(x,r)),\omega(B(y,r))\}$$ and denote
$d(x,y)=\min\limits_{\sigma\in G}\|x-\sigma(y)\|$ by the distance (the so-called Dunkl "metric") between two G-orbits $\mathcal{O}(x)$ and $\mathcal{O}(y)$.
Obviously, $d(x,y)\leqslant \|x-y\|, d(x,y)=d(y,x)$ and $d(x,y)\leqslant
d(x,z)+d(z,y)$ for all $x,y,z\in \R^N.$ Moreover,
$\omega(B(x,r))\sim \omega(B(y,r))$ when $d(x,y)\sim r,$ and
$\omega(B(x,r))\leqslant \omega(B_d(x,r))\leqslant
|G|\omega(B(x,r)),$ where $B_d(x,r):=\{y\in
\mathbb{R}^{N}:d(x,y)<r\}.$

A \textsl{ multiplicity function} $\kappa$ defined on $R$ (invariant
under $G$) is fixed $\geq0$ throughout this paper. Let
\begin{eqnarray*}
    d\omega(x)=\prod_{\alpha\in R}|\langle\alpha,x\rangle|^{\kappa(\alpha)}dx
\end{eqnarray*}
be the associated measure in $\mathbb{R}^{N},$ (see \cite{ADH}),
where, here and subsequently, $dx$ stands for the Lebesgue measure
in $\mathbb{R}^{N}$. We denote by $\textbf{N} = N
+\sum\limits_{\alpha\in R}\kappa(\alpha)$ the homogeneous dimension
of the system. Clearly,
$$\omega(B(tx, tr)) = t^\textbf{N}\omega(B(x, r))$$
and
$$\int_{\mathbb{R}^{N}}f(x)d\omega(x)=\int_{\mathbb{R}^{N}}\frac{1}{t^\textbf{N}}f(\frac{x}{t})d\omega(x)$$
for $f \in L^1(\R^N,\omega)$, $t>0$.

Observe that for $x\in \R^N$ and $r>0,$
$$\omega(B(x,r))\sim r^N \prod_{\alpha\in R}\big(|\langle\alpha,x\rangle|+r\big)^{\kappa(\alpha)}$$
and hence, $\inf\limits_{x\in \R^N}\omega(B(x,1))\geqslant c>0.$

Moreover,
\begin{equation}\label{rd1.1}
C^{-1}\bigg(\frac{r_2}{r_1}\bigg)^{N}\leqslant
\frac{w(B(x,r_2))}{w(B(x,r_1))}\leqslant C\bigg(\frac{r_2}{r_1}\bigg)^{\textbf N} \qquad\text{for}\quad
0<r_1<r_2.
\end{equation}
This implies that $d\omega(x)$ satisfies the doubling and reverse
doubling properties, that is, there is a constant $C > 0$ such that
for all $ x\in \mathbb{R}^{N}, r>0$ and $\lambda \geqslant 1,$
\begin{eqnarray}\label{rd}
C^{-1}\lambda^{N}\omega(B(x, r)) \leqslant \omega(B(x, \lambda
r))\leqslant C\lambda^{\textbf N}\omega(B(x, r)).
\end{eqnarray}

The Dunkl operators $T_j$ are defined by
$$T_jf(x)=\partial_j f(x) + \sum\limits_{\alpha\in R^+}\frac{\kappa(\alpha)}{2}\langle\alpha, e_j\rangle\frac{f(x)-f(\sigma_\alpha(x))}{\langle \alpha, x\rangle},$$
where $e_1, \cdots, e_N$ are the standard unit vectors of $\mathbb
R^N.$

The Dunkl Laplacian related to $R$ and $\kappa$ is the operator
$\triangle=\sum\limits^{N}_{j=1} T^2_j,$ which is equivalent to

$$\triangle f(x)=\triangle_{\mathbb R^N} f(x) + \sum\limits_{\alpha \in R}\kappa(\alpha)\delta_{\alpha}f(x),$$
where $
\delta_{\alpha}f(x)=\frac{\partial_{\alpha}f(x)}{\langle\alpha,
x\rangle}
-\frac{f(x)-f(\sigma_{\alpha}(x))}{\langle\alpha,
x\rangle^2}.$

The operator $\triangle$ is essentially self-adjoint on
$L^2(\R^N,\omega)$(\cite{ADH}) and generates the heat semigroup
$$H_tf(x)=e^{t\triangle}f(x)=\int_{\mathbb R^N}h_t(x,y)f(y)d\omega(y),$$
where the heat kernel $h_t(x,y)$ is a $C^\infty$ function for all
$t>0, x, y\in \mathbb R^N$ and satisfies $h_t(x, y)=h_t(y, x)>0$ and
$\int_{\mathbb R^N} h_t(x, y) d\omega(y)=1.$

The Poisson semigroup is given by
$$P_tf(x)=\pi^{-\frac{1}{2}}\int^\infty_0 e^{-u}exp(\frac{t^2}{4u}\triangle)f(x)\frac{du}{u^{\frac{1}{2}}}$$
and $u(x, t)=P_tf(x),$ so-called the Dunkl Poisson integral, solves
the boundary value problem
$$\begin{cases}
       & (\partial_t^2 + \triangle_x)u(x,t)=0,\\[3pt]
       & u(x,0)=f(x)
    \end{cases}
$$
in the half-space $\mathbb R^N_+,$ see \cite{ADH}.

All these tools, the Dunkl transform, Laplacian and Poisson integral
together with the Dunkl translation and convolution operators,
opened the door for developing the harmonic analysis related to the
Dunkl setting, which includes the Littlewood-Paley theory, Hardy
spaces and singular integral operators. To be more precise, in
\cite{ADH}, the Littlewood-Paley theory was established and the
Hardy space $H^1(\R^N)$ was characterized by the area integrals,
maximal fuction and the Riesz transforms, see also \cite{ABDH}. The atomic decomposition of $H^1(\R^N)$ was provided in \cite{DH1}. The boundedness of
singular integral convolution operators and the H\"ormander
multipliers was given by \cite{DH2} and \cite{DH3}, respectively.
See \cite{AAS, AGS, AH, AS, A, DW, DX, DK, D1, D3, Jeu, LL,
R1, R2, R3, R5, RV, S, TX2} for other topics related to the Dunkl
setting.

Some natural questions arise: Are there the Calder\'on-Zygmund
singular integral operator theory, particularly, the $T1$ theorem,
the Littlewood-Paley theory and Hardy spaces $H^p(\R^N)$ for $p\leq
1$ in the Dunkl setting?

In this paper, we answer these questions. To reach our goal, it is
well known that in the $\R^N$ and, even more general setting, spaces
of homogeneous type in the sense of Coifman and Weiss, the
Littlewood-Paley theory gives a uniform treatment of function spaces
and provides a powerful tool for providing the boundedness for
Calder\'on-Zygmund  singular integral operators on these function
spaces. The key point is that before developing the classical
Littlewood-Paley theory, one should establish the Calder\'on
reproducing formula. To see this, let $\psi$ be a Schwartz function
satisfying the following conditions:
\begin{itemize}
    \item[(i)] $\operatorname{supp} \widehat{\psi} \subset\left\{\xi \in {\R^N}: 1 / 2 \leqslant|\xi| \leqslant 2\right\}$;
    \item[(ii)] $\left|\widehat{\psi}(\xi)\right| \geqslant C>0 \text { for all } 3 / 5 \leqslant|\xi| \leqslant 5 / 3$.
\end{itemize}
Calder\'{o}n in \cite{C} provided the following formula:
\begin{equation}\label{crf1}
f=\sum\limits_{k=-\infty}^{\infty} \phi_{k} * \psi_{k} * f,
\end{equation}
where $\phi$ satisfies the properties similar to $\psi,
\psi_k(x)=2^{kN}\psi(2^kx),$ the series converges in
$L^2({\R^N}).$ We point out that the Fourier transform is the
main tool to show the above formula.

This formula is called to be the Calder\'on reproducing formula.
Based on this formula, one can define the following square function

$$S(f)(x)=\left\{\sum\limits_{k=-\infty}^{\infty}|\psi_k\ast f(x)|^2\right\}^{1 /2}.$$
The Littlewood-Paley $L^p, 1<p<\infty,$ theory is given by the
following estimates:
$$\|S(f)\|_p\sim \|f\|_p.$$
It is well known that the classical Hardy space $H^p$ can be
characterized by such a square function. However, in applications,
for example, to get the atomic decomposition and the dual space of
$H^p$, the discrete square function is a powerful version. To this
end, one needs the following discrete Calder\'on reproducing
formula:
\begin{equation}\label{1.4}
f(x)=\sum\limits_{k=-\infty}^{\infty}\sum\limits_{Q}|Q|
\phi_{Q}(x-x_Q) \psi_{Q}\ast f(x_Q),
\end{equation}
where $\phi$ and $\psi$ are same as in the formula \eqref{crf1}, $Q$
are all dyadic cubes with the side length $2^{-k}$,
$\phi_{Q}=\phi_k, \psi_{Q}=\psi_k$, $x_Q$ are all dyadic points at
the left lower corner of $Q$ and the series converges in
$L^p({\R^N}), 1<p<\infty$,
$S_{\infty}\left({\R^N}\right),$ Schwartz functions with all
moment cancellation conditions, and
$S^{\prime}\left({\R^N}\right) /
\mathcal{P}\left({\R^N}\right),$ the space of Schwartz
distributions modulo the space of all polynomials. Again, the proof
of  this discrete Calder\'on reproducing formula depends on the
compact support of the Fourier transform.

This discrete Calder\'on reproducing formula leads the following
discrete Littlewood-Paley square function
\begin{equation}\label{1.5}
S_d(f)(x)=\left\{\sum\limits_{k=-\infty}^{\infty}\sum\limits_{Q}|\psi_{Q}\ast
f(x_Q)|^2\chi_Q(x) \right\}^{1 /2}.
\end{equation}
See \cite{FJ} for more details.

To generalize the classical Calder\'on-Zygmund singular integral
operator theory to more general setting, Coifman and Weiss
introduced spaces of homogeneous type, see \cite{CW}. As we point
out that the classical harmonic analysis rely on the Fourier
transform. Needless to say, the Fourier transform does not exist on
spaces of homogeneous type in the sense of Coifman and Weiss.
However, David, Journ$\acute{\text{e}}$ and Semmes in \cite{DJS}
developed the Littlewood-Paley theory for spaces of homogeneous
type. The key tools they used are Coifman's approximation to the
identity and Coifman's decomposition of the identity on $L^2.$ To be
precise, spaces they considered are $(X, \rho, \mu)$ with quasi-metric
$\rho$ satisfying some regularity properties and measure $\mu$
satisfying the condition $\mu(B(x,r))\sim r^n, n>0,$ which is much
stronger than the doubling condition. Applying Coifman's
approximation to the identity, that is, $\{S_k\}_{k=-\infty}^\infty$
is an approximation to the identity on $L^2(X, \mu)$ where
$\{S_k(x,y)\}_{k=1}^\infty,$ the kernels of
$\{S_k\}_{k=-\infty}^\infty,$ satisfy certain size and regularity
conditions, and together using Coifman's decomposition of the
identity on $L^2$, namely
\begin{equation}\label{1. 6}
\begin{aligned}
I &=\sum\limits_{k=-\infty}^{\infty} D_{k}
=\left(\sum\limits_{k=-\infty}^{\infty} D_{k}\right)\left(\sum\limits_{j=-\infty}^{\infty} D_{j}\right) \\
&=\sum\limits_{|k-j| \leqslant N} D_{k} D_{j}+\sum\limits_{|k-j|>N}
D_{k} D_{j} =T_{N}+R_{N},
\end{aligned}
\end{equation}
where $D_k=S_{k+1}-S_k,$ David, Journ$\acute{\text{e}}$ and Semmes
proved that $R_N$ is bounded on $L^p(X)$, $1 <p< \infty$, with the
operator norm less than 1 if $N$ is large enough and thus,
$T_{N}^{-1},$ the inverse of $T_N,$ is bounded on $L^p(X),
1<p<\infty.$ Denote $D_{k}^N = \sum\limits_{|j|\leqslant N} D_{k+j}$
for $k\in \mathbb Z$, they obtained the following Calder\'{o}n-type
reproducing formulae:
\begin{equation}\label{1.7}
f=\sum\limits_{k=-\infty}^{\infty} T_{N}^{-1}D_{k}^{N} D_{k}(f)
=\sum\limits_{k=-\infty}^{\infty} D_{k}D_{k}^{N} T_{N}^{-1}(f)
\end{equation}
where the series converge in $L^p(X), 1 <p<\infty $. Using this
formula, they were able to obtain the Littlewood-Paley theory for
the space $L^p(X)$: There exists a constant $C > 0$ such that for
all $f \in L^p(X), 1 <p< \infty$,
\begin{equation*}
C^{-1}\|f\|_{L^p(X)} \leqslant \Big\| \Big\{
\sum\limits_{k={-\infty}}^{\infty} |D_k(f)|^2 \Big\}^\frac{1}{2}
\Big\|_{L^p(X)} \leqslant C\|f\|_{L^p(X)}.
\end{equation*}
Applying the above Littlewood-Paley estimates, they showed the
remarkable $T1$ theorem on such spaces of homogeneous type.

However, the presence of the operator $T_{N}^{-1}$ prevents one from
developing Littlewood-Paley characterizations for other spaces, such
as the Hardy spaces, by simply applying the Calder\'on-type formula
in \eqref{crf1}. To obtain the Calder\'on reproducing formula
similar to one given in \eqref{crf1}, in \cite{HS}, the space of
test functions, $\mathcal{M}(X),$ a suitable analogue of
$S_{\infty}\left(\mathbb{R}^{d}\right)$ was introduced. The key idea
is to show that $T_{N}^{-1}$ is a Calder\'on-Zygmund operator whose
kernel has some additional second order smoothness and is bounded on
$\mathcal{M}(X).$ Since then  $T_{N}^{-1}D_k^N$ satisfy nice
estimates sufficiently similar to those satisfied by $D_k$ itself.
Precisely, the Calder\'{o}n reproducing formula provided in
\cite{HS} is given by the following: There exist families of
operators $\left\{\widetilde{D}_{k}\right\}_{k=-\infty}^{\infty}$
and $\left\{\bar{D}_{k}\right\}_{k=-\infty}^{\infty}$ such that
\begin{equation}\label{crf2}
f=\sum\limits_{k=-\infty}^{\infty} \widetilde{D}_{k} D_{k}(f)
=\sum\limits_{k=-\infty}^{\infty} D_{k} \bar{D}_{k}(f),
\end{equation}
where the series converge in the norms of the space $L^p(X), 1 <p<
\infty$, the space $\mathcal{M}(X)$ and the dual space
$(\mathcal{M}(X) )'$, respectively.

Note that the formula in \eqref{crf2} is similar to that in
\eqref{crf1}. In \cite{H2}, applying Coifman's decomposition of the
identity together with the boundedness of Calder\'on-Zygmund
operators satisfying the second order smoothness on the test
function space, the following discrete Calder\'on reproducing
formula was established:
\begin{equation}\label{crf3}
f=\sum\limits_{k=-\infty}^{\infty}\sum\limits_{Q}\mu(Q)
{\widetilde {\widetilde{D}}}_{k}(x, x_Q) D_{k}(f)(x_Q)
=\sum\limits_{k=-\infty}^{\infty}\sum\limits_{Q}\mu(Q) D_{k}(x,x_Q)
{\bar{\bar{D}}}_{k}(f)(x_Q),
\end{equation}
where $Q$ are dyadic cubes in the sense of Christ, $x_Q$ are any
fixed point in $Q,$ and the series converge in the norms of the
space $L^p(X), 1 <p< \infty$, the space $\mathcal{M}(X)$ and the
dual space $(\mathcal{M}(X) )'$, respectively.

We remark that the Calder\'on-Zygmund operator theory plays a
crucial role for the proofs of the Calder\'on reproducing formulae
in \eqref{crf2} and \eqref{crf3}. Using formulae \eqref{crf3}, the
Hardy space $H^p$ theory was established, which includes the
Littlewood-Paley characterization of $H^p,$ atomic decomposition and
dual space of $H^p$, and the $T1$ theorem of the boundedness of
singular integrals on these spaces. See \cite{H1, H2, HS}.

In \cite{HMY}, motivated by the work of Nagel and Stein on the
several complex variables in \cite{NS}, they considered spaces of
homogeneous type $(X, \rho, \mu),$ where the quasi-metric $\rho$ satisfies
some regularity properties and the measure $\mu$ satisfies the
doubling condition and the reverse doubling condition, that is,
there are constants $\kappa >0$ and $c \in (0, 1]$ such that
\begin{eqnarray}\label{reverse doubling condition}
c \lambda^\kappa \mu ( B (x, r) ) \leqslant \mu ( B (x, \lambda r) )
\end{eqnarray}
for all $x \in X$, $0 < r < \displaystyle\sup_{x, y \in X} \rho (x, y)
/ 2$ and $1 \leqslant \lambda < \displaystyle\sup_{x, y \in X} \rho (x,
y) / 2r. $

Applying Coifman's approximation to the identity and Coifman's
decomposition of the identity together with some modified test
functions, they provided the discrete Calder\'on reproducing
formulae as in \eqref{crf3} and established the Hardy space theory
in this setting. We would like to point out that the reverse
doubling condition of $\mu$ and the Calder\'on-Zygmund operator
theory play a crucial role for the boundedness of $T_N^{-1}$ on the
space of test functions. See \cite{HMY} for the Littlewood-Paley
characterization of the Hardy space and its applications in this
setting.

Notice that the regularity of the quasi-metric $\rho$ is key fact for
constructing Coifman's approximation to the identity and the reverse
doubling condition is crucial to get the boundedness of $T_N^{-1}$
on the space of test functions. To develop the Littlewood-Paley
theory on space of homogeneous $(X, \rho, \mu),$ where the quasi-metric
$\rho$ has no any regularity and the measure $\mu$ satisfies the
doubling condition only, a new approach is required.

Adapting the developed randomized dyadic structure on space of
homogeneous type $(X, \rho, \mu)$ where $\rho$ is the quasi-metric without
any regularity and the measure $\mu$ satisfies the doubling
condition only, in \cite{AuH}, Auscher and Hyt\"onen builded a
remarkable orthonormal basis of H\"older-continuos wavelet with
exponential decay. Using this wavelet basis they provided a
universal Calder\'on reproducing formula to study and develop
function spaces theory and singular integrals. More precisely, they
discussed $L^p, 1<p<\infty$ spaces, BMO and gave a proof of the $T1$
theorem in this general setting. See more details in \cite{AuH}.
Applying Auscher-Hyt\"onen's orthonormal basis, the Hardy space and
the product Hardy space were developed in \cite{HHL, HLW}.\\

\subsection{Statement of main results }
\ \\

The main results of this paper are (i) the
Dunkl-Calder\'on-Zygmund singular integral operators and the
$T1$ theorem; (ii) the Littlewood-Paley theory and the
Hardy spaces; (iii) the boundedness for the
Dunkl-Calder\'on-Zygmund singular integral operators, particularly,
the Dunkl-Rieze transforms, on the Dunkl-Hardy space.

\subsubsection{ Dunkl-Calder\'on-Zygmund singular integral operator and $T1$ theorem}
\ \\

As mentioned above, we can consider the Dunkl setting, $(\Bbb R^N,
\|\cdot\|, \omega),$ as a space of homogeneous type in the sense of
Coifman and Wiess. Note that the measure $\omega$ satisfies the
doubling and the reverse doubling properties. Let $C^\eta_0(\Bbb
R^N), \eta>0,$ denote the space of continuous functions $f$ with compact
support and
$$\|f\|_{\boldsymbol{\eta}}:=\sup\limits_{x\ne y} \frac{|f(x)-f(y)|}{\|x-y\|^{\boldsymbol{\eta}}}<\infty. $$
The classical Calder\'on-Zygmund singular integral operator in
$(\Bbb R^N,\|\cdot\|, \omega)$ is defined by the following
\begin{definition}\label{SIO}
    An operator $T:C_0^\eta(\mathbb{R}^N)\rightarrow(C_0^\eta(\mathbb{R}^N))'$ with $\eta>0,$ is said to be a Calder\'on-Zygmund singular integral operator if $K(x,y),$ the kernel of $T,$ satisfies the following estimates: for some $0<\varepsilon\leqslant 1,$
    \begin{eqnarray}\label{size of C-Z-S-I-O}
    |K(x, y)|
    \leqslant {{C}\over {\omega(B(x,\|x- y\|))}}
    \end{eqnarray}
    for all $x\not= y$;
    \begin{eqnarray}\label{x smooth of C-Z-S-I-O}
    |  K(x, y) - K(x', y) |
    \leqslant \Big({\|x-x'\|\over \|x- y\|}\Big)^\varepsilon {C\over \omega(B(x,\|x- y\|))}
    \end{eqnarray}
    for $\|x- x'\|\leqslant \frac{1}{2} \|x- y\|$;
    \begin{eqnarray}\label{y smooth of C-Z-S-I-O}
    |  K(x, y) - K(x, y') |
    \leqslant \Big({\|y-y'\|\over \|x- y\|}\Big)^\varepsilon {C\over \omega(B(x,\|x- y\|))}
    \end{eqnarray}
    for $\|y- y'\|\leqslant \frac{1}{2} \|x- y\|.$
    Moreover,
    $$\langle T(f),g\rangle=\int_{\R^N}\int_{\R^N} K(x,y)f(x)g(y)d\omega(x)d\omega(y)$$
    for $\supp f\cap \supp g=\emptyset.$
\end{definition}
See \cite{HMY} for more details.

However, the following motivation leads to consider a new class of
the Calder\'on-Zygmund singular integral operators in the Dunkl
setting. Recall that $p_t(x,y)$ is the Dunkl-Poisson kernel.
Applying the size and smoothness conditions of $p_t(x,y)$(see
\cite{DH1}) implies that $K(x,y)=\int_0^\infty p_t(x,y)\frac{dt}{t}$
satisfies the following estimates( see the Proposition \ref{pr300}
in {\bf Section 2}): for any $0<\varepsilon<1,$

\begin{equation*}
|K(x,y)|\lesssim
\frac1{\omega(B(x,d(x,y)))}\boldsymbol{\Big(\frac{d(x,y)}{\|x-y\|}\Big)^\varepsilon}
\end{equation*}
for all $x\not= y;$
\begin{equation*}
|K(x,y)-K(x,y')|\lesssim
\Big(\frac{\|y-y'\|}{d(x,y)}\Big)^\varepsilon\frac{1}{\omega(B(x,d(x,y)))}\boldsymbol{\Big(\frac{d(x,y)}{\|x-y\|}\Big)^\varepsilon}
\end{equation*}
for $\|y-y'\|\leqslant  d(x,y)/2;$
\begin{equation*}
|K(x',y)-K(x,y)|\lesssim
\Big(\frac{\|x-x'\|}{d(x,y)}\Big)^\varepsilon\frac1{\omega(B(x,d(x,y)))}\boldsymbol{\Big(\frac{d(x,y)}{\|x-y\|}\Big)^\varepsilon}
\end{equation*}
for $\|x-x'\|\leqslant  d(x,y)/2.$

This motivation leas to introduce the following
Dunkl-Calder\'on-Zygmund singular integral operators.
\begin{definition}\label{sio}
    An operator $T: C_0^\eta(\mathbb{R}^N)\rightarrow(C_0^\eta(\mathbb{R}^N))'$ with $\eta>0,$ is said to be a Dunkl-Calder\'on-Zygmund singular integral operator if $K(x,y),$ the kernel of $T$,
    satisfies the following estimates: for some $0<\varepsilon\leqslant 1,$
    \begin{equation}\label{si}
    |K(x,y)|\lesssim \frac1{\omega(B(x,d(x,y)))}
    \boldsymbol{\Big(\frac{d(x,y)}{\|x-y\|}\Big)^\varepsilon}
    \end{equation}
    for all $x\not= y;$
    \begin{equation}\label{smooth y3}
    |K(x,y)-K(x,y')|\lesssim \boldsymbol{\Big(\frac{\|y-y'\|}{\|x-y\|}\Big)^\varepsilon}\frac{1}{\omega(B(x,d(x,y)))}
    \end{equation}
    for $\|y-y'\|\leqslant  d(x,y)/2;$
    \begin{equation}\label{smooth x3}
    |K(x',y)-K(x,y)|\lesssim \boldsymbol{\Big(\frac{\|x-x'\|}{\|x-y\|}\Big)^\varepsilon}\frac1{\omega(B(x,d(x,y)))}
    \end{equation}
    for $\|x-x'\|\leqslant  d(x,y)/2.$

    Moreover,
    $$\langle T(f),g\rangle=\int_{\R^N}\int_{\R^N} K(x,y)f(x)g(y)d\omega(x)d\omega(y)$$
    for $\verb"supp" f\cap \verb"supp" g=\emptyset.$

    A Dunkl-Calder\'on-Zygmund singular integral operator is said to be the Dunkl-Calder\'on-Zygmund operator if it extends a bounded operator on $L^2(\R^N).$
\end{definition}
We remark that the size and regularity conditions of the
Dunkl-Calder\'on-Zygmund singular integral operator are much weaker
than the classical Calder\'on-Zygmund singular integral operators
given in space of homogeneous type in the sense of Coifman and
Weiss. Indeed, by the reverse doubling condition on the measure
$\omega$ in \eqref{rd}, $\omega(B(x,\|x- y\|))=\omega(B(x,\frac{\|x-
y\|}{d(x,y)}\cdot d(x,y))\geqslant C \big(\frac{\|x-
y\|}{d(x,y)}\big)^{N}\omega(B(x,d(x,y))).$ Thus,
\begin{eqnarray}\label{size DCZ}
\frac{1}{\omega(B(x,\|x- y\|))}&\lesssim& \Big(\frac{d(x,y)}{\|x-
y\|}\Big)^{N}\frac{1}{\omega(B(x,d(x,y)))}\\ &\lesssim&
\Big(\frac{d(x,y)}{\|x-y\|}\Big)^\varepsilon\frac1{\omega(B(x,d(x,y)))}
\end{eqnarray}
and if $\|x-x'\|\leqslant \frac{1}{2}d(x,y)\leqslant
\frac{1}{2}\|x-y\|,$ then,

\begin{equation}\label{sm of DCZ}
\begin{aligned}
\Big({\|x-x'\|\over \|x-
y\|}\Big)^\varepsilon\frac{1}{\omega(B(x,\|x-y\|))}
&\leqslant C \Big(\frac{\|x-x'\|}{\|x-y\|}\Big)^\varepsilon \frac1{\omega(B(x,d(x,y)))}\Big(\frac{d(x,y)}{\|x- y\|}\Big)^{N}\\
&\leqslant C\Big(\frac{\|x-x'\|}{\|x-y\|}\Big)^\varepsilon
\frac1{\omega(B(x,d(x,y)))}.
\end{aligned}
\end{equation}

Further, if $K(x,y)$ satisfies the above size condition \eqref{si},
then $K(x,y)$ is locally integrable in $\{\R^N\times \R^N: x\not = y
\}.$ Indeed, for any fixed $x\in \R^N$ and $0<\delta<R<\infty,$ by
the doubling properties of the measure $\omega,$ (see the details in
{\bf Section 2})
$$\int_{\delta<\|x-y\|<R} |K(x,y)| d\omega(y)\leqslant C \frac1{\delta}\int_{d(x,y)<R} \frac{d(x,y)}{\omega(B(x,d(x,y))}d\omega(y)\leqslant C\frac{R}{\delta}<\infty.$$

It is well known that if $T$ is the classical Calder\'on-Zygmund
operator on $\R^N,$ then $T$ is bounded on $L^p(\R^N), 1<p<\infty,$
from $H^1(\R^N)$ to $L^1(\R^N)$ and from $L^\infty(\R^N)$ to
$BMO(\R^N).$ These results still hold for the
Dunkl-Calder\'on-Zygmund operators.
\begin{theorem}\label{th1.1}
    Suppose that $T$ is a Dunkl-Calder\'on-Zygmund operator. Then $T$ is bounded on $L^p(\R^N,\omega), 1<p<\infty,$ from $H_d^1(\R^N,\omega)$ to $L^1(\R^N,\omega)$
    and from $L^\infty(\R^N,\omega)$ to $BMO_d(\R^N,\omega).$
\end{theorem}
Here $H_d^1(\R^N,\omega)$ is the Dunkl-Hardy space introduced in
\cite {ADH} and $BMO_d(\R^N,\omega)$ is the classical $BMO$ function
defined in $(\Bbb R^N, \|\cdot\|, \omega),$ as a space of
homogeneous type in the sense of Coifman and Wiess.

If $T$ is a Dunkl-convolution operator, the Dunkl transform is the
main tool for providing the $L^2(\R^N, \omega)$ boundedness of $T.$
However, beyond the convolution operators, it becomes indispensable
to obtain a criterion for $L^2(\R^N,\omega)$ continuity. As Meyer
pointed out that without such a criterion of the $L^2(\R^N,
\omega),$ the theory of the Calder\'on-Zygmund singular integral
operators collapses like a house of cards. In the classical case,
one such criterion is the remarkable $T1$ theorem of David and
Journ\'e. In Dunkl setting, we have a similar $T1$ theorem.

Before describing the $T1$ theorem, we need to extend the definition
of the Dunkl-Calder\'on-Zygmund operators to bounded functions in
$C^\eta(\R^N).$ The idea for doing this is to define $T(b)$ for $b\in
C^\eta(\R^N),$ as a distribution on $C^\eta_{0,0}(\R^N)=\{f: f\in
C^\eta_0, \int_{\R^N} f(x)d\omega(x)=0\}.$ To this end, given $f\in
C^\eta_{0,0}(\R^N)$ with the support contained in the ball $B(x_0,
R)$ for some $x_0\in \R^N$ and $R>0.$ Let $\eta (x)= 1$ for $x\in
B_d(x_0, 2R)$ and $\eta(x)= 0$ for $x\in \big(B_d(x_0, 4R )\big)^c.$
Write $b=\eta b + (1-\eta)b$ and formlly, $\langle Tb,
f\rangle=\langle T(b\eta), f\rangle +\langle T(1-\eta)b, f\rangle.$
The first term $\langle Tb\eta, f\rangle$ is well defined. By the
cancellation condition of $f$ and the fact that if $x\in \supp f$
and $y\in \supp \eta,$ we can write
\begin{align*}
\langle T(1-\eta)b, f\rangle
&=\int_{\R^N}\int_{\R^N} K(x,y)(1-\eta(y))b(y)f(x)d\omega(y)d\omega(x)\\
&=\int_{\R^N}\int_{\R^N}
[K(x,y)-K(x_0,y)](1-\eta(y))b(y)f(x)d\omega(y) d\omega(x).
\end{align*}
Observe that when $x\in B(x_0, R)$ and $y\notin B_d(x_0, 2R),
\|x-x_0\|\leqslant R\leqslant \frac{1}{2}d(x_0,y).$ Applying the
smoothness condition of $K(x,y)$ implies that
\begin{align*}
&\int_{\R^N}\int_{\R^N} |K(x,y)-K(x_0,y)||f(x)|d\omega(y) d\omega(x)\\
&\qquad\leqslant C\int_{\|x-x_0\|\leqslant R}\int_{d(x_0,y)\geqslant 2R}\frac{1}{\omega(B(x_0,y))}\Big(\frac{\|x-x_0\|}{\|x_0-y\|}\Big)^\varepsilon |f(x)|d\omega(y) d\omega(x)\\
&\qquad\leqslant C\int_{\|x-x_0\|\leqslant R}\int_{d(x_0,y)\geqslant 2R}\frac{1}{\omega(B(x_0,y))}\Big(\frac{\|x-x_0\|}{d(x_0,y)}\Big)^\varepsilon |f(x)|d\omega(y) d\omega(x)\\
&\qquad \leqslant C\|f\|_1
\end{align*}
and thus, $\langle Tb, f\rangle$ is well defined. It is easy to see
that this definition is independent of the choice of the function
$\eta.$ Therefore, $T(b)$ is a distribution in
$\Big(C^\eta_{0,0}\Big)^\prime.$ Now $T(1)=0$ means that for any
$f\in C^\eta_{0,0}, \langle T1,f\rangle=0.$ This is equivalent to
$T^*(f)=0$ for any $f\in C^\eta_{0,0}.$ The definition of $T^*(1)=0$
is defined similarly.

If considering  $(\Bbb R^N, \|\cdot\|, \omega)$ as a space of
homogeneous type in the sense of Coifman and Wiess, the classical
weak boundedness property of $T$ says that for $f,g\in
C_0^\eta(\R^N)$ with supp$f,g\subseteq B(x_0,t)$ and $ x_0\in \Bbb
R^n,t>0$, then there exists a constant $C$ such that
$$|\langle Tf,g\rangle|\leqslant C\omega(B(x_0,t))t^{2\eta}\|f\|_{\eta}\|g\|_{\eta}.$$
Indeed, if $f\in C_0^\eta(\R^N)$ with supp$f\subseteq B(x_0,t)$ and
$ x_0\in \Bbb R^n, t>0$, then $$\|f\|^2_2=\int_{B(x_0,t)}
|f(x)|^2d\omega(x)=\int_{B(x_0,t)} |f(x)-f(x_1)|^2d\omega,$$ where
$t<\|x_1-x_0\|\leqslant 2t$ and thus, $f(x_1)=0.$ This implies that
$\int_{B(x_0,t)}|f(x)-f(x_1)|^2d\omega(x)\leqslant C
t^{2\eta}\|f\|^2_\eta \omega(B(x_0,t))$ and we get
$$\|f\|_2\leqslant C\big(\omega(B(x_0,t))\big)^{1/2} t^\eta \|f\|_\eta.$$
Suppose that $T$ is bounded on $L^2(\R^N,\omega)$ and $f,g\in
C_0^\eta(\R^N)$ with supp$f, g\subseteq B(x_0,t)$ and $x_0\in \Bbb
R^n,t>0,$ then
$$|\langle  Tf,g\rangle|\leqslant \|T\|\|f\|_2\|g\|_2\leqslant C\omega(B(x_0,t))t^{2\eta}\|f\|_{\eta}\|g\|_{\eta}.$$
In our situation, the {\it weak boundedness property} (WBP) in the
Dunkl setting is defined by following
\begin{definition}\label{wbp}
    The Dunkl-Calder\'on-Zygmund singular integral operator $T$ with the distribution kernel $K(x,y)$ is said to have the weak boundedness property
    if there exist $\eta>0$ and $C<\infty$ such that
    $$|\langle K, f\rangle |\leqslant C\max\{\omega(B(x_0,r)), \omega(B(y_0,r))\}$$
    for all $f\in C^\eta_0(\mathbb R^N\times \mathbb R^N)$ with supp$(f)\subseteq B(x_0,r)\times B(y_0,r), x_0, y_0\in \mathbb R^N, \|f\|_\infty\leqslant 1, \|f(\cdot,y)\|_\eta\leqslant r^{-\eta}$ for all $y\in \R^N$ and $\|f(x,\cdot)\|_\eta\leqslant r^{-\eta}$ for all $x\in \mathbb R^N$.
\end{definition}
If the operator $T$ has the weak boundedness property, we denote by
$T\in WBP.$ It is easy to see that if the operator $T$ is bounded on
$L^2(\R^N),$ then $T$ satisfies the weak boundedness property
defined by the Definition \ref{wbp}.

It is well known that in the classical case, the almost orthogonal
estimates are fundamental tools for the proof of the $T1$ theorem.
The following result provides such a tool in the Dunkl setting.

\begin{theorem}\label{1.3}
    Suppose that $T$ is the Dunkl-Calder\'on-Zygmund singular integral operator with $T(1)=T^*(1)=0$ and  $T\in WBP.$ Then $T$ maps $\mathbb M(\beta, \gamma, r, x_0)$ to  $\widetilde{\mathbb M}(\beta, \gamma', r, x_0)$
    with $0<\beta<\varepsilon, 0<\gamma'<\gamma<\varepsilon,$ where $\varepsilon$ is the exponent of the regularity of the kernel of $T.$ Moreover, there exists a constant $C$ such that
    $$\|T(f)\|_{\widetilde{\mathbb M}(\beta, \gamma', r, x_0)}\leqslant C\|f\|_{{\mathbb M}(\beta, \gamma, r, x_0)}.$$
\end{theorem}

Here $\mathbb M(\beta, \gamma, r, x_0)$ and  $\widetilde{\mathbb
M}(\beta, \gamma, r, x_0)$ are defined by following

\begin{definition}\label{sm}
    A function $f(x)$ is said to be a smooth molecule for $0<\beta\leqslant 1, \gamma>0, r>0$ and some fixed $x_0\in \R^N,$ if $f(x)$ satisfies the following conditions:
    \begin{center}
        \begin{equation}\label{sm1.17}
        |f(x)|\leqslant C \frac{1}{V(x,x_0,r+d(x,x_0))}\Big(\frac{r}{r+\boldsymbol{\|x-x_0\|}}\Big)^\gamma;
        \end{equation}
        \begin{equation}\label{sm1.18}
        \begin{aligned}
        |f(x)-f(x')|&\leqslant C  \Big(\frac{\|x-x'\|}{r}\Big)^\beta
        \Big\{   \frac{1}{V(x,x_0,r+d(x,x_0))}\Big(\frac{r}{r+\boldsymbol{\|x-x_0\|}}\Big)^\gamma\\
        &\qquad\qquad + \frac{1}{V(x',x_0,r+d(x',x_0))}\Big(\frac{r}{r+\boldsymbol{\|x'-x_0\|}}\Big)^\gamma\Big\};
        \end{aligned}
        \end{equation}
        \begin{equation}\label{sm1.19}
        \int_{\R^N} f(x) d\omega(x)=0.
        \end{equation}
    \end{center}
    If $f(x)$ is a smooth molecule, we denote $f(x)$ by $f\in \mathbb M(\beta, \gamma, r, x_0)$ and define the norm of $f$ by
    $$\|f\|_{\mathbb M(\beta, \gamma, r, x_0)}:=\inf\{C: (\ref{sm1.17})-(\ref{sm1.18})\ {\rm hold}\}.$$
\end{definition}
Observe that $t\partial_t p_t(x,y)$ with $p_t,$ the Poisson kernel,
is the smooth molecule with $\beta, \gamma<1, r=t, x_0=y$ for fixed
$y,$ and it is also the smooth molecule for $x_0=x$ for $x$ is
fixed.

\begin{definition}\label{wsm}
    A function $f(x)$ is said to be a weak smooth molecule for $0<\beta\leqslant 1, \gamma>0, r>0$ and some fixed $x_0\in \R^N,$ if $f(x)$ satisfies the following conditions:
    \begin{equation}\label{wsm1.20}
    |f(x)|\leqslant C \frac{1}{V(x,x_0,r+d(x,x_0))}\Big(\frac{r}{r+\boldsymbol{d(x,x_0)}}\Big)^\gamma;
    \end{equation}
    \begin{equation}\label{wsm1.21}
         \begin{aligned}
    |f(x)-f(x')|&\leqslant C  \Big(\frac{\|x-x'\|}{r}\Big)^\beta
    \Big\{   \frac{1}{V(x,x_0,r+d(x,x_0))}\Big(\frac{r}{r+\boldsymbol{d(x,x_0)}}\Big)^\gamma\\
    &\qquad\qquad + \frac{1}{V(x',x_0,r+d(x',x_0))}\Big(\frac{r}{r+\boldsymbol{d(x',x_0)}}\Big)^\gamma\Big\};
         \end{aligned}
    \end{equation}
    \begin{equation}\label{f1}
    \int_{\R^N} f(x) d\omega(x)=0.
    \end{equation}
    If $f(x)$ is a weak smooth molecule, we denote $f(x)$ by $f\in \widetilde {\mathbb M}(\beta, \gamma, r, x_0)$ and define the norm of $f$ by
    $$\|f\|_{\widetilde{\mathbb M}(\beta, \gamma, r, x_0)}:=\inf\{C: (\ref{wsm1.20})- (\ref{wsm1.21})\ {\rm hold} \}.$$
\end{definition}
We remark that all $\|x-x_0\|$ and $\|x'-x_0\|$ in the smooth molecule are placed by  $d(x,x_0)$ and $d(x',x_0)$ in the weak smooth molecule, respectively.

Now we can state the $T1$ theorem for Dunkl-Calder\'on-Zygmund
singular integral operators by the following
\begin{theorem}\label{1.2}
    Suppose that $T$ is a Dunkl-Calder\'on-Zygmund singular integral operator. Then $T$ extends to a bounded operator on $L^2(\R^N, \omega)$ if and only if
    (a) $T(1)(x)\in BMO_d(\R^N,\omega);$ (b) $T^*(1)(x)\in BMO_d(\R^N,\omega);$ (c) $T\in WBP.$

\end{theorem}

We remark that the theory of Dunkl-Calder\'on-Zygmund singular integral operators plays a fundamental role for establishing the weak-type discrete Calder\'on reproducing formulae,
the Littlewood-Paley theory and the Hardy spaces in the Dunkl setting.\\

\subsubsection{{ Calder\'on reproducing formula and  Littlewood-Paley theory on $L^p, 1<p<\infty$}}
\ \\

We begin with the Calder\'on reproducing formula. Thanks \cite{ADH},
the authors provided such a formula as follows: for $f\in L^2(\R^N,
\omega),$
\begin{equation}\label{ccrf}
f(x)=\int_0^\infty \psi_{t}\ast q_{t}\ast f(x)\frac{dt}{t},
\end{equation}
where $q_tf=t\frac{\partial}{\partial t}p_tf$ with $p_t$ is the
Poisson kernel as mentioned above and $\psi\in C^\infty_0(B(0,1/4))$
is a radial function with $\int_{\R^N}\psi(x)d\omega(x)=0$.

Write

$$f(x)=\int_0^\infty \psi_t\ast q_t\ast f(x)\frac{dt}{t}=T_M(f)(x) + R_1(f)(x) + R_M(f)(x),$$
where
$$T_M(f)(x)=-\ln r\sum\limits_{j=-\infty}^\infty\sum\limits_{Q\in Q^j}w(Q) \psi_{j}(x,x_{Q})q_{j}\ast
f(x_{Q}),$$
$$R_1(f)(x)=-\sum\limits_{j=-\infty}^\infty\int_{r^{-j}}^{r^{-j+1}}
\Big[\psi_t\ast q_t\ast f(x)- \psi_{j}\ast q_{j}\ast f(x)\Big]\frac{dt}{t}$$
and
$$R_M(f)(x)=-\ln r\sum\limits_{j=-\infty}^\infty\sum\limits_{Q\in Q^j}\int_{Q}
\Big[\psi_{j}(x,y)q_{j}\ast f(y)-\psi_{j}(x,x_{Q})q_{j}\ast
f(x_{Q})\Big]d\omega(y),$$ where $\psi_j=\psi_{r^j}, q_{j}=q_{r^j},$ with $1<r\leqslant r_0$ for some
fixed $r_0, Q^j$ is the collection of all "$r$-dyadic cubes" $Q$
with the side length $r^{-M-j}$ for $M$ is some fixed large integer,
and $x_{Q}$ is any fixed point in the cube $Q.$

Applying the Coifman's decomposition of the identity on $L^2(\R^N)$
gives
$$I=T_M + R_1 + R_M.$$
We will show that $R_1$ and $R_M$ are Dunkl-Calder\'on-Zygmund
operators and the boundedness of $R_1$ and $R_M$ on $L^2(\R^N,
\omega)$ and $L^p(\R^N, \omega), 1<p<\infty,$ follows from the
Cotlar-Stein Lemma and Theorem \ref{th1.1}, respectively. Moreover,
$\|R_1+R_M\|_{p,p}< 1$ which implies that $(T_M)^{-1},$ the inverse
of $T_M,$ exists and is bounded on $L^p(\R^N, \omega), 1<p<\infty.$
This yields the following weak-type discrete Calder\'on reproducing
formula:

\begin{theorem}\label{1.5}
    If $f\in L^2(\R^N,\omega)\cap L^p(\R^N, \omega), 1<p<\infty,$ then there exists a function $h\in L^2(\R^N,\omega)\cap L^p(\R^N, \omega),$ such that $\|f\|_2\sim \|h\|_2$ and $\|f\|_p\sim \|h\|_p,$
    \begin{align*}
    f(x)=\sum\limits_{j=-\infty}^\infty\sum\limits_{Q\in Q^j}\omega(Q)
    \psi_{j}(x,x_{Q})q_{j}
    h(x_{Q}),
    \end{align*}
    where $q_jh(x)=q_j\ast h(x),$ the series converges in $L^2(\R^N,\omega)\cap L^p(\R^N, \omega)$ with $\psi_j=\psi_{r^j}, q_{j}=q_{r^j}, 1<r\leqslant r_0,$ for some fixed $r_0,$
    $Q^j$ is the collection of all "r-dyadic cubes" $Q$ with the side length $r^{-M-j}$ for $M$ is some fixed large integer, and $x_{Q}$ is any fixed point in the cube $Q.$
\end{theorem}

This weak-type discrete Calder\'on reproducing formula leads to the
following discrete Littlewood-Paley square function.
\begin{definition}\label{dsf}
    For $f\in L^2(\R^N, \omega), S(f),$ the \emph{discrete Littlewood--Paley square function} of $f,$ is defined by
    \begin{eqnarray}\label{square_function}
    S(f)(x):= \Big\{
    \sum\limits_{j=-\infty}^\infty\sum\limits_{Q\in Q^j}|q_{Q}f(x_{Q})|^2\chi_{Q}
    (x) \Big\}^{1/2},
    \end{eqnarray}
    where $q_Q=q_j$ when $Q\in Q^j$ and $\chi_{Q}(x)$ is the characteristical function of the cube $Q.$
\end{definition}

Applying the Dunkl-Calder\'on-Zygmund singular integral operator
theory, the Littlewood-Paley $L^p, 1<p<\infty,$ estimates are given
by the following
\begin{theorem}\label{1.6}
    There exist two constants $C$ and $C'$ such that for $L^p(\R^N, \omega), 1<p<\infty,$
    $$C'\|f\|_p\leqslant \|S(f)\|_p\leqslant C\|f\|_p.$$
\end{theorem}

\subsubsection{{Littlewood-Paley theory and Hardy space}}
\ \\

The above discrete Littlewood-Paley theory leads to introduce the
Dunkl-Hardy space norm for $f\in L^2{(\R^N, \omega)}$ as follows:
\begin{definition}\label{hpn}
    For $f\in L^2{(\R^N, \omega)}, \|f\|_{H_d^p},$ the Dunkl-Hardy space norm of $f,$ is defined by $\|f\|_{H_d^p}:=\|S(f)\|_p$ for $0<p\leqslant 1.$
\end{definition}

Using the Dunkl-Calder\'on-Zygmund singular integral operator
theory, the weak-type discrete Calder\'on reproducing formula for
$f\in L^2{(\R^N, \omega)}$ with respect to the Dunkl-Hardy space
norm is given by the following:

\begin{theorem}\label{1.7}
    If $f\in L^2(\R^N, \omega)$ with $\|f\|_{H_d^p}<\infty,$ for $\frac{{\bf N}}{{\bf N}+1}<p\leqslant 1,$ then there exists a function $h\in L^2(\R^N,\omega),$ such that $\|f\|_2\sim \|h\|_2, \|f\|_{H_d^p} \sim \|h\|_{H_d^p}$ and
    \begin{align*}
    f(x)=\sum\limits_{j=-\infty}^\infty\sum\limits_{Q\in Q^j}\omega(Q)
    \psi_{Q}(x,x_{Q})q_{Q} h(x_{Q}),
    \end{align*}
    where $\psi_Q =\psi_j, q_Q=q_j$ for $Q\in Q^j$ and the series converges in $L^2{(\R^N,\omega)}$ norm and the Dunkl-Hardy space norm.
\end{theorem}
We remark that the Dunkl-Calder\'on-Zygmund singular integral
operator theory plays a crucial role to obtain the sharp range
$\frac{{\bf N}}{{\bf N}+1}<p\leqslant 1$ which is the same as the
classical case.

Applying the above Theorem \ref{1.7} implies the following duality
estimate which will be a key idea for developing the Dunkl-Hardy
space theory:
\begin{proposition}\label{pr1}
    For $f, g\in L^2{(\R^N, \omega)}$ and $\frac{{\bf N}}{{\bf N}+1}<p\leqslant 1,$ then there exists a constant $C$ such that
    $$|\langle f, g\rangle|\leqslant C \|f\|_{H_d^p}\|g\|_{CMO_d^p}.$$
\end{proposition}
Here for $L^2{(\R^N, \omega)}$ function, the norm of the
Dunkl-Carleson measure space $CMO_d^p(\R^N, \omega)$ is defined by

\begin{definition} \label{norm-CMO}
    Suppose that  $f\in L^2(\R^N, \omega).$
    The norm of $f\in {CMO}_d^p(\R^N, \omega)$ is defined by
    $$\|f\|_{CMO_d^p} := \sup\limits_{P}
    \Big\{ {1\over\omega(P)^{{2\over p} - 1}}
    \sum\limits_{ Q \subseteq P }\omega(Q)
    \big| \psi_{Q}f (x_Q)
    \big|^2 \Big\}^{1/2}<\infty$$
    for $0<p\leqslant 1$, where $P$ runs over all dyadic cubes and $\psi_{Q}=\psi_{j}$ when $Q\in Q^j$.
\end{definition}

The above Proposition \ref{pr1} means that each function $f\in
L^2(\R^N,\omega)$ with $\|f\|_{H_d^p}<\infty$ can be considered as a
continuous linear functional on $L^2{(\R^N, \omega)}\cap
CMO_d^p{(\R^N, \omega)},$ the subspace of $g\in L^2{(\R^N, \omega)}$
with the Dunkl-Carleson measure space norm
$\|g\|_{CMO_d^p}<\infty.$ Therefore, one can consider
$L^2{(\R^N, \omega)}\cap CMO_d^p{(\R^N, \omega)}$ as a new test
function space and define the Dunkl-Hardy space  $H_d^p$ as the
collection of some distributions on $L^2{(\R^N, \omega)}\cap
CMO_d^p{(\R^N, \omega)}$. More precisely, the Dunkl-Hardy space  is
defined by the following
\begin{definition} \label{hp}
    The Dunkl-Hardy space $H_d^p(\R^N, \omega), \frac{{\bf N}}{{\bf N}+1}<p\leqslant 1,$ is defined by the collection of all distributions $f\in (L^2{(\R^N, \omega)}\cap CMO_d^p{(\R^N, \omega)})^\prime$ such that $$f(x)=\sum\limits_{j=-\infty}^\infty\sum\limits_{Q\in Q^j}\omega(Q)
    \lambda_{Q}\psi_{Q}(x,x_{Q})$$
    with $\|\Big\{
    \sum\limits_{j=-\infty}^\infty\sum\limits_{Q\in Q^{j}}|\lambda_{Q}|^2\chi_{Q}
    \Big\}^{1/2}\|_p<\infty,$ where the series converges in the distribution sense and $\psi_Q=\psi_j$ if $Q\in Q^j.$

    If $f\in H_d^p(\R^N, \omega),$ the norm of $f$ is defined by
    $$\|f\|_{H_d^p}:=\inf \Big\{\Big\|\Big\{ \sum\limits_{j=-\infty}^\infty
    \sum\limits_{Q\in Q^{j}}|\lambda_{Q}|^2\chi_{Q} (x)
    \Big\}^{1/2}\Big\|_p\Big\},$$ where the infimum is taken over all
    $f(x)=\sum\limits_{j=-\infty}^\infty\sum\limits_{Q\in Q^{j}}\omega(Q)
    \lambda_{Q}\psi_{Q}(x,x_{Q}).$
\end{definition}
We remark that if $\|\Big\{
\sum\limits_{j=-\infty}^\infty\sum\limits_{Q\in Q^{j}}|\lambda_{Q}|^2\chi_{Q}
\Big\}^{1/2}\|_p<\infty,$ then the series $\sum\limits_{j=-\infty}^\infty\sum\limits_{Q\in Q^{j}}\omega(Q)
\lambda_{Q}\psi_{Q}(x,x_{Q})$ defines a distribution in $(L^2{(\R^N, \omega)}\cap CMO_d^p{(\R^N, \omega)})^\prime.$ See the proof in the {\bf Section 4}.

The following theorem is very useful in the proof of the boundedness
for the Dunkl-Calder\'on-Zygmund singular integral operators on the
Dunkl-Hardy spaces $H_d^p(\R^N,\omega).$
\begin{theorem}\label{1.8}
    $$H_d^p(\R^N, \omega)=\overline{L^2(\R^N,\omega)\cap H_d^p(\R^N, \omega)},$$
    where $\overline{L^2(\R^N,\omega)\cap H_d^p}$ is the collection of all distributions $f\in (L^2{(\R^N, \omega)}\cap CMO_d^p{(\R^N, \omega)})^\prime$ such that there exists a sequence $\{f_n\}_{n=1}^\infty$ in $L^2(\R^N,\omega)$ with $\|f_n-f_m\|_{H_d^p}\rightarrow 0$ as $n,m\rightarrow \infty.$ Moreover, $f_n$ converges to $f$ in $(L^2{(\R^N, \omega)}\cap CMO_d^p{(\R^N, \omega)})^\prime.$ 
\end{theorem}

As a direct consequence of Theorem \ref{1.8}, we obtain the following
\begin{corollary}\label{cor1}
    The subspace $L^2(\R^N, \omega)\cap H_d^p(\R^N,\omega)$ is dense in $H_d^p(\R^N,\omega)$ for $\frac{{\bf N}}{{\bf N}+1}<p\leqslant 1.$\\
\end{corollary}
It is well known that the atomic decomposition is a powerful tool
for the boundedness of the classical Calder\'on-Zygmund operator on
the classical Hardy space. The following theorem gives such an
atomic decomposition for the Dunkl-Hardy space. We recall that a
function $a(x)$ is an $(p,2)$ atom if (i) $\supp (a)\subseteq Q,$
where $Q$ is a cube in $\R^N;$ (ii) $\|a\|_2\leq
\omega(Q)^{\frac{1}{2}-\frac{1}{p}};$ (iii) $\int_{\R^N}
a(x)d\omega(x)=0.$
\begin{theorem}\label{atom}
    Suppose $\frac{\bf N}{{\bf N}+1}<p\leqslant 1.$ If $f\in H_d^p(\R^N, \omega)$ then $f$ has an atomic decomposition. More precisely,  $f(x)=\sum\limits_{j=-\infty}^\infty\lambda_j a_j(x),$ where all $a_j$ are $(2, p)$ atoms and
    $$\sum\limits_{j=-\infty}^\infty|\lambda_j|^p\leqslant C\|f\|^p_{H_d^p}$$
    for some constant $C.$

    Conversely, if $f$ has an atomic decomposition $f(x)=\sum\limits_{j=-\infty}^\infty\lambda_j a_j(x),$ then $f\in H_d^p(\R^N, \omega)$ and
    $$\|f\|^p_{H_d^p}\leqslant C \sum\limits_{j=-\infty}^\infty|\lambda_j|^p.$$
\end{theorem}

\subsubsection{Boundedness of Dunkl-Calder\'on-Zygmund singular integral operator on Dunkl-Hardy space}
\ \\

Let's recall  the singular integrals convolution operators in the
rational Dunkl setting, which was introduced in \cite{DH3}.
\begin{definition}\label{SIO1}
    For a positive integer $s,$ consider a kernel $K\in C^s(\R^N \setminus \{ 0 \})$ such that
    \begin{eqnarray*}
        {\rm (i)}\hspace{2cm}  \sup\limits_{0<a<b<\infty}\Big| \int_{a<|x|<b} K(x) d\omega(x) \Big|<\infty,
    \end{eqnarray*}

    \begin{eqnarray*}
        {\rm (ii)}\hspace{1cm}  \Big|\frac{\partial^\beta}{\partial x^\beta}K(x)\Big| \leqslant C\|x\|^{-N-|\beta|} \text{ for all } |\beta| \leqslant s,
    \end{eqnarray*}

    \begin{eqnarray*}
        {\rm (iii)}\hspace{1cm} \lim\limits_{\varepsilon \rightarrow 0} \int_{\varepsilon < |x| <1} K(x) d\omega(x) =L, \text{ where } L\in \C.
    \end{eqnarray*}
\end{definition}
Set $K^{t}(x)=K(x)(1-\phi(\frac{x}{t}))$, where $\phi$ is a fixed
radial $C^\infty$-function supported by the unit ball $B(0,1)$ such
that $\phi(x)=1$ for $\|x\|<\12$. The following result was shown in
\cite{DH3}:

\begin{theorem}\label{1.9}

    Suppose that $T^{t}(f)(x)=f\ast K^{t}(x)$ where $K(x)$ satisfies the above conditions. Then the limit $\lim\limits_{t\rightarrow 0}f\ast K^{t}(x)$ exists. Moreover, $T(f)(x)=\lim\limits_{t\rightarrow 0}f\ast K^{t}(x)$ is bounded on $L^p(\R^N,\omega)$ for $1<p<\infty$ and is of weak type $(1,1)$ as well.
\end{theorem}

The boundeness of the H\"ormander multiplier was proved in
\cite{DH2} as follows:
\begin{theorem}\label{1.10}
    Let $\psi\in C^\infty_c(\R^n)$ be a non-zero radial function such that supp $\psi\subseteq \R^N\setminus \{0\}.$ If $m$ is a function on $\R^N$ which satisfies the H\"omander condition
    $$ M=\sup_{t>0}\|\psi(\cdot)m(t\cdot)\|_{W^s_2}<\infty$$
    for some $s>\textbf{N},$ then the multiplier operator
    $$ T_m(f)={(m{\widehat f})}^{\vee},$$
    originally defined by the Dunkl trasform on $L^2(\R^N, \omega)\cap L^1(\R^N, \omega),$ is \\{\rm (A)} of weak type (1,1), \\{\rm(B)} of strong type (p,p) for $1<p<\infty,$ \\{\rm(C)} bounded on the Hardy space $H^1_{atom}.$
\end{theorem}
Here the classical Sobolev norm is defined by
$$\|m\|_{W_2^s}=\|{\widehat m}(x)(1+\|x\|)^s\|_{L^2(dx)}.$$

The boundedness of the Dunkl-Calder\'on-Zygmund operators on the
Dunkl-Hardy space are following:
\begin{theorem}\label{1.11}
    Suppose that the Dunkl-Calder\'on-Zygmund operator with the kernel $K(x,y)$ satisfies the smoothness conditions only: for $0<\varepsilon\leqslant 1,$
    \begin{equation}\label{smooth y}
    |K(x,y)-K(x,y')|\leqslant C\Big(\frac{\|y-y'\|}{\|x-y\|}\Big)^\varepsilon\frac1{\omega(B(x,d(x,y)))}\qquad {\rm for}\ \|y-y'\|\leqslant d(x,y)/2.
    \end{equation}
    Then $T$ is bounded from the Dunkl-Hardy space $H_d^p(\R^N, \omega)$ to $L^p(\R^N, \omega)$ for $\frac{{\bf N}}{{\bf N}+\varepsilon}<p\leqslant 1.$

    When $p=1,$ the above smoothness condition can be replaced by the following H\"ormander condition:
    $$\int_{\|y-y'\|\leqslant d(x,y)/2}|K(x,y)-K(x,y')|d\omega(x)\leqslant C$$
    and $T$ is also bounded from $H_d^1(\R^N,\omega)$ to $L^1(\R^N,\omega).$
\end{theorem}
It is well known that the molecule theory of the Hardy space is a
powerful tool for providing the boundedness of the classical
Calder\'on-Zygmund operators on the cassical Hardy space. The
molecule theory for the Hardy spaces was developed by Coifman and
Weiss for space of homogeneous type $(X,\rho,\mu)$ where $\rho$ is
the measure distance and the measure $\mu$ satisfies the doubling
property, see page 594 in \cite{CW2}. In \cite{HHL}, the molecule
theory was established for $(X,\rho,\mu),$ where $\rho$ is the
quasi-metric without any regularity and the measure $\mu$ satisfies
the doubling property only. In this paper, applying the similar idea
as in \cite{HHL}, we develop the molecule theory for $(\R^N,
\|\cdot\|,\omega)$ in the Dunkl seting.
\begin{definition}\label{def-molecule}
    Suppose $ {\bf N\over \bf N+1} <p\leq1$. A function $m(x)\in L^2(\R^N,\omega)$ is said to be an $(p,2,\varepsilon, \eta)$ molecule centered at $x_0\in \R^N$ for the Dunkl-Hardy space $H_d^p(\R^N,\omega)$
    if $1\geqslant \varepsilon >\eta>0$, $ {\bf N\over \bf N+\varepsilon-\eta} <p\leq1, \int_{\R^N} m(x)d\omega(x)=0$ and
    \begin{align}\label{molecule}
    \Big( \int_{\R^N} m(x)^2 d\omega(x) \Big)\Big( \int_{\R^N} m(x)^2 \omega(B(x_0,\|x-x_0\|))^{1+{2\varepsilon-2\eta\over \bf N}} d\omega(x)\Big)^{({{\bf N}+2\varepsilon -2\eta\over \bf N}{p\over 2-p}-1)^{-1}}\leqslant 1.
    \end{align}
\end{definition}
Note that the fact that $ {\bf N\over \bf N+\varepsilon-\eta} <p$
implies $\frac{\varepsilon -\eta}{\bf
N}>\frac{1}{p}-1=\frac{1-p}{p}.$ Thus,${{\bf
N}+2\varepsilon-2\eta\over \bf N}{p\over 2-p}-1 =
(1+\frac{2\varepsilon -2\eta}{\bf
N})\frac{p}{2-p}-1>(1+\frac{2-2p}{p})\frac{p}{2-p}-1>0.$

The following result shows that each $(p,2,\varepsilon,\eta)$
molecule $m(x)$ belongs to $H_d^p(\R^N).$
\begin{theorem}\label{moleculeinHp}
    Suppose that $m$ is an $(p,2,\varepsilon,\eta)$ molecule. Then $m\in H_d^p(\R^N,\omega)$ and moreover,
    $$\|m\|_{H_d^p}\leqslant C,$$
    where the constant $C$ is independent of $m.$
\end{theorem}
Applying the above Theorem \ref{moleculeinHp}, we obtain the $T1$
Theorem for the bounedness of Dunkl-Calder\'on-Zygmund operators on
the Hardy space $H_d^p(\R^N).$
 \begin{theorem}\label{1.12}
  Suppose that $T$ is a Dunkl-Calder\'on-Zygmund operator with the kernel
 $K(x,y)$ satisfying the following smoothness condition only: when $M>\frac{\bf N}{2}, 0<\varepsilon\leqslant 1$ and $\|y-y'\|\leqslant \frac{1}{2}d(x,y),$
\begin{align*}
|K(x,y)-K(x,y')|&\leqslant
C\Big(\frac{\|y-y'\|}{\|x-y\|}\Big)^\varepsilon \frac{1}{\omega(B(x,
d(x,y)))}\Big(\frac{d(x,y)}{\|x-y\|}\Big)^{M}.
\end{align*}
Then $T$ is bounded on the Dunkl-Hardy space $H_d^p(\R^N,\omega),
\frac{\bf N}{\bf N +\varepsilon}<p\leqslant 1,$ if and only if
${T^*}(1)=0.$
\end{theorem}
We remark that the size condition on the kernel of $T$ is not
required. Moreover, if $T$ is a classical Calder\'on-Zygmund
operator on space of homogeneous type $(\R^N, \|\cdot\|,\omega).$
Then, by the estimate in \eqref{sm of DCZ} with $N>\frac{\bf N}{2}, K(x,y),$ the kernel of
$T,$ satisfies the smoothness condition of the above Theorem
\ref{1.12}.

The general $T1$ Theorem for the Dunkl-Hardy space
$H_d^p(\R^N,\omega)$ is the following
\begin{theorem}\label{1.13}
    Suppose that $T$ is a Dunkl-Calder\'on-Zygmund operator. Then $T$ is bounded on $H_d^p(\R^N, \omega), \frac{{\bf N}}{{\bf N}+\varepsilon}<p\leqslant 1,$ where $\varepsilon$ is the exponent of the regularity of the kernel of $T,$ if and only if $T^*(1)=0.$
\end{theorem}

As a consequence of Theorems \ref{1.11} and \ref{1.13}, we obtain
the boundedness of the Dunkl-Riesz transforms in the Dunkl-Hardy
space $H_d^p(\R^N, \omega).$

\begin{theorem}\label{1.14}
    The Dunkl-Riesz transforms $R_j, 1\leqslant j\leqslant N,$ are bounded on the Hardy space $H_d^p(\R^N, \omega)$ and from $H_d^p(\R^N, \omega)$ to $L^p(\R^N, \omega), \frac{{\bf N}}{{\bf N}+1}<p\leqslant 1.$
\end{theorem}

The paper is organized as follows. In the next section, the almost
orthogonal estimates, the main tools in this paper will be provided.
The proofs of Theorems \ref{th1.1}, \ref{1.3} and \ref{1.2} will be
given. In section 3, we show Theorem \ref{1.5}, the weak-type
discrete Calder\'on reproducing formula Theorem \ref{1.5} and
Theorem \ref{1.6}, Littlewood-Paley theory on $L^p, 1<p<\infty.$ In
section 4, we discuss the Dunkl-Hardy theory and demonstrate Theorem
\ref{1.7}, Proposition \ref{pr1} and Theorems \ref{1.8} and
\ref{atom}. The boundedness of operators on the Hardy space,
molecule theory, Theorems \ref{1.11}, \ref{moleculeinHp} and
\ref{1.12}-\ref{1.13} will be included in the last section.

Before ending this section, some remarks must be in order. As
mentioned above, one can consider the Dunkl setting $(\Bbb R^N,
\|\cdot\|, \omega)$ as a space of homogeneous type in the sense of
Coifman and Wiess. As Meyer remarked in his preface to~\cite{DH},
\emph{``One is amazed by the dramatic changes that occurred in
analysis during the twentieth century. In the 1930s complex methods
and Fourier series played a seminal role. After many improvements,
mostly achieved by the Calder\'on--Zygmund school, the action takes
place today on spaces of homogeneous type. No group structure is
available, the Fourier transform is missing, but a version of
harmonic analysis is still present. Indeed the geometry is
conducting the analysis.''} The geometry involved in the Dunkl
setting, namely, the finite reflaction goups on $\R^N,$ plays a crucial role. More precisely, the geometric consideration is conducting the Dunkl transform,
translation and convolution operators. They are not appeared in genaral
spaces of homogeneous type in the sense of Coifman and Weiss.
Morever, they are also conducting the Riesz transforms, the
Dunkl-H\"olmander multipliers and the Dunkl-Hardy space in the Dunkl setting. The results of this paper indicate: (1) the operators, namely, the
Dunkl-Calder\'on-Zygmund singular integral theory conducted by the
geometry involved in the Dunkl setting are different from those
defined on spaces of homogeneous type in the sense of Coifman and
Weiss and this theory still plays a fundamental role in the
Dunkl setting; (2)  the Dunkl-Hardy spaces deduced by this geometry
are same as those defined on spaces of homogeneous type; (3) many
methods, such as, almost orthogonal estimates, Coifman's
approximation to the identity and the decomposition of the identity
operator, Meyer's commutation Lemma still can be applied to the
Dunkl setting.

\section{{Dunkl-Calder\'on-Zygmund Singular Integral Operators and $T1$ Theorem }}

\subsection{{Dunkl-Calder\'on-Zygmund Singular Integral Operators}}
\ \\

We begin with the following proposition which is the motivation for
introducing the Dunkl-Calder\'on-Zygmund singular integral
operators.
\begin{proposition}\label{pr300}
    Suppose that $S_t(x,y)$ satisfy the following conditions:
    \begin{eqnarray*}
        &\textup{(i)}& |S_t(x,y)|\lesssim  {1\over V(x,y,t+d(x,y))}\frac{t}{t+\|x-y\|},\\
        &\textup{(ii)}& |S_t(x,y)-S_t(x',y)|\\
        &&\lesssim {\|x-x'\|\over t} \Big({1\over V(\boldsymbol{x},y,t+d(\boldsymbol{x},y))}\frac{t}{t+\|\boldsymbol{x}-y\|}+{1\over V(\boldsymbol{x'},y,t+d(\boldsymbol{x'},y))}\frac{t}{t+\|\boldsymbol{x'}-y\|}\Big),\\
        &\textup{(iii)}& |S_t(x,y)-S_t(x,y')|\\
        &&\lesssim {\|y-y'\|\over t}\Big({1\over V(x,\boldsymbol{y},t+d(x,\boldsymbol{y}))}\frac{t}{t+\|x-\boldsymbol{y}\|}+{1\over V(x,\boldsymbol{y'},t+d(x,\boldsymbol{y'}))}\frac{t}{t+\|x-\boldsymbol{y'}\|}\Big),
    \end{eqnarray*}
    then $K(x,y)=\int_0^\infty S_t(x,y)\frac{dt}{t}$ satisfies the following estimates: for any $0<\varepsilon<1,$
    \begin{center}
        \begin{equation}\label{size3}
        |K(x,y)|\lesssim \Big(\frac{d(x,y)}{\|x-y\|}\Big)^\varepsilon\frac1{\omega(B(x,d(x,y)))};
        \end{equation}
        \begin{equation}\label{smooth y3}
        |K(x,y)-K(x,y')|\lesssim \Big(\frac{\|y-y'\|}{\|x-y\|}\Big)^\varepsilon\frac{1}{\omega(B(x,d(x,y)))} \qquad {\rm for}\ \|y-y'\|\leqslant  d(x,y)/2;
        \end{equation}
        \begin{equation}\label{smooth x3}
        |K(x',y)-K(x,y)|\lesssim \Big(\frac{\|x-x'\|}{\|x-y\|}\Big)^\varepsilon\frac1{\omega(B(x,d(x,y)))}\qquad {\rm for}\ \|x-x'\|\leqslant d(x,y)/2.
        \end{equation}
    \end{center}
\end{proposition}

\begin{proof}
    We first verify \eqref{size3}. By the condition (i),
    \begin{align*}
    |K(x,y)|&\lesssim \int_0^\infty |S_t(x,y)|\frac{dt}{t}= \int_0^{d(x,y)}\, {1\over V(x,y,t+d(x,y))}\frac{t}{t+\|x-y\|}\frac{dt}{t}\\
    &\qquad +\int_{d(x,y)}^{\|x-y\|}\, {1\over V(x,y,t+d(x,y))}\frac{t}{t+\|x-y\|}\frac{dt}{t}\\
    &\qquad +\int_{\|x-y\|}^\infty
    {1\over V(x,y,t+d(x,y))}\frac{t}{t+\|x-y\|}\frac{dt}{t}
    \\ &=:I_1+I_2+I_3.
    \end{align*}
    Note that $\omega(B(x,d(x,y)))\sim \omega(B(y,d(x,y)))\sim V(x,y,d(x,y))$, we obtain
    $$I_1\lesssim \frac{1}{\|x-y\|}\int_0^{d(x,y)} {1\over V(x,y,t+d(x,y))}dt\lesssim \frac{d(x,y)}{\|x-y\|}\frac1{\omega(B(x,d(x,y)))}.$$

    If $d(x,y)\leqslant t\leqslant \|x-y\|$, by using the reverse doubling condition on the measure $\omega,$ we get  $V(x,y, t+d(x,y))\geqslant \omega(B(x,t))\geqslant C\big(\frac{t}{d(x,y)}\big)^N\omega(B(x,d(x,y))).$
    Thus, we have
    \begin{align*}
    I_2\lesssim \frac{1}{\|x-y\|}\int_{d(x,y)}^{\|x-y\|}\big(\frac{d(x,y)}{t}\big)^N {1\over \omega(B(x,d(x,y)))}dt.
    \end{align*}
    Hence, if $N>1,$
    \begin{align*}
    I_2\lesssim\frac{d(x,y)}{\|x-y\|}\frac1{\omega(B(x,d(x,y)))}.
    \end{align*}
    If $N=1,$ then there exists $0<\varepsilon<1$ such that
    \begin{align*}
    I_2\lesssim \frac{d(x,y)}{\|x-y\|}\ln\Big(\frac{\|x-y\|}{d(x,y)}\Big){1\over \omega(B(x,d(x,y)))}
    \lesssim \Big(\frac{d(x,y)}{\|x-y\|}\Big)^\varepsilon{1\over \omega(B(x,d(x,y)))}.
    \end{align*}
    Again, by the reverse doubling property of the measure $\omega$, we see that
    $V(x,y,d(x,y))=V(x,y,\frac{d(x,y)}{t}t)\lesssim
    \big(\frac{d(x,y)}{t}\big)^{N} V(x,y,t)$ for
    $d(x,y)\leqslant\|x-y\|\leqslant t<\infty $.  Thus,
    $I_3$ can be estimated as follows
    \begin{align*}
    I_3&\lesssim \int_{\|x-y\|}^\infty {1\over V(x,y,d(x,y))}\Big(\frac{d(x,y)}{t}\Big)^{N}\frac{dt}{t} \lesssim \Big(\frac{d(x,y)}{\|x-y\|}\Big)^{N}{1\over \omega(B(x,d(x,y)))}\\
     &\lesssim \frac{d(x,y)}{\|x-y\|}{1\over \omega(B(x,d(x,y)))}.
    \end{align*}
    Hence we obtain that \eqref{size3} holds.

    To verify \eqref{smooth y3}, we write
    \begin{align*} |K(x,y)-K(x,y')|&\lesssim \int_0^\infty |S_t(x,y)-S_t(x,y')|\frac{dt}{t}
    \\&=\int_0^{\|y-y'\|} |S_t(x,y)-S_t(x,y')|\frac{dt}{t}+\int_{\|y-y'\|}^{\|x-y\|} |S_t(x,y)-S_t(x,y')|\frac{dt}{t}\\
    &\qquad +\int_{\|x-y\|}^\infty |S_t(x,y)-S_t(x,y')|\frac{dt}{t}
    \\&=:I\!I_1+I\!I_2+I\!I_3.
    \end{align*}

    Observe that when $d(y,y')\leqslant \|y-y'\|\leqslant \frac{1}{2}d(x,y)\leqslant \frac{1}{2}\|x-y\|,$   $d(x,y)\sim d(x,y')$ and $\|x-y\|\sim \|x-y'\|.$ Then applying condition (i), we get
    \begin{align*}
    I\!I_1&\lesssim \int_0^{\|y-y'\|} |S_t(x,y)-S_t(x,y')|\frac{dt}{t} \\
    & \lesssim \int_0^{\|y-y'\|} \Big\{{1\over V(x,\boldsymbol{y},t+d(x,\boldsymbol{y}))}\frac{t}{t+\|x-\boldsymbol{y}\|}+{1\over V(x,\boldsymbol{y'},t+d(x,\boldsymbol{y'}))}\frac{t}{t+\|x-\boldsymbol{y'}\|}\Big\}\frac{dt}{t}
    \\        &\lesssim \frac{\|y-y'\|}{\|x-y\|}{1\over \omega(B(x,d(x,y)))}.
    \end{align*}

    Applying condition (iii) implies that for any fixed $0<\varepsilon<1,$
    \begin{align*}
    I\!I_2&\lesssim \int_{\|y-y'\|}^{\|x-y\|} {\|y-y'\|\over t}\Big({1\over V(x,\boldsymbol{y},t+d(x,\boldsymbol{y}))}\frac{t}{t+\|x-\boldsymbol{y}\|}\\
    &\qquad\qquad\qquad +{1\over V(x,\boldsymbol{y'},t+d(x,\boldsymbol{y'}))}\frac{t}{t+\|x-\boldsymbol{y'}\|}\Big)\frac{dt}{t} \\
    & \lesssim  \frac{\|y-y'\|}{\|x-y\|}\ln\bigg(\frac{\|x-y\|}{\|y-y'\|}\bigg){1\over \omega(B(x,d(x,y)))}
    \\        &\lesssim \Big(\frac{\|y-y'\|}{\|x-y\|}\Big)^\varepsilon{1\over \omega(B(x,d(x,y)))}.
    \end{align*}

    Similarly,

    \begin{align*}
    I\!I_3&\lesssim \int_{\|x-y\|}^\infty  {\|y-y'\|\over t}\Big({1\over  V(x,\boldsymbol{y},t+d(x,\boldsymbol{y}))}\frac{t}{t+\|x-\boldsymbol{y}\|}\\
    &\hskip4.5cm +{1\over V(x,\boldsymbol{y'},t+d(x,\boldsymbol{y'}))}\frac{t}{t+\|x-\boldsymbol{y'}\|}\Big)\frac{dt}{t}\\
    & \lesssim  \int_{\|x-y\|}^\infty  {\|y-y'\|\over t}\Big({1\over  V(x,\boldsymbol{y},t+d(x,\boldsymbol{y}))}+{1\over V(x,\boldsymbol{y'},t+d(x,\boldsymbol{y'}))}\Big)\frac{dt}{t}
    \\        &\lesssim \frac{\|y-y'\|}{\|x-y\|}\cdot {1\over \omega(B(x,d(x,y)))}.
    \end{align*}

    The verification for \eqref{smooth x3} is similar and we omit the details here.

    The proof of Proposition \ref{pr300} is complete.
\end{proof}

Before proving  Theorem \ref{th1.1}, we first give the following
lemma.
\begin{lemma}\label{lem1}
    For any $\varepsilon,t>0,y\in \R^N$ , there exists a constant $C$ depending on $\varepsilon$ such that,
    $$\int_{\Bbb R^N}{1\over \omega(x,t+d(x,y))}\Big(\frac{t}{t+d(x,y)}\Big)^\varepsilon d\omega(x) \leqslant C.$$
\end{lemma}

\begin{proof}
    Let $\displaystyle S=\int_{\Bbb R^N}{1\over \omega(x,t+d(x,y))}\Big(\frac{t}{t+d(x,y)}\Big)^\varepsilon d\omega(x)$, then
    \begin{align*}
    S &\lesssim
    \sum\limits_{j=1}^\infty\int_{t2^{j-1}\leqslant d(x,y)<t2^j}\frac1{\omega(x,t+d(x,y))}\bigg(\frac{t}{t+d(x,y)}\bigg)^{\varepsilon}d\omega(x)
    \\ &\qquad +\int_{d(x,y)<t}\frac1{\omega(x,t+d(x,y))}d\omega(x).
    \end{align*}
    Note that
    \begin{align*}
    \int_{d(x,y)<t}\frac1{\omega(x,t+d(x,y))}d\omega(x)
    &\lesssim \sum\limits_{\sigma \in G}\int_{\|\sigma(y)-x\|<t}\frac1{\omega(B(x,t))}d\omega(x) \\
    &\lesssim \sum\limits_{\sigma \in G}\int_{\|\sigma(y)-x\|<t}\frac1{\omega(B(\sigma(y),t))}d\omega(x) \\ &\lesssim 1.
    \end{align*}

    And for $j \geqslant 1$,
    \begin{align*}
    &\int_{t2^{j-1}\leqslant d(x,y)<t2^j}\frac1{\omega(x,t+d(x,y))}\bigg(\frac{t}{t+d(x,y)}\bigg)^{\varepsilon}d\omega(x) \\
    &\qquad \lesssim  2^{-j\varepsilon}\sum\limits_{\sigma \in G}\int_{\|\sigma(y)-x\|<t2^j}\frac1{\omega(B(x,t2^j))}d\omega(x)  \\
    &\qquad \lesssim 2^{-j\varepsilon}\sum\limits_{\sigma \in G}\int_{\|\sigma(y)-x\|<t2^j}\frac1{\omega(B(\sigma(y),t2^j))}d\omega(x)\\
    &\qquad \lesssim 2^{-j\varepsilon}.
    \end{align*}

    Therefore,
    \begin{align*}
    S\lesssim \sum\limits_{j=1}^\infty2^{-j(\varepsilon)}+1 \lesssim  1.
    \end{align*}
    This completes the proof of Lemma \ref{lem1}.
\end{proof}

We now show Theorem \ref{th1.1}.
\begin{proof}[{\bf The proof of Theorem \ref{th1.1}}]
To begin with, we first show the
week type (1,1) estimate. The idea is to apply the classical
Calder\'on-Zygmund decomposition. To this end, let $f\in
L^2(\R^N,\omega)\cap L^1$ and let $\lambda$ be a positive real
number. Then there exist disjoint cubes $Q_j$ and $f$ can be written
as the sum of a function $g\in L^1\cap L^2$, and of a series of
functions $\{b_j\}^\infty_{j=1}$ such that each function $b_j$ is
contained in the corresponding cube $Q_j.$ More precisely,
$g(x)=f(x)$ if $x\notin \cup Q_j$, whereas $g(x)=|Q_j|^{-1}\int_{Q_j
}f(x)d\omega(x)$ for $x\in Q_j$ and
$b(x)=f(x)-g(x)=\sum\limits_{j=1}^\infty b_j(x)$ where
$b_j(x)=f(x)-|Q_j|^{-1}\int_{Q_j }f(x)d\omega(x)$ for $x\in Q_j.$ It
is easy to see that
\begin{eqnarray*}
    &\textup{(i)}& f(x)=g(x)+\sum\limits_{j=1}^\infty b_j(x), x\in \R^N,\\
    &\textup{(ii)}& |g(x)|=|f(x)|\leqslant \lambda, \text{for\ almost\ all\ } x\notin \bigcup\limits_{j=1}^\infty Q_j,\\
    &\textup{(iii)}& |g(x)|\leqslant 2^N\lambda, \forall x\in \bigcup\limits_{j=1}^\infty Q_j,\\
    &\textup{(iv)}& |\bigcup\limits_{j=1}^\infty Q_j|\leqslant \lambda^{-1}\|f\|_1,\\
    &\textup{(v)}&  \|g\|_2\leqslant 2^N \lambda^{1/2}\|f\|^{1/2}_1,\\
    &\textup{(vi)}&  \int_{Q_j} |b_j(x)| d\omega(x)\leqslant 2^{N+1}\lambda|Q_j|,\\
    &\textup{(vii)}&  \int_{Q_j} b_j(x) d\omega(x)=0.
\end{eqnarray*}
Let $\mathcal O(\bigcup\limits\limits_{j=1}^\infty
4\sqrt{N}Q_j)=\{x: d(x,x_{Q_j})\leqslant \l(4\sqrt{N}Q_j), \text{for
each}\ j\}$ with $x_{Q_j},$ the center of $Q_j$ and
$\l(4\sqrt{N}Q_j)$  the side length of $4\sqrt{N}Q_j.$ Then
$$\omega\big(\mathcal O(\bigcup\limits_{j=1}^\infty 4\sqrt{N}Q_j)\big)\lesssim \sum\limits_{j=1}^\infty|Q_j|\lesssim \lambda^{-1}\|f\|_1$$
and
$$\lambda|\big\{x\in \mathcal O(\bigcup\limits_{j=1}^\infty 4\sqrt{N}Q_j)^c: |Tb(x)|\geqslant \lambda\big\}|\leqslant  \int_{\mathcal O(\bigcup\limits_{j=1}^\infty 4\sqrt{N}Q_j)^c} |Tb(x)|d\omega(x).$$
We estimate the last term above as follows. For $x\in \mathcal
O(\bigcup\limits_{j=1}^\infty 4\sqrt{N}Q_j)^c,$
$$Tb_j(x)=\int_{\R^N} K(x,y)b_j(y)d\omega(y)=\int_{Q_j}[K(x,y)-K(x,x_{Q_j})]b_j(y)d\omega(y),$$
where the cancellation condition of $b_j$ is used.

Observe that if $x\in \mathcal O(\bigcup\limits_{j=1}^\infty
4\sqrt{N}Q_j)^c, y\in Q_j$ and $y_j$ is the center of $Q_j$, then
$d(x,y)\geqslant 2\|y-x_{Q_j}\|$ and hence, applying the smoothness
condition of $K(x,y)$ gives
\begin{align*}
&\int_{\mathcal O(\bigcup\limits_{j=1}^\infty 4\sqrt{N}Q_j)^c}|Tb_j(x)|d\omega(x)\\
&\lesssim \int_{\mathcal O(\bigcup\limits_{j=1}^\infty 4\sqrt{N}Q_j)^c}\int_{Q_j}\Big(\frac{\|y-x_{Q_j}\|}{d(x,y)}\Big)^\varepsilon \frac{1}{\omega(x, d(x,y))}|b_j(y)|d\omega(y)d\omega(x)\\
&\lesssim \int_{Q_j}
\int_{\R^N}\Big(\frac{\|y-x_{Q_j}\|}{\|y-x_{Q_j}\|+d(x,y)}\Big)^\varepsilon
\frac{1}{\omega(x, \|y-x_{Q_j}\|+d(x,y))}d\omega(x)
|b_j(y)|d\omega(y)\\
&\lesssim \int_{Q_j} |b_j(y)|d\omega(y)\\
& \lesssim
2^{N+1}\lambda|Q_j|,
\end{align*}
where we applying the Lemma \ref{lem1} in the second inequality and
the property(vi) in the last inequality above. And hence, by (iv),
$$\int_{\mathcal O(\bigcup\limits_{j=1}^\infty 4\sqrt{N}Q_j)^c}|Tb(x)|d\omega(x)\lesssim \sum\limits_{j=1}^\infty\lambda|Q_j|\lesssim \|f\|_1.$$
All these estimates together with the $L^2$ boundedness of $T$ and (v) imply that
\begin{align*}
\omega\Big(\Big\{x: |Tf(x)|\geqslant \lambda\Big\}\Big) &\lesssim \omega\Big(\Big\{x: |Tg(x)|\geqslant {\lambda\over2}\Big\}\Big) +\omega\Big(\Big\{x: |Tb(x)|\geqslant {\lambda\over2}\Big\}\Big)\\
& \lesssim  \Big(\frac{\|g\|_2}{\lambda}\Big)^2 +\omega\big(\mathcal
O(\bigcup\limits_{j=1}^\infty 4\sqrt{N}Q_j)\big) +\lambda^{-1}
\int_{\mathcal O(\bigcup\limits_{j=1}^\infty 4\sqrt{N}Q_j)^c}|Tb(x)|d\omega(x)\\
&\lesssim \lambda^{-1}\|f\|_1,
\end{align*}
which implies that $T$ is of  weak type (1,1).

By interpolation, $T$ is bounded on $L^p, 1<p\leqslant 2.$ The same
proof applies to $T^*$ gives the $L^p, 1<p\leqslant 2$ boundedness
of $T^*$ and then, by the duality, $T$ is bounded on $L^p,
2\leqslant p<\infty.$

To show that $T$ is bounded from $H_d^1(\R^N,\omega)$ (introduced in
\cite{ADH}) to $L^1(\R^N,\omega),$ we apply the the atomic
decomposition of $H_d^1(\R^N,\omega)$ provided in \cite{DH1}. it
suffices to show that if $a(x)$ is an atom, that is, $a$ satisfies
{\rm (i)} support of $a$ is contained in a cube $Q$ in $\R^N;$ {\rm (ii)}
$\|a\|_2\leqslant \omega(Q)^{-\frac 12};$ {\rm (iii)} $\int_{\R^N}
a(x)d\omega(x)=0;$ then $\|T(a)\|_1\leqslant C,$ where the constant
$C$ is independent of $a.$ To this end, let $B=\{x:
d(x,x_Q)\leqslant 4\sqrt{N}\l(Q) \},$ where $x_Q$ is the center of
$Q$ and $\l(Q)$ is the side length of $Q.$ Write
$$\int_{\R^N} |T(a)(x)|d\omega(x)=\int_B |T(a)(x)|d\omega(x)+\int_{B^c} |T(a)(x)|d\omega(x).$$
The H\"older inequality, the $L^2$ boundedness of $T,$ and the size condition of $a$ imply that
$$\int_B |T(a)(x)|d\omega(x)\leqslant C \omega(B)^{\frac 12}\|a\|_2\leqslant C\omega(Q)^{\frac 12}\|a\|_2\leqslant C.$$
If $x\in B^c$ and $y\in Q,$ then $\|y-x_Q\|\leqslant
\frac{1}{2}d(x,x_Q).$ By the cancellation condition of $a$ and the
smoothness condition of the kernel $K(x,y)$,
\begin{align*}
T(a)(x)&=\int_Q K(x,y)a(y)d\omega(y)=\int_Q [K(x,y)-K(x,x_Q)]a(y)d\omega(y)\\
&\leqslant C\Big(\frac{\l(Q)}{\|x-x_Q\|}\Big)^\varepsilon \frac{1}{\omega(x, d(x,x_Q))}\|a\|_1.
\end{align*}
Therefore,
\begin{align*}
\int_{B^c} |T(a)(x)|d\omega(x)&\lesssim \int_{B^c} \Big(\frac{\l(Q)}{\|x-x_Q\|}\Big)^\varepsilon \frac{1}{\omega(x, d(x,x_Q))}\|a\|_1d\omega(x) \\
&\lesssim\int_{d(x,x_Q)\geqslant \l(Q)}
\Big(\frac{\l(Q)}{d(x,x_Q)}\Big)^\varepsilon \frac{1}{\omega(x,
    d(x,x_Q))}d\omega(x)\\
    &\lesssim 1,
\end{align*}
since $\|a\|_1\leqslant C$ and hence, $T$ is bounded from
$H_d^1(\R^N,\omega)$ to $L^1(\R^N,\omega).$

We now prove the $L^\infty-BMO_d(\R^N,\omega)$ boundedness of $T.$
We first provide a strict definition of $Tf(x)$ when $f\in L^\infty.$
To this end, we
follow the idea given in \cite{MC}. If $f\in L^\infty(\R^N,
\omega),$ we define the functions $f_j(x)$ by $f_j(x)=f(x),$ when
$\|x\|\leqslant j,$ and $f_j(x)=0,$ if $\|x\|>j.$ Since $f_j\in
L^2(\R^N, \omega),$  $T(f_j)$ is well defined by the action of $T$
on $L^2(\R^N,\omega).$ We claim that there exists a sequence $\{c_j\}_j$
of constants such that $T(f_j)-c_j$ converges, uniformly on any
compact set in $\R^N,$ to a function in $BMO_d(\R^N,\omega)$ which
will be defined by $T(f)$ modulo the constant functions. Indeed, set
$c_j=\int_{1\leqslant d(0,y)\leqslant j}K(0,y)d\omega(y).$ Observe
that, by the size condition on the kernel $K(x,y),$
\begin{align*}
c_j
&\leqslant C \int_{1\leqslant d(0,y)\leqslant j}\frac{1}{\omega(B(0,d(0,y)))}\Big(\frac{d(0,y)}{\|y\|}\Big)^\varepsilon d\omega(y)\\
&\leqslant C\frac{\omega(B(0,j))}{\omega(B(0,1))}{j}^\varepsilon
<\infty.
\end{align*}

To show $T(f_j)-c_j$ converges uniformly on the compact ball
$B(0,R),$ we split $f_j$ into $g+h_j,$ where $g(x)=f(x),$ when
$d(0,x)\leqslant 2R,$ and $g(x)=0,$ if $d(0,x)>2R.$ Taking $j>2R,$
we have, for $\|x\|\leqslant R,$
\begin{align*}
T(f_j)(x)&=T(g)(x)+T(h_j)(x)=T(g)(x)+\int_{2R\leqslant d(0,y)\le
    j}K(x,y)f(y)d\omega(y)\\
&= T(g)(x)+\int_{2R\leqslant d(0,y)\le
    j}[K(x,y)-K(0,y)]f(y)d\omega(y)+c_j-C(R),
\end{align*} where $C(R)=
\int_{1\leqslant d(0,y)\leqslant 2R}K(0,y)d\omega(y).$ Observe that
when $\|x\|\leqslant R,$ by the smoothness condition on the kernel
$K(x,y),$ we get
\begin{align*}
&\int_{2R\leqslant d(0,y)}|K(x,y)-K(0,y)|\cdot|f(y)|d\omega(y)\\
&\leqslant C \int_{2R\leqslant d(0,y)}\frac{1}{\omega(B(0,d(0,y)))}\Big(\frac{\|x\|}{\|y\|}\Big)^\varepsilon d\omega(y)\|f\|_{\infty}\\
&\leqslant C \int_{2R\le
    d(0,y)}\frac{1}{\omega(B(0,d(0,y)))}\Big(\frac{R}{d(0,y)}\Big)^\varepsilon
d\omega(y)\|f\|_{\infty}\\
&\leqslant C \|f\|_{\infty}.
\end{align*}
Thus the integral $\int_{2R\leqslant d(0,y)\le
    j}|K(x,y)-K(0,y)|\cdot|f(y)|d\omega(y)$ converges uniformly on
$\|x\|\leqslant R$ as $j$ tends to $\infty,$ which implies that
$T(f_j)-c_j$ converges uniformly on any compact set in $\R^N.$ We
remark that the smoothness condition on the kernel $K(x,y)$ can be
replaced by
$$\int_{d(x,y)\geqslant 2\|x-x'\|} |K(x,y)-K(x',y)|d\omega(y)\leqslant C.$$
Once  the $T(f)$ is defined with $f\in L^\infty(R^N,\omega)$ by the above
claim, we can show that $T(f)\in BMO_d(\R^N,\omega)$ and moreover,
$\|T(f)\|_{BMO_d}\leqslant C\|f\|_{\infty}.$ To this end, let $B$
denote an arbitrary ball with center $x_0$ and radius $R,$ and
$\mathcal O_B=\{y: d(x_0,y)\leqslant 2R\}.$ Then we write
$f=f_1+f_2,$ where $f_1$ is the multiplication of $f$ with the characteristic
function of $\mathcal O_B.$ We now define $T(f_2)(x)$ for $x\in B$
by the following absolutely convergent integral
$$T(f_2)(x)=\int_{d(x_0,y)\geqslant 2R}[K(x,y)-K(x_0,y)]f(y)d\omega(y).$$
Indeed, by the smoothness condition of $T,$
\begin{align*}
&\int_{d(x_0,y)\geqslant 2R}|K(x,y)-K(x_0,y)|\cdot|f(y)|d\omega(y)\\
&\leqslant  C\|f\|_\infty \int_{d(x_0,y)\geqslant
2R}\frac{1}{\omega(x_0,
    d(y,x_0))}\Big(\frac{R}{d(x_0,y)}\Big)^\varepsilon d\omega(y)\\
&\leqslant C\|f\|_\infty.
\end{align*}
Moreover, since $f_1\in L^\infty$ is supported in a bounded set $\mathcal O_B,$ we see that
$f_1\in L^2(\R^N,\omega).$  Hence $T(f_1)(x)$ is well-defined.

We can now give a strict definition of $T(f)(x)$ as follows:
$T(f)(x)=T(f_1)(x)+T(f_2)(x)$.
Further, this definition of $T(f)(x)$
is only differing by a constant, depending on $x_0$ and $R.$ To see
the proof that $T(f)$ belongs to $BMO_d(\R^N,\omega),$ we have
$$\|T(f_2)\|_\infty\leqslant C\|f\|_\infty$$
and, further,
$$\|T(f_1)\|_2\leqslant \|T\|\|f_1\|_2\leqslant C\|f\|_\infty \omega(\mathcal O_B)^{1/2}\|T\|.$$
We thus get
\begin{align*}
&\Big(\int_B\Big|T(f)(x)-\frac{1}{\omega(B)}\int_B
T(f)(y)d\omega(y)\Big|^2
d\omega(x)\Big)^{1/2}\\
&\leqslant C\|f\|_\infty\omega(B)^{1/2} + C\omega(\mathcal
O_B)^{1/2}\|f\|_\infty\|T\|\\
&\leqslant
C'\big(\omega(B)\big)^{1/2}\|f\|_\infty,
\end{align*} where the last inequality follows from the fact that $
\omega(\mathcal O_B)\sim \omega(B).$ The proof of Theorem
\ref{th1.1} is complete.
\end{proof}

\subsection{{Meyer's commutation Lemma}}
\ \\

Now we prove the Theorem \ref{1.3}. To this end, we first recall
Meyer's commutation Lemma, which plays a fundamental role in the
Calder\'on-Zygmund singular integral operator theory. Given
$\phi(x)\in C^\infty_0$ with $\supp \phi\subseteq B(x_0, r), r>0.$
Let $\theta(x)$ be a function in $C^\infty_0(\R^N)$ with
$\theta(x)=1$ if $|x|\leqslant 1$ and $\theta(x)=0$ if $|x|>2,$ and
$\eta_0(x)=\theta(\frac{|x-x_0|}{2r}).$ Meyer's commutation Lemma is
the following
\begin{lemma}\label{mcl1}(\cite{M})
    Suppose that $T$ is a classical Calder\'on-Zygmund singular integral operator with the kernel $K(x,y)$ satisfing the size and smoothness conditions. Moreover, $T$ has the weak boundedness property and $T(1)=0.$ Then
    $$T(\phi)(x)=\int_{\R^N} K(x,y)[\phi(y)-\phi(x)]\eta_0(y)dy+\phi(x)\int_{\R^N} K(x,y)\eta_0(y)dy$$
    with $\int_{\R^N} K(x,y)[\phi(y)-\phi(x)]\eta_0(y)dy=\lim\limits_{\delta \to 0}\int_{|x-y|\ge\delta}K(x,y)[\phi(y)-\phi(x)]\eta_0(y)dy,$ where the limit exists.
\end{lemma}
Applying this lemma, Meyer obtained the boundedness of Calder\'on-Zygmund singular integral operators on smooth molecule space. As a consenquence, Meyer proved that all classical Calder\'on-Zygmund operators with
the conditions $T(1)=T^*(1)=0$ form an algebral. See \cite{M} for more details.

In \cite{DJS}, David, Journ\'e and Semmes stated Meyer's commutation Lemma as follows.
\begin{lemma}\label{mcl2}(\cite{DJS})
    Suppose that $T$ is a coutinuous operator from $C^\infty_0(\R^N)$ to $\big(C^\infty_0(\R^N)\big)^\prime$ with the kernel $K(x,y)$ satisfing the size condition and $T$ has the weak boundedness property. Then for all $f, g, h\in C^\infty_0(\R^N),$
    $$\langle f, T(gh)\rangle-\langle fg, T(h)\rangle=\int_{\R^N}\int_{\R^N} f(x)K(x,y)[g(y)-g(x)]h(y)dydx,$$
    where the integral absolutely converges.
\end{lemma}
Using this lemma, they showed that if $T$ has the weak boundedness property with $C_0^\eta(\R^N)$ for $\eta>0,$ then $T$ has the weak boundedness property with $C_0^\infty(\R^N).$ See \cite{DJS} for more details.

In \cite{HS}, Meyer's commutation Lemma was proved for spaces of
homogeneous type in the sense of Coifman and Weiss. More precisely,
suppose that $(X, \rho, \mu)$ is a space of homogeneous type with
the measure satisfying $\mu(B(x, r))\sim r$ with $B(x,r)=\{y\in X:
\rho(x,y)<r\}$ and $r>0.$ Let $\theta$ be the same as above and
$\eta_0(x)=\theta(\frac{\rho(x_0,x)}{2r}), \eta_0+\eta_1=1.$ Then
Meyer's commutation Lemma is given by
\begin{lemma}\label{mcl3}(\cite{HS})
    Suppose that $T$ is a classical Calder\'on-Zygmund singular integral operator defined on space $(X, d, \mu)$ with the kernel $K(x,y)$ satisfing the size and smoothness conditions. Moreover, $T$ has the strong-weak boundedness property
    and $T(1)=0.$ Then
    $$\langle T\phi,\psi\rangle=\int_{\R^N}\int_{\{y:y\not=x\}} K(x,y)\Big\{[\phi(y)-\phi(x)]\eta_0(y)-\eta_1(y)\phi(x)\Big\}d\mu(y)d\mu(x)$$
    and
    $$\langle K(x,y),[\phi(y)-\phi(x)]\eta_0(y)\rangle=\lim_{\delta \to 0}\int_{\rho(x,y)\ge\delta}K(x,y)[\phi(y)-\phi(x)]\eta_0(y)dy$$
    where the limit exists.
\end{lemma}
Applying this lemma, the boundedness on the test function space for classical Calder\'on-Zygmund integral operators defined on space $(X, \rho, \mu)$ with the kernel $K(x,y)$ satisfying the size and smoothness conditions together with some additional second order smoothness
and $T(1)=T^*(1)=0.$ And the Calder\'on reproducing formulae were established. See \cite {HS} for more details. See also \cite{HMY} for similar results on spaces of homogeneous type with the measure satisfing doubling and reverse doubling conditions.

To establish the $T1$ theorem for non-doubling measures, Tolsa
introduced the following definition.
\begin{definition}\label{mcl4}(\cite{Tol})
    Let $T$ be an SCZO with the kernel $K(x,y).$ We say that $T$ satisfies the commutation lemma of Meyer, and we write $T\in CLM,$ if for compactly supported functions $\phi, \psi, w\in L^\infty(\mu),$ with $\psi$ Lipschitz,
    the following identity holds:
    $$\langle T\phi, \psi w\rangle-\langle T(\phi\psi), w\rangle=\int_{\R^N} K(x,y)\big(\psi(y)-\psi(x)\big)\phi(x)w(y)d\mu(y)d\mu(x).$$
\end{definition}
This definition plays a crucial role in the proof of the $T1$ theorem for non-doubling measures. See \cite{Tol} for more details.

In this paper, we prove Meyer's commutation Lemma in the Dunkl setting as follows.
\begin{lemma}(Meyer's commutation Lemma)\label{lem2.5}
    Suppose that $T$ is a Dunkl-Calder\'on-Zygmund singular integral operator from $C^\eta_0$ to $(C^\eta_0)'$ satisfying
    $T\in WBP$ and $T(1)=0$.
    Then for any $M>1,$ there exists a positive constant $C_{M}$ depending on $M$ such that
    $$\|T\phi\|_{L^\infty(B(0,Mr)))} \leqslant C_{M}$$
    whenever there exist $x_0\in \mathbb R^N$ and $r>0$ such that $\supp$ $(\phi)\subseteq B(x_0,r)$
    with $\|\phi\|_\infty\leqslant 1$ and $\|\phi\|_\eta\leqslant r^{-\eta}$.
\end{lemma}

\begin{proof}
    Fix a function $\theta\in C^\infty(\mathbb R)$ with the following properties: $\theta(x)=1$ for $|x|\le1$ and $\theta(x)=0$ for $|x|>2$.
    Let $\chi_0(x)=\theta(\frac {d(x,x_0)}{2r})$ and $\chi_1=1-\chi_0$. Then $\phi=\phi\chi_0$ and for all $\psi\in C^\eta_0(\mathbb R^N)$ with $\supp \psi\subseteq B(x_0,Mr),$
    \begin{align*}
    \langle T\phi, \psi\rangle
    &= \langle K(x,y), \phi(y)\psi(x) \rangle =  \langle K(x,y), \chi_0(y)\phi(y)\psi(x) \rangle \\
    &= \langle K(x,y), \chi_0(y)[\phi(y)-\phi(x)]\psi(x) \rangle +
    \langle K(x,y), \chi_0(y)\phi(x)\psi(x) \rangle \\
    &=: p+q,
    \end{align*}
    where $K(x,y)$ is the distribution kernel of $T$.

    To estimate $p$, let $\lambda_\delta(x,y)=\theta(\frac{\|x-y\|}\delta)$. Then
    \begin{equation}\label{T2.5}
    \begin{aligned}
    p&=\langle K(x,y), (1-\lambda_\delta(x,y))\chi_0(y)[\phi(y)-\phi(x)]\psi(x) \rangle \\
    &\qquad + \langle K(x,y), \lambda_\delta(x,y)\chi_0(y)[\phi(y)-\phi(x)]\psi(x) \rangle \\
    &=: p_{1,\delta}+p_{2,\delta}.
    \end{aligned}
    \end{equation}

    Since $K$ is locally integrable on $\Omega=\{(x,y)\in \mathbb R^N\times \mathbb R^N: x\ne y\}.$
    By the size condition on $K(x,y)$ and the smoothness condition on $\phi$ together with the fact that if $\chi_0(y)\not=0, \psi(x)\not=0$ and $1-\lambda_\delta(x,y)\not=0,$
    then $\delta\leqslant \|x-y\|$ and $d(x,y)\leqslant (M+4)r.$ Thus,    the first term on the right side of \eqref{T2.5} satisfies

    \begin{eqnarray*}
        |p_{1,\delta}|&=& \bigg| \iint\limits_{\Omega} K(x,y)(1-\lambda_\delta(x,y))\chi_0(y)[\phi(y)-\phi(x)]\psi(x) d\omega(y)d\omega(x)\bigg| \\
        &\lesssim& \iint\limits_{d(x,y)\leqslant (M+4)r} \frac{1}{\omega(B(x,d(x,y)))} \Big(\frac{d(x,y)}{\|x-y\|}\Big)^\eta \Big(\frac{\|x-y\|}{r}\Big)^\eta|\psi(x)|
        d\omega(y)d\omega(x)\\
        &\lesssim& \|\psi\|_1.
    \end{eqnarray*}

    It remains to show that $\lim\limits_{\delta\to 0}p_{2,\delta}=0,$ that is,
    \begin{equation}\label{T2.6}
    \lim_{\delta\to 0}  \langle K(x,y), \lambda_\delta(x,y)\chi_0(y)[\phi(y)-\phi(x)]\psi(x) \rangle=0.
    \end{equation}
    To this end,  let $\{y_j\}_{j\in \mathbb Z}\in \mathbb R^N$ be the maximal collection of points satisfying
    \begin{equation}\label{T2.8}
    \frac12\delta<\inf\limits_{j\ne k} \|y_j-y_k\|\leqslant \delta.
    \end{equation}
    By observing that   $\{y_j\}_{j\in \mathbb Z}$ is a  maximal collection, we get that for each $x\in \mathbb R^N$ there exists a point $y_j$ such that $\|x-y_j\|\leqslant\delta$.
    Let $\eta_j(y)=\theta(\frac{\|y-y_j\|}\delta)$ and $\bar\eta_j(y)=[\sum\limits_{i=1}^\infty \eta_i(y)]^{-1}\eta_j(y)$. To see that $\bar \eta_j$ is well defined,
    it suffices to show that for any $y\in \mathbb R^N$, there are only finitely many $\eta_j$ with $\eta_j(y)\ne0$. This follows from the following fact:
    $\eta_j(y)\ne0 $ if and only if $\|y-y_j\|\leqslant 2\delta$ and hence this implies that $B(y_j,\delta)\subseteq B(y,4\delta)$. Inequality \eqref{T2.8} shows $B(y_j,\frac\delta4)\cap B(y_k, \frac\delta4)=\emptyset$
    for $j\ne k$ and hence there are at most $C_0$ points $y_j\in \mathbb R^N$ such that $B(y_j,\frac\delta4)\subseteq B(y, 4\delta)$. Now let $\Gamma=\{j: \bar\eta_j(y)\chi_0(y)\ne 0\}$.
    Then \#$\Gamma\leqslant Cr^N/\delta^N$ since $\supp$ $(\chi_0)\subseteq \bigcup\limits_{\sigma\in G}B(\sigma(x_0), 2r)$ and $\supp (\bar \eta_j)\subseteq B(y_j,2\delta)$.   We write
    $$\lambda_\delta(x,y)\chi_0(y)[\phi(y)-\phi(x)]\psi(x)
    =\sum\limits_{j\in \Gamma} \lambda_\delta(x,y)\bar\eta_j(y)\chi_0(y)[\phi(y)-\phi(x)]\psi(x),$$
    and
    \begin{align*}
    &\langle K(x,y), \lambda_\delta(x,y)\chi_0(y)[\phi(y)-\phi(x)]\psi(x) \rangle \\
    &=\sum\limits_{j\in \Gamma} \langle K(x,y), \lambda_\delta(x,y)\bar\eta_j(y)\chi_0(y)[\phi(y)-\phi(x)]\psi(x) \rangle.
    \end{align*}
    It is then easy to check that supp$(\lambda_\delta(x,y)\bar\eta_j(y)\chi_0(y)[\phi(y)-\phi(x)]\psi(x))\subseteq B(y_j,4\delta)\times B(y_j,2\delta)$ and
    $$\|\lambda_\delta(x,y)\bar\eta_j(y)\chi_0(y)[\phi(y)-\phi(x)]\psi(x)\|_\infty\leqslant C\delta^\eta,$$
    where $C$ is a constant depending only on $\theta,\phi,\psi,x_0,$ and $r$ but not on $\delta$ and $j$.

    We claim that
    \begin{equation}\label{T2.9}
    \|\lambda_\delta(\cdot,y)\bar\eta_j(y)\chi_0(y)[\phi(y)-\phi(\cdot)]\psi(\cdot)\|_\eta\lesssim 1,
    \end{equation}
    and
    \begin{equation}\label{T2.10}
    \|\lambda_\delta(x,\cdot)\bar\eta_j(\cdot)\chi_0(\cdot)[\phi(\cdot)-\phi(x)]\psi(x)\|_\eta\lesssim 1.
    \end{equation}

    Assuming \eqref{T2.9} and \eqref{T2.10} for the moment, since $T\in WBP$, we have
    \begin{align*}
    &|\langle K(x,y), \lambda_\delta(x,y)\chi_0(y)[\phi(y)-\phi(x)]\psi(x) \rangle|\\
    &\quad\leqslant \sum\limits_{j\in \Gamma} |\langle K(x,y), \lambda_\delta(x,y)\bar\eta_j(y)\chi_0(y)[\phi(y)-\phi(x)]\psi(x) \rangle | \\
    &\quad \leqslant \sum\limits_{j\in \Gamma} \omega( B(y_j,4\delta))\delta^\eta
    \lesssim \frac{r^N}{\delta^N}\sup_{j\in \Gamma}\omega(B(y_j,1))\delta^N\delta^\eta
    \lesssim \sup_{j\in \Gamma}\omega(B(y_j,1))r^N\delta^\eta.
    \end{align*}
    hence, \eqref{T2.6} holds.

    It remains to show \eqref{T2.9} nad \eqref{T2.10}. To check \eqref{T2.9}, it suffices to show that for given $x_1,x_2\in \mathbb R^N$ with $\|x_1-x_2\|\leqslant \delta$,
    $$|\bar\eta_j(y)\chi_0(y)| |\lambda_\delta(\boldsymbol{x_1},y)[\phi(y)-\phi(\boldsymbol{x_1})]\psi(\boldsymbol{x_1})-
    \lambda_\delta(\boldsymbol{x_2},y)[\phi(y)-\phi(\boldsymbol{x_2})]\psi(\boldsymbol{x_2})|\lesssim \|x_1-x_2\|^\eta,$$
    since if $\|x_1-x_2\|\geqslant \delta$, then the expansion on the left above is clearly bounded by
    \begin{align*}
    &|\bar\eta_j(y)\chi_0(y)| \{|\lambda_\delta(x_1,y)[\phi(y)-\phi(x_1)]\psi(x_1)|+
    |\lambda_\delta(x_2,y)[\phi(y)-\phi(x_2)]\psi(x_2)|\}\\
    &\lesssim \delta^\eta \leqslant \|x_1-x_2\|^\eta.
    \end{align*}
    By the construction of $\bar\eta_j$, it follows that
    $$|\bar\eta_j(y)\chi_0(y)|\lesssim 1$$
    for all $y\in \mathbb R^N.$ Thus
    \begin{align*}
    &|\bar\eta_j(y)\chi_0(y)| |\lambda_\delta(x_1,y)[\phi(y)-\phi(x_1)]\psi(x_1)-
    \lambda_\delta(x_2,y)[\phi(y)-\phi(x_2)]\psi(x_2)| \\
    &\quad\lesssim \big|\lambda_\delta(x_1,y)[\phi(y)-\phi(x_1)]\psi(x_1)-
    \lambda_\delta(x_2,y)[\phi(y)-\phi(x_2)]\psi(x_2)\big|  \\
    &\quad\lesssim  \big|\lambda_\delta(x_1,y)-\lambda_\delta(x_2,y)|[\phi(y)-\phi(x_1)]\psi(x_1)\big|
    + \big|\lambda_\delta(x_2,y)[\phi(x_1)-\phi(x_2)]\psi(x_1)\big|\\
    &\qquad    +\big|\lambda_\delta(x_2,y)[\phi(y)-\phi(x_2)]|\psi(x_1)-\psi(x_2)\big| \\
    &\quad=: I_1 +I_2 +I_3.
    \end{align*}
    Recall that $\|x_1-x_2\|\leqslant \delta$. If $\|x_1-y\|> 4\delta$, then $\lambda_\delta(x_1,y)=\lambda_\delta(x_2,y)=0$, so $I_1=0$. Thus we may assume that
    $\|x_1-y\|\leqslant 4\delta$,
    $$I_1\lesssim \Bigg|\frac{\|x_1-y\|}\delta-\frac{\|x_2-y\|}\delta\Bigg| \|x_1-y\|^\eta\lesssim \delta^{\eta-1}\|x_1-x_2\|
    \lesssim \|x_1-x_2\|^\eta,
    $$
    since we may assume $\eta\le1$. Terms $I_2$ and $I_3$ are easy to estimate:
    $$I_2+I_3 \lesssim  \|x_1-x_2\|^\eta,$$
    since we may assume that $\delta<1$.

    To check \eqref{T2.10} it suffices to show that for $y_1,y_2\in \mathbb R^N$ with $\|y_1-y_2\|\leqslant \delta$,
    $$|\lambda_\delta(x,\boldsymbol{y_1})\bar\eta_j(\boldsymbol{y_1})\chi_0(\boldsymbol{y_1})[\phi(\boldsymbol{y_1})-\phi(x)]\psi(x)-
    \lambda_\delta(x,\boldsymbol{y_2})\bar\eta_j(\boldsymbol{y_2})\chi_0(\boldsymbol{y_2})[\phi(\boldsymbol{y_2})-\phi(x)]\psi(x)|\lesssim |y_1-y_2|^\eta.$$
    Similarly, if $\|y_1-y_2\|\geqslant \delta$, then the expansion  on the left-hand side above is clearly bounded by
    \begin{align*}
    &|\lambda_\delta(x,y_1)\bar\eta_j(y_1)\chi_0(y_1)[\phi(y_1)-\phi(x)]\psi(x)|+
    |\lambda_\delta(x,y_2)\bar\eta_j(y_2)\chi_0(y_2)[\phi(y_2)-\phi(x)]\psi(x)|\\
    &\quad\lesssim \delta^\eta \leqslant \|y_1-y_2\|^\eta.
    \end{align*}
    Hence, suppose $\|y_1-y_2\|\leqslant \delta$ and write
    \begin{align*}
    &|\lambda_\delta(x,y_1)\bar\eta_j(y_1)\chi_0(y_1)[\phi(y_1)-\phi(x)]\psi(x)-
    \lambda_\delta(x,y_2)\bar\eta_j(y_2)\chi_0(y_2)[\phi(y_2)-\phi(x)]\psi(x)|\\
    &\leqslant \big|\lambda_\delta(x,y_1)-\lambda_\delta(x,y_2)|
    \bar\eta_j(y_1)\chi_0(y_1)[\phi(y_1)-\phi(x)]\psi(x)\big| \\
    &\quad+ \big|\lambda_\delta(x,y_2)[\bar\eta_j(y_1)-\bar\eta_j(y_2)]\chi_0(y_1)[\phi(y_1)-\phi(x)]\psi(x)\big| \\
    &\quad + \big|\lambda_\delta(x,y_2)\bar\eta_j(y_2)[\chi_0(y_1)-\chi_0(y_2)][\phi(y_1)-\phi(x)]\psi(x)\big| \\
    &\quad  + \big|\lambda_\delta(x,y_2)\bar\eta_j(y_2)\chi_0(y_2)[\phi(y_1)-\phi(y_2)]\psi(x)\big|\\
    &=:  I\!I_1+I\!I_2+I\!I_3+I\!I_4.
    \end{align*}
    If $\|x-y_1\|> 4\delta$, then $\lambda_\delta(x,y_1)=\lambda_\delta(x,y_2)=0$, so $I\!I_1=I\!I_2=I\!I_3=I\!I_4=0$. Thus we may assume that
    $\|x-y_1\|\leqslant 4\delta$,
    $$I\!I_1\lesssim \Bigg|\frac{\|x-y_1\|}\delta-\frac{\|x-y_2\|}\delta\Bigg| \|x-y_1\|^\eta\lesssim \delta^{\eta-1}\|y_1-y_2\|
    \lesssim \|y_1-y_2\|^\eta. $$
    And
    $$I\!I_2\lesssim  \Bigg|\frac{\|y_1-y_j\|}\delta-\frac{\|y_2-y_j\|}\delta\Bigg| \|y_1-x\|^\eta\lesssim \delta^{\eta-1}\|y_1-y_2\|
    \lesssim \|y_1-y_2\|^\eta.$$
    Similarly,
    $$I\!I_3\lesssim  \Bigg|\frac{d(y_1,x_0)}\delta-\frac{d(y_2,x_0)}\delta\Bigg| \|y_1-x\|^\eta\lesssim \delta^{\eta-1}d(y_1,y_2)\lesssim \delta^{\eta-1}\|y_1-y_2\|
    \lesssim \|y_1-y_2\|^\eta.$$
    It is clear that
    $$I\!I_4\lesssim \|y_1-y_2\|^\eta.$$
    This completes the proofs of \eqref{T2.9} and \eqref{T2.10} and we obtain
    $$|p|\lesssim \|\psi\|_{1}.$$

    To finish the proof of Lemma \ref{lem2.5}, we now estimate $q$. It suffices to show that for $x\in B(x_0,r)$,
    \begin{equation}\label{T2.11}
    |T\chi_0(x)|\lesssim 1.
    \end{equation}
    To see this, it is easy to check that $q=\langle T\chi_0, \phi\psi\rangle$, and hence \eqref{T2.11} implies
    $$|q|\leqslant \|T\chi_0\|_{L^\infty(B(x_0,r))}\|\phi\psi\|_{L^1(B(x_0,r))}\lesssim \|\psi\|_{1}.$$
    To show \eqref{T2.11}, let $\psi\in C^\eta(\mathbb R^N)$ with supp$(\psi) \subseteq B(x_0,r)$
    and $\int_{\R^N}\psi(x)d\omega(x)=0$. By the facts that $T(1)=0$ and $\int_{\R^N}\psi(x)d\omega(x)=0$, we obtain
    \begin{align*}
    |\langle T\chi_0, \psi\rangle| &=|- \langle T\chi_1,\psi\rangle|\\
    &=\Big|\int_{\R^N}\int_{\R^N} [K(x,y)-K(x_0,y)]\chi_1(y)\psi(x)d\omega(y)d\omega(x)\Big|.
    \end{align*}
    Observe that the supports of $\chi_1(y)$ and $\psi(x)$ imply $d(y,x_0)> 2r$ and $\|x-x_0\|\leqslant r,$ respectively. The smoothness condition of $K$ yields
    \begin{align*}
    |\langle T\chi_0, \psi\rangle| &\lesssim \iint\limits_{d(y,x_0)> 2r\geqslant 2\|x-x_0\|}\frac{1}{\omega(B(x,d(y,x_0)))} \Big(\frac{\|x-x_0\|}{\|y-x_0\|}\Big)^\varepsilon d\omega(y)|\psi(x)|d\omega(x)\\
    &\lesssim \iint\limits_{d(y,x_0)> 2r\geqslant 2\|x-x_0\|}\frac{1}{\omega(B(x,d(y,x_0)))} \Big(\frac{r}{d(y,x_0)}\Big)^\varepsilon d\omega(y)|\psi(x)|d\omega(x)\\
    &\lesssim  \int_{\R^N} |\psi(x)|d\omega(x).
    \end{align*}
    This implies that $T\chi_0(x)=\alpha+\gamma(x)$ for $x\in B(x_0,r)$ with $\alpha$ is a constant depending on $\chi_0$ and $\|\gamma(x)\|_\infty\leqslant C_0$ for some constant $C_0$ independent of $\chi_0.$
    To estimate $\alpha$, choose $\varphi
    \in C_0^\eta(\mathbb{R}^N)$ with supp $\varphi\subseteq B(x_0,r),  0\leqslant \varphi\leqslant 1, \|\varphi\|_\eta\leqslant r^{-\eta}$
    and $\int_{\R^N} \varphi(x)d\omega(x)=C_1 \omega(B(x_0,r)),$ for some constant $C_1$ independent of $r.$ We then use $T\in WBP$ to get
    $$\bigg|C_1\omega(B(x_0,r))\alpha + \int_{\R^N} \varphi(x)\gamma(x)d\omega(x)\bigg|=|\langle T\chi_0,\varphi\rangle |\leqslant C \omega(B(x_0,r)),$$
    which implies $|\alpha|\leqslant C_0+\frac{C}{C_1}$ and hence, the proof of Lemma \ref{lem2.5} is complete.
\end{proof}

We remark that if the operator $T$ and functions $\phi, \chi_0$ satisfy the conditions as given in the Lemma \ref{lem2.5},
then $T\phi(x)$ is a locally bounded function rather than a distribution. This fact will play a crucial role in the following proof of Theorem \ref{1.3}.

\subsection{{Boundedness of Dunkl-Calder\'on-Zygmund oprators on smooth molecule functions}}
\ \\

\begin{proof}[{\bf Proof of Theorem \ref{1.3}}]
    Suppose that $f(x)$ is a smooth molecule in $\mathbb M(\beta, \gamma, r, x_0),$
    we will show that $\|T(f)\|_{\widetilde{\mathbb M}(\beta, \gamma', r, x_0)}\leqslant C\|f\|_{{\mathbb M}(\beta, \gamma, r,
    x_0)},$where $0<\beta<\varepsilon,$$0<\gamma<\gamma'<\varepsilon$ and $\varepsilon$ is the exponent of the regularity of the kernel of $T.$
    We first estimate the size condition for $Tf(x).$ To this end, we consider two cases: Case (1): $d(x,x_0)\leqslant 5r$ and
    Case (2): $d(x,x_0)=R> 5r.$

    For the first case, set $1=\xi(y)+\eta(y)$, where $\xi(y)=\theta\Big(\frac{d(y,x_0)}{10r}\Big)$ with $\theta\in C_0^\infty(\mathbb{R}), \theta(x)=1$ for $\|x\|\leqslant 1$ and
    $\theta(x)=0$ for $\|x\|\geqslant 2.$ Applying the Lemma \ref{lem2.5}, we write
    \begin{align*}
    Tf(x)
    &= \langle K(x,y),(\xi(y)+\eta(y))f(y)\rangle
    \\&=  \int_{\R^N} K(x,y)\xi(y)(f(y)-f(x))d\omega(y)+f(x)\langle
    K(x,y),\xi(y)\rangle
    \\&\qquad +\int_{\R^N} K(x,y)\eta(y)f(y)d\omega(y)
    \\&=:  I_1+I_2+I_3.
    \end{align*}

    Applying the size condition for the kernel $K(x,y)$ in \eqref{si} and the smoothness condition for $f$ in \eqref{sm of DCZ}, we have
    \begin{align*}|I_1|
    &\lesssim  \int_{d(x,y)\leqslant 20r} |K(x,y)|\cdot|f(y)-f(x)|d\omega(y)
    \\&\lesssim  \int_{d(x,y)\leqslant 20r} \frac{1}{\omega(B(x,d(x,y)))}\Big(\frac{d(x,y)}{\|x-y\|}\Big)^\beta\Big(\frac{\|x-y\|}{r}\Big)^\beta\Big\{\frac{1}{V(\boldsymbol{y},x_0,r+d(\boldsymbol{y},x_0))}
    \\&\qquad\qquad\times  \Big(\frac{r}{r+\|\boldsymbol{y}-x_0\|}\Big)^\gamma+ \frac{1}{V(\boldsymbol{x},x_0,r+d(\boldsymbol{x},x_0))}\Big(\frac{r}{r+\|\boldsymbol{x}-x_0\|}\Big)^\gamma\Big\} d\omega(y).
    \end{align*}

    Note that if $d(y,x)\leqslant 20r$ and $d(x,x_0)\leqslant 5r,$ then $\omega(B(y,r+d(x,x_0)))\sim  \omega(B(x,r+d(x,x_0)))$. Thus, we obtain
    \begin{align*}|I_1|
    &\lesssim  \frac{1}{r^\beta}\frac{1}{V(x,x_0,r+d(x,x_0))}\int_{d(x,y)\leqslant 20r}\frac{1}{\omega(B(x,d(x,y)))}d(x,y)^\beta d\omega(y)
    \\&\lesssim  \frac{1}{V(x,x_0,r+d(x,x_0))}
    \\&\lesssim  \frac{1}{V(x,x_0,r+d(x,x_0))}\Big(\frac{r}{r+d(x,x_0)}\Big)^\gamma.
    \end{align*}

    Similar to the proof of \eqref{T2.11} in Lemma \ref{lem2.5}, we
    can get $|T(\xi)(x)|\lesssim 1$ and thus
    $$I_2\lesssim |f(x)|\lesssim  \frac{1}{V(x,x_0,r+d(x,x_0))}\Big(\frac{r}{r+d(x,x_0)}\Big)^\gamma.$$

    For the last term $I_3$, observing that $d(x,x_0)\leqslant 5r$ and the support of $\eta(y)$ is contained in $\{y\mid d(y,x_0)\geqslant 10r\},$ so $d(x,y)\geqslant 5r$ and $d(x,y)\sim d(y,x_0),$ and thus,
    \begin{align*}|I_3|
    &\lesssim  \int_{d(y,x_0)\geqslant 10 r\atop d(x,y)\geqslant 5r} \frac{1}{\omega(B(x,d(x,y)))}\frac{1}{V(y,x_0,r+d(y,x_0))}\Big(\frac{r}{r+\|y-x_0\|}\Big)^\gamma d\omega(y)
    \\&\lesssim  \frac{1}{\omega(B(x,r))} \int_{d(y,x_0)\geqslant 10 r} \frac{1}{\omega(B(x_0,d(y,x_0)))}\Big(\frac{r}{d(y,x_0)}\Big)^\gamma d\omega(y)
    \\&\lesssim  \frac{1}{\omega(B(x,r))}        \lesssim \frac{1}{V(x,x_0,r+d(x,x_0))}\Big(\frac{r}{r+d(x,x_0)}\Big)^\gamma.
    \end{align*}

    Case 2. $d(x,x_0)=R > 5r$.

    Set $1=I(y)+J(y)+L(y)$, where $I(y)=\theta\big(\frac{16d(y,x)}{R}\big), J(y)=\theta\big(\frac{16d(y,x_0)}{R}\big)$ and $f_1(y)=I(y)f(y),f_2(y)=J(y)f(y)$ and
    $f_3(y)=L(y)f(y).$

    Observing that, if $y$ is in the support of $f_1(y),$ then $d(y,x_0)\sim d(x,x_0)=R,$ and thus,
    \begin{align*}
    \textup{(i)}&\ |f_1(y)|  \lesssim |I(y)|\frac{1}{V(y,x_0,r+d(y,x_0))}\Big(\frac{r}{r+\|y-x_0\|}\Big)^\gamma\lesssim \frac{1}{V(x,x_0,R)}\Big(\frac{r}{R}\Big)^\gamma.
    \\
    \textup{(ii)}&\ \int_{\R^N} |f_1(y)|d\omega(y)  \lesssim  \int_{d(y,x_0)\geqslant \frac{7R}{8}}\frac{1}{V(y,x_0,d(y,x_0))}\Big(\frac{r}{d(y,x_0)}\Big)^\gamma d\omega(y)
    \lesssim \Big(\frac{r}{R}\Big)^\gamma.\\
    \textup{(iii)}&\ |f_1(y)-f_1(x)|\lesssim  \Big(\frac{\|y-x\|}{r}\Big)^\beta\frac{1}{V(x,x_0,d(x,x_0))}\Big(\frac{r}{R}\Big)^\gamma.
    \\
    \textup{(iv)}&\ \int_{\R^N} |f_3(y)|\omega(y)dy  \lesssim   \int_{d(y,x_0)\geqslant \frac{R}{16}}\frac{1}{V(y,x_0,r+d(y,x_0))}\Big(\frac{r}{d(y,x_0)}\Big)^\gamma d\omega(y)
    \lesssim \Big(\frac{r}{R}\Big)^\gamma.
    \end{align*}

    By the fact $\int_{\R^N} f(y)d\omega(y)=0$, we have
    \begin{align*}
    \textup{(v)}\ \Big|\int_{\R^N} f_2(y)d\omega(y)\Big|  =\Big|-\int_{\R^N} f_1(y)d\omega(y) -\int_{\R^N} f_3(y)d\omega(y) \Big|
    \lesssim \Big(\frac{r}{R}\Big)^\gamma.\ \ \ \ \ \ \ \ \ \ \ \ \ \ \ \ \ \
    \end{align*}

    We first estimate $Tf_1(x)$ as follows.

    Set $u(y)=\theta\Big(\frac{2d(y,x)}{R}\Big)$. Then $f_1(y)=u(y)f_1(y).$ By using Lemma \ref{lem2.5}, we have
    \begin{align*}Tf_1(x)
    &=  \langle K(x,y)u(y)f_1(y)\rangle
    \\ &=  \int_{\R^N} K(x,y)u(y)[f_1(y)-f_1(x)]d\omega(y)+f_1(x)\langle
    K(x,\cdot),u(\cdot)\rangle
    \\&=:  I+I\!I.
    \end{align*}

    Similar to the proof of \eqref{T2.11} in Lemma \ref{lem2.5}, we
    can get $|T(u)(x)|\lesssim 1$ and thus
    $$ |I\!I|\lesssim |f(x)|\lesssim \frac{1}{V(x,x_0,r+d(x,x_0))}\Big(\frac{r}{r+d(x,x_0)}\Big)^\gamma.$$

    For the term $I$, we write it in two parts.
    \begin{align*}I&=\int_{d(x,y)\leqslant r} K(x,y)u(y)[f_1(y)-f_1(x)]d\omega(y)+\int_{r<d(x,y)\leqslant R} K(x,y)u(y)[f_1(y)-f_1(x)]d\omega(y)
    \\&=:I_1+I_2.
    \end {align*}

    Applying the size condition on the kernel $K(x,y)$ and the property \textup{(iii)} above implies that
    \begin{align*}|I_1|
    &\lesssim  \int_{d(x,y)\leqslant r} \frac{1}{\omega(B(x,d(x,y)))}\Big(\frac{d(x,y)}{\|x-y\|}\Big)^\beta \Big(\frac{\|x-y\|}{r}\Big)^\beta\frac{1}{V(x,x_0,d(x,x_0))}\Big(\frac{r}{R}\Big)^\gamma  d\omega(y)
    \\&= \frac{1}{V(x,x_0,d(x,x_0))}\Big(\frac{r}{R}\Big)^\gamma\int_{d(x,y)\leqslant r} \frac{1}{\omega(B(x,d(x,y)))}\Big(\frac{d(x,y)}{r}\Big)^\beta d\omega(y)
    \\&\lesssim  \frac{1}{V(x,x_0,r+d(x,x_0))}\Big(\frac{r}{r+d(x,x_0)}\Big)^\gamma.
    \end{align*}

    Applying the size conditions for the kernel $K(x,y)$ and property \textup{(i)} above, we obtain that for $\delta =\gamma - \gamma',$
    \begin{align*}|I_2|
    &\lesssim  \int_{r<d(x,y)\leqslant \frac{R}{4}} \frac{1}{\omega(B(x,d(x,y)))}\Big(\frac{d(x,y)}{\|y-x\|}\Big)^\delta[|f_1(y)|+|f_1(x)|]d\omega(y)
    \\&\lesssim  \frac{1}{V(x,x_0,d(x,x_0))}\Big(\frac{r}{R}\Big)^\gamma \Big(\frac{1}{r}\Big)^\delta \int_{d(x,y)\leqslant \frac{R}{4}} \frac{1}{\omega(B(x,d(x,y)))}d(x,y)^\delta d\omega(y)
    \\&\lesssim  \Big(\frac{R}{r}\Big)^\delta\Big(\frac{r}{R}\Big)^\gamma\frac{1}{V(x,x_0,d(x,x_0))}
    \lesssim  \frac{1}{V(x,x_0,d(x,x_0))}\Big(\frac{r}{r+d(x,x_0)}\Big)^{\gamma'}.
    \end{align*}

    To estimate $Tf_2(x)$,  we decompose it in two parts.
    \begin{eqnarray*}Tf_2(x)
        =  \int_{\R^N} [K(x,y)-K(x,x_0)]f_2(y)d\omega(y)dy+K(x,x_0)\int_{\R^N} f_2(y)d\omega(y)
        =:  I\!I_1+I\!I_2.
    \end{eqnarray*}

    By the estimate in (v) above,
    \begin{eqnarray*}|I\!I_2|
        \lesssim  |K(x,x_0)|\Big(\frac{r}{R}\Big)^{\gamma}
        \lesssim   \frac{1}{\omega(B(x,d(x,x_0)))}\Big(\frac{r}{R}\Big)^{\gamma}
        \lesssim  \frac{1}{V(x,x_0,d(x,x_0))}\Big(\frac{r}{r+d(x,x_0)}\Big)^\gamma.
    \end{eqnarray*}

    For the term $I\!I_1$, we write it by
    \begin{align*}I\!I_1
    &=\bigg(  \int_{\|y-x_0\|\leqslant \frac{R}{4}} +\int_{d(y,x_0)\leqslant \frac{R}{4}\leqslant \|y-x_0\|}\bigg)[K(x,y)-K(x,x_0)]f_2(y)d\omega(y)
    \\&=: I\!I_{11}+I\!I_{12}.
    \end{align*}

    Applying the size condition for $f_2$ and the smoothness condition on the kernel $K(x,y)$ in \eqref{smooth y3} with $\|y-x_0\|\leqslant \frac{1}{2}d(x,x_0)$ for term $I\!I_{11}$ implies that
    \begin{align*}|I\!I_{11}|
    &\lesssim  \int_{d(y,x_0)\leqslant \frac{R}{8}} \frac{1}{\omega(B(x,d(x,x_0)))}\Big(\frac{\|y-x_0\|}{\|x-x_0\|}\Big)^{\gamma'}  \Big(\frac{r}{r+\|y-x_0\|}\Big)^\gamma\frac{1}{V(y,x_0,r+d(y,x_0))}d\omega(y)
    \\&\lesssim  \frac{1}{\omega(B(x,d(x,x_0)))}\Big(\frac{r}{R}\Big)^{\gamma'}\int_{d(y,x_0)\leqslant \frac{R}{8}} \Big(\frac{r}{r+ \|y-x_0\|}\Big)^{\gamma-\gamma'}\frac{1}{V(y,x_0,r+d(y,x_0))}d\omega(y)
    \\&\lesssim  \frac{1}{V(x,x_0,d(x,x_0))}\Big(\frac{r}{r+d(x,x_0)}\Big)^{\gamma'}.
    \end{align*}

    For the term $I\!I_{12},$ since $d(y,x_0)\leqslant \frac{R}{4}$ implies $d(x,y)\geqslant d(x,x_0)-d(y,x_0)\geqslant \frac{3}{4}d(x,x_0)$. Applying the size conditions for the kernel $K(x,y)$ and $K(x,x_0)$ yields
    \begin{align*}|I\!I_{12}|
    &\lesssim \int_{d(y,x_0)\leqslant \frac{R}{4}\leqslant \|y-x_0\|}\Big\{\frac{1}{\omega(B(x,d(x,y)))} +\frac{1}{\omega(B(x,d(x,x_0)))}\Big\}
    \\ & \times  \Big(\frac{r}{r+\|y-x_0\|}\Big)^\gamma\frac{1}{V(y,x_0,r+d(y,x_0))}d\omega(y)
    \\&\lesssim  \frac{1}{\omega(B(x,d(x,x_0)))}\Big(\frac{r}{R}\Big)^{\gamma'}\int_{R^N} \Big(\frac{r}{r+d(y,x_0)}\Big)^{\gamma-\gamma'}\frac{1}{V(y,x_0,r+d(y,x_0))}d\omega(y)
    \\&\lesssim  \frac{1}{V(x,x_0,d(x,x_0))}\Big(\frac{r}{r+d(x,x_0)}\Big)^{\gamma'}.
    \end{align*}

    Finally,
    \begin{align*}|Tf_3(y)|
    &\lesssim  \int_{ d(y,x)\geqslant \frac{R}{16}, \atop d(y,x_0)\geqslant \frac{R}{16}}\frac{1}{\omega(B(x,d(x,y)))}
    \Big(\frac{r}{r+\|y-x_0\|}\Big)^\gamma\frac{1}{V(y,x_0,d(y,x_0))}d\omega(y)
    \\&\lesssim   \frac{1}{\omega(B(x,d(x,x_0)))} \int_{  d(y,x_0)\geqslant \frac{R}{16}}
    \Big(\frac{r}{d(y,x_0)}\Big)^\gamma\frac{1}{V(y,x_0,d(y,x_0))}d\omega(y)
    \\&\lesssim  \frac{1}{\omega(B(x,d(x,x_0)))} \Big(\frac{r}{R}\Big)^\gamma
    \\&\lesssim  \frac{1}{V(x,x_0,d(x,x_0))}\Big(\frac{r}{r+d(x,x_0)}\Big)^{\gamma}.
    \end{align*}

    It remains to show the regularity of $T(f),$ that is the following estimate:

    \begin{align*}
    |Tf(x)-Tf(x')|&\lesssim
    \big(\frac{\|x-x'\|}{r}\big)^\beta \Big\{\frac{1}{V(\boldsymbol{x},x_0,r+d(\boldsymbol{x},x_0))}\Big(\frac{r}{r+d(\boldsymbol{x},x_0)}\Big)^{\gamma'}\\
    &\qquad\qquad  + \frac{1}{V(\boldsymbol{x'},x_0,r+d(\boldsymbol{x'},x_0))}\Big(\frac{r}{r+d(\boldsymbol{x'},x_0)}\Big)^{\gamma'}\Big\}.
    \end{align*}

    Observing that we only need to consider the case where $\|x-x'\|\leqslant \frac{1}{20}r.$ Indeed, if $\|x-x'\|\geqslant \frac{1}{20}r,$ by the size estimate of $T(f),$
    \begin{align*}
    |Tf(x)-Tf(x')|&\leqslant |Tf(x)|+|Tf(x')|\\
    &\lesssim \frac{1}{V(\boldsymbol{x},x_0,r+d(\boldsymbol{x},x_0))}\Big(\frac{r}{r+d(\boldsymbol{x},x_0)}\Big)^{\gamma'}\\
    &\qquad\qquad + \frac{1}{V(\boldsymbol{x'},x_0,r+d(\boldsymbol{x'},x_0))}\Big(\frac{r}{r+d(\boldsymbol{x'},x_0)}\Big)^{\gamma'},
    \end{align*}
    which gives the desired regularity estimate of $T(f).$

    Set $\|x-x'\|=\delta\leqslant \frac{1}{20}r.$ We will consider it in the following two cases: $d(x,x_0)=R\geqslant 10r$ and $d(x,x_0)< 10r.$

    Case (1): $d(x,x_0)=R\geqslant 10r.$
    Let $I(y)=\theta(\frac{8d(y,x)}{R}), J(y)=1-I(y).$ Denote $f_1(y)=I(y)f(y), f_2(y)=J(y)f(y).$ Write

    \begin{align*}
    Tf_1(x)&=            \int_{\R^N} K(x,y)u(y)[f_1(y)-f_1(x)]d\omega(y)\\
    &\qquad+ \int_{\R^N} K(x,y)v(y)f_1(y)d\omega(y)+\int_{\R^N} K(x,y)u(y)f_1(x)d\omega(y),
    \end{align*}
    where $u(y)=\theta(\frac{d(y,x)}{2\delta})$ and $v(y)=1- u(y).$

    Let $p(x)=\int_{\R^N} K(x,y)u(y)[f_1(y)-f_1(x)]d\omega(y)$ and $q(x)= \int_{\R^N} K(x,y)v(y)f_1(y)d\omega(y)+\int_{\R^N} K(x,y)u(y)f_1(x)d\omega(y)$.
    Then we have
    \begin{align*}
    |p(x)|&\lesssim \int_{d(x,y)\leqslant 4\delta}|K(x,y)|\cdot|f_1(y)-f_1(x)|d\omega(y)\\
    &\lesssim \int_{d(x,y)\leqslant 4\delta} \frac{1}{\omega(B(x,d(y,x)))}\Big(\frac{d(x,y)}{\|x-y\|}\Big)^\beta\Big(\frac{\|x-y\|}{r}\Big)^\beta \\ &
    \times\Big\{\frac{1}{V(\boldsymbol{y},x_0,r+d(\boldsymbol{y},x_0))}\Big(\frac{r}{r+d(\boldsymbol{y},x_0)}\Big)^\gamma
    + \frac{1}{V(\boldsymbol{x},x_0,r+d(\boldsymbol{x},x_0))}\Big(\frac{r}{r+d(\boldsymbol{x},x_0)}\Big)^\gamma\Big\}d\omega(y)\\
    &\lesssim\frac{1}{r^\beta} \frac{1}{V(x,x_0,r+d(x,x_0))}\Big(\frac{r}{r+d(x,x_0)}\Big)^\gamma\int_{d(x,y)\leqslant 4\delta} \frac{1}{\omega(B(x,d(y,x)))}\big(d(x,y)\big)^\beta d\omega(y) \\
    &\lesssim \Big(\frac{\delta}{r}\Big)^\beta\frac{1}{V(x,x_0,r+d(x,x_0))}\Big(\frac{r}{r+d(x,x_0)}\Big)^\gamma,
    \end{align*}
    since $d(y,x)\leqslant \frac{1}{4}R$ and $d(x,x_0)=R,$ so $d(y,x_0)\sim d(x,x_0).$

    If replacing $x$ by $x'$, we still have

    \begin{eqnarray*}
        |p(x')|\lesssim \Big(\frac{\delta}{r}\Big)^{\beta}\Big(\frac{r}{r+d(x,x_0)}\Big)^\gamma\frac{1}{V(x,x_0,r+d(x,x_0))}.
    \end{eqnarray*}
    Therefore,
    $$|p(x)-p(x')|\lesssim \Big(\frac{\delta}{r}\Big)^{\beta}\Big(\frac{r}{r+d(x,x_0)}\Big)^\gamma\frac{1}{V(x,x_0,r+d(x,x_0))}.$$
    Observing that by $T(1)=0,$ we can write

    \begin{align*}
    q(x)-q(x')&=\int_{d(x,y)\geqslant 2\delta } [K(x,y)-K(x',y)]v(y)[f_1(y)-f_1(x)]d\omega(y)
    \\&\qquad+[f_1(x)-f_1(x')]\int_{\R^N} K(x',y)u(y)d\omega(y)=: I+I\!I.
    \end{align*}
    Similar to the proof of \eqref{T2.11} in Lemma \ref{lem2.5}, we
    can get $|T(u)(x')|\lesssim 1$ and thus
    \begin{align*}
    I\!I&\lesssim |f_1(x)-f_1(x')|\\
    &\lesssim \Big(\frac{\delta}{r}\Big)^\beta\Big\{\frac{1}{V(\boldsymbol{x},x_0,r+d(\boldsymbol{x},x_0))}\Big(\frac{r}{r+d(\boldsymbol{x},x_0)}\Big)^\gamma\\
    &\qquad \qquad\qquad+\frac{1}{V(\boldsymbol{x'},x_0,r+d(\boldsymbol{x'},x_0))}\Big(\frac{r}{r+d(\boldsymbol{x'},x_0)}\Big)^\gamma\Big\}.
    \end{align*}
    For term I, applying the smoothness condition of $K(x,y)$ with $\|x-x'\|= \delta\leqslant \frac{1}{2}d(x,y)$ and the smoothness condition for $f_1$ implies that
    \begin{align*}
    |I|&\lesssim\int_{d(x,y)\geqslant 2\delta}\frac{1}{\omega(B(x,d(y,x)))}\Big(\frac{\|x-x'\|}{\|x-y\|}\Big)^\varepsilon
    \Big(\frac{\|y-x\|}{r}\Big)^\beta\frac{1}{V(x,x_0,d(x,x_0))}\Big(\frac{r}{R}\Big)^\gamma d\omega(y)\\
    &\lesssim \frac{\delta^\varepsilon}{r^\beta} \Big(\frac{r}{r+d(x,x_0)}\Big)^\gamma\frac{1}{V(x,x_0,r+d(x,x_0))}\int_{d(x,y)\geqslant 2\delta}\frac{1}{\omega(B(x,d(y,x)))}\frac{1}{(d(x,y))^{\varepsilon-\beta}}d\omega(y) \\
    &\lesssim \Big(\frac{\delta}{r}\Big)^\beta \frac{1}{V(x,x_0,r+d(x,x_0))} \Big(\frac{r}{r+d(x,x_0)}\Big)^\gamma,
    \end{align*}
    since $d(y,x_0)\sim d(x,x_0).$
    The estimates of $I$ and $I\!I$ gives the desired estimate for $Tf_1(x)-Tf_1(x').$

    To see the estimate for $Tf_2(x)-Tf_2(x'),$ note that if $f_2(y)\not=0,$ then $d(x,y)\geqslant \frac{1}{8}R\geqslant 2\|x-x'\|.$
    Therefore,
    \begin{align*}
    &|Tf_2(x)-Tf_2(x')|\\&\leqslant \int_{d(y,x)\geqslant \frac{3}{4}R\geqslant 2\delta}  |K(x,y)-K(x',y)|\cdot |f_2(y)| d\omega(y)\\
    &\lesssim \int_{d(y,x)\geqslant \frac{3}{4}R}   \frac{1}{\omega(B(x,d(y,x)))}\Big(\frac{\|x-x'\|}{\|x-y\|}\Big)^\varepsilon \frac{1}{V(y,x_0,r+d(y,x_0))}\Big(\frac{r}{r+d(y,x_0)}\Big)^\gamma d\omega(y)\\
    &\lesssim \Big(\frac{\delta}{R}\Big)^\varepsilon \frac{1}{\omega(B(x,d(x,x_0)))}
    \int_{\R^N} \frac{1}{V(y,x_0,r+d(y,x_0))}\Big(\frac{r}{r+d(y,x_0)}\Big)^\gamma d\omega(y)\\
    &\lesssim \Big(\frac{\delta}{r}\Big)^\varepsilon \Big(\frac{r}{r+d(x,x_0)}\Big)^\varepsilon \frac{1}{\omega(B(x,d(x,x_0)))}\lesssim \Big(\frac{\delta}{r}\Big)^\beta \frac{1}{V(x,x_0,r+d(x,x_0))}\Big(\frac{r}{r+d(x,x_0)}\Big)^\gamma.
    \end{align*}

    Cases 2: $d(x,x_0)< 10r.$ The proof of this case is easier. Indeed, set $1=\xi(y)+\eta(y)$, where $\xi(y)=\theta\Big(\frac{d(y,x)}{5\delta}\Big)$ and again write $Tf(x)=p(x)+q(x),$
    where $p(x)=\int_{\R^N} K(x,y)[f(y)-f(x)]\xi(y)d\omega(y)$ and
    $$q(x)=\int_{\R^N} K(x,y)f(y)\eta(y)d\omega(y)+f(x)\int_{\R^N} K(x,y)\xi(y)d\omega(y).$$
    Applying the size condition for $K(x,y)$ and the smoothness condition for $f$ implies that
    \begin{align*}
    |p(x)|&\lesssim \int_{d(x,y)\leqslant 10\delta} \frac{1}{\omega(B(x,d(x,y)))}\Big(\frac{d(x,y)}{\|x-y\|}\Big)^\beta \Big(\frac{\|x-y\|}{r}\Big)^\beta\\
    &\qquad\times\Big\{\frac{1}{V(\boldsymbol{y},x_0,r+d(\boldsymbol{y},x_0))}\Big(\frac{r}{r+\|\boldsymbol{y}-x_0\|}\Big)^\gamma\\
    &\qquad\qquad\qquad+\frac{1}{V(\boldsymbol{x},x_0,r+d(\boldsymbol{x},x_0)))}\Big(\frac{r}{r+\|\boldsymbol{x}-x_0\|}\Big)^\gamma\Big\}d\omega(y)\\
    &\lesssim \frac{1}{r^\beta}\frac{1}{(V(x,x_0, r))}
    \int_{d(x,y)\leqslant 10\delta}  \frac{1}{\omega(B(x,d(x,y)))}\big(d(x,y)\big)^\beta d\omega(y)\\
    &\lesssim \Big(\frac{\delta}{r}\Big)^\beta \frac{1}{V(x,x_0,
        r+d(x,x_0))}\Big(\frac{r}{r+d(x,x_0)}\Big)^\gamma.
    \end{align*}

    Repeating the same proof implies that
    $$
    |p(x')|\lesssim \Big(\frac{\delta}{r}\Big)^\beta \frac{1}{V(x,x_0,r+d(x,x_0))}\Big(\frac{r}{r+d(x,x_0)}\Big)^\gamma.
    $$
    Similarly, by $T(1)=0,$ we have
    \begin{align*}
    &q(x)-q(x')\\&= \int_{\R^N} [K(x,y)-K(x',y)]\eta(y)[f(y)-f(x)]d\omega(y) +
    [f(x)-f(x')]\int_{\R^N} K(x',y)\xi(y)d\omega(y)\\&:=I+I\!I.
    \end{align*}
    Observing that if $d(y,x)\geqslant 5\delta$, then
    $|K(x,y)-K(x',y)|\lesssim \Big(\frac{\delta}{\|x-y\|}\Big)^\varepsilon \frac{1}{\omega(B(x,d(y,x)))}$ and
    \begin{align*}
    |f(y)-f(x)|&\lesssim
    \Big(\frac{\|x-y\|}{r}\Big)^\beta
    \Big\{\frac{1}{V(\boldsymbol{y},x_0,r+d(\boldsymbol{y},x_0))}\Big(\frac{r}{r+\|\boldsymbol{y}-x_0\|}\Big)^\gamma\\
    &\qquad\qquad +\frac{1}{V(\boldsymbol{x},x_0,r+d(\boldsymbol{x},x_0)))}\Big(\frac{r}{r+\|\boldsymbol{x}-x_0\|}\Big)^\gamma\Big\}.
    \end{align*}
    Note that $r+d(x,x_0)\lesssim r+d(y,x_0),$ therefore
    \begin{align*}
    |I|&\lesssim \frac{\delta^{\varepsilon}}{r^\beta} \frac{1}{V(x,x_0,r+d(x,x_0)))} \int_{d(y,x)\geqslant 5\delta}\frac{1}{d(y,x)^{\varepsilon-\beta}}\frac{1}{\omega(B(x,d(y,x)))}d\omega(y)\\
    &\lesssim\Big(\frac{\delta}{r}\Big)^\beta \frac{1}{V(x,x_0,r+d(x,x_0)))} \Big(\frac{r}{r+d(x,x_0)}\Big)^\gamma.
    \end{align*}
    Similar to the proof of \eqref{T2.11} in Lemma \ref{lem2.5}, we
    can get $|T(\xi)(x')|\lesssim 1$ and thus
    \begin{align*}
    |I\!I|&\lesssim |f(x)-f(x')|\\
    &\lesssim \Big(\frac{\delta}{r}\Big)^\beta\Big\{ \frac{1}{V(\boldsymbol{x},x_0,r+d(\boldsymbol{x},x_0)))} \Big(\frac{r}{r+d(\boldsymbol{x},x_0)}\Big)^\gamma \\
    &\qquad\qquad+ \frac{1}{V(\boldsymbol{x'},x_0,r+d(\boldsymbol{x'},x_0)))} \Big(\frac{r}{r+d(\boldsymbol{x'},x_0)}\Big)^\gamma\Big\} .
    \end{align*}
    The fact that $\int_{\R^N} T(f)(x)d\omega(x)=0$ follows from $T^*(1)=0.$

    The proof of Theorem \ref{1.3} is complete.
\end{proof}

\subsection{{Proof of $T1$ Theorem}}
\ \\

To show Theorem \ref{1.2}, the $T1$ theorem, observe that the
necessary conditions of the $T1$ theorem follow from Theorem
\ref{th1.1}, namely $T(1), T^*(1)\in BMO_d(\R^N,\omega)$ and $T\in WBP$ by the
definition of the weak boundedness of property.

To show the sufficent conditions of Theorem \ref{1.2}, we need to
apply Coifman's approximation to the identity. For this purpose, we
first extend $T$ to a continuous linear operator from $\Lambda^s
\cap L^2(\R^N,\omega)$ into $(C^s_0)^\prime$ where $\Lambda^s(\Bbb
R^N)$ denotes the closure of $C^\eta_0(\Bbb R^N)$ with respect to
the norm $\|\cdot\|_s, 0<s<\eta$. To be precise, given $g\in C^s_0,
0<s<1,$ with the support contained in a ball $B(x_0, r),$ and set
$\theta\in C^s_0$ with $\theta(x)=1$ for $d(x,x_0)\leqslant 2r$ and
$\theta(x)=0$ for $d(x,x_0)\geqslant 4r.$ Given $f\in \Lambda^s \cap
L^2(\R^N,\omega),$ we write
$$\langle Tf, g \rangle =\langle T(\theta f), g\rangle
+\langle  T((1-\theta) f), g \rangle. $$
The first term makes sense
since $\theta f\in C^s_0.$ To see that the second term is also well defined, by the size condition of $K(x,y)$ and the fact $f\in L^2(\R^N,\omega)$ together with the doubling and the reverse doubling conditions of the measure $\omega,$ we first write

$$\langle T((1-\theta) f), g \rangle
= \int_{\Bbb R^N} g(x)\int_{\{y:d(x,y)>r\}} K(x,y)(1-\theta(y))f(y)d\omega(y)d\omega(x).
$$
By H\"older's inequality,

\begin{align*}
&\bigg|\int_{\{y:d(x,y)>r\}} K(x,y)(1-\theta(y))f(y)d\omega(y)\bigg|\\
&\lesssim \|f\|_2\Big( \int_{\{y:d(x,y)>r\}}|K(x,y)|^2d\omega(y)\Big)^{\frac12} \\
&\lesssim \|f\|_2 \Big(\sum\limits_{j=0}^\infty\int_{\{2^jr\leqslant d(x,y)\leqslant 2^{j+1}r\}} \frac1{\omega(B(x,d(x,y)))^2}d\omega(y)\Big)^{\frac12} \\
&\lesssim \|f\|_2 \Big(\sum\limits_{j=0}^\infty\int_{\{2^jr\leqslant d(x,y)\leqslant 2^{j+1}r\}} \frac1{\omega(B(x,2^jr))^2}d\omega(y)\Big)^{\frac12} \\
&\lesssim \|f\|_2\Big( \sum\limits_{j=0}^\infty
2^{-jN}\frac{1}{\omega(B(x,r))}\Big)^{\frac12}\\
& <\infty,
\end{align*}
where the last inequality follows from the fact that $\inf\limits_x
\omega(B(x,r))>0.$

This implies that $\langle  T((1-\theta) f), g \rangle $ is well defined. Moreover, this extension is independent of the choice of $\theta.$

We now describe the properties of Coifman's approximation to the identity acting on $\Lambda^s\cap L^2(\R^N,\omega).$ Let's begin with considering $(\Bbb R^N,\|\cdot\|, \omega)$ as space of homogeneous type in the sense of Coifman and Wiess. Note that the measure $\omega$ satisfies the doubling and the reverse doubling properties. Therefore, in this case, the Littlewood-Paley theory has already established in \cite{HMY}. We recall main results for $(\Bbb R^N, \|\cdot\|, \omega).$ Here and
throughout, $V_k(x)$ always denotes the measure
$\omega(B(x,r^{-k}))$ for $r>1, k\in \Bbb Z$ and $x\in \Bbb R^N$. We
also denote by $V(x,y)=\omega(B(x,\|x-y\|))$ for $x,y\in \Bbb R^N$.

Let $\theta: \mathbb R\mapsto [0,1]$ be a smooth function which is 1
for $\|x\|\leqslant r$ and vanishes for  $\|x\|\geqslant 2r$ with
some fixed $r>1.$  Applying the construction of Coifman's
approximation to the identity, we define
$$T_k(f)(x)=\int_{\Bbb R^N} \theta(r^k \|x-y\|)f(y)d\omega(y), \quad k\in \Bbb Z.$$
Then
$$T_k(1)(x)\leqslant \int_{\|x-y\|\leqslant 2r^{-k}}d\omega(y)\leqslant C\omega(B(x,r^{-k})).$$
Conversely,
$$T_k(1)(x)\geqslant \int_{\|x-y\|< r^{-k}}d\omega(y)=\omega(B(x,r^{-k})).$$
Hence, $T_k(1)(x)\sim \omega(B(x,r^{-k}))=V_k(x)$. It is easy to
check $V_k(x) {\sim} V_k(y)$ whenever $\|x-y\|\leqslant  r^{5-k}$.
Thus,
\begin{align*}
T_k\bigg(\frac1{T_k(1)}\bigg)(x)
&=\int_{\Bbb R^N} \theta(r^k \|x-y\|)\frac1{T_k(1)(y)}d\omega(y) \\
&\sim\int_{\Bbb R^N} \theta (r^k \|x-y\|) \frac 1{V_k(y)}d\omega(y) \\
&\sim \frac 1{V_k(x)}\int_{\Bbb R^N} \theta(r^k \|x-y\|)d\omega(y) \\
&= \frac 1{V_k(x)}T_k(1)(x)\sim 1.
\end{align*}
Let $M_k$ be the operator of multiplication by $M_k(x):=\frac1{T_k(1)(x)}$ and let $W_k$ be the operator of multiplication by $W_k(x):=\big[T_k\big(\frac1{T_k(1)}\big)(x)\big]^{-1}$.
Coifman's approximation to the identity is constructed by $S_k=M_kT_kW_kT_kM_k,$ where
the kernel of $S_k$ is
$$S_k(x,y)=\int_{\Bbb R^n} M_k(x)\theta(r^k\|x-z\|)W_k(z)\theta(r^k\|z-y\|)M_k(y)d\omega(z).$$
In \cite{HMY}, it was proved that kernels $S_k(x,y)$ defined on $\Bbb R^N\times \Bbb R^N$ satisfy the following properties.
\begin{enumerate}
    \item[(i)] $S_k(x,y)=S_k(y,x);$
    \item[(ii)] $S_k(x,y)=0$ if $\|x-y\|>r^{4-k}$\ \ and\ \ $\displaystyle |S_k(x,y)|\leqslant \frac C{V_k(x)+V_k(y)};$
    \item[(iii)] $\displaystyle |S_k(x,y)-S_k(x',y)|\leqslant C\frac{r^k\|x-x'\|}{V_k(x)+V_k(y)}$\qquad for $\|x-x'\|\leqslant r^{8-k};$\\[3pt]
    \item[(iv)] $\displaystyle |S_k(x,y)-S_k(x,y')|\leqslant C\frac{r^k\|y-y'\|}{V_k(x)+V_k(y)}$\qquad for $\|y-y'\|\leqslant r^{8-k};$
    \item[(v)] $\displaystyle \big|[S_k(x,y)-S_k(x',y)]-[S_k(x,y')-S_k(x',y')]\big|\leqslant C\frac{r^k\|x-x'\|r^k\|y-y'\|}{V_k(x)+V_k(y)}$
    \item[] for $\|x-x'\|\leqslant r^{8-k}$ and $\|y-y'\|\leqslant r^{8-k};$
    \item[(vi)] $\displaystyle \int_{\Bbb R^N} S_k(x,y)d\omega(x)=1\qquad \text{for all}\ y\in \Bbb R^N;$
    \item[(vii)] $\displaystyle \int_{\Bbb R^N} S_k(x,y)d\omega(y)=1\qquad \text{for all}\ x\in \Bbb R^N.$
\end{enumerate}

Coifman's decomposition of the identity on $L^2(\R^N, \omega)$ is
given as follows. Let ${D}_k:= {S}_k - {S}_{k-1}$. The identity
operator $I$ on $L^2(\R^N,\omega)$ can be written as
$$I=\sum\limits_{k=-\infty}^\infty {D}_k=\sum\limits_{k=-\infty}^\infty\sum\limits_{j=-\infty}^\infty {D}_k{D}_j={T}_{M} +{R}_{M},$$
where ${T}_{M}=\sum\limits_{\{j,k\in
    \Bbb Z:\, |k-j|\leqslant {M}\}} {D}_k{ D}_j=\sum\limits_{k\in \Bbb Z} D_kD_k^{M}$ with $D_k^M=\sum\limits_{\{j\in \Bbb Z:|j|\leqslant M\}} D_{k+j}$ and
${R}_{M}=\sum\limits_{\{j,k\in \Bbb Z:\,
    |k-j|> {M}\}}{D}_k{D}_j.$
It is known, see \cite{HMY}, that there exists a constant $C$ such that
\begin{equation}\label{aoe h}
|D_jD_k(x,y)|\leqslant Cr^{-|j-k|}\frac 1{V_{j \wedge k}(x)+V_{j
        \wedge k}(y)},
\end{equation}
where $j \wedge k=\min\{j,k\}.$

This estimate implies
$$\|D_jD_k\|_{L^2(\omega) \mapsto L^2(\omega)}\lesssim r^{-|j-k|}.$$
By the Cotlar-Stein Lemma we obtain
$$\|R_{M}(f)\|_{L^2(\omega)}\leqslant Cr^{-M}\|f\|_{L^2(\omega)}$$
and then for a fixed large $M, T_{M}^{-1},$ the inverse of $T_M,$ is bounded on $L^2(\R^N,\omega)$. This yields that $T_M$ converges to the identity in the $L^2$ norm and moreover,
$$I=T_M^{-1}T_M=\sum\limits_{k=-\infty}^\infty {T}_{M}^{-1}D_k^{M}D_k=T_MT_M^{-1}=\sum\limits_{k=-\infty}^\infty  D_k^{M}D_k{T}_{M}^{-1} \qquad\mbox{in}\ L^2(\R^N,\omega).$$

The following lemma describes the properties of operators ${T}_{M}$ acting on $\Lambda^s.$
\begin{lemma}\label{Lemma 3.6}
    Suppose $0<s<\12 .$ Then
    \begin{enumerate}
        \item[(i)] ${T}_{M}=\sum\limits_{k=-\infty}^\infty D_kD^{M}_k$ converges uniformly and in the norm of $\Lambda^s$,
        \item[(ii)] ${T}_{M}$ is bounded on $\Lambda^s$,
        \item[(iii)] $\|{T}_{M}-I\|_s\to 0$ as $M\to +\infty$.
    \end{enumerate}
\end{lemma}

To prove Lemma \ref{Lemma 3.6}, we need the following estimates for
$D_k$ and $D_k^{M}$.

\begin{lemma}\label{Lemma 3.8}
    Let $0<s<1$. Then
    \begin{enumerate}
        \item[(i)] $\|D_kf\|_{L^\infty} \lesssim r^{-ks}\|f\|_s$,
        \item[(ii)] $\|D_kf\|_s \lesssim r^{ks}\|f\|_{L^\infty}$,
        \item[(iii)]  $\|D_kf\|_\beta\lesssim r^{k(\beta-s)}\|f\|_s$
        \qquad if\ \ $0<s\leqslant\beta< 1$,
        \item[(iv)] $\|D^{M}_kf\|_s\lesssim M\|f\|_s$.
    \end{enumerate}
\end{lemma}

\begin{proof}
    For (i), the cancellations of $D_k$ gives
    $$D_kf(x)=
    \int_{\Bbb R^N}D_k(x,y)[f(y)-f(x)]d\omega(y).$$
    Since $D_k(x,y)=0$ for $\|x-y\|\geqslant r^{4-k}$,
    the size condition of $D_k$ and the smoothness condition of $f$ yield
    \begin{align*}
    |D_kf(x)|&\lesssim \|f\|_s\int_{\|x-y\|\leqslant r^{4-k}}\frac1{V_k(x)+V_k(y)}\|x-y\|^sd\omega(y) \\
    &\lesssim r^{-ks}\|f\|_s.
    \end{align*}
    For (ii), the smoothness condition of $D_k$ gives
    \begin{align*}
    |D_kf(x)-D_kf(y)|
    &=\bigg|\int_{\Bbb R^N} (D_k(x,z)-D_k(y,z))f(z)d\omega(z)\bigg|\\
    &\leqslant \|f\|_{L^\infty}  \bigg(\int_{\|x-z\|\leqslant r^{4-k}}+\int_{\|y-z\|\leqslant r^{4-k}}\bigg)\frac{(r^k\|x-y\|)^s}{V_k(x)+V_k(z)}d\omega(z),
    \end{align*}
    which implies
    $$\|D_kf\|_s \lesssim r^{ks}\|f\|_{\infty}.$$

    To estimate (iii),  if  $\|x-y\|\leqslant r^{6-k},$ using
    the the cancellations of $D_k$ and the smoothness condition of $f,$ we get
    \begin{align*}
    |D_kf(x)-D_kf(y)|
    &=\bigg|\int_{\Bbb R^N} [D_k(x,z)-D_k(y,z)]f(z)d\omega(z)\bigg|\\
    &=\bigg|\int_{\Bbb R^N} [D_k(x,z)-D_k(y,z)][f(z)-f(x)]d\omega(z)\bigg|\\
    &\lesssim  \bigg(\int_{\|x-z\|\leqslant r^{4-k}}+\int_{\|y-z\|\leqslant r^{4-k}}\bigg)\frac{(r^k\|x-y\|)^\beta}{V_k(x)+V_k(z)}\|x-z\|^s\|f\|_sd\omega(z)\\
    &\lesssim \|x-y\|^\beta r^{k(\beta-s)}\|f\|_s.
    \end{align*}

    When $\|x-y\|> r^{6-k}$, (i) gives
    $$
    |D_kf(x)-D_kf(y)|\lesssim r^{-ks}\|f\|_s
    \lesssim \|x-y\|^\beta r^{k(\beta-s)}\|f\|_s.
    $$
    (iii) follows from these estimates.
    The estimate of (iv) follows from (iii) with $\beta = s.$
\end{proof}

We now show Lemma \ref{Lemma 3.6}.

\begin{proof}[Proof of Lemma \ref{Lemma 3.6}]
    We first show that $T_M(f)(x)=\sum\limits_{k=-\infty}^\infty D_kD^{M}_k(f)(x)$ is well defined on $\Lambda_s(\Bbb R^N).$ To this end, let $f\in C^\eta_0$ with $\eta>s$ and set $G_kf(x):=D_kD^{M}_kf(x)$. Observe that if $f\in C^\eta_0$ then $f\in L^\infty$ and hence, $\|D^M_k(f)\|_1\leqslant CM\|f\|_\infty.$ By (iv), $\|D^M_k(f)\|_\eta\lesssim M\|f\|_\eta.$  Therefore,
    \begin{align*}
    |G_k(f)(x)|&=|\int_{\R^N} D_k(x,y)D^M_k(f)(y)d\omega(y)|\leqslant C\frac{1}{V_k(x)}\|f\|_\infty\\
    &\lesssim r^k\frac{1}{V_0(x)}\|f\|_\infty\lesssim r^k\|f\|_\infty
    \end{align*}
    since $\inf\limits_x V_0(x)\geqslant C>0.$
    And
    \begin{align*}
    |G_k(f)(x)|&=\Big|\int_{\R^N} D_k(x,y)D^M_K(f)(y)d\omega(y)\Big|=\Big|\int_{\R^N} D_k(x,y)[D^M_K(f)(y)-D^M_K(f)(x)]d\omega(y)\Big|\\
    &\lesssim \|f\|_\eta\int_{\R^N} |D_k(x,y)|\|x-y\|^\eta d\omega(y)\lesssim r^{-k\eta}\|f\|_\eta.
    \end{align*}
    These two estimates imply that if $f\in C^\eta_0$ then the series $\sum\limits_{k=-\infty}^\infty D_kD^{M}_k(f)(x)$ converges uniformly. Moreover, for given
    $x,y\in \Bbb R^N$, choose $k_0\in \Bbb Z$ such that
    $r^{-k_0}\leqslant \|x-y\|\leqslant r^{-k_0+1}$. Then
    Lemma \ref{Lemma 3.8} implies that
    \begin{equation}\label{eq 2.5}
    \begin{aligned}
    \bigg|\sum\limits_{k=-\infty}^\infty [G_kf(x)-G_kf(y)]\bigg|
    &\leqslant \sum\limits_{k\geqslant k_0} |G_kf(x)-G_kf(y)|
    +\sum\limits_{k< k_0}|G_kf(x)-G_kf(y)| \\
    &\lesssim  \sum\limits_{k\geqslant k_0}2\|G_kf\|_{L^\infty}
    +\sum\limits_{k< k_0}\|x-y\|^\beta r^{k(\beta-s)}\|f\|_s \\
    &\lesssim r^{-k_0s}\|f\|_s+r^{k_0(\beta-s)}\|x-y\|^\beta\|f\|_s\\
    &\lesssim \|x-y\|^s\|f\|_s.
    \end{aligned}
    \end{equation}
    Hence if $f\in C^\eta_0$ with $\eta>s$,
    then the series $\sum\limits_kD_kD^{M}_kf$ converges in $\Lambda^s$ norm. Observe that $C^\eta_0$ with $\eta > s$ is dense in $\Lambda^s.$ This implies that $T_M$ extends to $\Lambda^s.$ Indeed, if $f\in \Lambda^s,$ then there exists a sequence $f_n\in C^\eta_0, \eta>s,$ such that  $\|f_n-f\|_s$ tends to zero as $n$ tends to $\infty.$ Let $T_M(f)(x)=\lim\limits_{n\to \infty}T_M(f_n)(x).$ Then $T_M$ is bounded on $\Lambda^s$ and moreover, $\|T_M(f)\|_{s}\lesssim \|f\|_s$ for $f\in \Lambda^s.$

    To show $\|{T}_{M}-I\|_s\to 0$ as $M\to +\infty,$ it is sufficient to prove $\|{R}_{M}\|_{s,s}\to 0$ as $M\to +\infty$. To this end, we rewrite
    \begin{align*}
    {R}_{M}f
    &=\sum\limits_{k=-\infty}^\infty\sum\limits_{\{j\in \Bbb Z:\, |k-j|> {M}\}}{D}_k{D}_j
    =\sum\limits_{\{\ell \in \Bbb Z:\,|\ell|> M\}}\sum\limits_{k=-\infty}^\infty D_kD_{k+\ell} f \\
    &=\sum\limits_{k=-\infty}^\infty D_k(I-S_{k+M})f
    +\sum\limits_{k=-\infty}^\infty D_kS_{k-M-1}f.
    \end{align*}
    Since $\int_{\Bbb R^N} S_k(x,y)d\omega(y)=1$ for $k\in \Bbb Z$, we have
    \begin{align*}
    |(I-S_{k+M})f(x)|
    &=|f(x)-S_{k+M}f(x)|\\
    &=\bigg|\int_{\Bbb R^N} S_{k+M}(x,y)[f(x)-f(y)]d\omega(y)\bigg|\\
    &\leqslant Cr^{-(M+k)s}\|f\|_s
    \end{align*}
    and hence $$\|(I-S_{k+M})f\|_{L^\infty}\leqslant Cr^{-(M+k)s}\|f\|_s.$$
    The above estimate together with Lemma \ref{Lemma 3.8} and applying the same proof for \eqref{eq 2.5} imply that
    $\big\|\sum\limits_{k=-\infty}^\infty D_k(I-S_{k+M})f\big\|_s\leqslant r^{-Ms}\|f\|_s,$ which gives $\big\|\sum\limits_{k=-\infty}^\infty D_k(I-S_{k+M})\big\|_{s,s}\to 0$ as $M\to +\infty$.

    To estimate $\sum\limits_{k=-\infty}^\infty D_kS_{k-M-1}f,$ the second term of $R_Mf,$    let $H_kf=D_kS_{k-M-1}f$ and denote $H_k(x,y)$ by the kernel of $H_k$.  Then $\int_{\Bbb R^N} H_k(x,y)d\omega(y)=D_kS_{k-M-1}(1)=D_k(1)=0$
    and $H_k(x,y)=0$ if $\|x-y\|\geqslant r^{6-(k-M)}$.
    By the cancellation of $D_k$ and the smothness of $S_{k-M-1}$,
    \begin{align*}
    |H_k(x,y)|
    &=\bigg|\int_{\Bbb R^N}
    D_k(x,z)\big[S_{k-M-1}(z,y)-S_{k-M-1}(x,y)\big]d\omega(z)\bigg|\\
    &\leqslant C\int_{|x-z|\leqslant r^{2-k}}(V_k(x))^{-1}
    \frac{r^{k-M-1}\|x-z\|}{V_{k-M-1}(x)}d\omega(z)\\
    &\leqslant Cr^{-M}(V_{k-M-1}(x))^{-1}.
    \end{align*}
    Thus, for $f\in \Lambda^s,$
    \begin{align*}
    |H_kf(x)|
    &=\bigg|\int_{\Bbb R^N}
    H_k(x,y)[f(y)-f(x)]d\omega(y)\bigg|\\
    &\leqslant C\int_{|x-y|\leqslant r^{6-(k-M)}}r^{-M}(V_{k-M-1}(x))^{-1}\|x-y\|^s\|f\|_sd\omega(y)\\
    &\leqslant Cr^{-M}r^{-(k-M)s}\|f\|_s.
    \end{align*}
    This implies that
    \begin{equation}\label{eq 2.6}
    \|H_kf\|_{L^\infty}\lesssim r^{-M}r^{-(k-M)s}\|f\|_s.
    \end{equation}
    If $\|x-x'\|\leqslant r^{6-(k-M)}$, then
    \begin{equation}\label{eq 2.7}
    \begin{aligned}
    &|H_k(x,y)-H_k(x',y)|\\
    &= \bigg|\int_{\Bbb R^N} \big[D_k(x,z)-D_k(x',z)\big]S_{k-M-1}(z,y)d\omega(z)\bigg| \\
    &\leqslant C \int_{\{\|x-z\|\leqslant r^{2-k}\ \text{or}\ \|x'-z\|\leqslant r^{2-k}\}}\frac{r^k\|x-x'\|}{V_k(x)+V_k(z)}(V_{k-M-1}(y))^{-1}d\omega(z) \\
    &\leqslant Cr^k\|x-x'\|(V_{k-M-1}(y))^{-1}.
    \end{aligned}
    \end{equation}
    For $\|x-y\|\leqslant r^{6-(k-M)}$, applying \eqref{eq 2.7} yields
    \begin{align*}
    &|H_kf(x)-H_kf(y)|\\
    &=\bigg|\int_{\Bbb R^N}[H_k(x,z)-H_k(y,z)]f(z)d\omega(z) \bigg|\\
    &=\bigg|\int_{\Bbb R^N}[H_k(x,z)-H_k(y,z)][f(z)-f(x)]d\omega(z) \bigg|\\
    &\lesssim \int_{\{\|x-z\|\leqslant r^{6-(k-M)}\ \text{or}\ \|y-z\|\leqslant r^{6-(k-M)}\}}r^k\|x-y\|(V_{k-M-1}(y))^{-1}\|x-z\|^s\|f\|_s d\omega(z)\\
    &\lesssim r^kr^{-(k-M)s}\|x-y\|\|f\|_s.
    \end{align*}
    For $\|x-y\|> r^{6-(k-M)}$, the estimate \eqref{eq 2.6} implies
    $$
    |H_kf(x)-H_kf(y)|\lesssim r^{-M}r^{-(k-M)s}\|f\|_s
    \lesssim r^kr^{-(k-M)s}\|x-y\|\|f\|_s.
    $$
    These estimates imply that $H_k(f)(x)$ is a Lipschitz function with the Lipschitz norm bounded by
    \begin{equation}\label{eq 2.8}
    \|H_kf\|_{Lip}  \lesssim r^kr^{-(k-M)s}\|f\|_s.
    \end{equation}
    Using the fact that $\|f\|_\beta\leqslant \|f\|_{\infty}^{1-\beta}\|f\|_{Lip}^\beta, 0<\beta<1,$ the estimates \eqref{eq 2.6} and \eqref{eq 2.8} yield
    \begin{equation}\label{eq 2.9}
    \|H_kf\|_\beta \lesssim r^{-M(1-2\beta)}r^{-(k-M)(s-\beta)}\|f\|_s.
    \end{equation}
    Given
    $x,y\in \Bbb R^N$, choose $k_1\in \Bbb Z$ such that
    $r^{-k_1}\leqslant \|x-y\|\leqslant r^{-k_1+1}$.
    The estimates in \eqref{eq 2.6} and \eqref{eq 2.9} imply that for $s<\beta,$
    \begin{align*}
    &\bigg|\sum\limits_{k=-\infty}^\infty [H_kf(x)-H_kf(y)]\bigg|\\
    &\leqslant \sum\limits_{\{k:k\geqslant k_1\}} |H_kf(x)-H_kf(y)|
    +\sum\limits_{\{k:k< k_1\}}|H_kf(x)-H_kf(y)| \\
    &\lesssim  \sum\limits_{\{k:k\geqslant k_1\}}2\|H_kf\|_{L^\infty}
    +\sum\limits_{\{k:k< k_1\}}\|x-y\|^\beta \|H_kf\|_\beta \\
    &\lesssim  \sum\limits_{\{k:k\geqslant k_1\}} r^{-M}r^{-(k-M)s}\|f\|_s
    + \sum\limits_{\{k:k< k_1\}}\|x-y\|^\beta r^{-M(1-2\beta)}r^{-(k-M)(s-\beta)}\|f\|_s \\
    &\lesssim r^{-k_1s}r^{-M+Ms}\|f\|_s+ r^{-M(1-2\beta)}r^{M(s-\beta)}r^{k_1(\beta-s)}\|x-y\|^\beta\|f\|_s\\
    &\lesssim  r^{-M(1-2\beta)}\Big(r^{M(s-2\beta)}+r^{M(s-\beta)}\Big)\|x-y\|^s\|f\|_s\\
    &\lesssim r^{-M(1-2\beta)}\|x-y\|^s\|f\|_s.
    \end{align*}
    Therefore, we have
    $$\Big\|\sum\limits_{k=-\infty}^\infty H_k f\Big\|_s\lesssim r^{-M(1-2\beta)}\|f\|_s\qquad \text{for}\quad s<\beta<1.$$ If $s<\12 $, we can choose $\beta$
    so that $r^{-M(1-2\beta)}\to 0$ as $M\to +\infty.$ The proof of Lemma \ref{Lemma 3.6} is finished.
\end{proof}
We are now ready to give the proof of sufficient conditions of the
$T1$ theorem under the assumptions that $T(1)=T^*(1)=0.$ Notice that
${T}_{M}$ converges strongly on $L^2(\R^N,\omega)$ since, by the
almost orthogonal estimates and the Cotlar-Stein Lemma,
$$\sup\limits_{L_1,L_2}\Big\|\sum\limits_{k=L_1}^{L_2}
D_kD^{M}_k\Big\|_{L^2(\R^N,\omega)\mapsto
L^2(\R^N,\omega)}<+\infty.$$ Thus, by Lemma \ref{Lemma 3.6},
${T}_{M}$ converges strongly on $\Lambda^s\cap L^2$.

It is clear that $\Lambda^s\cap L^2(\R^N,\omega)$ is dense in $L^2$.
To prove the sufficient condition of  Theorem \ref{1.2}, it suffices
to show that
$$\big|\langle g_0, Tf_0 \rangle\big|\leqslant C\|g_0\|_{L^2(\R^N,\omega)}\|f_0\|_{L^2(\R^N,\omega)}$$
for any $g_0, f_0\in \Lambda^s\cap L^2(\R^N,\omega)$ with compact
supports. For given $f_0\in \Lambda^s\cap L^2(\R^N,\omega)$ with
compact support, by Lemma\ref{Lemma 3.6}, set
$f_1={T}_{M}^{-1}f_0\in \Lambda^s\cap L^2(\R^N,\omega) $ and let
$$U_{L_1,L_2}
=\sum\limits_{k=L_1}^{L_2}D_kD^{M}_k.$$ By Lemma \ref{Lemma 3.6},
$\lim\limits_{\overset{L_1\to -\infty}{L_2\to +\infty}}
U_{L_1,L_2}f_1=f_0$ in $\Lambda^s\cap L^2(\R^N,\omega)$. Observe
that operator $T$ extends to a continuous linear operator from
$\Lambda^s \cap L^2(\R^N,\omega)$ into $(C^s_0)^\prime.$ Hence, for
each $g_0\in \Lambda^s\cap L^2(\R^N,\omega)$ with compact support,
$$\langle g_0, Tf_0 \rangle =\lim\limits_{\overset{L_1\to -\infty}{L_2\to +\infty}}
\langle g_0, TU_{L_1,L_2}f_1 \rangle.$$ Similarly, let
$g_1={T}_{M}^{-1}g_0$. Then $g_1\in \Lambda^s\cap L^2(\R^N,\omega)$
and $\lim\limits_{\overset{L'_1\to -\infty}{L'_2\to +\infty}}
U_{L'_1,L'_2}g_1=g_0$ in $\Lambda^s\cap L^2(\R^N,\omega)$. Thus,
$$\langle g_0, Tf_0 \rangle =\lim\limits_{\overset{L_1\to -\infty}{L_2\to +\infty}}
\lim\limits_{\overset{L'_1\to -\infty}{L'_2\to +\infty}}
\langle U_{L'_1,L'_2}g_1, TU_{L_1,L_2}f_1 \rangle.$$
Observe that
\begin{align*}
\langle U_{L'_1,L'_2}g_1, TU_{L_1,L_2}f_1 \rangle
&\quad=\sum\limits_{k=L_1}^{L_2}\sum\limits_{k'=L'_1}^{L'_2}\Big\langle
D^{M}_{k'}g_1, D^*_{k'} T D_{k}D^{M}_{k}f_1\Big\rangle.
\end{align*}
The following almost orthogonal estimate is crucial.

\begin{lemma}\label{Lem aoe}
    Let $T$ be a Dunkl-Calder\'on-Zygmund singular integral satisfying $T(1)=T^*(1)=0$ and $T\in WBP$. Then
    \begin{align*}
    &\bigg|\int_{\Bbb R^N}\int_{\Bbb R^N} D_k(x,u)K(u,v)D_j(v,y)
    d\omega(u)d\omega(v)\bigg|\\
    &\quad \lesssim r^{-|k-j|\varepsilon'}
    \frac1{V(x,y, r^{(-j)\vee (-k)}+d(x,y))}\Big(\frac{r^{(-j)\vee (-k)}}{r^{(-j)\vee (-k)}+d(x,y)}\Big)^\gamma
    \end{align*}
    where $\gamma, \varepsilon'\in (0,\varepsilon)$ and $\varepsilon$ is the regularity exponent of the kernel of $T$ given in
    \eqref{x smooth of C-Z-S-I-O} and \eqref{y smooth of C-Z-S-I-O}, $a\vee b=max\{a,b\}.$
\end{lemma}
Assuming the Lemma \ref{Lem aoe} for the moment, then
$$\|D^*_kT D_{k'}\|_{L^2(\omega) \mapsto L^2(\omega)}\lesssim 2^{-|k-k'|}.$$ Applying the Cotlar-Stein lemma yields
\begin{align*}
|\langle U_{L_1,L_2}g_1, TU_{L'_1,L'_2}f_1 \rangle|\lesssim
\|f_1\|_{L^2(\R^N,\omega)}\|g_1\|_{L^2(\R^N,\omega)}\lesssim
\|f_0\|_{L^2(\R^N,\omega)}\|g_0\|_{L^2(\R^N,\omega)}
\end{align*}
for all $L_1, L_2, L'_1$ and $L'_2.$
Hence,
$$\big|\langle g_0, Tf_0 \rangle\big|\leqslant C\|g_0\|_{L^2(\R^N,\omega)}\|f_0\|_{L^2(\R^N,\omega)}.$$
The proof of Theorem \ref{1.2} with the assumptions $T(1)=T^*(1)=0$ is complete.

To show the Lemma \ref{Lem aoe}, we need the following lemma, which will be used to establish the discrete weak-type Calder\'on reproducing formula and the boundedness of the Dunkl-Calder\'on-Zygmund operators on the Dunkl-Hardy spaces.

\begin{lemma}\label{aoe}
    Let $x,y\in \R^N$ and $\varepsilon_0,t,s>0$ with $t\geqslant s.$  Suppose that $f_t(x,\cdot)$ is a weak smooth molecule function in $\widetilde{\mathbb M}(\varepsilon_0,\varepsilon_0, t, x)$ and
    $g_s(\cdot,y)$ is a  smooth molecule function  in $\mathbb M(\varepsilon_0,\varepsilon_0, s, y)$. Then for any $0<\varepsilon_1,\varepsilon_2<\varepsilon_0$,
    there exists $C>0$ depending on $\varepsilon_0,\varepsilon_1,\varepsilon_2$, such that for all $t\geqslant s>0,$
    \begin{equation}\label{aoe1}
    \int_{\Bbb R^N}f_t(x,u)g_s(u,y)d\omega(u)\leqslant C \bigg(\frac{s}{t}\bigg)^{\varepsilon_1} \frac1{V(x,y,t+d(x,y))}\Big(\frac{t}{t+d(x,y)}\Big)^{\varepsilon_2}.
    \end{equation}
    If $f_t(x,\cdot)$ and $g_s(\cdot,y)$ both are smooth molecule functions in $\mathbb M(\varepsilon_0,\varepsilon_0, t, x)$ and $\mathbb M(\varepsilon_0,\varepsilon_0, s, y),$ respectively,
    then for any $0<\varepsilon_1,\varepsilon_2<\varepsilon_0$, there exists $C>0$  depending on $\varepsilon_0,\varepsilon_1,\varepsilon_2$, such that for all $t, s>0,$
    \begin{align}\label{aoe2}
    &\int_{\Bbb R^N}f_t(x,u)g_s(u,y)d\omega(u)\\
    &\leqslant C \bigg(\frac{s}{t}\wedge \frac{t}{s}\bigg)^{\varepsilon_1} \frac1{V(x,y,(t\vee s)+d(x,y))}\Big(\frac{t\vee s}{(t\vee s)+\|x-y\|}\Big)^{\varepsilon_2},\nonumber
    \end{align}
    where $a\wedge b=min\{a,b\}$ and $a\vee b=max\{a,b\}.$
\end{lemma}

Before proving the above lemma, we first give the following lemma.
\begin{lemma}\label{lemtan2}
    For any $\varepsilon_1,\varepsilon_2,t,s>0$ , Let
    $$T=\int_{\Bbb R^N}
    \frac1{V(x,z,t+d(x,z))}\Big(\frac{t}{t+d(x,z)}\Big)^{\varepsilon_1}\frac1{V(z,y,s+d(z,y))}\Big(\frac{s}{s+d(z,y)}\Big)^{\varepsilon_2}d\omega(z),$$
    then there exists a constant $C$ depending on $\varepsilon_1,\varepsilon_2$ such that,
    $$
    T \leqslant \frac{C}{V(x,y,(t\vee s)+d(x,y))}.
    $$
\end{lemma}

\begin{proof}
    Without loss of generality, we can assume $t\geqslant s$. We just
    need to show that
    $$
    T \leqslant \frac{C}{V(x,y,t+d(x,y))}.
    $$

    Case 1: $d(x,y)\leqslant t$,
    \begin{align*}
    T  &\lesssim \int_{\Bbb R^N} \frac1{V(x,z,t)}\frac1{V(z,y,s+d(z,y))}\Big(\frac{s}{s+d(z,y)}\Big)^{\varepsilon_2}d\omega(z)\\
    & \lesssim \frac1{\omega(B(x,t))}\int_{\Bbb R^N} \frac1{V(z,y,s+d(z,y))}\Big(\frac{s}{s+d(z,y)}\Big)^{\varepsilon_2}d\omega(z).
    \end{align*}

    By the condition $d(x,y)\leqslant t$ and Lemma \ref{lem1}, we have
    \begin{align*}
    T \lesssim  \frac1{V(x,y,t+d(x,y))} .
    \end{align*}

    Case 2: $d(x,y)\geqslant t$,
    since $d(x,z)+d(y,z)\geqslant d(x,y)$, we have
    \begin{align*}
    T&\leqslant  \int_{d(x,z) \geqslant \12 d(x,y)} \frac1{V(x,z,t+d(x,z))}\Big(\frac{t}{t+d(x,z)}\Big)^{\varepsilon_1}\frac1{V(z,y,s+d(z,y))}\Big(\frac{s}{s+d(z,y)}\Big)^{\varepsilon_2}d\omega(z)\\
    &+  \int_{d(y,z) \geqslant \12 d(x,y)} \frac1{V(x,z,t+d(x,z))}\Big(\frac{t}{t+d(x,z)}\Big)^{\varepsilon_1}\frac1{V(z,y,s+d(z,y))}\Big(\frac{s}{s+d(z,y)}\Big)^{\varepsilon_2}d\omega(z)\\
    &=: T_1+T_2.
    \end{align*}

    For term $T_1$, we have
    \begin{align*}
    T_1&\lesssim \int_{d(x,z) \geqslant \12 d(x,y)} \frac1{\omega(B(x,t+d(x,y)))}\frac1{V(z,y,s+d(z,y))}\Big(\frac{s}{s+d(z,y)}\Big)^{\varepsilon_2}d\omega(z) \\
    &\lesssim \frac1{V(x,y,t+d(x,y))}\int_{\Bbb R^N} \frac1{V(z,y,s+d(z,y))}\Big(\frac{s}{s+d(z,y)}\Big)^{\varepsilon_2}d\omega(z).
    \end{align*}

    By Lemma \ref{lem1}, we have
    \begin{align*}
    T_1 \lesssim   \frac1{V(x,y,t+d(x,y))}.
    \end{align*}

    For term $T_2$,
    \begin{align*}
    T_2&\lesssim \int_{d(y,z) \geqslant \12 d(x,y)}  \frac1{V(x,z,t+d(x,z))}\Big(\frac{t}{t+d(x,z)}\Big)^{\varepsilon_1}\frac1{\omega (B(y,d(x,y)))}d\omega(z) \\
    &=  \frac1{\omega (B(y,d(x,y)))}\int_{\R^N}  \frac1{V(x,z,t+d(x,z))}\Big(\frac{t}{t+d(x,z)}\Big)^{\varepsilon_1}d\omega(z).
    \end{align*}

    By the condition $d(x,y)\geqslant t$ and Lemma \ref{lem1}, we have
    \begin{align*}
    T_2 \lesssim  \frac1{V(x,y,d(x,y))} \lesssim
    \frac1{V(x,y,t+d(x,y))}.
    \end{align*}
    This complete the proof of the Lemma \ref{lemtan2}.
\end{proof}

Now we prove the Lemma \ref{aoe}.

\begin{proof}
    We begin with the estimate \eqref{aoe1}. Let $\varepsilon=\max\{\varepsilon_1,\varepsilon_2\},$ then we just need to show that
    $$
    \int_{\Bbb R^N}f_t(x,u)g_s(u,y)d\omega(u)\leqslant C \bigg(\frac{s}{t}\bigg)^{\varepsilon}
    \frac1{V(x,y,t+d(x,y))}\Big(\frac{t}{t+d(x,y)}\Big)^{\varepsilon},
    \quad \text{for }t\geqslant s.
    $$
    we write
    \begin{align*}
    &S=\int_{\Bbb R^N}f_t(x,u)g_s(u,y)d\omega(u)=\int_{\Bbb R^N}\Big(f_t(x,u)-f_t(x,y)\Big)g_s(u,y)d\omega(u).
    \end{align*}

    Note that
    \begin{align*}
    |S| &\leqslant \int_{\|u-y\|\leqslant t}  |f_t(x,u)-f_t(x,y)| \cdot |g_s(u,y)|d\omega(u) \\
    &\quad+\int_{\|u-y\|> t}  \big(|f_t(x,u)|+|f_t(x,y)|\big) \cdot |g_s(u,y)|d\omega(u)\\
    &=: I+I\!I,
    \end{align*}

    where
    \begin{align*}
    I &\lesssim \int_{\|u-y\|\leqslant t} \Big(\frac{\|u-y\|}t\Big)^{\varepsilon_0} \bigg(
    \frac1{V(x,\boldsymbol{u},t+d(x,\boldsymbol{u}))}\Big(\frac{t}{t+d(x,\boldsymbol{u})}\Big)^{\varepsilon_0}
    \\
    &\qquad + \frac1{V(x,\boldsymbol{y},t+d(x,\boldsymbol{y}))}\Big(\frac{t}{t+d(x,\boldsymbol{y})}\Big)^{\varepsilon_0}\bigg) \times \frac1{V(u,y,s+d(u,y))}\Big(\frac{s}{s+\|u-y\|}\Big)^{\varepsilon_0}d\omega(u)
    \end{align*}
    and
    \begin{align*}I\!I &\lesssim \int_{\|u-y\|> t} \bigg(
    \frac1{V(x,\boldsymbol{u},t+d(x,\boldsymbol{u}))}\Big(\frac{t}{t+d(x,\boldsymbol{u})}\Big)^{\varepsilon_0}
    + \frac1{V(x,\boldsymbol{y},t+d(x,\boldsymbol{y}))}\Big(\frac{t}{t+d(x,\boldsymbol{y})}\Big)^{\varepsilon_0}\bigg)\\
    &\qquad \times
    \frac1{V(u,y,s+d(u,y))}\Big(\frac{s}{s+\|u-y\|}\Big)^{\varepsilon_0}d\omega(u).
    \end{align*}

    For term I, since $\|u-y\|\leqslant t,$ we have
    \begin{align*}
    \Big(\frac{\|u-y\|}t\Big)^{\varepsilon_0}
    \Big(\frac{s}{s+\|u-y\|}\Big)^{\varepsilon_0}&\leqslant \bigg(\frac{\|u-y\|}t\bigg)^{\varepsilon}
    \bigg(\frac{s}{\|u-y\|}\bigg)^{\varepsilon}
    \bigg(\frac{s}{s+\|u-y\|}\bigg)^{\varepsilon_0-\varepsilon}
    \\&=\bigg(\frac{s}t\bigg)^{\varepsilon}
    \bigg(\frac{s}{s+\|u-y\|}\bigg)^{\varepsilon_0-\varepsilon}.
    \end{align*}

    Thus
    \begin{align*}
    I &\lesssim \bigg(\frac{s}t\bigg)^{\varepsilon}
    \int_{\|u-y\|\leqslant t}
    \bigg( \frac1{V(x,\boldsymbol{u},t+d(x,\boldsymbol{u}))}\Big(\frac{t}{t+d(x,\boldsymbol{u})}\Big)^{\varepsilon_0}
    + \frac1{V(x,\boldsymbol{y},t+d(x,\boldsymbol{y}))}\Big(\frac{t}{t+d(x,\boldsymbol{y})}\Big)^{\varepsilon_0}\bigg)\\
    &\qquad \times \frac1{V(u,y,s+d(u,y))} \bigg(\frac{s}{s+\|u-y\|}\bigg)^{\varepsilon_0-\varepsilon}d\omega(u).
    \end{align*}
    Let
    $$
    I_1= \int_{\|u-y\|\leqslant t}
    \frac1{V(x,u,t+d(x,u))}\Big(\frac{t}{t+d(x,u)}\Big)^{\varepsilon_0}
    \frac1{V(u,y,s+d(u,y))}\bigg(\frac{s}{s+\|u-y\|}\bigg)^{\varepsilon_0-\varepsilon}d\omega(u)$$
    and
    $$
    I_2=\int_{\|u-y\|\leqslant t}
    \frac1{V(x,y,t+d(x,y))}\Big(\frac{t}{t+d(x,y)}\Big)^{\varepsilon_0}\frac1{V(u,y,s+d(u,y))}\bigg(\frac{s}{s+\|u-y\|}\bigg)^{\varepsilon_0-\varepsilon}d\omega(u),
    $$
    then $I\lesssim \big(\frac{s}t\big)^{\varepsilon} \cdot(I_1+I_2).$ Note that
    \begin{align*}
    I_2&\leqslant \frac1{V(x,y,t+d(x,y))}\Big(\frac{t}{t+d(x,y)}\Big)^{\varepsilon_0}\int_{\Bbb
        R^N}
    \frac1{V(u,y,s+d(u,y))} \bigg(\frac{s}{s+d(u,y)}\bigg)^{\varepsilon_0-\varepsilon} d\omega(u)\\
    &\lesssim
    \frac1{V(x,y,t+d(x,y))}\bigg(\frac{t}{t+d(x,y)}\bigg)^{\varepsilon},
    \end{align*}
    where we apply   Lemma \ref{lem1} in the last inequality above.

    Note that $$
    I_1\leqslant \int_{d(u,y)\leqslant t}
    \frac1{V(x,u,t+d(x,u))}\Big(\frac{t}{t+d(x,u)}\Big)^{\varepsilon_0}
    \frac1{V(u,y,s+d(u,y))}\bigg(\frac{s}{s+d(u,y)}\bigg)^{\varepsilon_0-\varepsilon}d\omega(u).$$
    We will discuss it in the following two
    cases: $d(x,y) \leqslant 2t$ and $d(x,y)> 2t.$

    Case 1: $d(x,y) \leqslant 2t,$ applying the Lemma \ref{lemtan2}, we have
    \begin{align*}
    I_1  \lesssim \frac{1}{V(x,y,t+d(x,y))}\lesssim
    \frac1{V(x,y,t+d(x,y))}\bigg(\frac{t}{t+d(x,y)}\bigg)^{\varepsilon}.
    \end{align*}

    Case 2:$d(x,y)> 2t,$ by $d(u,y)\leqslant t<\frac{1}{2}d(x,y)$ and
    $d(x,u)+d(y,u)\geqslant d(x,y),$ we obtain $d(x,u)\geqslant
    \frac{1}{2}d(x,y).$ And hence,
    \begin{align*}
    I_1 &\leqslant \int_{d(u,y)\leqslant t}
    \frac1{V(x,u,t+d(x,u))}\bigg(\frac{t}{t+d(x,u)}\bigg)^{\varepsilon} \bigg(\frac{t}{t+d(x,u)}\bigg)^{\varepsilon_0-\varepsilon} \\
    &\qquad \times
    \frac1{V(u,y,s+d(u,y))}\bigg(\frac{s}{s+d(u,y)}\bigg)^{\varepsilon_0-\varepsilon}d\omega(u)\\
    &\lesssim \int_{\Bbb R^N}
    \frac1{V(x,u,t+d(x,u))}\bigg(\frac{t}{t+d(x,y)}\bigg)^{\varepsilon} \bigg(\frac{t}{t+d(x,u)}\bigg)^{\varepsilon_0-\varepsilon} \\
    &\qquad \times
    \frac1{V(u,y,s+d(u,y))}\bigg(\frac{s}{s+d(u,y)}\bigg)^{\varepsilon_0-\varepsilon}d\omega(u)
    \\ &\lesssim
    \frac1{V(x,y,t+d(x,y))}\bigg(\frac{t}{t+d(x,y)}\bigg)^{\varepsilon},
    \end{align*}
    where we apply the  Lemma \ref{lemtan2} in the last inequality above.

    Therefore
    $$
    I\lesssim \big(\frac{s}t\big)^{\varepsilon} \cdot(I_1+I_2)\lesssim \big(\frac{s}t\big)^{\varepsilon}
    \frac1{V(x,y,t+d(x,y))}\Big(\frac{t}{t+d(x,y)}\Big)^{\varepsilon}.
    $$

    For term $I\!I$, Let
    \begin{align*}I\!I_1 = \int_{\|u-y\|> t}
    \frac1{V(x,u,t+d(x,u))}\Big(\frac{t}{t+d(x,u)}\Big)^{\varepsilon_0}
    \frac1{V(u,y,s+d(u,y))}\Big(\frac{s}{s+\|u-y\|}\Big)^{\varepsilon_0}d\omega(u)
    \end{align*}

    and
    \begin{align*}I\!I_2 = \int_{\|u-y\|> t}  \frac1{V(x,y,t+d(x,y))}\Big(\frac{t}{t+d(x,y)}\Big)^{\varepsilon_0}\frac1{V(u,y,s+d(u,y))}\Big(\frac{s}{s+\|u-y\|}\Big)^{\varepsilon_0}d\omega(u)
    \end{align*}

    Note that
    \begin{align*}I\!I_2 &= \frac1{V(x,y,t+d(x,y))}\Big(\frac{t}{t+d(x,y)}\Big)^{\varepsilon_0} \int_{\|u-y\|> t}  \frac1{V(u,y,s+d(u,y))}\\
    &\qquad \times
    \bigg(\frac{s}{s+\|u-y\|}\bigg)^{\varepsilon_0-\varepsilon}\bigg(\frac{s}{s+\|u-y\|}\bigg)^{\varepsilon}d\omega(u)\\
    &\lesssim \frac1{V(x,y,t+d(x,y))}\bigg(\frac{t}{t+d(x,y)}\bigg)^{\varepsilon} \int_{\Bbb R^N}  \frac1{V(u,y,s+d(u,y))}\\
    &\qquad \times
    \bigg(\frac{s}{s+d(u,y)}\bigg)^{\varepsilon_0-\varepsilon}\bigg(\frac{s}{t}\bigg)^{\varepsilon}d\omega(u)\\
    &\lesssim \big(\frac{s}t\big)^{\varepsilon}
    \frac1{V(x,y,t+d(x,y))}\Big(\frac{t}{t+d(x,y)}\Big)^{\varepsilon},
    \end{align*}
    where we apply the  Lemma \ref{lem1} in the last inequality above.

    For term $I\!I_1,$ since  $d(x,u)+d(y,u)\geqslant d(x,y),$ we have
    \begin{align*}I\!I_1 &\leqslant \int_{d(x,u)\geqslant \frac{1}{2}d(x,y)\atop\|u-y\|> t}
    \frac1{V(x,u,t+d(x,u))}\Big(\frac{t}{t+d(x,u)}\Big)^{\varepsilon_0}
    \frac1{V(u,y,s+d(u,y))}\Big(\frac{s}{s+\|u-y\|}\Big)^{\varepsilon_0}d\omega(u)\\
    &\ \ \ +\int_{d(y,u)\geqslant \frac{1}{2}  d(x,y) \atop\|u-y\|> t}
    \frac1{V(x,u,t+d(x,u))}\Big(\frac{t}{t+d(x,u)}\Big)^{\varepsilon_0}
    \frac1{V(u,y,s+d(u,y))}\Big(\frac{s}{s+\|u-y\|}\Big)^{\varepsilon_0}d\omega(u)\\
    &=: I\!I_{11}+I\!I_{12}.
    \end{align*}
    Note that
    \begin{align*}I\!I_{11} &\leqslant \int_{\Bbb R^N}
    \frac1{V(x,u,t+d(x,u))}\bigg(\frac{t}{t+d(x,y)}\bigg)^\varepsilon
    \bigg(\frac{t}{t+d(x,u)}\bigg)^{\varepsilon_0-\varepsilon}\\
    &\qquad \times \frac1{V(u,y,s+d(u,y))}\bigg(\frac{s}{t}\bigg)^\varepsilon
    \bigg(\frac{s}{s+d(u,y)}\bigg)^{\varepsilon_0-\varepsilon}d\omega(u)\\
    &\lesssim \big(\frac{s}t\big)^{\varepsilon}
    \frac1{V(x,y,t+d(x,y))}\bigg(\frac{t}{t+d(x,y)}\bigg)^{\varepsilon},
    \end{align*}
    where we apply  Lemma \ref{lemtan2} in the last inequality above.
    Moreover,
    \begin{align*}I\!I_{12} &= \int_{d(y,u)\geqslant \frac{1}{2}  d(x,y)\atop\|u-y\|> t}
    \frac1{V(x,u,t+d(x,u))}\Big(\frac{t}{t+d(x,u)}\Big)^{\varepsilon_0}\\
    &\qquad \times
    \frac1{V(u,y,s+d(u,y))}\bigg(\frac{s}{s+\|u-y\|}\bigg)^\varepsilon\bigg(\frac{s}{s+\|u-y\|}\bigg)^{\varepsilon_0-\varepsilon}d\omega(u).
    \end{align*}
    Since $d(y,u) \geqslant \frac{1}{2}  d(x,y)$ and $\|u-y\|> t,$ we have
    $\|u-y\|\geqslant \frac{1}{2}(t+\|u-y\|) \geqslant \frac{1}{2}(t+d(u,y)) \geqslant \frac{1}{4}(t+d(x,y)).$

    Therefore $$ \bigg(\frac{s}{s+\|u-y\|}\bigg)^\varepsilon\lesssim
    \bigg(\frac{s}{t+d(x,y)}\bigg)^\varepsilon=\bigg(\frac{s}{t}\bigg)^\varepsilon
    \bigg(\frac{t}{t+d(x,y)}\bigg)^\varepsilon,
    $$
    which implies
    \begin{align*}I\!I_{12} &\lesssim \bigg(\frac{s}{t}\bigg)^\varepsilon
    \bigg(\frac{t}{t+d(x,y)}\bigg)^\varepsilon \int_{\Bbb R^N}
    \frac1{V(x,u,t+d(x,u))}\Big(\frac{t}{t+d(x,u)}\Big)^{\varepsilon_0}\\
    &\qquad \times
    \frac1{V(u,y,s+d(u,y))}\bigg(\frac{s}{s+d(u,y)}\bigg)^{\varepsilon_0-\varepsilon}d\omega(u)\\
    &\lesssim \big(\frac{s}t\big)^{\varepsilon}
    \frac1{V(x,y,t+d(x,y))}\bigg(\frac{t}{t+d(x,y)}\bigg)^{\varepsilon},
    \end{align*}
    where we apply the  Lemma \ref{lemtan2} in the last inequality
    above. This completes the proof of the estimate \eqref{aoe1}. The proof of the estimate \eqref{aoe2} is almost the same. To be precise, replacing $d(x,u)$ by $\|x-u\|$ for all fractions $\frac{t}{t+d(x,u)}$ and $d(x,y)$ by $\|x-y\|$ for all fractions $\frac{t}{t+d(x,y)},$ respectively, yields the proof of the estimate \eqref{aoe2}. We leave the details to the reader.
\end{proof}
We point out that the estimate \eqref{aoe1} of Lemma \ref{aoe} will
be used for the proof of Lemma \ref{Lem aoe} and while the estimate
\eqref{aoe2} will be needed for establishing the weak-type discrete
Calder\'on reproducing formula in next {\bf Section}.

We return to the proof of Lemma \ref{Lem aoe}, that is, we show that
if $T$ is a Dunkl-Calder\'on-Zygmund singular integral satisfying
$T(1)=T^*(1)=0$ and $T\in WBP,$ then
\begin{align*}
&\bigg|\int_{\Bbb R^N}\int_{\Bbb R^N} D_k(x,u)K(u,v)D_j(v,y)
d\omega(u)d\omega(v)\bigg|\\
&\lesssim  r^{-|k-j|\varepsilon'} \frac1{V(x,y, r^{-j\vee
        -k}+d(x,y))}\Big(\frac{r^{-j\vee -k}}{r^{-j\vee
        -k}+d(x,y)}\Big)^\gamma,
\end{align*}
where $\gamma, \varepsilon'\in (0,\varepsilon)$ and $\varepsilon$ is the regularity exponent of the kernel of $T.$

To this end, we may assume $k\leqslant j.$ Observe that
$D_k(x,\cdot)$ is a smooth molecule in $\mathbb M(1,1, t, x)$ with
$t=r^{-k}, x\in \R^N$ and $D_j(\cdot,y)$ is a smooth molecule in
$\mathbb M(1,1, s, y)$ with $s=r^{-j}, y\in \R^N$  Set
$\widetilde{D}_k(x,v)=\int_{\Bbb R^N} D_k(x,u)K(u,v)d\omega(u).$ By
Theorem \ref{1.3}, for any $0<\varepsilon_0<1,$
$\widetilde{D}_k(x,\cdot)$ is a weak smooth molecule in
$\widetilde{\mathbb M}(\varepsilon_0,\varepsilon_0, t, x)$ with
$t=r^{-k}.$ Note that when $k\leqslant j$,then $t\geqslant s.$
Applying the estimate \eqref{aoe1} in the Lemma \ref{aoe} yields
\begin{align*}
&\bigg|\int_{\Bbb R^N}\int_{\Bbb R^N} D_k(x,u)K(u,v)D_j(v,y)
d\omega(u)d\omega(v)\bigg|\\
&\lesssim r^{(k-j)\varepsilon'} \frac1{V(x,y,
    r^{-k}+d(x,y))}\Big(\frac{r^{-k}}{r^{-k}+d(x,y)}\Big)^\gamma,
\end{align*}
where $\varepsilon', \gamma <\varepsilon_0.$

Similarly, if $j\leqslant k,$ then $\int_{\Bbb R^N} K(u,v)D_j(v,y)
d\omega(v)$ is a weak smooth molecule and repeating the same proof
gives the desired estimate.

Finally, to finish the proof of the Theorem \ref{1.2}, it remains to
consider the general case: $T(1)\in BMO(\omega)(\R^N,\omega)$ and
$T^*(1)\in BMO(\R^N,\omega).$ To handle this case, we recall the
paraproduct operators on space of homogeneous type. We begin with
the following definition of the test functions in space of
homogeneous type $(\Bbb R^N,\|\cdot\|,\omega):$

\begin{definition}\label{test} A function $f(x)$ defined on $\R^N$ is said to be a test function if there exits a constant $C$ such that for $0<\beta\leqslant 1, \gamma>0, r>0$ and $x_0\in \R^N,$
    \begin{enumerate}
        \item[(i)] $\displaystyle f(x)\leqslant \frac C{V(x, r+\|x-x_0\|)}\Big(\frac{r}{r+\|x-x_0\|}\Big)^\gamma;$\\[4pt]
        \item[(ii)] $\displaystyle |f(x)-f(x')|\leqslant C\Big(\frac{\|x-x'\|}{r+\|x-x_0\|}\Big)^\beta\frac{1}{V(x,r+\|x-x_0\|)}\Big(\frac{r}{r+\|x-x_0\|}\Big)^\gamma, \\ \quad \text{for}\quad \|x-x'\|\leqslant \frac{1}{2}(r+\|x-x_0\|);$
        \item[(iii)] $\displaystyle \int_{\R^N} f(x)d\omega(x)=0.$
    \end{enumerate}
    We denote such a test function by $f\in \mathcal M(\beta,\gamma,r,x_0)$ and $\|f\|_{\mathcal M(\beta,\gamma,r,x_0)},$ the norm in $\mathcal M(\beta,\gamma,r,x_0),$ is defined by the smallest $C$ satisfying the above conditions (i) and (ii).
\end{definition}

Applying Coifman's decomposition for the identity operator and the Calder\'on-Zygmund operator theory, the discrete Calder\'on reproducing formula in space of homogeneous type is given by the following

\begin{theorem}\label{dCRh}
    Let $\{S_k\}_{k\in \Bbb Z}$ be a Coifman's approximations to the identity and set
    ${D}_k := {S}_k - {S}_{k-1}$. Then there exists a family of operators
    $\{\widetilde{D}_k\}_{k\in \Bbb Z}$ such that for any fixed $x_{Q}\in Q$ with $k\in \Bbb Z$ and $Q$ are "$r-$dyadic cubes" with the side length $r^{-M-k},$
    $$f(x)=\sum\limits_{k=-\infty}^\infty \sum\limits_{Q\in Q^k}\omega(Q){\widetilde D}_k(x,x_{Q})D_k(f)(x_{Q}),$$
    where the series converge in $L^p(\omega)$, $1<p<\infty$, $\mathcal M(\beta,\gamma,r,x_0),$ and in $(\mathcal M(\beta,\gamma,r,x_0))^\prime,$ the dual of in $\mathcal M(\beta,\gamma,r,x_0),$ and moreover, the kernels of the operators $\widetilde{D}_k$
    satisfy the the following conditions:
    \begin{enumerate}
        \item[(i)] $\displaystyle|\widetilde{D}_k(x,y)|\leqslant C\frac1{V_k(x)+V_k(y)+V(x,y)}\frac{r^{-k}}{r^{-k}+\|x-y\|};$
        \item[(ii)] $\displaystyle|\widetilde{D}_k(x,y)-\widetilde{D}_k(x',y)|\leqslant C\frac{\|x-x'\|}{r^{-k}+\|x-x'\|}\frac1{V_k(x)+V_k(y)+V(x,y)}\frac{r^{-k}}{r^{-k}+\|x-y\|}$,
        \item[]for $\|x-x'\|\leqslant (r^{-k}+\|x-y\|)/2;$
        \item[(iii)] $\displaystyle \int_{\Bbb R^N} \widetilde{D}_k(x,y)d\omega(x)=0\qquad \text{for all}\ y\in \Bbb R^N;$
        \item[(iv)] $\displaystyle \int_{\Bbb R^N} \widetilde{D}_k(x,y)d\omega(y)=0\qquad \text{for all}\ x\in \Bbb R^N.$
    \end{enumerate}
    Similarly, there exists a family of linear operators
    $\{\widetilde{\widetilde{D}}_k\}_{k\in \Bbb Z}$ such that for any fixed $x_{Q}\in Q,$
    $$f(x)=\sum\limits_{k=-\infty}^\infty \sum\limits_{Q\in Q^k}\omega(Q)D_k(x,x_{Q})\widetilde{\widetilde{D}}_k(f)(x_{Q}),$$
    where the kernels of the operators $\widetilde{\widetilde{D}}_k$
    satisfy the above conditions {\rm (i), (iii), (iv)} and {\rm (ii)} with $x$ and $y$ interchanged,
\end{theorem}
The papraproduct operator is defined by
\begin{definition} Supporse that $\{S_k\}, \{D_k\}$ and $\{\widetilde{\widetilde{D}}_k\}$ are same as defined above. The paraproduct operator of $f\in \mathcal M(\beta,\gamma,r,x_0)^\prime$ is defined by
    $$\Pi_b(f)(x)=\sum\limits_{k=-\infty}^\infty \sum\limits_{Q\in Q^k}\omega(Q)D_k(x,x_{Q})\widetilde{\widetilde{D}}_k(b)(x_{Q})S_k(f)(x_Q),$$
    where $b\in BMO_d(\mathbb R^N, \omega).$
\end{definition}
We need the following result on space of homogeneous type:
\begin{theorem}\label{para} The paraproduct operator is the Calder\'on-Zygmund operator. Moreover, $\Pi_b(1)=b$ in the topology $(H^1_d, BMO_d), (\Pi_b)^*(1)=0$ and there exists a constant $C$ such that for $1<p<\infty,$
    $$ \|\Pi_b(f)\|_p\leqslant C\|b\|_{BMO_d}\|f\|_p.$$
\end{theorem}
See \cite{HMY} for all these results and the details of the proofs.

Observe that the classical Calder\'on-Zygmund operator on space of homogeneous type is also the Dunkl-Calder\'on-Zygmund operator. Suppose now that both $T(1)$ and $T^*(1)$ belong to $BMO_d(\mathbb
R^N, \omega).$ Set ${\widetilde T}=T -\Pi_{T(1)}-(\Pi_{T^*(1)})^*.$
Then ${\widetilde T}$ is a Dunkl-Calder\'on-Zygmund singular
integral operator. Moreover, ${\widetilde T}(1)=({\widetilde
T})^*(1)=0$ and ${\widetilde T}\in WBP.$ Therefore, ${\widetilde T}$
is bounded on $L^2(\R^N, \omega)$ and hence, $T$ is also bounded on
$L^2(\R^N, \omega)$.

The proof of Theorem \ref{1.2} is completed.

\section{{Weak-Type Discrete Calder\'on Reproducing Formula and   Littlewood-Paley Theory on $L^p, 1<p<\infty$}}

In this section, we will apply the Dunkl-Calder\'on-Zygmund operator Theory, namely the Cotlar-Stein Lemma and the $L^p$ boundedness of the Dunkl-Calder\'on-Zygmund operators to establish the weak-type discrete Calder\'on reproducing formula and   Littlewood-Paley theory on $L^p, 1<p<\infty.$

\subsection{{Weak-Type Discrete Calder\'on Reproducing Formula}}
\ \\

We begin with the following Calder\'on reproducing formula provided in \cite{ADH}.

\begin{theorem}\label{CRF}
    For $f\in L^2(\R^N,\omega)$,
    \begin{eqnarray}\label{CRF1}f(x)=\int_0^\infty \psi_{t}\ast q_{t}\ast f(x)\frac{dt}{t} \end{eqnarray}
    where $q_t=t\partial_{t}p_t$ with the Poisson kernel $p_t, q_{t}\ast  f(x)=\int_{\mathbb{R}^{N}}q_{t}(x,y)f(y)d\omega(y)$  and $\psi_{t}\ast  f(x)=\int_{\mathbb{R}^{N}}\psi_{t}(x,y)f(y)d\omega(y)$ with $\psi(x)$
    being a radial Schwartz function supported in the unit ball $B(0,1).$
\end{theorem}
    We remark that in \cite{ADH}, the authors established the estimates for $p_t(x,y),$ the Poisson kernel, as follows:
    \begin{equation}
    |\partial^m_t\partial^\alpha_x\partial^\beta_y p_t(x,y)|
    \lesssim t^{-m-|\alpha|-|\beta|}p_t(x,y).
    \end{equation}
    In \cite{DH1}, the authors improved the estimates for $p_t(x,y)$ by
    $$|p_t(x,y)|\lesssim  \frac1{V(x,y,t+d(x,y))}\frac{t}{t+\|x-y\|}$$
    and hence,
    \begin{equation}\label{heat estimate}
    |\partial^m_t\partial^\alpha_x\partial^\beta_y p_t(x,y)|
    \lesssim t^{-m-|\alpha|-|\beta|}\frac1{V(x,y,t+d(x,y))}\frac{t}{t+\|x-y\|}.
    \end{equation}
    These estimates indicade that $q_{t}(x,y)$ for all $x,y\in \mathbb{R}^{N}$ and $t>0,$ satisfy the following conditions:
    \begin{eqnarray*}
        &\textup{(i)}& |q_t(x,y)|\leqslant  {1\over V(x,y,t+d(x,y))}\frac{t}{t+\|x-y\|},\\
        &\textup{(ii)}& |q_t(x,y)-q_t(x',y)|\\
        &&\leqslant {\|x-x'\|\over t} \Big({1\over V(\boldsymbol{x},y,t+d(\boldsymbol{x},y))}\frac{t}{t+\|\boldsymbol{x}-y\|}+{1\over V(\boldsymbol{x'},y,t+d(\boldsymbol{x'},y))}\frac{t}{t+\|\boldsymbol{x'}-y\|}\Big),\\
        &\textup{(iii)}& |q_t(x,y)-q_t(x,y')|\\
        &&\leqslant {\|y-y'\|\over t}\Big({1\over V(x,\boldsymbol{y},t+d(x,\boldsymbol{y}))}\frac{t}{t+\|x-\boldsymbol{y}\|}+{1\over V(x,\boldsymbol{y'},t+d(x,\boldsymbol{y'}))}\frac{t}{t+\|x-\boldsymbol{y'}\|}\Big),\\
        &\textup{(vi)}& \int_{\mathbb{R}^{N}}q_t(x,y) d\omega(y)
        = \int_{\mathbb{R}^{N}}q_t(x,y) d\omega(x)=0.
    \end{eqnarray*}
    And $\psi_{t}(x,y)$ for all $x,y\in \mathbb{R}^{N}$, $t>0$ satisfy the similar conditions as $q_t(x,y)$ but $\psi_{t}(x,y)$ is supported in $\{(x,y): d(x,y)\leqslant t\}.$

It is east to check that $q_t(x,y)$ are smooth molecules. Indeed, $q_t(\cdot,y)\in \mathbb M(1, 1, t, y)$ for any fixed $y$ and
$q_t(x,\cdot)\in \mathbb M(1, 1, t, x)$ for any fixed $x,$ and similarly for $\psi_t(x,y).$

Now we show Theorem \ref{1.5} with $p=2.$ The main tools are the almost orthogonal estimates and the Cotlar-Stein Lemme.

\begin{proof}[\bf Proof of Theorem \ref{1.5} with $p=2$] Let $1<r\leq
r_0,$ where $r_0$ will be chosen later, and $t_j=r^{-j}$ and $\psi_j=\psi_{r^{-j}}$ and $q_j=q_{r^{-j}}.$ For given $f\in L^2(\R^N,\omega),$ we decompose $f$ as follows.
\begin{align*}
f(x)&=\int_0^\infty \psi_t\ast q_t\ast f(x)\frac{dt}{t}\\
&=-\sum\limits_{j=-\infty}^\infty\int_{r^{-j}}^{r^{-j+1}}
\psi_{j}\ast q_{j}\ast f(x)\frac{dt}{t}+\sum\limits_{j=-\infty}^\infty\int_{r^{-j}}^{r^{-j+1}} \Big[\psi_{j}\ast q_{j}\ast f(x)-\psi_t\ast q_t\ast f(x)\Big]\frac{dt}{t}\\
&=-\ln r\sum\limits_{j=-\infty}^\infty \psi_{j}\ast q_{j}\ast
f(x)+R_1(f),
\end{align*}
where
$R_1f(x)=\sum\limits_{j=-\infty}^\infty\int_{r^{-j}}^{r^{-j+1}}
\Big[\psi_t\ast q_t\ast f(x)- \psi_{j}\ast q_{j}\ast
f(x)\Big]\frac{dt}{t}.$ Further, we decompose
\begin{align*}
&-\ln r\sum\limits_{j=-\infty}^\infty
\psi_{j}\ast q_{j}\ast f(x) \\
&=-\ln r\sum\limits_{j=-\infty}^\infty\sum\limits_{Q\in Q^j}
\int_{Q}
\psi_{j}(x,y)q_{j}\ast f(y)d\omega(y) \\
&=-\ln r\sum\limits_{j=-\infty}^\infty\sum\limits_{Q\in Q^j}w(Q)
\psi_{j}(x,x_{Q})q_{j}\ast
f(x_{Q}) \\
&\quad +\ln r\sum\limits_{j=-\infty}^\infty\sum\limits_{Q\in
Q^j}\int_{Q} \Big[\psi_{j}(x,x_{Q})q_{j}\ast
f(x_{Q})-\psi_{t}(x,y)q_{j}\ast
f(y)\Big]d\omega(y)\\
&=T_{M}(f)(x)+R_{M}{f}(x),
\end{align*}
where
$$T_M(f)(x)=-\ln r\sum\limits_{j=-\infty}^\infty\sum\limits_{Q\in Q^j}w(Q) \psi_{j}(x,x_{Q})q_{j}\ast
f(x_{Q})$$
and
$$R_M(f)(x)=\ln r\sum\limits_{j=-\infty}^\infty\sum\limits_{Q\in Q^j}\int_{Q}
\Big[\psi_{j}(x,x_{Q})q_{j}\ast
f(x_{Q})-\psi_{t}(x,y)q_{j}\ast
f(y)\Big]d\omega(y),$$
$M$ is any fixed integer and  $Q^j$ are all ``$r-$dyadic cubes"   with the side length $r^{-M-j}$, and $x_{Q}$ is any fixed point in $Q$.

The identity operator on $L^2(\R^N,\omega)$ can be written by the
following
$$I=T_{M}+R_1 + R_{M}.$$
We claim that there exist $r_0$ for any $1<r\leqslant r_0$ and $M$
such that
$$\|R_1(f)\|_2\leqslant C(r-1)\|f\|_2$$
and
$$\|R_M(f)\|_2\leqslant r^{-M}\|f\|_2.$$
Assuming the claim for the moment, if we choose $r_0$ to be close to
1 and $M$ to be sufficiently large, then $\|R_1+R_M\|_{2,2}<1.$
Observing that $T_M=I-R_1-R_{M},$ therefore, $T_M$ is bounded on
$L^2{(\R^N,\omega)}$ and moreover, $(T_M)^{-1},$ the inverse of
$T_M,$ is also bounded on $L^2{(\R^N,\omega)}$ since
$(T_M)^{-1}=(I-R_1-R_{M})^{-1}.$ For each $f\in L^2{(\R^N,\omega)}$
setting $h=(T_M)^{-1}f$ then $h\in L^2(\R^N,\omega)$ with
$\|f\|_2\sim \|h\|_2.$ We obtain the weak-type discrete Calder\'on
reproducing formula
$$f(x)=T_M(T_M)^{-1}f(x)=-\ln r\sum\limits_{j=-\infty}^\infty\sum\limits_{Q\in Q^j}w(Q)
\psi_{j}(x,x_{Q})q_{j}\ast h(x_{Q}).$$ To see that the above series
converges in $L^2(\R^N,\omega),$ we need the Littlewood-Paley
estimates on $L^2(\R^N,\omega),$ namely, $\|S(f)\|_2\leqslant
C\|f\|_2,$ where $S(f)$ is the square function of $f.$ See details
in next {\bf Subsection}. Indeed,
\begin{align*}
&\|\sum\limits_{|j|\geqslant n}^\infty\sum\limits_{Q\in Q^j}w(Q)
\psi_{Q}(\cdot,x_{Q})q_{Q}\ast h(x_{Q})\|^2_2\\
&\leqslant C\|\Big(\sum\limits_{|j|\geqslant n}\sum\limits_{Q\in
Q^j}|q_{Q}\ast h(x_{Q})|^2\chi_Q(x)\Big)^{\frac{1}{2}}\|_2^2,
\end{align*}
where the last term tends to zero as $n$ tends to $\infty.$ See more details in next {\bf Subsection}.

We now return to the proof of the claim. The proof for $R_1$ follows from the Cotlar-Stein Lemma. To this end, we have
$$R_1f(x)=-\sum\limits_{j=-\infty}^\infty\int_{r^{-j}}^{r^{-j+1}} \Big[\psi_t\ast q_t\ast f(x)-
\psi_{j}\ast q_{j}\ast
f(x)\Big]\frac{dt}{t}=-\sum\limits_{j=-\infty}^\infty\int_{r^{-j}}^{r^{-j+1}}S_j(f)(x)\frac{dt}{t},$$
where $S_jf(x)=S_j\ast f(x)=\int_{\R^N} S_j(x,y)f(y)d\omega(y)$ with
$S_j(x,y)=\psi_{t}\ast q_{ t}(x,y)-\psi_{j}\ast q_{j}(x,y).$

We first show that $S_j(\cdot,y)\in \mathbb M(1, 1, r^{-j}, y)$ and the proof for $S_j(x,\cdot)\in \mathbb M(1, 1, r^{-j}, x)$ is similar.
Note that
$$S_j(x,y)=\int_{\R^N}\Big\{\big(\psi_{t}(x,z)-\psi_{j}(x,z)\big)q_{t}(z,y)+
\psi_{j}(x,z)\big(q_{t}(z,y)-q_{j}(z,y)\big)\Big\}d\omega(z).$$ To
estimate the size condition of $S_j(x,y)$, applying the estimates in
\eqref{heat estimate} with $\partial _tq_t(x,y)=\partial_t
p_t(x,y)+t\partial^2 _tp_t(x,y)$ implies that for $r^{-j}\leqslant
t\leqslant r^{-j+1},$
\begin{equation}\label{d1}
\begin{aligned}
|q_{t}(z,y)-q_{j}(z,y)|
&\leqslant C\frac{(r^{-j+1}-r^{-j})} {{r^{-j}}}\frac1{V(z,y,r^{-j}+d(z,y))}\frac{r^{-j}}{r^{-j}+\|z-y\|}\\
&\leqslant
C(r-1)\frac1{V(z,y,r^{-j}+d(z,y))}\frac{r^{-j}}{r^{-j}+\|z-y\|}.
\end{aligned}
\end{equation}
Similarly,
\begin{equation}\label{d2}
|\psi_{t}(x,z)-\psi_{j}(x,z)| \leqslant
C(r-1)\frac1{V(x,z,r^{-j}+d(x,z))}\frac{r^{-j}}{r^{-j}+\|x-z\|}.
\end{equation}
Therefore, Lemma \eqref{lemtan2} gives
\begin{equation}\label{M-size}
|S_j(x,y)| \leqslant
C(r-1)\frac1{V(x,y,r^{-j}+d(x,y))}\frac{r^{-j}}{r^{-j}+\|x-y\|}.
\end{equation}
To see the smooth condition of $S_j(\cdot, y)$, we write
\begin{align*}
&S_j(x,y)-S_j(x',y)\\
&\qquad =\big(\psi_{t}-\psi_{j}\big)q_{t}(x,y)+
\psi_{j}\big(q_{t}-q_{j}\big)(x,y)-\big(\psi_{t}-\psi_{j}\big)q_{t}(x',y)-
\psi_{j}\big(q_{t}-q_{j}\big)(x',y)\\
&\qquad=\int_{\R^N} \Big[\big(\psi_{t}(x,z)-\psi_{t}(x',z)\big) -\big(\psi_{j}(x,z)-\psi_{j}(x',z)\big)\Big]q_{t}(z,y)d\omega(z)\\
&\qquad\qquad+
\int_{\R^N}\Big[\psi_{j}(x,z)-\psi_j(x',z)\Big]\big[q_{t}(z,y)-q_{j}(z,y)\big]d\omega(z)
\end{align*}
To estimate the first term for $r^{-j}\leqslant t\leqslant
r^{-j+1},$ applying Lemma \eqref{heat estimate} with $m=1,
|\alpha|=1, |\beta|=0$ yields
\begin{align*}
&|\big(\psi_{t}(x,z)-\psi_{t}(x',z)\big) -\big(\psi_{j}(x,z)-\psi_{j}(x',z)\big)|\\
&\leqslant C\frac{(r^{-j+1}-r^{-j})}{ {r^{-j}}}\frac{\|x-x'\|}{r^{-j}} \\
&\qquad\times
\Big[\frac{1}{V(\boldsymbol{x},z,r^{-j}+d(\boldsymbol{x},z))}\frac{r^{-j}}{r^{-j}+\|\boldsymbol{x}-z\|}
+\frac{1}{V(\boldsymbol{x'},z,r^{-j}+d(\boldsymbol{x'},z))}\frac{r^{-j}}{r^{-j}+\|\boldsymbol{x'}-z\|}\Big].
\end{align*}
 By Lemma \eqref{lemtan2}, for $r^{-j}\leqslant t\leqslant r^{-j+1}$ we have
\begin{align*}
&\bigg|\int_{\R^N} \Big[\big(\psi_{t}(x,z)-\psi_{t}(x',z)\big) -\big(\psi_{j}(x,z)-\psi_{j}(x',z)\big)\Big]q_{t}(z,y)d\omega(z)\bigg|\\
&\leqslant C\frac{(r^{-j+1}-r^{-j})}{ {r^{-j}}}\frac{\|x-x'\|}{r^{-j}}\Big[\frac{1}{V(\boldsymbol{x},y,r^{-j}+d(\boldsymbol{x},y))}\frac{r^{-j}}{r^{-j}+\|\boldsymbol{x}-y\|} \\
&\hskip6cm+\frac{1}{V(\boldsymbol{x'},y,r^{-j}+d(\boldsymbol{x'},y))}\frac{r^{-j}}{r^{-j}+\|\boldsymbol{x'}-y\|}\Big]\\
&\leqslant C(r-1)\frac{\|x-x'\|}{r^{-j}}\Big[\frac{1}{V(\boldsymbol{x},y,r^{-j}+d(\boldsymbol{x},y))}\frac{r^{-j}}{r^{-j}+\|\boldsymbol{x}-y\|}\\
&\hskip6cm+\frac{1}{V(\boldsymbol{x'},y,r^{-j}+d(\boldsymbol{x'},y))}\frac{r^{-j}}{r^{-j}+\|\boldsymbol{x'}-y\|}\Big].
\end{align*}
To estimate the second term, applying the smoothness condition on $\psi_j$ and the estimate in \eqref{d1} and then apply the Lemma \eqref{lemtan2} yield
\begin{align*}
&\bigg|\int_{\R^N}\Big[\psi_{j}(x,z)-\psi_j(x',z)\Big]\big[q_{t}(z,y)-q_{j}(z,y)\big]d\omega(z)\bigg|\\
&\leqslant C\frac{(r^{-j+1}-r^{-j})}{ {r^{-j}}}\frac{\|x-x'\|}{r^{-j}}\Big[\frac{1}{V(\boldsymbol{x},y,r^{-j}+d(\boldsymbol{x},y))}\frac{r^{-j}}{r^{-j}+\|\boldsymbol{x}-y\|} \\
&\hskip6cm+\frac{1}{V(\boldsymbol{x'},y,r^{-j}+d(\boldsymbol{x'},y))}\frac{r^{-j}}{r^{-j}+\|\boldsymbol{x'}-y\|}\Big]\\
&\leqslant C(r-1)\frac{\|x-x'\|}{r^{-j}}\Big[\frac{1}{V(\boldsymbol{x},y,r^{-j}+d(\boldsymbol{x},y))}\frac{r^{-j}}{r^{-j}+\|\boldsymbol{x}-y\|} \\
&\hskip6cm+\frac{1}{V(\boldsymbol{x'},y,r^{-j}+d(\boldsymbol{x'},y))}\frac{r^{-j}}{r^{-j}+\|\boldsymbol{x'}-y\|}\Big].
\end{align*}
We obtain
$$|S_j(x,y)|
\leqslant
C(r-1)\frac1{V(x,y,r^{-j}+d(x,y))}\frac{r^{-j}}{r^{-j}+\|x-y\|}$$
and
\begin{align*}
|S_j(x,y)-S_j(x',y)| &\leqslant C(r-1)\frac{\|x-x'\|}{r^{-j}}
      \Big[\frac{1}{V(\boldsymbol{x},y,r^{-j}+d(\boldsymbol{x},y))}\frac{r^{-j}}{r^{-j}+\|\boldsymbol{x}-y\|} \\
&\qquad
+\frac{1}{V(\boldsymbol{x'},y,r^{-j}+d(\boldsymbol{x'},y))}\frac{r^{-j}}{r^{-j}+\|\boldsymbol{x'}-y\|}\Big],
\end{align*}
which together with the fact $\int_{\R^N}S_j(x,y)d\omega(x)=0$
implies that for any fixed $y, S_j(x,y))$ is a smooth molecule in
$\mathbb M(1, 1, r^{-j}, y).$ Moreover,
$$\|S_j(\cdot,y)\|_{\mathbb M(1, 1, r^{-j}, y)}\leqslant C(r-1).$$
Similarly, for any fixed $x$, $S_j(x,\cdot)\in \mathbb M(1, 1, r^{-j}, x)$
and
$$\|S_j(x,\cdot)\|_{\mathbb M(1, 1, r^{-j}, x)}\leqslant C(r-1).$$
Applying the same argument, we also obtain that $S^*_j(\cdot,y)\in \mathbb M(1, 1, r^{-j}, y)$ for any fixed $y$ and $S^*_j(x,\cdot)\in \mathbb M(1, 1, r^{-j}, x)$ for any fixed $x$.

By \eqref{aoe2} in the Lemma \eqref{aoe}, there exists $\varepsilon<1$ such that

$$|S_jS^*_k(x,y)|\leqslant C(r-1)^2r^{-|j-k|\varepsilon}\frac1{V(x,y,r^{-j\vee -k}+d(x,y))}\bigg(\frac{r^{-j\vee -k}}{r^{-j\vee -k}+d(x,y)}\bigg)^{\varepsilon}.$$
Let $T_j(x,y)=\int_{r^{-j}}^{r^{-j+1}}S_j(x,y)\frac{dt}{t}.$ Then
\begin{align*}
|T_jT^*_k(x,y)|
&\leqslant \int_{r^{-j}}^{r^{-j+1}}\int_{r^{-k}}^{r^{-k+1}}|S_jS^*_k(x,y)|\frac{dt}{t}\frac{ds}{s} \\
&\leqslant C(r-1)^2r^{-|j-k|\varepsilon}\frac1{V(x,y,r^{-j\vee
-k}+d(x,y))}\bigg(\frac{r^{-j\vee -k}}{r^{-j\vee
-k}+d(x,y)}\bigg)^{\varepsilon}
\end{align*}
Let $f\in L^2(\R^N,\omega)$. Then
\begin{align*}
\|T_jT^*_kf\|_{L^2(\R^N,\omega)}^2 &=\int_{\mathbb R^N}
\bigg|\int_{\mathbb
R^N}T_jT^*_k(x,y)f(y)d\omega(y)\bigg|^2d\omega(x)
\end{align*}
By the definition, $d(x,y)=\min\limits_{\sigma\in
G}\|\sigma(x)-y\|$,
\begin{align*}
|T_jT^*_k(x,y)| &\leqslant  \sum\limits_{\sigma\in
G}C(r-1)^2r^{-|j-k|\varepsilon}\frac1{\omega\big(B(y,r^{-j\vee
-k}+\|\sigma(x)-y\|)\big)}
\bigg(\frac{r^{-j\vee -k}}{r^{-j\vee -k}+\|\sigma(x)-y\|}\bigg)^{\varepsilon}\\
&\sim \sum\limits_{\sigma\in
G}C(r-1)^2r^{-|j-k|\varepsilon}\frac1{\big(B(\sigma(x),r^{-j\vee
-k}+\|\sigma(x)-y\|)\big)} \bigg(\frac{r^{-j\vee -k}}{r^{-j\vee
-k}+\|\sigma(x)-y\|}\bigg)^{\varepsilon}.
\end{align*}
Since $G$ is finite,
\begin{align*}
\|T_jT^*_kf\|_{L^2(\R^N,\omega)}^2 &\lesssim \sum\limits_{\sigma\in
G}(r-1)^4r^{-2|j-k|\varepsilon}\int_{\mathbb R^N}
\big(Mf(\sigma(x))\big)^2d\omega(x)\\
&= \sum\limits_{\sigma\in
G}(r-1)^4r^{-2|j-k|\varepsilon}\int_{\mathbb R^N}
\big(Mf(x)\big)^2d\omega(x)\\
&\lesssim (r-1)^4r^{-2|j-k|\varepsilon}\|f\|_{L^2(\R^N,\omega)}^2,
\end{align*}
where $M$ denotes the Hardy-Littlewood maximal operator on $(\Bbb
R^N, \|\cdot\|,\omega)$. Hence
$$\|T_jT^*_k\|_{L^2(\R^N,\omega)\mapsto L^2(\R^N,\omega)}\leqslant C(r-1)^2r^{-|j-k|\varepsilon}. $$
By the Cotlar-Stein's lemma, for $1<r\leqslant r_0,$
$$\left\|\sum\limits_{j=-\infty}^\infty\int_{r^{-j}}^{r^{-j+1}} \Big[\psi_tq_tf(x)-
\psi_{r^{-j}}q_{t_j}f(x)\Big]\frac{dt}{t}\right\|_{L^2(\R^N,\omega)}\le
C(r-1)^2.$$ This implies that $\|R_1f\|_{2,2}\leqslant C(r-1).$

It remains to show the claim for $R_M.$ To do this, we write
\begin{align*}
R_Mf(x)&= -\ln r\sum\limits_{j=-\infty}^\infty\sum\limits_{Q\in
Q^j}\int_{Q}
\big(\psi_{Q}(x,y)-\psi_{Q}(x,x_{Q})\big)q_{Q}f(y)d\omega(y) \\
&\quad +\ln r\sum\limits_{j=-\infty}^\infty\sum\limits_{Q\in
Q^j}\int_{Q}
\psi_{Q}(x,x_{Q})\big(q_{Q}f(x_Q)-q_{Q}f(y)\big)d\omega(y)\\
&=R^1_Mf(x)+R^2_Mf(x),
\end{align*}
where $Q^j$ are all cubes with the side length $r^{-M-j}.$

We estimate $\|R^2_M(f)\|_2$ only since the proof for
$\|R^1_M(f)\|_2$ is similar.

Denote $R^2_M(f)(x)=\int_{\R^N} E_{j}(x,z)f(z)d\omega(z),$ where
$$E_{j}(x,z)=
\sum\limits_{Q\in
Q^j}\int_{Q}\psi_{Q}(x,x_{Q})\big(q_{Q}(x_{Q},z)-q_{Q}(y,z))\big)d\omega(y).$$
We show that $E_j(x,z)$ are smooth molecules. Precisely,
$E_j(\cdot,z)\in \mathbb M(1, 1, r^{-j}, z)$ and $E_j(x,\cdot)\in
\mathbb M(1, 1, r^{-j}, x)$. Moreover, $\|E_j(x, \cdot)\|_{\mathbb
M(1, 1, r^{-j}, {x})}\leqslant C r^{-M}$ and similarly for
$E_j(\cdot,x).$ We show $E_j(x,\cdot)\in \mathbb M(1, 1, r^{-j}, x)$
only. It is clear that $\int_{\mathbb R^N} E_j(x,z)d\omega(z)=0.$
Observe that if $y\in Q\in Q^j,$
\begin{align*}
&|q_{Q}(x_{Q},z)-q_{Q}(y,z)|\\
&\leqslant  C\frac{\|y-x_Q\|}{r^{-j}}
\Big[\frac{1}{V(\boldsymbol{x_Q},z,r^{-j}+d(\boldsymbol{x_Q},z))}\frac{r^{-j}}{r^{-j}+\|\boldsymbol{x_Q}-z\|}+
    \frac{1}{V(\boldsymbol{y},z,r^{-j}+d(\boldsymbol{y},z))}\frac{r^{-j}}{r^{-j}+\|\boldsymbol{y}-z\|} \Big]\\
&\leqslant
C\frac{r^{-j-M}}{r^{-j}}\frac{1}{V(y,z,r^{-j}+d(y,z))}\frac{r^{-j}}{r^{-j}+\|y-z\|},
\end{align*}
since $r^{-j}+d(y,z)\sim r^{-j}+d(x_Q,z)$ and $r^{-j}+\|y-z\|\sim r^{-j}+\|x_Q-z\|$ for $y\in Q\in Q^j.$ Note that $\psi_{Q}(x,x_{Q})\sim \psi_{Q}(x,y),$ thus,
\begin{align*}
|E_{j}(x,z)|
&\leqslant C\sum\limits_{Q\in Q^j}\int_{Q}|\psi_{Q}(x,x_Q)\big(q_{Q}(x_{Q},z)-q_{Q}(y,z)\big)|d\omega(y)\\
&\leqslant C r^{-M}\int_{\R^N}|\psi_{r^{-j}}(x,y)|\frac{1}{V(y,z,r^{-j}+d(y,z))}\frac{r^{-j}}{r^{-j}+\|y-z\|}d\omega(y)\\
&\leqslant Cr^{-M}
\frac1{V(x,z,r^{-j}+d(x,z))}\frac{r^{-j}}{r^{-j}+\|x-z\|}.
\end{align*}
Now we verify the smooth condition of $E_j(x,\cdot).$ To this end,we write
\begin{align*}
&E_{j}(x,z)-E_{j}(x,z')\\
&=\sum\limits_{Q\in
Q^j}\int_{Q}\psi_{Q}(x,x_{Q})\Big[\big(q_{Q}(y,z)-q_{Q}(x_{Q},z)\big)
-\big(q_{Q}(y,z')-q_{Q}(x_{Q},z')\Big]d\omega(y).
\end{align*}
Applying the estimate in \eqref{heat estimate} with $|\alpha|=|\beta|=1$ yields

\begin{align*}
&|[q_Q(y,z)-q_Q(x_Q,z)]-[q_Q(y,z')-q_Q(x_Q,z')]|\\
&\lesssim \frac{\|y-x_Q}{r^{-j}}\frac{\|z-z'\|}{r^{-j}}\Bigg\{
{\frac{1}{V(\boldsymbol{y},z,r^{-j}+d(\boldsymbol{y},z))}}{\frac{r^{-j}}{r^{-j}+\|\boldsymbol{y}-z\|}}\\
&\qquad
+\frac1{V(\boldsymbol{x_Q},z,r^{-j}+d(\boldsymbol{x_Q},z))}\frac{r^{-j}}{r^{-j}+\|\boldsymbol{x_Q}-z\|}+
\frac1{V(\boldsymbol{y},z',r^{-j}+d(\boldsymbol{y},z'))}\frac{r^{-j}}{r^{-j}+\|\boldsymbol{y}-z'\|}\\
&\qquad
+\frac1{V(\boldsymbol{x_Q},z',r^{-j}+d(\boldsymbol{x_Q},z'))}\frac{r^{-j}}{r^{-j}+\|\boldsymbol{x_Q}-z'\|}\Bigg\}.
\end{align*}
Observe that if $y\in Q\in Q^j,$ then $(y,z,r^{-j}+d(y,z))\sim (x_Q,z,r^{-j}+d(x_Q,z))$ and $(y,z',r^{-j}+d(y,z'))\sim (x_Q,z',r^{-j}+d(x_Q,z')),$ and $\psi_{Q}(x,x_Q)\sim \psi_{Q}(x,y)$ thus,
\begin{align*}
&|E_{j}(x,z)-E_j(x,z')|\\
&\leqslant C r^{-M}\frac{\|z-z'\|}{r^{-j}}
\sum\limits_{Q\in Q^j}\int_{Q}|\psi_{Q}(x,y)|\Big(\frac{1}{V(y,\boldsymbol{z},r^{-j}+d(y,\boldsymbol{z}))}\frac{r^{-j}}{r^{-j}+\|y-\boldsymbol{z}\|}\\
&\hskip7cm +\frac{1}{V(y,\boldsymbol{z'},r^{-j}+d(y,\boldsymbol{z'}))}\frac{r^{-j}}{r^{-j}+\|y-\boldsymbol{z'}\|}\Big)d\omega(y)\\
&\leqslant Cr^{-M}\frac{\|z-z'\|}{r^{-j}}\int_{\R^N}|\psi_{r^{-j}}(x,y)|\Big(\frac{1}{V(y,\boldsymbol{z},r^{-j}+d(y,\boldsymbol{z}))}\frac{r^{-j}}{r^{-j}+\|y-\boldsymbol{z}\|}\\
&\hskip7cm+\frac{1}{V(y,\boldsymbol{z'},r^{-j}+d(y,\boldsymbol{z'}))}\frac{r^{-j}}{r^{-j}+\|y-\boldsymbol{z'}\|}\Big)d\omega(y)\\
&\leqslant Cr^{-M}\frac{\|z-z'\|}{r^{-j}}\Big(\frac{1}{V(x,\boldsymbol{z},r^{-j}+d(x,\boldsymbol{z}))}\frac{r^{-j}}{r^{-j}+\|x-\boldsymbol{z}\|}\\
&\hskip7cm
+\frac{1}{V(x,\boldsymbol{z'},r^{-j}+d(x,\boldsymbol{z'}))}\frac{r^{-j}}{r^{-j}+\|x-\boldsymbol{z'}\|}\Big),
\end{align*}
which implies that $E_j(x,\cdot)\in \mathbb M(1, 1, r^{-j}, x)$ with
$\|E_j(x, \cdot)\|_{\mathbb M(1, 1, r^{-j}, {x})}\leqslant C
r^{-M}.$

The same estimates hold for $E^*_{j},$ the adjoint operator of $E_{j}.$
By \eqref{aoe2} in the Lemma \ref{aoe},
$$ |E_{j}E^*_{k}(x,y)|
\lesssim r^{-2M}r^{-|j-k|\varepsilon}\frac1{V(x,y,r^{-j\vee -k}+d(x,y))}\bigg(\frac{r^{-j\vee -k}}{r^{-j\vee -k}+\|x-y\|}\bigg)^{\varepsilon}.$$
Similar to the proof for $R_1$, applying Cotlar-Stein's Lemma gets $\|R^2_{M}f\|_{L^2} \lesssim r^{-2M}\|f\|_{L^2}$. The claim is proved and hence, the proof of Theorem \ref{1.5} with $p=2$ is complete.
\end{proof}

\begin{proof}[\bf Proof of Theorem \ref{1.5} with $1<p<\infty$]

The main toll of the proof of the Theorem \ref{1.5} with $1<p<\infty$ is the
Dunkl-Calder\'on-Zygmund singular integral operator theory. Namely,
we will show that $R_1$ and $R_M$ are the Dunkl-Calder\'on-Zygmund
operators. To this end, observe that these two operators have been
proved to be bounded on $L^2(\R^N, \omega)$ with the operator norms
less than $C(r-1)$ and $Cr^{-M}$ for $1<r\leqslant r_0,$
respectively. It suffices to show that the kernels of $R_1$ and
$R_M$ are the Dunkl-Calder\'on-Zygmund singular integral operator
kernels.

We first verify the kernel of $R_1.$ We recall
$$R_1f(x)=-\sum\limits_{j=-\infty}^\infty\int_{r^{-j}}^{r^{-j+1}} \Big[\psi_tq_tf(x)-
\psi_{j}q_{j}f(x)\Big]\frac{dt}{t}=-\sum\limits_{j=-\infty}^\infty\int_{r^{-j}}^{r^{-j+1}}S_j(f)(x)\frac{dt}{t},$$
where $S_j(f)(x)=\int_{\R^N} S_j(x,y)f(y)d\omega(y)$ with
$S_j(x,y)=\psi_{t}q_{ t}(x,y)-\psi_{j}q_{j}(x,y).$ $R_1(x,y),$ the
kernel  of $R_1,$ can be written as
$$R_1(x,y)=-\sum\limits_{j=-\infty}^\infty\int_{r^{-j}}^{r^{-j+1}} S_j(x,y)\frac{dt}{t}.$$
Observe that if $r^{-j}\leqslant t\leqslant r^{-j+1},$
\begin{align*}
|S_j(x,y)|&\leqslant C(r-1)\frac1{V(x,y,r^{-j}+d(x,y))}\Big(\frac{r^{-j}}{r^{-j}+\|x-y\|}\Big)^{\varepsilon}\\
&\leqslant
C(r-1)\frac1{V(x,y,t+d(x,y))}\Big(\frac{t}{t+\|x-y\|}\Big)^{\varepsilon}
\end{align*}
for any $0<\varepsilon<1.$
Hence,
\begin{align*}
|R_1(x,y)|&\lesssim \sum\limits_{j=-\infty}^\infty\int_{r^{-j}}^{r^{-j+1}} |S_j(x,y)|\frac{dt}{t} \\
&\lesssim (r-1)\sum\limits_{j=-\infty}^\infty\int_{r^{-j}}^{r^{-j+1}} \frac1{V(x,y,t+d(x,y))}\Big(\frac{t}{t+\|x-y\|}\Big)^{\varepsilon}\frac{dt}{t}\\
&\lesssim (r-1)\int_{0}^\infty
\frac1{V(x,y,t+d(x,y))}\Big(\frac{t}{t+\|x-y\|}\Big)^{\varepsilon}\frac{dt}{t}.
\end{align*}
Applying the same proof as given in the Proposition \ref{pr300}, implies that for any fixed $0<\varepsilon<1,$
$$|R_1(x,y)|\lesssim (r-1) \Big(\frac{d(x,y)}{\|x-y\|}\Big)^\varepsilon\frac1{\omega(B(x,d(x,y)))}.$$
To verify the regularity conditions for $R_1,$ we apply the
following estimate for $r^{-j}\leqslant t\leqslant r^{-j+1},$
\begin{align*}
& |S_j(x,y)-S_j(x',y)| \\
&\lesssim (r-1){\|x-x'\|\over r^{-j}} \\
&\quad\times \Big({1\over V(\boldsymbol{x},y,r^{-j}+d(\boldsymbol{x},y))}\Big(\frac{r^{-j}}{r^{-j}+\|\boldsymbol{x}-y\|}\Big)^\varepsilon+{1\over V(\boldsymbol{x'},y,r^{-j}+d(\boldsymbol{x'},y))}\Big(\frac{r^{-j}}{r^{-j}+\|\boldsymbol{x'}-y\|}\Big)^\varepsilon\Big)\\
&\lesssim (r-1){\|x-x'\|\over {t}} \\
&\quad\times \Big({1\over
V(\boldsymbol{x},y,t+d(\boldsymbol{x},y))}\Big(\frac{t}{t+\|\boldsymbol{x}-y\|}\Big)^\varepsilon+{1\over
V(\boldsymbol{x'},y,+d(\boldsymbol{x'},y))}\Big(\frac{t}{t+\|\boldsymbol{x'}-y\|}\Big)^\varepsilon\Big).
\end{align*}
Therefore, for any $0<\varepsilon<1$ and $\|y-y'\|\leqslant
d(x,y)/2,$
\begin{align*}
|R_1(x,y)-R_1(x',y)|&\leqslant C(r-1)\int_{0}^{\infty}{\|x-x'\|\over
t}
\Big({1\over V(\boldsymbol{x},y,t+d(\boldsymbol{x},y))}\Big(\frac{t}{t+\|\boldsymbol{x}-y\|}\Big)^\varepsilon \\
&\hskip3cm+{1\over V(\boldsymbol{x'},y,+d(\boldsymbol{x'},y))}\Big(\frac{t}{t+\|\boldsymbol{x'}-y\|}\Big)^\varepsilon\Big)\frac{dt}{t}\\
&\lesssim
C(r-1)\Big(\frac{\|x-x'\|}{\|x-y\|}\Big)^\varepsilon\frac1{\omega(B(x,d(x,y)))}\qquad
\mbox{for } \|x-x'\|\leqslant d(x,y)/2.
\end{align*}
Similarly, for any $0<\varepsilon<1$ and $\|y-y'\|\leqslant
d(x,y)/2,$
$$|R_1(x,y)-R_1(x,y')|\lesssim (r-1) \Big(\frac{\|y-y'\|}{\|x-y\|}\Big)^\varepsilon\frac1{\omega(B(x,d(x,y)))}.$$
Now we verify that $R_M$ is a Dunkl-Calder\'on-Zygmund operator. We recall
\begin{align*}
R_Mf(x)&= -\ln r\sum\limits_{j=-\infty}^\infty\sum\limits_{Q\in
Q^j}\int_{Q}
\big(\psi_{Q}(x,y)-\psi_{Q}(x,x_{Q})\big)q_{Q}f(y)d\omega(y)  \\
&\quad- \ln r\sum\limits_{j=-\infty}^\infty\sum\limits_{Q\in
Q^j}\int_{Q}
\psi_{Q}(x,x_Q)\big(q_{Q}f(x_Q)-q_{Q}f(y)\big)d\omega(y)\\
&=R^1_Mf(x)+R^2_Mf(x).
\end{align*}

Denote $R^1_M(f)(x)=-\ln r\sum\limits_{j=-\infty}^\infty\int_{\R^N}
E_{j}(x,z)f(z)d\omega(z),$ where
$$E_{j}(x,z)=\sum\limits_{Q\in Q^j}\int_{Q}\big(\psi_{Q}(x,y)-\psi_{Q}(x,x_{Q})\big)q_{Q}(y,z)d\omega(y).$$
Applying the same proof for $0<\varepsilon<1$,
$$
|E_j(x,z)|\lesssim r^{-M}\frac1{V(x,z,r^{-j}+d(x,z))}\Big(\frac{r^{-j}}{r^{-j}+\|x-z\|}\Big)^\varepsilon.
$$
Note that $R^1_M(x,z)=-\ln r\sum\limits_{j=-\infty}^\infty
E_j(x,z)$. Similar to the estimate for $R_1,$ we obtain that for any
$0<\varepsilon<1,$
\begin{center}
    \begin{equation}\label{R1m1}
    |R^1_M(x,z)|\lesssim r^{-M}\Big(\frac{d(x,z)}{\|x-z\|}\Big)^\varepsilon \frac1{\omega(B(x,d(x,y)))};
    \end{equation}
    \begin{equation}\label{R1m2}
    |R^1_M(x,z)-R^1_M(x,z')|\lesssim  r^{-M}\Big(\frac{\|z-z'\|}{\|x-z\|}\Big)^\varepsilon\frac{1}{\omega(B(x,d(x,z)))}
    \end{equation}
    for $\|z-z'\|\leqslant  d(x,z)/2;$
    \begin{equation}\label{R1m3}
    |R^1_M(x',z)-R^1_M(x,z)|\lesssim  r^{-M}\Big(\frac{\|x-x'\|}{\|x-z\|}\Big)^\varepsilon\frac1{\omega(B(x,d(x,z)))}
    \end{equation}
    for $\|x-x'\|\leqslant d(x,z)/2.$
\end{center}
Observe that
$$R^2_M(x,z)=\ln r\sum\limits_{j=-\infty}^\infty\sum\limits_{Q\in Q^j}\int_{Q}\psi_{Q}(x,x_{Q})\big(q_{Q}(y,x_Q)-q_Q(z,y)\big)d\omega(y).$$
Similarly, the kernel $R^2_M(x,z)$ also satisfies the conditions \eqref{R1m1}-\eqref{R1m3}.

Suppose that $T$ is the Dunkl-Calder\'on-Zygmund operator. We denote
$\|T\|_{dcz}=\|T\|_{2,2}+\|K\|_{dcz}$ by the
Dunkl-Calder\'on-Zygmund operator norm, where $\|K\|_{dcz}$ the
minimum of the constants in \eqref{size3}, \eqref{smooth y3} and
\eqref{smooth x3}. The $L^2(\R^N,\omega)$ boundedness and all size
and smoothness conditions for $R_1$ and $R_M$ obtained above imply
that $R_1$ and $R_M$ are the Dunkl-Calder\'on-Zygmund operator with
the operator norm $\|R_1\|_{dcz}\lesssim (r-1)$ and
$\|R_M\|_{dcz}\lesssim r^{-M},$ where $1<r\leqslant r_0.$

Applying Theorem \ref{th1.1}, $R_1$ and $R_M$ are bounded on $L^p(\R^N,\omega), 1<p<\infty,$ with $\|R_1\|_{p,p}\lesssim (r-1)$ and $\|R_M\|_{p,p}\lesssim r^{-M}.$ Therefore, if take $r_0$ to be close to 1 and $M$ is large enough, then $T_M=I-R_1-R_M$ is invertible and the inverse of $T_M$ is bounded on $L^p(\R^N,\omega)$ for all $1<p<\infty.$ For $f\in L^2(\R^N,\omega)\cap L^p(\R^N,\omega), 1<p<\infty,$ set $h=(T_M)^{-1}f.$ Then $\|h\|_p\sim\|f\|_{p}.$ Moreover,
$$f(x)=T_M(T_M)^{-1}f(x)=-\ln r\sum\limits_{j=-\infty}^\infty\sum\limits_{Q\in Q^j}w(Q) \psi_{j}(x,x_{Q})q_{j}
h(x_{Q}).$$

It remains to show that the above series converges in the
$L^p(\R^N,\omega), 1<p<\infty,$ norm. The proof will follow from
Theorem \ref{1.6}, the Littlewood-Paley estimates on
$L^p(\R^N,\omega)$ for $1<p<\infty.$ See the details in next {\bf
Subsection}.
\end{proof}

\subsection{{Littlewood-Paley Theory on $L^p, 1<p<\infty$}}
\ \\

We now give the proof of Theorem \ref{1.6}.

\begin{proof}[\bf Proof of Theorem \ref{1.6} with $p=2$]

We first show
$$\|S(f)\|_2^2=\sum\limits_{j=-\infty}^\infty\sum\limits_{Q\in Q^j}\omega(Q)|q_{Q}f(x_{Q})|^2\leqslant C\|f\|_2^2.$$
Te proof follows from a duality argument together with the almost orthogonal estimates. Indeed,
\begin{align*}
&\sup \Big\{\Big( \sum\limits_{j=-\infty}^\infty\sum\limits_{Q\in Q^j}\omega(Q)|q_{Q}f(x_{Q})|^2\Big)^{\frac{1}{2}}: f\in L^2, \|f\|_2\leqslant 1\Big \}\\
&=\sup \Big\{ \sum\limits_{j=-\infty}^\infty\sum\limits_{Q\in Q^j}\omega(Q)| q_{Q}f(x_{Q})\cdot g_Q|: f\in L^2, \|f\|_2\leqslant 1, \sum\limits_{j=-\infty}^\infty\sum\limits_{Q\in Q^j}\omega(Q)|g_Q|^2\leqslant 1 \Big\}\\
&=\sup \Big\{ \|\sum\limits_{j=-\infty}^\infty\sum\limits_{Q\in
Q^j}\omega(Q)g_Q q_{Q}(x_{Q},\cdot)\|_2:
\sum\limits_{j=-\infty}^\infty\sum\limits_{Q\in
Q^j}\omega(Q)|g_Q|^2\leqslant 1 \Big \}.
\end{align*}
Observe that
\begin{align*}
\|\sum\limits_{j=-\infty}^\infty\sum\limits_{Q\in
Q^j}\omega(Q)g_Qq_{Q}(x_{Q},\cdot)\|^2_2
&=\Big\langle\sum\limits_{j=-\infty}^\infty\sum\limits_{Q\in Q^j}\omega(Q)g_Qq_{Q}(x_{Q},\cdot), \sum\limits_{j'=-\infty}^\infty\sum\limits_{Q'\in Q^{j'}}\omega(Q')g_{Q'}q_{Q'}(x_{Q'},\cdot)\Big\rangle \\
&=\sum\limits_{j=-\infty}^\infty\sum\limits_{Q\in
Q^j}\sum\limits_{j'=-\infty}^\infty\sum\limits_{Q'\in
Q^{j'}}\omega(Q)\omega(Q')\langle q_{Q'}(x_{Q'},\cdot),
q_{Q}(x_Q,\cdot)\rangle g_Qg_{Q'}
\end{align*}
and for $Q\in Q^j$ and $Q'\in Q^{j'},$ applying the almost orthogonal estimates implies
\begin{align*}
&|\langle q_{Q'}(x_{Q'},\cdot), q_{Q}(x_Q,\cdot)\rangle|\\
&\lesssim r^{-|j-j'|\varepsilon}\frac{1}{\omega(B(x_Q, r^{-j-M}\vee
r^{-j'-M}+d(x_Q,x_{Q'})))}\Big(\frac{r^{-j-M}\vee
r^{-j'-M}}{r^{-j-M}\vee r^{-j'-M}+\|x_Q-x_{Q'}\|}\Big)^\varepsilon.
\end{align*}
The above estimate yields
$$\sum\limits_{j=-\infty}^\infty\sum\limits_{Q\in Q^j}\omega(Q)|\langle q_{Q'}(x_{Q'},\cdot), q_{Q}(x_Q,\cdot)\rangle|\leqslant C.$$
Therefore,
$$\sum\limits_{j=-\infty}^\infty\sum\limits_{Q\in Q^j}\sum\limits_{j'=-\infty}^\infty\sum\limits_{Q'\in Q^{j'}}\omega(Q)\omega(Q')\langle q_{Q'}(x_{Q'},\cdot), q_{Q}(x_Q,\cdot)\rangle g_Qg_{Q'}$$
$$\leqslant C \Big(\sum\limits_{j=-\infty}^\infty\sum\limits_{Q\in Q^j}\omega(Q)|g_Q|^2\Big)^{\frac{1}{2}}\Big(\sum\limits_{j'=-\infty}^\infty\sum\limits_{Q'\in Q^{j'}}\omega(Q')|g_{Q'}|^2\Big)^{\frac{1}{2}}.$$
We have proved the estimate $\|S(f)\|_2\leqslant C\|f\|_2,$ which
together with the duality argument implies the estimate $\|f\|_2\leq
C\|S(f)\|_2.$ We leave the detail to the reader.
\end{proof}

\begin{proof}[\bf Proof of Theorem \ref{1.6} with $1<p<\infty$]
The idea of the proof is to consider $S(f)$ as a
vector-valued Dunkl-Calder\'on-Zygmund operator. To this end, we recall
\begin{align*}
\mathcal Sf(x)&=\Big(\sum\limits_{Q}|q_{Q}f(x_{Q})|^2\chi_Q(x)\Big)^{\frac{1}{2}}\\
&=\Big(\sum\limits_{Q}|\int_{\R^N}
\chi_Q(x)q_{Q}(x_Q,y)f(y)d\omega(y)|^2\Big)^{\frac{1}{2}}.
\end{align*}
This leads to introduce the $H-$ valued function $\{f_Q(x)\},$ where $Q$ are all $r$-dyadic cubes in $\R^N$ and the norm of $\{f_Q(x)\}$ is defined by
$$\|f_Q(x)\|_{H}:= \Big(\sum\limits_{Q}|f_Q(x)|^2\Big)^{\frac{1}{2}}.$$
Supporse that $T$ is an $L^2-$ bounded $H-$ valued operator defined
by $$T(f)(x)=\bigg\{\int_{\R^N} K_Q(x,y)(f)(y)d\omega(y) \bigg\},$$
where $K(x,y)=\{K_Q(x,y)\}$ satisfies the following condition: for
some $0<\varepsilon\leqslant 1,$
\begin{center}
    \begin{equation}\label{smooth y4}
    \|K(x,y)-K(x,y')\|_{H}\lesssim \Big(\frac{\|y-y'\|}{\|x-y\|}\Big)^\varepsilon\frac{1}{\omega(B(x,d(x,y)))} \qquad {\rm for}\ \|y-y'\|\leqslant  d(x,y)/2;
    \end{equation}
\end{center}
Then $T$ is bounded from $L^p(\R^N,\omega)$ to $L^p(H),$ that is, there exists a constant $C$ such that for $1<p<\infty,$
$$\|Tf\|_{L^p(H)}\leqslant C\|f\|_p.$$
The proof of this argument is same as in the proof of Theorem \ref{th1.1}. We leave the details of the proof to the reader.

Now we define the $H-$ valued operator $\mathcal
Sf(x)=\{K_Q(f)(x)\},$ where
$$K_Q(f)(x)=q_Q(f)(x_Q)\chi_Q(x)=\int_{\R^N}
\chi_Q(x)q_Q(x_Q,y)f(y)d\omega(y).$$

Observe that $\|\mathcal S(f)\|_{L^2(H)}=\|S(f)\|_2$ and hence
$$\|\mathcal S(f)\|_{L^2(H)}=\|S(f)\|_2\leqslant C \|f\|_2.$$
To see that $\mathcal S(f)$ is bounded from $L^p(\R^N,\omega)$ to $L^p(H), 1<p<\infty,$ it remains to verify that the kernel of $\mathcal S$ satisfies the condition \eqref{smooth y4}. To this end, we write
$K(x,y),$ the kernel of $\mathcal S,$ as $K(x,y)=\{K_Q(x,y)\}=\{\chi_Q(x)q_Q(x_Q,y)\}.$
It follows that
$$\|K(x,y)-K(x,y')\|_{H}=\|\{\chi_Q(x)[q_Q(x_Q,y)-q_Q(x_Q,y')]\}\|_{H}.$$
Therefore,
$$\|K(x,y)-K(x,y')\|_{H}=\Big(\sum\limits_{Q}\chi_Q(x)|q_Q(x_Q,y)-q_Q(x_Q,y'|^2) \Big)^{\frac{1}{2}}.$$
By the size condition on $q_Q$ with $x\in Q$ and $Q\in Q^j,$ we have
\begin{align*}
|q_Q(x_Q,y)-q_Q(x_Q,y')|&\lesssim  \frac{1}{V(x_Q,\boldsymbol{y},r^{-j}+d(x_Q,\boldsymbol{y}))}\frac{r^{-j}}{r^{-j}+\|x_Q-\boldsymbol{y}\|}\\
&\qquad
+\frac{1}{V(x_Q,\boldsymbol{y'},r^{-j}+d(x_Q,\boldsymbol{y'}))}\frac{r^{-j}}{r^{-j}+\|x_Q-\boldsymbol{y'}\|}.
\end{align*}
and by the smoothness condition on $q_Q$,

\begin{align*}
|q_Q(x_Q,y)-q_Q(x_Q,y')|&\lesssim \frac{\|y-y'\|}{r^{-j}} \Big(\frac{1}{V(x_Q,\boldsymbol{y},r^{-j}+d(x_Q,\boldsymbol{y}))}\frac{r^{-j}}{r^{-j}+\|x_Q-\boldsymbol{y}\|}\\
&\qquad+\frac{1}{V(x_Q,\boldsymbol{y'},r^{-j}+d(x_Q,\boldsymbol{y'}))}\frac{r^{-j}}{r^{-j}+\|x_Q-\boldsymbol{y'}\|}\Big).
\end{align*}

Observe that when $d(y,y')\leqslant \|y-y'\|\leq
\frac{1}{2}d(x,y)\leqslant \frac{1}{2}\|x-y\|$ then $d(x,y)\sim
d(x,y')$ and $\|x-y\|\sim \|x-y'\|.$ Moreover, if $x\in Q$ and $Q\in
Q^j$ then $V(x_Q, y, r^{-j}+d(x_Q,y))\sim V(x, y, r^{-j}+d(x,y))\sim
V(x_Q, y', r^{-j}+d(x_Q,y'))\sim V(x, y', r^{-j}+d(x,y'))$ and
$\|x_Q-y\|\sim \|x-y\|\sim \|x_Q-y'\|\sim \|x-y'\|.$ This yields
$$|q_Q(x_Q,y)-q_Q(x_Q,y')\lesssim \frac{1}{V(x,y,r^{-j}+d(x,y))}\frac{r^{-j}}{r^{-j}+\|x-y\|}$$
and
$$|q_Q(x_Q,y)-q_Q(x_Q,y')\lesssim \frac{\|y-y'\|}{r^{-j}} \frac{1}{V(x,y,r^{-j}+d(x,y))}\frac{r^{-j}}{r^{-j}+\|x-y\|}.$$
Thus, if $\|y-y'\|\leqslant \frac{1}{2}d(x,y),$ we consider three
cases: (i) $\|y-y'\|\geqslant r^{-j};$ (ii) $\|y-y'\|\leqslant
r^{-j}\leqslant \|x-y\|;$ (iii) $r^{-j}\geqslant \|x-y\|.$ For the
first case, we have
\begin{align*}
&\sum\limits_{j:\rm (i)}\sum\limits_{Q\in Q^j}\chi_Q(x)|q_Q(x_Q,y)-q_Q(x_Q,y')|^2\\
&\qquad \lesssim
\sum\limits_{j:\rm (i)}\sum\limits_{Q\in Q^j}\chi_Q(x) \Big( \frac{1}{V(x,y,r^{-j}+d(x,y))}\frac{r^{-j}}{r^{-j}+\|x-y\|}\Big)^2\\
&\qquad \lesssim
\Big(\frac{\|y-y'\|}{\|x-y\|}\frac{1}{\omega(B(x,d(x,y)))}\Big)^2.
\end{align*}
For the second case, it follows
\begin{align*}
&\sum\limits_{j:\rm (ii)}\sum\limits_{Q\in Q^j}\chi_Q(x)|q_Q(x_Q,y)-q_Q(x_Q,y')|^2\\
&\qquad \lesssim
\sum\limits_{j:\rm (ii)}\sum\limits_{Q\in Q^j}\chi_Q(x) \Big(\frac{\|y-y'\|}{r^{-j}} \frac{1}{V(x,y,r^{-j}+d(x,y))}\frac{r^{-j}}{r^{-j}+\|x-y\|}\Big)^2 \\
&\qquad \lesssim \Big(\frac{\|y-y'\|}{\|x-y\|}\ln\frac{\|x-y\|}{\|y-y'\|}\frac{1}{\omega(B(x,d(x,y)))}\Big)^2 \\
&\qquad \lesssim
\Big(\Big(\frac{\|y-y'\|}{\|x-y\|}\Big)^\varepsilon\frac{1}{\omega(B(x,d(x,y)))}\Big)^2.
\end{align*}
The last case gives
\begin{align*}
&\sum\limits_{j:\rm (iii)}\sum\limits_{Q\in Q^j}\chi_Q(x)|q_Q(x_Q,y)-q_Q(x_Q,y')|^2 \\
&\qquad \lesssim
\sum\limits_{j:\rm (iii)}\sum\limits_{Q\in Q^j}\chi_Q(x) \Big(\frac{\|y-y'\|}{r^{-j-M}} \frac{1}{V(x,y,r^{-j}+d(x,y))}\frac{r^{-j}}{r^{-j}+\|x-y\|}\Big)^2 \\
&\qquad \lesssim
\Big(\frac{\|y-y'\|}{\|x-y\|}\frac{1}{\omega(B(x,d(x,y)))}\Big)^2.
\end{align*}
These estimates imply that if $\|y-y'\|\leqslant \frac{1}{2}d(x,y),$
$$\|K(x,y)-K(x,y')\|_{H}\lesssim \Big(\frac{\|y-y'\|}{\|x-y\|}\Big)^\varepsilon\frac{1}{\omega(B(x,d(x,y)))}.$$
Again, the estimate $\|S(f)\|_p\leqslant C\|f\|_p, 1<p<\infty,$
together with the duality argument implies the estimate $\|f\|_p\leq
C\|S(f)\|_p, 1<p<\infty$.

The proof of Theorem \ref{1.6} is
complete.
\end{proof}
We now return to the proof on the convergence in Theorem \ref{1.5}, that is, for $f\in L^2(\R^N,\omega)\cap L^p(\R^N,\omega), 1<p<\infty,$ there exists a function $h$ with $\|h\|_p\sim \|f\|_p$ such that
\begin{align*}
f(x)=\sum\limits_{j=-\infty}^\infty\sum\limits_{Q\in Q^j}\omega(Q)
\psi_{Q}(x,x_{Q})q_{Q} h(x_{Q}),
\end{align*}
where the series converges in $L^2(\R^N,\omega)\cap L^p(\R^N,\omega).$

It is sufficient to show that
$$\|\sum\limits_{|j|>n}\sum\limits_{Q\in Q^j}\omega(Q)\psi_{Q}(\cdot,x_{Q})q_{Q}h(x_{Q})\|_p\rightarrow 0$$
as $n\rightarrow \infty.$

By Theorem \ref{1.6}, we only need to show
$$\|S\Big(\sum\limits_{|j|>n}\sum\limits_{Q\in Q^j}\omega(Q)\psi_{Q}(\cdot,x_{Q})q_{Q}h(x_{Q})\Big)\|_p\rightarrow 0$$
as $n\rightarrow \infty.$
Observe that
\begin{align*}
&S\Big(\sum\limits_{|j|>n}\sum\limits_{Q\in Q^j}\omega(Q)\psi_{Q}(\cdot,x_{Q})q_{Q}h(x_{Q})\Big)(x)\\
&\qquad=
\Big(\sum\limits_{Q'}\Big|q_{Q'}\Big(\sum\limits_{|j|>n}\sum\limits_{Q\in
Q^j}\omega(Q)\psi_{Q}(\cdot,x_{Q})q_{Q}h(x_{Q})\Big)(x_{Q'})\Big|^2\chi_{Q'}(x)\Big)^{\frac{1}{2}}.
\end{align*}
Note that
$$q_{Q'}\Big(\sum\limits_{|j|>n}\sum\limits_{Q\in Q^j}\omega(Q)\psi_{Q}(\cdot,x_{Q})q_{Q}h(x_{Q})\Big)(x_{Q'})=\sum\limits_{|j|>n}\sum\limits_{Q\in Q^j}\omega(Q)q_{Q'}\psi_{Q}(x_{Q'},x_{Q})q_{Q}h(x_{Q}).$$
Applying the almost orthogonal estimate \ref{aoe2} in the Lemma \ref{aoe} on $\big(q_{Q'}\psi_Q\big)(x_{Q'},x_Q)$,
we obtain that for any $\varepsilon\in (0,1)$ there exists a constant $C>0$ such that
$$|\big(q_{Q'}\psi_Q\big)(x_{Q'},x_Q)|
\lesssim r^{-|j-j'|\varepsilon} \frac1{V(x_{Q^{'}},x_{Q},r^{-j\vee
-j'}+d(x_{Q^{'}},x_{Q}))} \bigg(\frac{r^{-j \vee -j'}}{r^{-j \vee
-j'} +d(x_{Q^{'}},x_{Q})}\bigg)^{\varepsilon}.$$ By the definition,
$d(x,y)=\min\limits_{\sigma\in G}\|\sigma(x)-y\|$, for $x\in Q'$ we
get
\begin{align*}
&|q_{Q'}\psi_Q(x_{Q'},x_Q)|\chi_{Q'}(x) \\
&\lesssim r^{-|j-j'|\varepsilon} \frac1{\omega(B(x_{Q},r^{-j\vee -j'}+d(x,x_{Q})))}
\bigg(\frac{r^{-j \vee -j'}}{r^{-j \vee -j'} +d(x,x_{Q})}\bigg)^{\varepsilon}\\
&\lesssim \sum\limits_{\sigma\in G} r^{-|j-j'|\varepsilon}
\frac1{\omega(B(x_{Q},r^{-j\vee -j'}+\|\sigma(x)-x_{Q}\|))}
\bigg(\frac{r^{-j \vee -j'}}{r^{-j \vee -j'} +\|\sigma(x)-x_{Q}\|}\bigg)^{\varepsilon}\chi_{Q'}(x)\\
&\lesssim \sum\limits_{\sigma\in G} r^{-|j-j'|\varepsilon}
\frac1{\omega(B(\sigma(x),r^{-j\vee -j'}+\|\sigma(x)-x_{Q}\|))}
\bigg(\frac{r^{-j \vee -j'}}{r^{-j \vee -j'}
+\|\sigma(x)-x_{Q}\|}\bigg)^{\varepsilon}\chi_{Q'}(x)
\end{align*}
for $x\in Q'$.
Let $M$ denote the Hardy-Littlewood maximal operator on $\Bbb R^N$.  Following a discrete version of the Hardy-Littlewood maximal function estimate, see more detail of the proof of Lemma \ref{ex2} in next {\bf Subsection 4.1}, yields
\begin{equation}\label{H-L max}
\begin{aligned}
&\sum\limits_{Q\in Q^j}w(Q)\frac1{\omega(B(\sigma(x), r^{-j\vee -j'}+\|\sigma(x)-x_{Q}\|))}\\
&\qquad \times \bigg(\frac {r^{-j\vee -j'}}{r^{-j\vee -j'}+\|\sigma(x)-x_{Q}\|}\bigg)^{\varepsilon}|q_{Q}h(x_{Q})|\\
&\leqslant Cr^{|-j'-(-j'\vee -j)|{\bf N}(1-\frac 1\theta)}
\bigg\{M\Big(\sum\limits_{Q\in
Q^j}|q_{Q}h(x_{Q})|^\theta\chi_{Q}\Big)(\sigma(x))\bigg\}^{1/\theta},
\end{aligned}
\end{equation}
where $\theta$ satisfies $\frac{\bf N}{{\bf
N}+\varepsilon}<\theta<1.$ Hence,
\begin{align*}
&\sum\limits_{Q\in Q^j}\omega(Q)q_{Q'}\psi_{Q}(x_{Q'},x_{Q})q_{Q}h(x_{Q})\chi_{Q'}(x) \\
&\lesssim \sum\limits_{\sigma\in
G}r^{-|j-j'|\varepsilon}r^{|-j'-(-j'\vee -j)|{\bf N}(1-\frac
1\theta)} \bigg\{M\Big(\sum\limits_{Q\in
Q^j}|q_{Q}h(x_{Q})|^\theta\chi_{Q}\Big)(\sigma(x))\bigg\}^{1/\theta}\chi_{Q'}(x).
\end{align*}
It is clear that for $\frac{\bf N}{{\bf N}+\varepsilon}<\theta<1,$
$$\sup_{j'}\sum\limits_{j\in \Bbb Z} r^{-|j-j'|\varepsilon}r^{|-j'-(-j'\vee -j)|{\bf N}(1-\frac 1\theta)}<\infty.$$
By H\"older's inequality, we have
\begin{align*}
&\Big|\sum\limits_{|j|>n}\sum\limits_{Q\in Q^j}\omega(Q)
q_{Q'}\psi_{Q}(x_{Q'},x_{Q})q_{Q}h(x_{Q})\Big|^2\chi_{Q'}(x) \\
&\lesssim \sum\limits_{\sigma\in G}\sum\limits_{|j|>n}
r^{-|j-j'|\varepsilon}r^{|-j'-(-j'\vee -j)|{\bf N}(1-\frac 1\theta)}
\bigg\{M\Big(\sum\limits_{Q^j}|q_{Q^j}h(x_{Q^j})|^\theta\chi_{Q^j}\Big)(\sigma(x))\bigg\}^{2/\theta}\chi_{Q'}(x).
\end{align*}
This implies
\begin{align*}
&S\Big(\sum\limits_{|j|>n}\sum\limits_{Q\in Q^j}\omega(Q)\psi_{Q}(x,x_{Q})q_{Q}h(x_{Q})\Big)\\
&\lesssim \bigg\{\sum\limits_{\sigma\in G}\sum\limits_{|j|>n}
\bigg\{M\Big(\sum\limits_{Q\in
Q^j}|q_{Q}h(x_{Q})|^\theta\chi_{Q}\Big)(\sigma(x))\bigg\}^{2/\theta}\bigg\}^{1/2},
\end{align*}
where the estimate is used: for $\frac{\bf N}{{\bf
N}+\varepsilon}<\theta<1,$
$$\sup_j\sum\limits_{j'\in \Bbb Z} r^{-|j-j'|\varepsilon}r^{|-j'-(-j'\vee -j)|{\bf N}(1-\frac 1\theta)}<\infty.$$
The Fefferman-Stein vector valued maximal function inequality with $\theta<1<p<\infty$ yields
$$\Big\|S\Big(\sum\limits_{|j|>n}\sum\limits_{Q\in Q^j}\omega(Q)\psi_{Q}(\cdot,x_{Q})q_{Q}h(x_{Q})\Big)\Big\|_p\lesssim \sum\limits_{\sigma\in G}
\Big\|\Big(\sum\limits_{|j|>n}\sum\limits_{Q\in
Q^j}|q_{Q}h(x_{Q})|^2\chi_Q(\sigma(x))\Big)^{\frac{1}{2}}\Big\|_p.$$
Since $G$ is finite group and
$$\int_{\Bbb R^N}  f(\sigma(x))d\omega(x)=\int_{\Bbb R^N} f(x)d\omega(x),$$
we get
\begin{align*}
&\Big\|S\Big(\sum\limits_{|j|>n}\sum\limits_{Q\in Q^j}\omega(Q)\psi_{Q}(\cdot,x_{Q})q_{Q}h(x_{Q})\Big)\Big\|_p \\
&\lesssim \Big\|\Big(\sum\limits_{|j|>n}\sum\limits_{Q\in
Q^j}|q_{Q}h(x_{Q})|^2\chi_Q(x)\Big)^{\frac{1}{2}}\Big\|_p,
\end{align*}
where the last term tends to zero as $n\rightarrow \infty$ since
$h\in L^p(\R^N,\omega), 1<p<\infty,$ and hence, $\|S(h)\|_p\leqslant
C$ by Theorem \ref{1.6}.

\section{{Littlewood-Paley Theory and Dunkl-Hardy Space}}

In this section, we establish the Littlewood-Paley Theory for
$L^p(\mathbb R^N, \omega), p\leqslant 1$ and develop the Dunkl-Hardy
Space.

\subsection{{Littlewood-Paley square function $S(f)$ in $L^p,\,  p\leqslant 1$}}
\ \\

We begin with the proof of Theorem \ref{1.7}.

\begin{proof}[{\bf The proof of Theorem \ref{1.7}}]

    As mentioned before,
    the Dunkl-Calder\'on-Zygmund operator theory plays a crucial role.
    Let's recall the proof of Theorem \ref{1.6}. First, we decompose the
    identity operator on $L^2(\R^N,\omega)$ by $I= T_M + R_1 + R_M$ and
    then applying the Dunkl-Calder\'on-Zygmund operator theory, namely,
    Theorem \ref{th1.1}, to estimate $R_1$ and $R_M$ on
    $L^p(\R^N,\omega).$ We obtained that the $L^p, 1<p<\infty,$ norm of $R_1+
    R_M$ is less than 1. It turns out that $T_M$ is
    invertibal and $(T_M)^{-1},$ the inverse of $T_M,$ is bounded on
    $L^2(\R^N,\omega)\cap L^p(\R^N,\omega).$ Applying the same strategy, to show
    Theorem \ref{1.7}, we need to estimate $R_1$ and $R_M$ on $L^2(\R^N,\omega)\cap
    H_d^p(\R^N,\omega)$ and show that the norms of $R_1$ and $R_M$ on
    $L^2(\R^N,\omega)\cap H_d^p(\R^N,\omega)$ are less than 1. To this
    end, we recall the Littlewood-Paley theory and Hardy space on space
    of homogeneous type $(\R^N, \| \cdot\|, \omega)$ in the sense of Coifman and Weiss. The discrete
    Calder\'on reproducing formula in Theorem \ref{dCRh} leads the
    following discrete square function on space of homogeneous type
    $(\Bbb R^N, \|\cdot\|, \omega)$:
    \begin{definition}\label{scw}
        Suppose that $f\in (\mathcal M(\beta,\gamma,r,x_0))^\prime.$ $S_{cw}(f),$ the Littlewood-Paley square function of $f$ for space of homogeneous type $(\Bbb R^N, \|\cdot\|, \omega),$ is defined by
        $$S_{cw}(f)(x)=\bigg\{\sum\limits_{k=-\infty}^\infty\sum\limits_{Q\in Q^k}
        |D_kf(x_{Q})|^2\chi_{Q}(x)\bigg\}^{1/2},$$
        where $D_k, Q^k$ are the same as given by Theorem \ref{dCRh}.
    \end{definition}
    See \cite{HMY} for more details.

    The strategy for estimating $R_1$ and $R_M$ on $L^2(\R^N,\omega)$ is to establish the following estimates:
    $$\|S(R_1(f)\|_p\leqslant C \|S_{cw}(R_1(f)\|_p)\leqslant C \|R_1\|_{dcz} \|S_{cw}(f)\|_p\leqslant C\|R_1\|_{dcz}\|S(f)\|_p $$
    and the similar estimates hold for $R_M.$

    In Theorem \ref{1.6}, $R_1$ has been proved to be the Dunkl-Calder\'on-Zygmund operator with $\|R_1\|_{dcz}\leqslant C (r-1)$ for some $1<r\leqslant r_0$ with $r_0$ is closed 1.
    The similar estimates hold for $R_M$ with $\|R_M\|_{dcz}\leqslant C r^{-M}$ for $1<r<r_0.$  Now we show the above  estimates for $R_1$ by the following steps.

    ${\textbf{Step 1}: \|S(R_1(f)\|_p\leqslant C \|S_{cw}(R_1(f))\|_p}$

    Indeed, we only need to show that for each $f\in L^2(\R^N,\omega),$
    $$\|S(f)\|_p\leqslant C \|S_{cw}(f)\|_p$$
    since $R_1$ is bounded on $L^2(\R^N,\omega).$

    To this end, by the discrete Calder\'on reproducing formula of $f\in L^2(\R^N,\omega)$ on space of homogeneous type given in Theorem \ref{dCRh}, we have
    $$f(x)=\sum\limits_{k\in \Bbb Z}\sum\limits_{Q\in Q^k}\omega(Q)D_k(x,x_{Q})\widetilde{\widetilde{D}}_k(f)(x_{Q})$$
    and hence,
    \begin{align*}
    S(f)(x)
    &=\Big(\sum\limits_{k'\in \Bbb Z}\sum\limits_{Q'\in Q^{k'}}|q_{Q'}\Big(\sum\limits_{k\in \Bbb Z}\sum\limits_{Q\in Q^k}\omega(Q)D_k(\cdot,x_{Q})(x_{Q'})\widetilde{\widetilde{D}}_k(f)(x_{Q})\Big)(x_{Q'})|^2\chi_{Q'}(x)\Big)^{\frac{1}{2}}\\
    &=\Big(\sum\limits_{k'\in \Bbb Z}\sum\limits_{Q'\in Q^{k'}}|\Big(\sum\limits_{k\in \Bbb Z}\sum\limits_{Q\in Q^k}\omega(Q)q_{Q'}D_k(x_{Q'},x_{Q})\widetilde{\widetilde{D}}_k(f)(x_{Q})\Big)|^2\chi_{Q'}(x)\Big)^{\frac{1}{2}}.
    \end{align*}
    To estimate $q_{Q'}D_k(x_{Q^{'}},x_{Q})$, we need the following
    \begin{lemma}\label{exchange}
        Let $k$ and $k'$ belong to $\mathbb Z.$ Suppose that $S_{k,k'}(x,y)$ satisfies the following condition:
        $$|S_{k,k'}(x,y)|\leqslant Cr^{-|k-k'|\varepsilon} \frac1{V(x,y,r^{-k\vee -k'}+d(x,y))}
        \bigg(\frac{r^{-k \vee -k'}}{r^{-k \vee -k'} +d(x,y)}\bigg)^{\varepsilon},$$
        for $0<\varepsilon\leqslant 1.$ Then for $\frac {\bf N}{{\bf N}+\varepsilon }<p\leqslant 1,$
        \begin{align*}
        &\Big\|\bigg(\sum\limits_{k'\in \mathbb Z}\sum\limits_{Q'\in Q^{k'}}
        \Big|\sum\limits_{k\in \mathbb Z}\sum\limits_{Q\in Q^k} \omega(Q)S_{k,k'}(x_{Q'},x_Q)\lambda_Q\Big|^2\chi_{Q'}\bigg)^{\frac{1}{2}}\Big\|_p \\
        &\qquad\lesssim \Big\| \bigg(\sum\limits_{k\in \mathbb Z}\sum\limits_{Q\in Q^k} |\lambda_Q|^2\chi_{Q}\bigg)^{\frac{1}{2}}\Big\|_p.
        \end{align*}
    \end{lemma}
    \begin{proof}
        Observing that $\omega(B(x,r^{-k\vee -k'}+d(x,x_{Q^k})))\sim \omega(B(y,r^{-k\vee -k'}+d(y,x_{Q^k})))$ for $x,y\in Q^{k'},$ hence,
        \begin{align*}
        &|S_{k,k'}(x_{Q'},x_Q)|\chi_{Q'}(x) \\
        &\lesssim Cr^{-|k-k'|\varepsilon} \frac1{V(x_Q,x_{Q'},r^{-k\vee -k'}+d(x_{Q^{'}},x_{Q}))}
        \bigg(\frac{r^{-k \vee -k'}}{r^{-k \vee -k'} +d(x_{Q^{'}},x_{Q})}\bigg)^{\varepsilon} \\
        &\lesssim Cr^{-|k-k'|\varepsilon} \frac1{\omega(B(x_{Q},r^{-k\vee -k'}+d(x,x_{Q}))}
        \bigg(\frac{r^{-k \vee -k'}}
        {r^{-k \vee -k'} +d(x,x_{Q})}\bigg)^{\varepsilon}.
        \end{align*}
        Applying $d(x,y)=\min\limits_{\sigma\in G}\|\sigma(x)-y\|$ gives
        \begin{align*}
        &|S_{k,k'}(x_{Q'},x_Q)|\chi_{Q'}(x) \\
        &\lesssim \sum\limits_{\sigma\in G} r^{-|k-k'|\varepsilon} \frac1{\omega(B(x_Q,r^{-k\vee -k'}+\|\sigma(x)-x_{Q}\|))}
        \bigg(\frac{r^{-k \vee -k'}}{r^{-k \vee -k'} +\|\sigma(x)-x_{Q}\|}\bigg)^{\varepsilon}\\
        &\lesssim \sum\limits_{\sigma\in G} r^{-|k-k'|\varepsilon} \frac1{\omega(B(\sigma(x),r^{-k\vee -k'}+\|\sigma(x)-x_{Q}\|))}
        \bigg(\frac{r^{-k \vee -k'}}{r^{-k \vee -k'} +\|\sigma(x)-x_{Q}\|}\bigg)^{\varepsilon}.
        \end{align*}
        Let $\theta$ satisfy that $\frac{\bf N}{{\bf N}+\varepsilon}<\theta<p\leqslant 1.$
        Then
        \begin{align*}
        &\sum\limits_{Q\in Q^k}w(Q)\frac1{\omega(B(\sigma(x),r^{-k\vee -k'}+\|\sigma(x)-x_{Q}\|))}
        \bigg(\frac {r^{-k\vee -k'}}{r^{-k\vee -k'}+\|\sigma(x)-x_{Q}\|}\bigg)^{\varepsilon}|a_Q(x_{Q})|\\
        &\leqslant \bigg\{\sum\limits_{Q\in Q^k}w(Q)^\theta \frac1{\omega(B(\sigma(x),r^{-k\vee -k'}+\|\sigma(x)-x_{Q}\|))
            )^\theta}
        \bigg(\frac {r^{-k\vee -k'}}{r^{-k\vee -k'}+\|\sigma(x)-x_{Q}\|}\bigg)^{\theta\varepsilon}|\lambda_Q|^\theta\bigg\}^{1\over\theta}.
        \end{align*}
        Denote by $c_Q$ the center point of $Q$. Let $A_0=\{Q\in Q^k : \|c_Q-\sigma(x)\|\leqslant r^{-k\vee -k'}\}$
        and $A_\ell=\{Q\in Q^k : r^{\ell-1+(-k\vee -k')}<\|c_Q-\sigma(x)\|\leqslant r^{\ell+(-k\vee -k')}\}$ for $\ell\in \mathbb N$.
        We use \eqref{rd1.1} to obtain that for $Q\in Q^k,$
        $$\omega(Q)\chi_{Q}(z)\sim \omega(B(z,
        r^{-k}))\chi_{Q}(z)\sim
        \omega(B(\sigma(z), r^{-k}))\chi_{Q}(z)\quad\mbox{for }\sigma\in G$$
        and
        $$\omega(B(x_{Q},r^{-k\vee -k'}))\lesssim
        r^{[k+(-k\vee -k')]{\bf N}}\omega(B(x_{Q},r^{-k})).$$
        Hence,
        \begin{align*}
        &\sum\limits_{Q\in Q^k} w(Q)^\theta \frac1{\omega(B(\sigma(x),r^{-k\vee -k'}+\|\sigma(x)-x_{Q}\|)))^\theta}
        \bigg(\frac {r^{-k\vee -k'}}{r^{-k\vee -k'}+\|\sigma(x)-x_{Q}\|}\bigg)^{\theta\varepsilon}|
        \lambda_Q|^\theta\\
        &=\sum\limits_{\ell=0}^\infty\sum\limits_{Q\in A_\ell}w(Q)^\theta \frac1{\omega(B(\sigma(x),r^{-k\vee -k'}+\|\sigma(x)-x_{Q}\|))
            )^\theta}
        \bigg(\frac {r^{-k\vee -k'}}{r^{-k\vee -k'}+\|\sigma(x)-x_{Q}\|}\bigg)^{\theta\varepsilon}|
        \lambda_Q|^\theta\\
        &\lesssim r^{[-k-(-k\vee -k')]{\bf N}(\theta-1)]}\sum\limits_{\ell=0}^\infty \Big(\frac {\omega(B(\sigma(x),r^{-k\vee -k'})))}{\omega(B(\sigma(x),r^{\ell+(-k\vee -k')})))}
        \Big)^{\theta-1} \frac 1{r^{\theta\varepsilon\ell}}\\
        &\qquad\qquad\times \frac1{\omega(B(\sigma(x),r^{\ell+(-k\vee -k')})))}\int_{\|\sigma(x)-z\|\leqslant 2r^{\ell+(-k\vee -k')}}
        \sum\limits_{Q\in A_\ell}|\lambda_Q|^\theta\chi_{Q}(z)d\omega(z)\\
        &\lesssim r^{[-k-(-k\vee -k')]{\bf N}(\theta-1)]}\sum\limits_{\ell=0}^\infty \frac1{r^{\ell[\theta\varepsilon+{\bf N}(\theta-1)]}} M\Big(\sum\limits_{Q\in Q^k}|\lambda_Q|^\theta\chi_{Q} \Big)(\sigma(x))\\
        &\lesssim  r^{[-k-(-k\vee -k')]{\bf N}(\theta-1)]}M\Big(\sum\limits_{Q\in Q^k}|\lambda_Q|^\theta\chi_{Q} \Big)(\sigma(x)),
        \end{align*}
        where $M$ denote the Hardy-Littlewood maximal operator on $(\Bbb R^N, \|\cdot\|,\omega)$.  Therefore,
        \begin{align*}
        &\sum\limits_{Q\in Q^k}w(Q)\frac1{\omega(B(\sigma(x),r^{-k\vee -k'}+\|\sigma(x)-x_{Q}\|))}
        \bigg(\frac {r^{-k\vee -k'}}{r^{-k\vee -k'}+\|\sigma(x)-x_{Q}\|}\bigg)^{\varepsilon}|a_Q(x_{Q})|\\
        &\lesssim r^{[-k-(-k\vee -k')]{\bf N}(1-1/\theta)]}
        \bigg\{M\Big(\sum\limits_{Q\in Q^k}|\lambda_Q|^\theta\chi_{Q}\Big)(\sigma(x))\bigg\}^{1/\theta}
        \end{align*}
        and
        \begin{align*}
        &\sum\limits_{Q\in Q^k}\omega(Q)S_{k,k'}(x_{Q'},x_{Q})\lambda_Q)\chi_{Q'}(x) \\
        &\lesssim \sum\limits_{\sigma\in G}r^{[-k-(-k'\vee -k)]{\bf N}(1-\frac 1\theta)}
        \bigg\{M\Big(\sum\limits_{Q\in Q^k}|\lambda_Q|^\theta\chi_{Q}\Big)(\sigma(x))\bigg\}^{1/\theta}\chi_{Q'}(x).
        \end{align*}
        It is clear that for $\frac{\bf N}{{\bf N}+\varepsilon}<\theta<p\leqslant 1,$
        $$\sup_{k'}\sum\limits_{k\in \Bbb Z} r^{-|k-k'|\varepsilon}r^{[-k-(-k'\vee -k)]{\bf N}(1-\frac 1\theta)}<\infty.$$
        By H\"older's inequality, we have
        \begin{align*}
        &\Big|\sum\limits_{k\in\mathbb Z}\sum\limits_{Q\in Q^k}\omega(Q)
        S_{k,k'}(x_{Q'},x_{Q})\lambda_Q\Big|^2\chi_{Q'}(x) \\
        &\lesssim \sum\limits_{\sigma\in G}\sum\limits_{k\in \mathbb Z} r^{-|k-k'|\varepsilon}r^{[-k-(-k'\vee -k)]{\bf N}(1-\frac 1\theta)}
        \bigg\{M\Big(\sum\limits_{Q^k}|\lambda_{Q^k}|^\theta\chi_{Q^k}\Big)(\sigma(x))\bigg\}^{2/\theta}\chi_{Q'}(x).
        \end{align*}
        This implies
        \begin{align*}
        &\sum\limits_{k'\in \mathbb Z}\sum\limits_{Q'\in Q^{k'}}
        \Big|\sum\limits_{k\in \mathbb Z}\sum\limits_{Q\in Q^k} \omega(Q)S_{k,k'}(x_{Q'},x_Q)\lambda_Q\Big|^2\chi_{Q'} \\
        &\qquad\lesssim \bigg\{\sum\limits_{\sigma\in G}\sum\limits_{k\in \mathbb Z}
        \bigg\{M\Big(\sum\limits_{Q\in Q^k}|\lambda_Q|^\theta\chi_{Q}\Big)(\sigma(x))\bigg\}^{2/\theta}\bigg\}^{1/2},
        \end{align*}
        where the estimate is used: for $\frac{\bf N}{{\bf N}+\varepsilon}<\theta<p\leqslant 1,$
        $$\sup_k\sum\limits_{k'\in \Bbb Z} r^{-|k-k'|\varepsilon}r^{[-k-(-k'\vee -k)]{\bf N}(1-\frac 1\theta)}<\infty.$$
        The Fefferman-Stein vector valued maximal function inequality with $\theta<p\leqslant 1$ yields
        \begin{align*}
        &\Big\|\sum\limits_{k'\in \mathbb Z}\sum\limits_{Q'\in Q^{k'}}
        \Big|\sum\limits_{k\in \mathbb Z}\sum\limits_{Q\in Q^k} \omega(Q)S_{k,k'}(x_{Q'},x_Q)\lambda_Q\Big|^2\chi_{Q'}\Big\|_p \\
        &\qquad\lesssim \sum\limits_{\sigma\in G}
        \Big\|\Big(\sum\limits_{k\in \mathbb Z}\sum\limits_{Q\in Q^k}|\lambda_Q|^2\chi_Q(\sigma(x))\Big)^{\frac{1}{2}}\Big\|_p.
        \end{align*}
        Since $G$ is finite group and
        $$\int_{\Bbb R^N}  f(\sigma(x))d\omega(x)=\int_{\Bbb R^N} f(x)d\omega(x),$$
        we have
        \begin{align*}
        &\Big\|\sum\limits_{k'\in \mathbb Z}\sum\limits_{Q'\in Q^{k'}}
        \Big|\sum\limits_{k\in \mathbb Z}\sum\limits_{Q\in Q^k} \omega(Q)S_{k,k'}(x_{Q'},x_Q)\lambda_Q\Big|^2\chi_{Q'}\Big\|_p \\
        &\qquad \lesssim
        \Big\|\Big(\sum\limits_{k\in \mathbb Z}\sum\limits_{Q\in Q^k}|\lambda_Q|^2\chi_Q(x)\Big)^{\frac{1}{2}}\Big\|_p.
        \end{align*}
        The proof of the Lemma \ref{exchange} is completed.
    \end{proof}
    \begin{rem}\label{ex2}
        Using the similar proof of Lemma \ref{exchange},  we also obtain
        \begin{align*}
        &\Big\|\bigg(\sum\limits_{k'\in \mathbb Z}\sum\limits_{Q'\in Q^{k'}}
        \Big|\sum\limits_{|k|>j}\sum\limits_{Q\in Q^k} \omega(Q)S_{k,k'}(x_{Q'},x_Q)\lambda_Q\Big|^2\chi_{Q'}\bigg)^{\frac{1}{2}}\Big\|_p \\
        &\qquad\lesssim \Big\| \bigg(\sum\limits_{|k|>j}\sum\limits_{Q\in Q^k} |\lambda_Q|^2\chi_{Q}\bigg)^{\frac{1}{2}}\Big\|_p
        \qquad\mbox{for $j\in \mathbb N$.}
        \end{align*}
    \end{rem}

    We now return to the proof of {\bf Step 1}. Applying the almsot orthogonal estimate given in the Lemma \ref{aoe} for any $\varepsilon\in (0,1)$ yields
    \begin{align*}
    &|q_{Q'}D_k(x_{Q^{'}},x_{Q})|\chi_{Q'}(x) \\
    &\lesssim r^{-|k-k'|\varepsilon} \frac1{V(x_{Q^{'}},x_{Q},r^{-k\vee -k'}+\|x_{Q^{'}}-x_{Q}\|)}
    \bigg(\frac{r^{-k \vee -k'}}{r^{-k \vee -k'} +\|x_{Q^{'}}-x_{Q}\|}\bigg)^{\varepsilon}\chi_{Q'}(x).
    \end{align*}
    The Lemma \ref{exchange} gives
    \begin{align*}
    \|S(f)(x)\|_p
    &=\|\Big(\sum\limits_{k'\in \Bbb Z}\sum\limits_{Q'\in Q^{k'}}|\Big(\sum\limits_{k\in \Bbb Z}\sum\limits_{Q\in Q^k}\omega(Q)q_{Q'}D_k(x_{Q'},x_{Q})\widetilde{\widetilde{D}}_k(f)(x_{Q})\Big)|^2\chi_{Q'}(x)\Big)^{\frac{1}{2}}\|_p\\
    &\lesssim \bigg\|\bigg\{\sum\limits_{k\in \Bbb Z}\sum\limits_{Q\in Q^k}|\widetilde{\widetilde{D}}_k f(x_{Q})|^2\chi_{Q}\bigg\}^{\frac12}\bigg\|_p\\
    &\lesssim \|S_{cw}(f)\|_p.
    \end{align*}

    ${\textbf{Step 2}: \|S_{cw}(R_1(f)\|_p\leqslant C \|R_1\|_{dcz} \|S_{cw}(f)\|_p}$

    To show the above inequality, we write
    $$\|S_{cw}(R_1f)(x)\|_p=\bigg\|\bigg\{\sum\limits_{k=-\infty}^\infty\sum\limits_{Q\in Q^k}
    |D_kR_1(f)(x_{Q})|^2\chi_{Q}(x)\bigg\}^{1/2}\bigg\|_p.$$
    The $L^2$ boundedness of $R_1$ together with the discrete Calder\'on reproducing formula of $f\in L^2(\R^N,\omega)$ on space of homogeneous type given in Theorem \ref{dCRh} yields
    $$\|S_{cw}(R_1(f)\|_p=
    \bigg\|\bigg\{\sum\limits_{k\in \Bbb Z}\sum\limits_{Q\in Q^k}
    |\sum\limits_{k'\in \Bbb Z}\sum\limits_{Q'\in Q^{k'}}\omega(Q')D_kR_1D_{k'}(x_Q,x_{Q'})\widetilde{\widetilde{D}}_{k'}(f)(x_{Q'})|^2\chi_{Q'}(x)\bigg\}^{1/2}\bigg\|_p.$$
    Applying the almost orthogonal estimate given in Lemma \ref{Lem aoe} to $D_kR_1D_{k'}(x_Q,x_{Q'}),$ we obtain
    that for any $\varepsilon\in (0,1),$
    \begin{align*}
    &|D_kR_1D_{k'}(x_Q,x_{Q'})|\chi_{Q'}(x) \\
    &\lesssim \|R_1\|_{dcz} r^{-|k-k'|\varepsilon} \frac1{V(x_{Q^{'}},x_{Q},r^{-k\vee -k'}+d(x_{Q^{'}},x_{Q}))}
    \bigg(\frac{r^{-k \vee -k'}}{r^{-k \vee -k'} +d(x_{Q^{'}},x_{Q})}\bigg)^{\varepsilon}\chi_{Q'}(x)\\
    &\lesssim \|R_1\|_{dcz}\sum\limits_{\sigma\in G}r^{-|k-k'|\varepsilon} \frac1{V(x_{Q^{'}},x_{Q},r^{-k\vee -k'}+\|\sigma(x_{Q^{'}})-x_{Q})\|}
    \bigg(\frac{r^{-k \vee -k'}}{r^{-k \vee -k'} +\|\sigma(x_{Q^{'}})-x_{Q}\|}\bigg)^{\varepsilon}\chi_{Q'}(x)\\
    &\lesssim \|R_1\|_{dcz}\sum\limits_{\sigma\in G}r^{-|k-k'|\varepsilon} \frac1{\omega(B(x_{Q},r^{-k\vee -k'}+\|\sigma(x)-x_{Q})\|)}
    \bigg(\frac{r^{-k \vee -k'}}{r^{-k \vee -k'} +\|\sigma(x)-x_{Q}\|}\bigg)^{\varepsilon}\chi_{Q'}(x)\\
    &\lesssim \|R_1\|_{dcz}\sum\limits_{\sigma\in G}r^{-|k-k'|\varepsilon} \frac1{\omega(B(\sigma(x),r^{-k\vee -k'}+\|\sigma(x)-x_{Q})\|)}
    \bigg(\frac{r^{-k \vee -k'}}{r^{-k \vee -k'} +\|\sigma(x)-x_{Q}\|}\bigg)^{\varepsilon}\chi_{Q'}(x),
    \end{align*}
    where we use the fact that if $x\in Q'$ then $r^{-k\vee -k'}+\|\sigma(x_{Q^{'}})-x_{Q})\|\sim r^{-k\vee -k'}+\|\sigma(x)-x_{Q})\|.$

    Appying the Lemma \ref{ex2} implies
    \begin{align*}
    &\sum\limits_{Q'\in Q^k}w(Q')\frac1{\omega(B(\sigma(x), r^{-k\vee -k'}+\|\sigma(x)-x_{Q}\|)}\\
    &\qquad \times \bigg(\frac {r^{-k\vee -k'}}{r^{-k\vee -k'}+\|\sigma(x)-x_{Q}\|}\bigg)^{\varepsilon}|\widetilde{\widetilde{D}}_k(f)(x_{Q})|\\
    &\leqslant Cr^{[-k'-(-k'\vee -k)]{\bf N}(1-\frac 1\theta)}
    \bigg\{M\Big(\sum\limits_{Q\in Q^k}|\widetilde{\widetilde{D}}_k(f)(x_{Q})|^\theta\chi_{Q}\Big)(\sigma(x))\bigg\}^{1/\theta},
    \end{align*}
    where $\theta$ satisfies
    $\frac{\bf N}{{\bf N}+\varepsilon}<\theta<p\leqslant 1.$
    By H\"older's inequality, we have
    \begin{align*}
    &\Big|\sum\limits_{k\in \Bbb Z}\sum\limits_{Q\in Q^k}\omega(Q)D_kR_1D_{k'}(x_Q,x_{Q'})\widetilde{\widetilde{D}}_k(f)(x_{Q})\Big|^2\chi_{Q^{'}}(x)\\
    &\lesssim \|R_1\|_{dcz}\sum\limits_{\sigma\in G}\sum\limits_{k\in \mathbb Z} r^{-|k-k'|\varepsilon}r^{[-k'-(-k'\vee -k)]{\bf N}(1-\frac 1\theta)}
    \bigg\{M\Big(\sum\limits_{Q\in Q^k}|\widetilde{\widetilde{D}}_k(f)(x_{Q})|^\theta\chi_{Q}\Big)(\sigma(x))\bigg\}^{2/\theta}\chi_{Q^{'}}(x),
    \end{align*}
    which implies
    $$S_{cw}(R_1(f)(x)\lesssim \|R_1\|_{dcz}\bigg\{\sum\limits_{\sigma\in G}\sum\limits_{k\in \Bbb Z}
    \bigg\{M\Big(\sum\limits_{Q\in Q^k}|\widetilde{\widetilde{D}}_k(f)(x_{Q})|^\theta\chi_{Q}\Big)(\sigma(x))\bigg\}^{2/\theta}\bigg\}^{1/2}.$$
    The Fefferman-Stein vector valued maximal function inequality with $\theta<p\leqslant 1$ yields
    \begin{align*}
    \|S_{cw}R_1(f)\|_p
    &\lesssim \|R_1\|_{dcz}\bigg\|\bigg\{\sum\limits_{k\in \Bbb Z}\sum\limits_{Q\in Q^k}|\widetilde{\widetilde{D}}_k f(x_{Q})|^2\chi_{Q}(\sigma(x))\bigg\}^{\frac12}\bigg\|_{L^p(\omega)}\\
    &\lesssim \|R_1\|_{dcz}\bigg\|\bigg\{\sum\limits_{k\in \Bbb Z}\sum\limits_{Q\in Q^k}|\widetilde{\widetilde{D}}_k f(x_{Q})|^2\chi_{Q}(x)\bigg\}^{\frac12}\bigg\|_{L^p(\omega)}\\
    &\lesssim \|R_1\|_{dcz}\|S_{cw}(f)\|_p.
    \end{align*}
    Applying the similar proof, we still have $\|S_{cw}R_M(f)\|_p\lesssim \|R_M\|_{dcz}\|S_{cw}(f)\|_p.$

    We remark that the Dunkl-Calder\'on-Zygmund operator theory plays a crucial role for the proof of this estimate. More precisely, we obtain the sharp range for $p: \frac{\bf N}{{\bf N}+\varepsilon}<\theta<p\leqslant 1,$ which is same as in the classical case.

    ${\textbf{Step 3}}: \|S_{cw}(f)\|_p\leqslant C\|S(f)\|_p$
    To show this estimate, the key point is to write
    \begin{align*}
    f(x)&=-\ln r\sum\limits_{j=-\infty}^\infty\sum\limits_{Q}w(Q)
    \psi_{j}(x,x_{Q})q_{j}f(x_{Q})
    +R_1(f)(x)+R_M(x)\\
    &=T_M(f)(x)+R_1(f)(x)+R_M(x).
    \end{align*}
    By the estimates in ${\textbf{Step 2}}$ for $\frac{\bf N}{{\bf N}+1}<p\leqslant 1,$ we have
    $$\|S_{cw}R_1(f)\|_p\leqslant C \|R_1\|_{dcz} \|S_{cw}(f)\|_p\leqslant \frac{1}{4} \|S_{cw}(f)\|_p $$
    and
    $$\|S_{cw}R_M(f)\|_p\leqslant C \|R_M\|_{dcz} \|S_{cw}(f)\|_p\leqslant \frac{1}{4}\|S_{cw}(f)\|_p.$$

    These estimates imply
    $$\|S_{cw}(f)\|^p_p\leqslant \|S_{cw}\big(T_M(f)(x)+R_1(f)(x)+R_M(x)\big)\|^p_p
    \leqslant \|S_{cw}(T_Mf)\|^p_p +\frac{1}{2}\|S_{cw}(f)\|^p_p$$
    and, hence,
    $$\|S_{cw}(f)\|_p\leqslant C_p \|S_{cw}(T_Mf)\|_p.$$
    We claim $\|S_{cw}(T_Mf)\|_p\leqslant C \|S(f)\|_p.$ Indeed, observing that
    $$T_M(f)(x)= -\ln r\sum\limits_{j=-\infty}^\infty\sum\limits_{Q\in Q^j}w(Q) \psi_{Q}(x,x_{Q})q_{Q}f(x_{Q}),$$
    and
    \begin{align*}
    &\bigg|D_kT_M(f)(x)\bigg|=\bigg|D_k\bigg(-\ln
    r\sum\limits_{j=-\infty}^\infty\sum\limits_{Q\in Q^j}w(Q)       \psi_{Q}(\cdot,x_{Q})q_{Q}f(x_{Q})\bigg)(x)\bigg|\\
    &\lesssim\sum\limits_{j=-\infty}^\infty\sum\limits_{Q\in Q^j}w(Q)|D_k \psi_{Q}(x,x_{Q})| |q_{Q}f(x_{Q})|.
    \end{align*}
    Following the same proof as in {\bf Step 1} and applying the estimate \eqref{aoe2} in the Lemma \ref{aoe}, there exists a constant $C>0$ such that for $0<\varepsilon<1$ and $x_{Q}\in Q, Q\in Q^j, Q'\in Q^k,$
    \begin{align*}
    &|D_k \psi_{Q}(x,x_{Q})|\chi_{Q'}(x) \\
    &\lesssim r^{-|j-k|\varepsilon} \frac1{V(x,x_{Q},r^{-j\vee -k}+d(x,x_{Q}))}
    \bigg(\frac{r^{-j \vee -k}}{r^{-j \vee -k} +\|x-x_{Q}\|}\bigg)^{\varepsilon}\chi_{Q'}(x).
    \end{align*}
    We obtain
    \begin{align*}
    &|D_kT_{M}f(x)|^2\chi_{Q'}(x) \\
    &\lesssim \sum\limits_{\sigma\in G}\sum\limits_{j=-\infty}^\infty
    r^{-|j-k|\varepsilon}r^{[-k-(-k\vee -j)]{\bf N}(1-\frac 1\theta)}
    \bigg\{M\Big(\sum\limits_{Q_j}|q_{j}
    f(x_{Q_j})|^\theta\chi_{Q}\Big)(x)\bigg\}^{2/\theta}\chi_{Q'}(x),
    \end{align*}
    which the fact that $\frac{\bf N}{{\bf N}+\varepsilon}<\theta<p\leqslant 1$ together with the Fefferman-Stein vector-valued maximal inequality implies
    \begin{align*}
    \|S_{cw}(T_Mf)\|_{L^p(\omega)}&=\bigg\|\bigg\{\sum\limits_{k\in \Bbb Z}\sum\limits_{Q'\in Q^k}|D_kT_{M}f(x_{Q'})|^2\chi_{Q'}(x)\bigg\}^{\frac12}\bigg\|_{L^p}\\
    &\lesssim \bigg\|\bigg\{\sum\limits_{j\in \Bbb Z}\sum\limits_{Q\in Q^j}|q_Q f(x_{Q})|^2\chi_{Q}(x)\bigg\}^{\frac12}\bigg\|_{L^p}\\
    &\lesssim \|S(f)\|_p.
    \end{align*}
    The proof of {\bf Step 3 } is complete.

    Observing that $f(x)=T_M(f)(x)+R_1(f)(x) +R_M(f)(x)$ and applying the above estimates, namely, $\|S(R_1+R_M)(f)\|_p\leqslant C\big(\|R_1\|_{dcz}+\|R_M\|_{dcz}\big)\|S(f)\|_p,$ imply that
    $$\|S\big(I-T_M\big)(f)\|_p=\|S\big(R_1(f)(x) +R_M(f)\big)\|_p\leqslant \frac{1}{2}\|S(f)\|_p.$$
    Therefore, $\|S\big((T_M)^{-1}f\big)\|_p\leqslant C\|Sf\|_p.$ Set $h=(T_M)^{-1}f,$ we obtain
    $$f(x)=T_Mh(x)=-\ln r\sum\limits_{j=-\infty}^\infty\sum\limits_{Q\in Q^j}w(Q)\psi_{Q}(x,x_{Q})q_{Q}
    h(x_{Q}),$$
    where $\|f\|_2\sim \|h\|_2$ and $\|f\|_{H_d^p}\sim \|h\|_{H_d^p},$ for $\frac{\bf N}{{\bf N}+1}<p\leqslant 1.$

    It remains to show that the above series converges in $L^2(\R^N,\omega)\cap H_d^p(\R^N,\omega).$ To this end, we only need to prove
    $$\bigg\|S\big(\sum\limits_{|j|>n}\sum\limits_{Q\in Q^j}w(Q)
    \psi_{Q}(x,x_{Q})q_{Q}h(x_{Q})\big)\bigg\|_p\rightarrow 0$$
    as $n\rightarrow \infty.$ Indeed, repeating the same proof as {\bf Step 1},
    \begin{equation}\label{rem}
    \bigg\|S\big(\sum\limits_{|j|>n}\sum\limits_{Q\in Q^j}w(Q)
    \psi_{Q}(x,x_{Q})q_{Q}h(x_{Q})\big)\bigg\|_p
    \lesssim \bigg\|\bigg\{\sum\limits_{|j|>n}\sum\limits_{Q\in Q^j}|q_Qh(x_Q)|^2\chi_{Q}(x)\bigg\}^{\frac12}\bigg\|_p,
    \end{equation}
    where by the fact $\|S(h)\|_p\leqslant C\|f\|_p,$ the last term tends to 0 as $n\rightarrow \infty.$

    The proof of Theorem \ref{1.7} is complete.
\end{proof}

\subsection{{Dunkl-Hardy space $H_d^p, \frac{\bf N}{\bf N +1}<p\leqslant 1$}}
\ \\

As mentioned, the departure for introducing the Hardy space $H_d^p(\R^N,\omega),
\frac{N}{N+1}<p\leqslant 1,$ in the Dunkl setting is the Proposition \ref{pr1}.
\begin{proof}[{\bf The proof of Proposition \ref{pr1}}]
    Applying the weak-type discrete Calder\'on-type reproducing formula given in Theorem \ref{1.7} for $f\in L^2(\R^N,\omega)\cap H_d^p(\R^N,\omega)$, we write
    $$f=-\ln r\sum\limits_{j=-\infty}^\infty\sum\limits_{Q\in Q^j}w(Q)
    \psi_{Q}(x,x_{Q})q_{Q}h(x_{Q}),$$
    where $\|h\|_2\sim \|f\|_2$ and $\|S(h)\|_p\sim \|S(f)\|_p.$

    Set $\Omega_{\ell}=\big\{x\in \mathbb{R}^{N}: S(h)(x)>2^{\ell}\big\}$ and
    $$B_{\ell}=\Big\{ Q: Q \ \ \text{are r-dyadic cubes}, \ \  \omega\big(Q\cap \Omega_{\ell}\big)>\frac{1}{2}\omega(Q) \ \ \text{and} \ \ \omega\big(Q\cap \Omega_{\ell+1}\big)\leqslant\frac{1}{2}\omega(Q)\Big\}.$$

    Denote $B^*_\ell:=\{ Q^*_\ell\}$ by the maximal $r-$dyadic cubes in $B_\ell$ for $l\in \mathbb{Z}$.
    We claim that
    \begin{eqnarray*}f=
        -\ln r \sum\limits_{\ell\in \mathbb{Z}}
        \sum\limits_{Q^*_\ell\in B^*_\ell}
        \sum\limits_{Q\in B^*_\ell \atop Q\in B_\ell}\omega(Q)\psi_{Q}(x,x_{Q})q_{Q}h(x_{Q}).
    \end{eqnarray*}

    In order to prove the above claim, we only need to show that if the dyadic cube $Q \not\in B_\ell$ for all $\ell \in \mathbb Z$, then
    \begin{eqnarray*}
        \omega(Q)\psi_{Q}(x,x_{Q})q_{Q}h(x_{Q})=0.
    \end{eqnarray*}
    Observe that by the stopping time argument, each dyadic cube $Q$ can be in one and only one in $B_\ell$, that is, if $Q$ belongs to both $B_\ell$ and $B_{\ell'},$ then $\ell=\ell'.$ We now can assume that $\omega(Q)\not=0.$ Otherwise, the above equality
    holds obviously. Note that   $\omega(\Omega_{\ell}) < 2^{-2\ell}
    \|S_L(f)\|^2_{L^2(X)}\to 0$ as $\ell \to +\infty.$ As a consequence, if $Q \not\in
    B_\ell$ for all $\ell \in \mathbb Z$, then $\omega(Q\cap
    \Omega_{\ell})\leqslant {1\over 2}\omega(Q)$, for all $\ell \in
    \mathbb Z$ since, otherwise, there exists an $\ell_0 \in \mathbb Z$, such that $\omega(Q\cap
    \Omega_{\ell_0})> {1\over 2}\omega(Q)$. However $\omega(Q\cap\Omega_{\ell}
    ) \to 0$ as $\ell \to +\infty$ and $\{\omega(Q\cap\Omega_\ell)\}_\ell$
    is a decreasing sequence. This implies that there must be one critical
    index $\ell_1$ such that $\omega(Q\cap \Omega_{\ell_1})>{1\over
        2}\omega(Q)$ and $\omega(Q\cap \Omega_{\ell_1+1})\leqslant {1\over 2}\omega(Q),$
    that is, $ Q  \in B_{\ell_1}.$ This is contradict to the fact that $Q $ is not in
    $B_\ell$ for all $\ell \in \mathbb Z.$

    The fact $\omega(Q\cap
    \Omega_{\ell})\leqslant {1\over 2}\omega(Q)$ for all $\ell \in \mathbb Z$ implies that
    $\mu(Q\cap \Omega_{\ell}^c)\geqslant {1\over 2}\mu(Q)$ for all $\ell \in
    \mathbb Z$. Now set $K=\{x\in X: S(h)(x) =0 \}$. Note that $\cap_{\ell
        \in \mathbb Z}\Omega_{\ell}^c=\cap_{\ell \in \mathbb Z}\{x\in X:\ S(h)(x)\leqslant
    2^\ell\}=K$. As a consequence,
    \begin{eqnarray*}
        \omega(Q\cap
        K)=\lim\limits_{\ell \to -\infty}\omega(Q\cap\Omega^c_{\ell})\geqslant
        {1\over 2}\omega(Q)>0.
    \end{eqnarray*}

    Since  for all $x\in K$,
    $0=S(h)(x)=\Big(\sum\limits_{k\in \Bbb Z}\sum\limits_{Q\in Q^k}|q_Qh(x_{Q})|^2\chi_{Q}(x)\Big)^{\frac{1}{2}},$ we have
    $q_Qh(x_{Q})=0$ and hence, the claim follows.

    We get
    \begin{eqnarray*}\langle f,g\rangle&=&
        -\ln r \sum\limits_{\ell\in \mathbb{Z}}\sum\limits_{Q^*_\ell\in B^*_\ell}
        \sum\limits_{Q\in B^*_\ell \atop Q\in B_\ell}w(Q)q_Qh(x_Q)\psi_Qg(x_Q).
    \end{eqnarray*}

    Applying first the Cauchy-Schwartz inequality and then the $p, 0<p\leqslant 1,$ inequality in the last summation implies that
    \begin{align*}|\langle f,g\rangle |&\lesssim \sum\limits_{\ell\in \mathbb{Z}}
    \sum\limits_{Q^*_\ell\in B^*_\ell}\Big(\sum\limits_{Q\in B^*_\ell \atop Q\in B_\ell}w(Q)|q_Qh(x_Q)|^2\Big)^{\frac{1}{2}}\Big(\sum\limits_{Q\in B^*_\ell \atop Q\in B_\ell}w(Q)|\psi_Qg(x_Q)|^2\Big)^{\frac{1}{2}}
    \\&\lesssim
    \Bigg\{\sum\limits_{\ell\in \mathbb{Z}}
    \sum\limits_{Q^*_\ell\in B^*_\ell}\Big(\sum\limits_{Q\in B^*_\ell \atop Q\in B_\ell}w(Q)|q_Qh(x_Q)|^2\Big)^{\frac{p}{2}}\Big(\sum\limits_{Q\in B^*_\ell \atop Q\in B_\ell}w(Q)
    |\psi_Qg(x_Q)|^2\Big)^{\frac{p}{2}}\Bigg\}^{\frac{1}{p}}.
    \end{align*}

    Observing that the last summation above is dominated by
    $$\Big(\sum\limits_{Q\in B^*_\ell \atop Q\in B_\ell}w(Q)|\psi_Qg(x_Q)|^2\Big)^{\frac{p}{2}}\leqslant
    \omega\big(Q_l^*\big)^{1-\frac{p}{2}}\|g \|_{CMO_d^p}^p,$$
    which implies that
    \begin{eqnarray*}|\langle f,g\rangle |\lesssim \|g \|_{CMO_d^p}\Bigg\{\sum\limits_{\ell\in \mathbb{Z}}
        \sum\limits_{Q^*_\ell\in B^*_\ell}\omega\big(Q_l^*\big)^{1-\frac{p}{2}}\Big(\sum\limits_{Q\in B^*_\ell \atop Q\in B_\ell}w(Q)|q_Qh(x_Q)|^2\Big)^{\frac{p}{2}}
        \Bigg\}^{\frac{1}{p}}.\end{eqnarray*}

    By the H\"older inequality, we have
    $$|\langle f,g\rangle|\lesssim \|g \|_{CMO_d^p}\Bigg\{\sum\limits_{\ell\in \mathbb{Z}}
    \Big[\sum\limits_{Q^*_\ell\in B^*_\ell}\omega\big(Q_l^*\big)\Big]^{1-\frac{p}{2}}\Big(\sum\limits_{Q\in B_\ell }w(Q)|q_Qh(x_Q)|^2\Big)^{\frac{p}{2}}
    \Bigg\}^{\frac{1}{p}}.$$

    To  estimate the last term above, let ${{\widetilde \Omega}_\ell }=\{x\in \mathbb{R}^N: M(\chi_{\Omega_\ell})(x)>\frac{1}{2}\},$ where $M$ is the Hardy-Littlewood maximal function on $\mathbb{R}^N$ with the measure $d\omega$ and $\chi_{{\Omega}_\ell}(x)$ is the indicate
    function of $\Omega_\ell$. It is not difficult to see that if $Q \in B_\ell$ then $Q\subseteq {\widetilde \Omega}_\ell.$ Since all $Q_{l}^*$ are disjoint, thus,
    $$\Big[\sum\limits_{Q^*_\ell\in B^*_\ell}\omega\big(Q_\ell^*\big)\Big]^{1-\frac{p}{2}}\leqslant \omega({\widetilde \Omega}_\ell)^{1-\frac{p}{2}}\leqslant C\omega({ \Omega}_\ell)^{1-\frac{p}{2}},$$
    where the first inequality follows from the facts that
    $\bigcup\limits_{Q^*_\ell\in
        B^*_\ell}Q^*_\ell \subseteq \widetilde{\Omega}_\ell$ and
    $\sum\limits_{Q^*_\ell\in B^*_\ell}\omega(Q^*_\ell) \leq
    \omega(\widetilde{\Omega}_\ell)$ and, by the $L^2$-boundedness
    of the Hardy--Littlewood maximal function, the last inequality follows from  the estimate $\omega(\widetilde{\Omega}_\ell)\leqslant
    C\omega({\Omega}_\ell).$

    We claim that
    \begin{eqnarray}\label{6}\sum\limits_{Q\in B_l}w(Q)|q_Qh(x_Q)|^2\leqslant C  2^{2\ell}\omega(\Omega_\ell).\end{eqnarray}

    Assuming this claim \eqref {6} for the moment, we get
    \begin{align*}\label{5}|\langle f,g\rangle|&\leqslant C \|g \|_{CMO_d^p}\Big(\sum\limits_{\ell=-\infty}^\infty\omega(\Omega_\ell)^{1-\frac{p}{2}} 2^{p\ell}\omega(\Omega_\ell)^{\frac{p}{2}}\Big)^{\frac{1}{p}}\nonumber
    \\&\leqslant C \|g \|_{CMO_d^p}\Big(\sum\limits_{\ell=-\infty}^\infty 2^{p\ell}\omega(\Omega_\ell)\Big)^{\frac{1}{p}}\\&\leqslant C\|S(h)\|_{p} \|g \|_{CMO_d^p}\\
    &\leqslant C\|S(f)\|_{p} \|g \|_{CMO_d^p}.
    \end{align*}

    It remains to show the claim \eqref{6}. To this end, we begin with the following estimate
    $$\int_{{\widetilde \Omega}_\ell\backslash\Omega_{\ell+1}} S(h)^2(x) d\omega(x)\leqslant C2^{2\ell}\omega\big({\widetilde \Omega}_\ell \big) \leqslant C 2^{2\ell}\omega\big(\Omega_\ell \big).$$
    Note that
    \begin{eqnarray*}\int_{{\widetilde \Omega}_\ell\backslash\Omega_{\ell+1}} S(h)^2(x) d\omega(x)\geq
        \sum\limits_{Q\in B_\ell}|q_Qh(x_Q)|^2w\big(\big({\widetilde \Omega}_\ell/\Omega_{\ell+1}\big)\cap Q\big).
    \end{eqnarray*}

    Since for each $Q \in B_\ell$, $Q\subseteq \widetilde{\Omega}_{\ell}$ and $ \Omega_{\ell+1} \subset
    \Omega_{\ell},$ hence,
    $$w\big(\big({\widetilde \Omega}_\ell/\Omega_{\ell+1}\big)\cap Q\big)= w(Q)-w(\Omega_{\ell+1}\cap Q) \geqslant \frac{1}{2}w(Q).$$

    Therefore,
    \begin{eqnarray*}\int_{{\widetilde \Omega}_\ell/\Omega_{\ell+1}} |S(h)(x)|^2 d\omega(x)\geqslant  C\sum\limits_{Q\in B_l}w(Q)|q_Qh(x_Q)|^2,\end{eqnarray*}
    which implies the claim \eqref{6}. The proof of Proposition  \ref{pr1} is concluded.
\end{proof}

The Proposition \ref{pr1} indicates that if
$\{f_n\}_{n=1}^\infty$ is a sequence in $L^2(\R^N,\omega)$ with
$\|S(f_n-f_m)\|_p\rightarrow 0$ as $n, m\rightarrow \infty,$ then
for each $g\in L^2(\R^N,\omega)$ with $\|g\|_{CMO_d^p}<\infty,
\lim\limits_{n,m\rightarrow \infty}\langle f_n-f_m,g\rangle =0.$
Therefore, there exists $f,$ as a distribution on $L^2(\R^N,\omega)\cap CMO_d^p$,
such that for each $g\in L^2(\R^N,\omega)$ with $\|g\|_{CMO_d^p}<\infty,$
$$\langle f,g\rangle=\lim\limits_{n \rightarrow \infty}\langle f_n, g\rangle.$$

Before introducing the Dunkl-Hardy space, we need the following
\begin{lemma}\label{test}
    Let $\frac{\bf N}{{\bf N}+1}<p\le1.$ Then $q_{j}(\cdot,y)$ is in $L^2(\R^N,\omega)\cap CMO_d^p$ for any fixed $j$ and $y\in \mathbb R^N.$ Moreover,
    $$\sup_P \frac1{\omega(P)^{\frac2p-1}}\sum\limits_{Q\subseteq P} \omega(Q) |\psi_Q(q_{j}(\cdot,y))(x_Q)|^2\leqslant C,$$
    where both $P$ and $Q$ are $r-$dyadic cubes on $\mathbb R^N$ and the constant $C$ which depends on $j$ but is independent of $y$.
\end{lemma}

\begin{proof}
    Fix a dyadic cube $P$ with the side length $r^{-M-k_0}$ and the center $x_P$.
    For $Q\in Q^k$, the almost orthogonal estimate \eqref{aoe2} in the Lemma \eqref{Lem aoe} implies
    \begin{equation}\label{cmo-aoe}
    \begin{aligned}
    &|\psi_Q(q_{j}(\cdot,y))(x)|=\bigg|\int_{\R^N}\psi_k(x,z)q_j(z,y)d\omega(z)\bigg|\\
    &\qquad\leqslant Cr^{-|k-j|\varepsilon}
    \frac 1{V(x,y,r^{-k\vee -j}+d(x,y))}\bigg(\frac{r^{-k\vee -j}}{r^{-k\vee -j}+\|x-y\|}\bigg)^\varepsilon\\
    &\qquad \leqslant Cr^{-|k-j|\varepsilon}\frac 1{\omega(B(x,r^{-k\vee -j}))}\\
    &\qquad \leqslant Cr^{-|k-j|\varepsilon}\frac 1{\omega(B(x,r^{-j}))}.
    \end{aligned}
    \end{equation}
    For $k_0\geqslant j$, we use \eqref{cmo-aoe} to get
    \begin{align*}
    &\frac1{\omega(P)^{\frac2p-1}}\sum\limits_{Q\subseteq P} \omega(Q) |\psi_Q(q_{j}(\cdot,y))(x_Q)|^2\\
    &\qquad=\frac1{\omega(P)^{\frac2p-1}}\sum\limits_{k=k_0}^\infty\sum\limits_{\{Q\in Q^k : Q\subseteq P \}} \omega(Q) |\psi_Q(q_{j}(\cdot,y))(x_{Q})|^2\\
    &\qquad\lesssim \frac1{\omega(P)^{\frac2p-1}}\sum\limits_{k=k_0}^\infty\sum\limits_{\{Q\in Q^k : Q\subseteq P \}} \omega(Q)\frac 1{\omega(B(x_Q,r^{-j}))^2} r^{-2|k-j|\varepsilon}\\
    &\qquad\lesssim \sup_{x\in P}\frac 1{\omega(B(x,r^{-j}))^2}\cdot \frac1{\omega(P)^{\frac2p-1}}\sum\limits_{k=k_0}^\infty\sum\limits_{\{Q\in Q^k : Q\subseteq P \}} \omega(Q)r^{-2|k-j|\varepsilon}\\
    &\qquad\lesssim \sup_{x\in P}\frac 1{\omega(B(x,r^{-j}))^2} \cdot \frac1{\omega(P)^{\frac2p-2}}\sum\limits_{k=k_0}^\infty r^{-2|k-j|\varepsilon}.
    \end{align*}
    If $k_0\geqslant j$, then doubling property of the measure $\omega$ implies,
    \begin{equation*}
    \omega(B(x_P,r^{-j}))\lesssim r^{(-j+k_0){\bf N}}\omega(B(x_P,r^{-k_0}))\lesssim r^{(-j+k_0){\bf N}}\omega(B(x_P,r^{-k_0}))\omega(P) .
    \end{equation*}
    Thus,
    \begin{align*}
    &\frac1{\omega(P)^{\frac2p-1}}\sum\limits_{Q\subseteq P} \omega(Q) |\psi_Q(q_{j}(\cdot,y))(x_Q)|^2\\
    &\lesssim \sup_{x\in P}\frac 1{\omega(B(x,r^{-j}))^2}\cdot \frac1{\omega(B(x_P,r^{-j}))^{\frac2p-2}}\sum\limits_{k=k_0}^\infty r^{-2|k-j|\varepsilon}r^{(-j+k_0){\bf N}(\frac{2}{p}-2)}\\
    &\lesssim \sup_{x\in P}\frac 1{\omega(B(x,r^{-j}))^{\frac2p}}\sum\limits_{k=k_0}^\infty r^{-2|k-j|\varepsilon}r^{(-j+k_0){\bf N}(\frac{2}{p}-2)}.
    \end{align*}
    Since $\frac {\bf N}{{\bf N}+\varepsilon }<p$, we have $N(\frac{2}{p}-2)-2\varepsilon\le0$ and then $\sum\limits_{k=k_0}^\infty r^{(-j+k_0){\bf N}(\frac{2}{p}-2)-2(k_0-j)\varepsilon}\lesssim 1$.
    The above inequality yields
    $$\sup_P\frac1{\omega(P)^{\frac2p-1}}\sum\limits_{Q\subseteq P} \omega(Q) |\psi_Q(q_{j}(\cdot,y))(x_Q)|^2\leqslant C_j\sup_{x\in P}\frac 1{\omega(B(x,r^{-j}))^{\frac2p}}.$$
    For $k_0\leqslant j$, applying \eqref{cmo-aoe} gets
    \begin{align*}
    &\frac1{\omega(P)^{\frac2p-1}}\sum\limits_{Q\subseteq P} \omega(Q) |\psi_Q(q_{j}(\cdot,y))(x_Q)|^2\\
    &\qquad\lesssim \frac1{\omega(P)^{\frac2p-1}}\sum\limits_{k=k_0}^\infty\sum\limits_{\{Q\in Q^k : Q\subseteq P \}} \omega(Q)\frac 1{\omega(B(x_Q,r^{-j}))^2} r^{-2|k-j|\varepsilon}\\
    &\qquad\lesssim \sup_{x\in P}\frac 1{\omega(B(x,r^{-j}))^2} \frac1{\omega(P)^{\frac2p-1}}\sum\limits_{k=k_0}^\infty\sum\limits_{\{Q\in Q^k : Q\subseteq P \}} \omega(Q)r^{-2|k-j|\varepsilon}\\
    &\qquad\lesssim \sup_{x\in P}\frac 1{\omega(B(x,r^{-j}))^2}\frac1{\omega(P)^{\frac2p-1}}\sum\limits_{k=k_0}^\infty\omega(P) r^{-2|k-j|\varepsilon}\\
    &\qquad\lesssim \sup_{x\in P}\frac 1{\omega(B(x,r^{-j}))^2}\frac1{\omega(P)^{\frac2p-2}}\sum\limits_{k=k_0}^\infty r^{-2|k-j|\varepsilon}.
    \end{align*}
    Since
    $\omega(B(x_Q,r^{-j}))\leqslant \omega(B(x_Q,r^{-k_0}))\sim \omega(P)$,  we have
    \begin{align*}
    \frac1{\omega(P)^{\frac2p-1}}\sum\limits_{Q\subseteq P} \mu(Q) |\psi_Q(q_{j}(\cdot,y))(x_Q)|^2
    &\lesssim  \sup_{x\in P}\frac 1{\omega(B(x,r^{-j}))^{\frac2p}}\sum\limits_{k=k_0}^{\infty } r^{-2|k-j|\varepsilon}\\
    &\leqslant C_j \sup_{x\in P}\frac 1{\omega(B(x,r^{-j}))^{\frac2p}}.
    \end{align*}
    Taking the supremum over all dyadic cubes $P$ and using
    $$\inf\limits_{x\in \R^N} \omega(B(x,1))\geqslant C,$$ we obtain
    $$\sup_P \frac1{\omega(P)^{\frac2p-1}}\sum\limits_{Q\subseteq P} \omega(Q)  |\psi_Q(q_{j}(\cdot,y))(x_Q)|^2\leqslant C_j.$$
    and the proof is complete.
\end{proof}
We denote $L^2(\R^N,\omega)\cap CMO_d^p$ by the subspace of all $f\in L^2(\R^N,\omega)$ with the norm $\|f\|_{CMO_d^p}<\infty. $ Based on the above Lemma \ref{test}, if $f\in (L^2(\R^N,\omega)\cap CMO_d^p))^\prime,$ then $q_j(f)(x)$ is well defined since for each fixed $x, q_j(x,y)\in L^2(\R^N,\omega)\cap CMO_d^p.$

The following result describes an important property for such a distribution $f.$ More precisely, we establish the following weak-type discrete Calder\'on reproducing formula in the distribution sense:
\begin{proposition}\label{CRFdis1}
    Suppose that $\{f_n\}_{n=1}^\infty$ is a sequence in $L^2(\R^N,\omega)$ with $\|S(f_n-f_m)\|_p\rightarrow 0$ as $n, m\rightarrow \infty.$ Then there exists $f,$ as a distribution in $(L^2(\R^N,\omega)\cap CMO_d^p)^\prime $, such that (i)  $\|S(f)\|_p=\lim\limits_{n \rightarrow \infty}\|S(f_n)\|_p<\infty;$ (ii) there exists a distribution $h\in (L^2(\R^N,\omega)\cap CMO_d^p)^\prime$ with $\|f\|_2\sim \|h\|_2, \|S(f)\|_{p}\sim \|S(h)\|_{p},$ such that for each $g\in L^2(\R^N,\omega)\cap CMO_d^p, f$ has the following weak-type discrete Calder\'on reproducing formula in the distribution sense:
    \begin{align*}
    \langle f,g\rangle&:=\langle -\ln r\sum\limits_{j=-\infty}^\infty\sum\limits_{Q\in Q^j}w(Q)\psi_Q(\cdot,x_{Q})q_{Q}h(x_{Q}), g\rangle\\
    &=-\ln r\sum\limits_{j=-\infty}^\infty\sum\limits_{Q\in Q^j}w(Q)\psi_Qg(x_{Q})q_{Q}h(x_{Q}),
    \end{align*}
where the last serise converges absolutely.
\end{proposition}

\begin{proof}
    By  the Proposition \ref{pr1}, there exists $f\in (L^2(\R^N,\omega)\cap CMO_d^p)^\prime$ such that for each $g\in L^2(\R^N,\omega)\cap CMO_d^p,$
    $$\langle f,g\rangle=\lim\limits_{n \rightarrow \infty} \langle f_n, g\rangle.$$
    Observing that $\|S(f-f_n)\|_p=\|S(\lim\limits_{m \rightarrow \infty}(f_m-f_n))\|_p \leqslant \liminf\limits_{m \rightarrow \infty}
    \|S(f_m-f_n)\|_p,$ and hence, $\|S(f-f_n)\|_p\rightarrow 0$ as $n\rightarrow \infty.$ This implies that $\|S(f)\|_p=\lim\limits_{n \rightarrow \infty}\|S(f_n)\|_p<\infty.$ Applying Theorem \ref{1.7}, for each $f_n$ there exists an $h_n$ such that $\|f_n\|_2\sim \|h_n\|_2$ and $\|f_n\|_{H_d^p}\sim \|h_n\|_{H_d^p}.$ Thus, by the Proposition \ref{pr1}, there exists $h\in (L^2(\R^N,\omega)\cap CMO_d^p)^\prime$ such that for each $g\in L^2(\R^N,\omega)\cap CMO_d^p,$
    $$\langle h,g\rangle=\lim\limits_{n \rightarrow \infty} \langle h_n, g\rangle.$$
    Therefore, $\|S(h_n-h_m)\|_p\rightarrow 0$ and $\|S(h)\|_p=\lim\limits_{n \rightarrow \infty}\|S(h_n)\|_p\sim \lim\limits_{n \rightarrow \infty} \|S(f_n)\|_p\sim \|S(f)\|_p.$

    To show that $f$ has a weak-type discrete Calder\'on reproducing formula in the distribution sense, for each $g\in L^2(\R^N,\omega)\cap CMO_d^p,$ applying the proof of the Proposition \ref{pr1},
    $$\bigg|\sum\limits_{j=-\infty}^\infty\sum\limits_{Q\in Q^j}w(Q)\psi_Qg(x_{Q})q_{Q}h(x_{Q})\bigg|\leqslant C\|f\|_{H_d^p}\|g\|_{CMO_d^p},$$
    which implies that the series $\sum\limits_{j=-\infty}^\infty\sum\limits_{Q\in Q^j}w(Q)\psi_Q(x,x_{Q})q_{Q}h(x_{Q})$ is a distribution in $(L^2(\R^N,\omega)\cap CMO_d^p)^\prime.$ Moreover, by the weak-type discrete Calder\'on reproducing formula of $f_n$ in Theorem \ref{1.7}, for each $g\in L^2(\R^N,\omega)\cap CMO_d^p,$
    \begin{eqnarray*}
        \langle f,g\rangle=\lim\limits_{n \rightarrow \infty} \langle f_n, g\rangle
        =\lim\limits_{n \rightarrow \infty}\Big\langle -\ln r\sum\limits_{j=-\infty}^\infty\sum\limits_{Q\in Q^j}w(Q)\psi_Q(x,x_{Q})q_{Q}h_n(x_{Q}),
        g\Big\rangle,
    \end{eqnarray*}
    where $\|f_n\|_2\sim \|h_n\|_2$ and $\|S(f_n)\|_p\sim \|S(h_n)\|_p.$

    Observe that, by the same proof of Proposition \ref{pr1},

    $$\Big|\Big\langle -\ln r\sum\limits_{j=-\infty}^\infty\sum\limits_{Q\in Q^j}w(Q)\psi_Q(x,x_{Q})q_{Q}(h-h_n)(x_{Q}), g\Big\rangle\Big|\leqslant C\|S(h_n-h)\|_p\|g\|_{CMO_d^p},$$
    where the last term above tends to zero as $n\rightarrow \infty$ and hence,

    $$\langle f,g\rangle=\lim\limits_{n \rightarrow \infty}\langle f_n,g\rangle=\Big\langle -\ln r\sum\limits_{j=-\infty}^\infty\sum\limits_{Q\in Q^j}w(Q)\psi_Q(x,x_{Q})q_{Q}
    h(x_{Q}), g\Big\rangle.$$

    The proof of Proposition \ref{CRFdis1} is complete.
\end{proof}
The Proposition \ref{CRFdis1} indicates that one can consider $L^2(\R^N,\omega)\cap CMO_d^p,$  the subspace of all $f\in L^2(\R^N,\omega)$ with the norm $\|f\|_{CMO_d^p}<\infty,$ as the test function space and  $(L^2(\R^N,\omega)\cap CMO_d^p))^\prime,$ as the distribution space. The Dunkl-Hardy space is defined by Definition \ref{hp}.
We remark that in the Definition \ref{hp}, the series
$\sum\limits_{j=-\infty}^\infty\sum\limits_{Q\in
Q^j}w(Q)\lambda_{Q}\psi_Q(x,x_{Q})$ with
$\|\{\sum\limits_{j=-\infty}^\infty\sum\limits_{Q\in
Q^j}|\lambda_{Q}|^2\chi_{Q}\}^{\frac{1}{2}}\|_p<\infty$ defines a
distribution in $(L^2(\R^N,\omega)\cap CMO_d^p)^\prime.$ Indeed,
applying the proof of Proposition \ref{pr1}, for each $g\in
L^2(\R^N,\omega)\cap CMO_d^p,$
$$\Big|\sum\limits_{j=-\infty}^\infty\sum\limits_{Q\in Q^j}w(Q)\lambda_{Q}\psi_Qg(x_{Q})\Big|\leqslant C\Big\|\Big\{\sum\limits_{j=-\infty}^\infty\sum\limits_{Q\in Q^j}|\lambda_{Q}|^2\chi_{Q}\Big\}^{\frac{1}{2}}\Big\|_p \|g\|_{CMP_d^p}.$$

We now show Theorem \ref{1.8}.

\begin{proof}[\bf {The proof of Theorem \ref{1.8}}]
    Suppose $f\in H_d^p(\R^N,\omega).$ Then $f\in (L^2(\R^N,\omega)\cap CMO_d^p)^\prime$ and $f$ has a wavelet-type decomposition $f(x)=\sum\limits_{j=-\infty}^\infty\sum\limits_{Q\in Q^j}w(Q)\lambda_{Q}\psi_Q(x,x_{Q})$ in $(L^2(\R^N,\omega)\cap CMO_d^p)^\prime$ with $ \Big\|\{\sum\limits_{j}\sum\limits_{Q\in Q^j}|\lambda_{Q}|^2\chi_{Q}\}^{\frac{1}{2}}\Big\|_p<\infty.$  Set
    $$f_n(x)=\sum\limits_{|j|\leqslant n}\sum\limits_{\substack{Q\in Q_j\\ Q \subseteq B(0, n)} }w(Q)\lambda_{Q}\psi_Q(x,x_{Q}).$$
    Then $f_n\in L^2(\R^N,\omega)\cap H_d^p(\R^N,\omega)$ and $f_n$ converges  to $f$ in $(L^2(\R^N,\omega)\cap CMO_d^p)^\prime$ as $n$ tends to $\infty.$ To see that $f\in \overline{L^2(\R^N,\omega)\cap H^p(\R^N, \omega)},$ by Proposition \ref{CRFdis1}, it suffices to show that  $\|f_n-f_m\|_{H_d^p}\rightarrow 0$ as $n,m\rightarrow \infty.$ Indeed, if let $E_n=\{(j, Q): |j|\leqslant n, Q\in Q^j\subseteq B(0,n)\}$ and $E^c_{n,m}=E_n/E_m$ with $n\geqslant m,$
    \begin{align*}
    \|f_n-f_m\|_{H_d^p}&= \Big\|\Big(\sum\limits_{k\in \Bbb Z}\sum\limits_{Q'\in Q^k}|q_{Q'}(f_n-f_m)(x_{Q'})|^2\chi_{Q'}(x)\Big)^{\frac{1}{2}}\Big\|_p\\
    &\leqslant \Big\|\Big(\sum\limits_{k\in \Bbb Z}\sum\limits_{Q'\in Q^k}|q_{Q'}\big(\sum\limits_{E^c_{n,m}}w(Q)\lambda_{Q}\psi_Q(\cdot,x_{Q})  \big)(x_{Q'})|^2\chi_{Q'}(x)\Big)^{\frac{1}{2}}\Big\|_p\\
    &\leqslant C\Big\|\{\sum\limits_{E^c_{n,m}}|\lambda_{Q}|^2\chi_{Q}\}^{\frac{1}{2}}\Big\|_p\rightarrow 0,
    \end{align*}
    as $ n, m$ tend to $\infty,$ where the last inequality follows from the same proof of {\bf Step 1} in the Theorem \ref{1.7} and hence, $f\in \overline{L^2(\R^N,\omega)\cap H^p(\R^N, \omega)}.$

    Conversely, if $f\in \overline{L^2(\R^N,\omega)\cap H^p(\R^N, \omega)}$ by Proposition \ref{CRFdis1}, then there exists $h\in (L^2(\R^N,\omega)\cap CMO^p)^\prime$ with $\|S(h)\|_p\sim \|S(f)\|_p$ such that for each $g\in L^2(\R^N,\omega)\cap CMO_d^p,$
    $$\langle f,g\rangle=\bigg\langle -\ln r\sum\limits_{j=-\infty}^\infty\sum\limits_{Q\in Q^j}w(Q)\psi_Q(x,x_{Q})q_{Q}h(x_{Q}), g\bigg\rangle.$$
    Set $\lambda_Q=-\ln r q_{Q}h(x_{Q})$ with $Q\in Q^j.$ We obtain a wevelet-type decomposition of $f$ in $(L^2(\R^N,\omega)\cap CMO_d^p)^\prime$ in the distribution sense:
    $$f=\sum\limits_{j=-\infty}^\infty\sum\limits_{Q\in Q^j}w(Q)\lambda_Q\psi_Q(x,x_{Q})$$
    and hence, $f\in H_d^p(\R^N,\omega).$ Moreover
    $$\|f\|_{H_d^p}=\inf \Big\{\Big\|\{\sum\limits_{j=-\infty}^\infty\sum\limits_{Q\in Q^j}|\lambda_{Q}|^2\chi_{Q}\}^{\frac{1}{2}}\Big\|_p\Big\}\leqslant C\|S(h)\|_p\leqslant C\|S(f)\|_p.$$

    The proof of Theorem \ref{1.8} is complete.
\end{proof}

To describe the relationship between the Dunkl-Hardy space $H_d^p(\R^N,\omega)$ and the Hardy space $H_{cw}^p(\Bbb R^N, \|\cdot\|, \omega)$ on space of homogeneous type in the sense of Coifman and Weiss, we recall $H^p_{cw}$ as follows.
\begin{definition}\label{hpcw}
    $H^p_{cw}(\Bbb R^N, \|\cdot\|, \omega)$ is the collection of all distributions $f\in (\mathcal M(\beta,\gamma,r,x_0))^\prime$ such that
    $\|f\|_{H^p_{cw}}=\|S_{cw}(f)\|_p<\infty,$ where the square function $S_{cw}(f)$ is defined in the Definition \ref{scw}.
\end{definition}
The relationship between the Dunkl-Hardy space $H_d^p(\R^N,\omega)$ and the Hardy space $H_{cw}^p(\Bbb R^N, \|\cdot\|, \omega)$ is given by the following

\begin{theorem}\label{tm3}
    Suppose $\frac{\bf N}{{\bf N}+1}<p\leqslant 1.$ The Hardy space $H_d^p(\R^N, \omega)$ is equivalent to the Hardy space $H^p_{cw}(\mathbb R^N, \|\cdot\|, \omega)$ in the sense that if $f\in H_d^p(\R^N, \omega)$ then $f\in H^p_{cw}$ and there exists a constant $C$ such that $\|f\|_{H^p_{cw}}\leqslant C\|f\|_{H_d^p}.$
    Conversely, if $f\in H^p_{cw}(\mathbb R^N, \|\cdot\|, \omega)$ then $f$ can extend to a distribution $ \widetilde{f}$ on $(L^2(\R^N,\omega)\cap CMO_d^p)^\prime$ such that $\langle\widetilde{f}, g\rangle=\langle f,g \rangle$ for all $g\in \mathcal M(\beta,\gamma,r,x_0))^\prime$ and $\widetilde{f}\in H_d^p(\R^N, \omega),$ Moreover, $\|\widetilde{f}\|_{H_d^p}\leqslant C\|f\|_{H^p_{cw}}.$
\end{theorem}

\begin{proof}
    The proof of Theorem \ref{tm3} is based on Theorems \ref{1.7} and \ref{1.8}. Indeed, for $\frac{\bf N}{{\bf N}+1}<p\leqslant 1$ and $f\in L^2(\R^N,\omega),$ by Theorem \ref{1.7}, $\|S(f)\|_p\sim \|S_{cw}(f)\|_p.$ Therefore $\overline{ L^2(\R^N,\omega)\cap H_d^p(\R^N, \omega)}=\overline{ L^2(\R^N,\omega)\cap H^p_{cw}(\R^N, \|\cdot\|, \omega)}$ with the equivalent norms. Given $f\in H_d^p(\R^N, \omega),$ by Theorem \ref{1.8}, there exists a sequence $\{f_n\}_{n=1}^\infty$ such that for each $f_n\in L^2(\R^N,\omega)\cap H_d^p(\R^N, \omega)$ and $f_n$ converges to $f$ in $(L^2(\R^N,\omega)\cap CMO_d^p)^\prime.$ Moreover, $\|S(f)\|_p =\lim\limits_{n\rightarrow \infty}=\|S(f_n)\|_{p}.$ It is well known that by a classical result, $L^2(\R^N,\omega)\cap H^p_{cw}(\R^N, \|\cdot\|, \omega)$ is dense in $H^p_{cw}(\R^N, \|\cdot\|, \omega)$ and by Lemma \ref{test},  $(L^2(\R^N,\omega)\cap CMO_d^p)^\prime\subseteq (\mathcal M(\beta,\gamma,r,x_0))^\prime.$ Hence, $f_n$ is also converges to $f$ in $(\mathcal M(\beta,\gamma,r,x_0))^\prime$ and $\|S_{cw}(f)\|_p =\lim\limits_{n\rightarrow \infty}\|S_{cw}(f_n)\|_{p}.$ This implies that
     $$\|f\|_{H^p_{cw}}\leqslant C\|f\|_{H_d^p}.$$
    Suppose $f\in H^p_{cw}(\R^N, \|\cdot\|, \omega).$ there exists a sequence $\{f_n\}\in L^2(\R^N, \omega)$ such that $f_n$ converges $f$ in $(\mathcal M(\beta,\gamma,r,x_0))^\prime$ and $\|f\|_{H^p_{cw}}=\lim\limits_{n\rightarrow \infty}|f_n\|_{H_{cw}^p}.$ By the proof of Theorem \ref{1.7}, $\|S(f_n-f_m)\|_p\sim \|S_{cw}(f_n-f_m)\|_p $ and hence, $\|f_n-f_m\|_{H_d^p}=\|S(f_n-f_m)\|_p$ tends to zero as $n, m$ tends to $\infty.$
    Therefore, by Proposition \ref{CRFdis1}, $f_n$ tends to $\widetilde{f}$ in $(L^2(\R^N,\omega)\cap CMO_d^p)^\prime.$ It is clear that $ \widetilde{f}=f$ in $(\mathcal M(\beta,\gamma,r,x_0))^\prime.$ Moreover, by Theorem \ref{1.8},
    $$\|\widetilde{f}\|_{H_d^p}=\|S(\widetilde{f})\|_p=\lim\limits_{n\rightarrow \infty}\|S(f_n)\|_p\leqslant C\lim\limits_{n\rightarrow \infty}\|f_n\|_{H^p_{cw}}=C\|f\|_{H^p_{cw}}.$$

The proof of Theorem \ref{tm3} is complete.
\end{proof}

The proof of the Theorem \ref {atom} follows from the Theorem \ref{tm3}.
This atomic decomposition of $H_d^p(\R^N,\omega)$ is crucial for providing the boundedness of the Dunkl-Calder\'on-Zygmund operators from $H_d^p(\R^N,\omega)$ to $L^p(\R^N,\omega).$

\section{{Boundedness of Dunkl-Calder\'on-Zygmund Operators on $H_d^p$ }}

\subsection{{Boundedness of Dunkl-Calder\'on-Zygmund Operator from $H_d^p$  to $L^p$}}
\ \\

It is well known that the atomic decomposition is the main tool to
prove the boundedness of the classical Calder\'on-Zygmund singular
integral operators from $H^p(\R^N)$ to $L^p(\R^N)$ for $0<p\leqslant
1.$ Note that the proof of the theorem \ref{1.11} for $p=1$ was
given in the proof of Theorem \ref{th1.1}. We prove the Theorem
\ref{1.11} for $\frac{\bf N}{\bf N+\varepsilon}<p< 1$ only.

\begin{proof}[{\bf The proof of Theorem \ref{1.11}}]
    It is well known that the atomic decomposition of the Hardy spaces is a key method to show the boundedness of operators from $H_d^p(\R^N,\omega)$ to $L^p(\R^N,\omega).$ Observe that $L^2(\R^N,\omega)\cap H_d^p(\R^N,\omega)$ is dense in $H_d^p(\R^N,\omega).$ Moreoevr, if $f\in L^2(\R^N,\omega)\cap H_d^p(\R^N,\omega),$ by the Theorem \ref{atom}, then $f$ has an atomic decomposition which converges in both $L^2(\R^N,\omega)$ and $H_d^p(\R^N,\omega),$ respectively. Therefore, it suffices to show that if $a(x)$ is an $(p,2)$ atom of $H_d^p(\R^N,\omega),$ then $\|T(a)\|_p\leqslant C,$ where the constant $C$ is independent of $a.$ To this end, let $\supp a(x)\subseteq Q$ and $B=\{x: d(x,x_Q)\leqslant 4{\sqrt N}\l(Q) \},$ where $x_Q$ is the center of $Q$ and $\l(Q)$ is the side length of $Q.$ Write
    $$\int_{\R^N} |T(a)(x)|^pd\omega(x)=\int_B |T(a)(x)|^pd\omega(x)+\int_{B^c} |T(a)(x)|^pd\omega(x).$$
    The H\"older inequality, the $L^2$ boundedness of $T,$ and the size condition of $a$ imply that
    $$\int_B |T(a)(x)|^pd\omega(x)\leqslant C \omega(B)^{1-\frac{p}{2}}\|a\|^p_2\leqslant C \omega(Q)^{1-\frac{p}{2}} \omega(Q)^{p(\frac{1}{2}-\frac{1}{p})}\leqslant C,$$
    where the fact $\omega(B)\sim \omega(Q)$ is used.

    If $x\in B^c$ and $y\in Q,$ then $\|y-x_Q\|\leqslant \frac{1}{2}d(x,x_Q).$ By the cancellation condition of $a$ and the smoothness condition of the kernel $K(x,y),$
      \begin{align*}
    |T(a)(x)|&=\bigg|\int_Q K(x,y)a(y)d\omega(y)\bigg|=\bigg|\int_Q [K(x,y)-k(x,x_Q)]a(y)d\omega(y)\bigg|\\
                &\leqslant C\Big(\frac{\l(Q)}{\|x-x_Q\|}\Big)^\varepsilon \frac{1}{\omega(x, d(x,x_Q))}\|a\|_1.
     \end{align*}
    Therefore,
    $$\int_{B^c} |T(a)(x)|^pd\omega(x)\leqslant C \int_{B^c}\Big(\frac{\l(Q)}{\|x-x_Q\|}\Big)^{\varepsilon p} \Big(\frac{1}{\omega(x, d(x,x_Q))}\Big)^p\|a\|^p_1 d\omega(x).$$
    Applying the doubling on the measure $\omega$ and size condition on $a(x)$ yields
     \begin{align*}
&\int_{B^c}\Big(\frac{\l(Q)}{\|x-x_Q\|}\Big)^{\varepsilon p} \Big(\frac{1}{\omega(x, d(x,x_Q))}\Big)^p\|a\|^p_1 d\omega(x)\\
&\qquad = \sum\limits_{j\geqslant 1} \int_{2^j\l(Q)\leqslant d(x, x_Q)<2^{j+1}\l(Q)} \Big(\frac{\l(Q)}{d(x,x_Q)}\Big)^{\varepsilon p} \Big(\frac{1}{\omega(x, d(x,x_Q))}\Big)^p\|a\|^p_1 d\omega(x)\\
 &\qquad \leqslant C \sum\limits_{j\geqslant 1} 2^{-j\varepsilon p} \|a\|^p_1 \big(\omega(x_Q, 2^{j}\l(Q))\big)^{1-p}\\
  &\qquad \leqslant C \sum\limits_{j\geqslant 1} 2^{-j\varepsilon p} 2^{j{\bf N}(1-p)}\big(\omega(x_Q, \l(Q))\big)^{1-p}\big(\omega(Q)^{\frac 12}\|a\|_2\big)^p \\
    &\qquad \leqslant C \sum\limits_{j\geqslant 1} 2^{j(\bf N - (\varepsilon p +\bf N p))}\\
    &\qquad\leqslant C,
    \end{align*}
since $\frac{\bf N}{\bf N +\varepsilon}<p.$

    This implies $\|T(a)\|_p\leqslant C,$ where the constant $C$ is independent of $a(x).$
\end{proof}
We prove Theorem \ref{moleculeinHp}.

\begin{proof}[{\bf The proof of Theorem \ref{moleculeinHp}}]
    Set $\sigma=\|m\|_{L^2}^{-\alpha}$ with $\alpha={2p\over (2-p) \bf N}$.
    Observe that $B(x_0,2^{i+1}\sigma) \rightarrow \R^N$ as $i$ tends to $\infty.$ Thus, there exists an integer $i_0$ such that
    \begin{eqnarray}\label{molecule i0}
    \mu(B(x_0,2^{i_0+1}\sigma))>\sigma^{\bf N} \textup{\ \ \   and\ \ \   }
    \mu(B(x_0,2^{i_0}\sigma))\leqslant \sigma^{\bf N}.
    \end{eqnarray}

    Set $\chi_0=B(x_0, 2^{i_0}\sigma)$ and $ \chi_i=B(x_0,2^i2^{i_0}\sigma)\backslash B(x_0,2^{i-1}2^{i_0}\sigma)=\{x\in {\R^N}:  2^{i-1} 2^{i_0}\sigma \leqslant \|x-x_0\| < 2^i 2^{i_0}\sigma\}$ for $i\geqslant 1.$
    Let $\chi_i(x)$ be the characteristic function of $\chi_i$ for $i\geqslant 0.$
    We claim that there exists an integer $j_1> 1$ such that
    $\omega\Big(\bigcup\limits_{\ell=1}^{j_1} \chi_\ell \Big) > \omega(\chi_0)$
    and $\omega\Big(\bigcup\limits_{\ell=1}^{j_1-1} \chi_\ell \Big) \leqslant \mu(\chi_0).$  Indeed, suppose that such $j_1$ does not exist. Then for every integer $j>1$, we would have $\omega\Big(\bigcup\limits_{\ell=1}^{j} \chi_\ell \Big) \leqslant \omega(\chi_0)$. This implies that for any $j>1,
    \omega(B(x_0, 2^j2^{i_0}\sigma)\leqslant \omega(B(x_0,2^{i_0}\sigma)),$ which is impossible since $\omega\Big(B(x_0, 2^j2^{i_0}\sigma)\Big)\rightarrow +\infty$ as $j\rightarrow +\infty.$

    Applying the same stopping time argument yields that there exists a sequence $\{j_k\}_k$ such that $j_{k+1}>j_{k}+1,$
    $$  \omega\Big(\bigcup\limits_{\ell=j_{k}+1}^{j_{k+1}} \chi_\ell \Big) > \omega (B(x_0,2^{j_k}2^{i_0}\sigma))
    $$
    and
    $$
    \omega\Big(\bigcup\limits_{\ell=j_{k}+1}^{j_{k+1}-1} \chi_\ell \Big) \leqslant \omega(B(x_0,2^{j_k}2^{i_0}\sigma)). $$

    Observe that
    \begin{align}\label{molecule e1}
    \omega(B(x_0,2^{j_{k+1}}2^{i_0}\sigma))= \omega (B(x_0,2^{j_k+1}2^{i_0}\sigma)) = \omega\Big(\bigcup\limits_{\ell=j_{k}+1}^{j_{k+1}} \chi_\ell \Big)\geqslant 2\omega(B(x_0,2^{j_k}2^{i_0}\sigma))
    \end{align}
    for each integer $k\geq0$. Here we set $j_0=0$.

    Applying (\ref{molecule e1}) and the induction together with the doubling condition of the measure $\omega$  yields
    \begin{equation}\label{molecule e3}
    \begin{aligned}
    \omega(B(x_0,2^{j_k}2^{i_0} \sigma))&\geqslant  2^k \omega(B(x_0,2^{i_0} \sigma))\\
    &\geqslant  {C_\omega^{-1}} 2^{-\bf N} 2^k \omega(B(x_0,2^{i_0+1} \sigma))
    \geqslant {C_\omega^{-1}}  2^{-\bf N} 2^k  \sigma^{\bf N}.
    \end{aligned}
    \end{equation}
    We point out that for each integer $k\geq1$, if $j_k=j_{k-1}+1$, then we directly obtain that  $\omega( B(x_0, 2^{j_k}2^{i_0}\sigma ) ) \leqslant C_\omega 2^{\bf N} \omega( B(x_0, 2^{j_{k-1}}2^{i_0}\sigma ) )$ from the doubling property of the measure $\omega$. While if $j_k>j_{k-1}+1$, then
    \begin{eqnarray*}
        \omega( B(x_0, 2^{j_k}2^{i_0}\sigma ) ) &\leqslant& C_\omega 2^{\bf N} \omega( B(x_0, 2^{j_k-1}2^{i_0}\sigma ) ) \nonumber\\
        &=& C_\omega 2^{\bf N}\omega( B(x_0, 2^{j_k-1}2^{i_0}\sigma )\backslash B(x_0, 2^{j_{k-1}}2^{i_0}\sigma ) )+C_\omega 2^{\bf N}\omega( B(x_0, 2^{j_{k-1}}2^{i_0}\sigma ) ).
    \end{eqnarray*}
    Note that
    \begin{eqnarray*}
        \omega( B(x_0, 2^{j_k-1}2^{i_0}\sigma )\backslash B(x_0, 2^{j_{k-1}}2^{i_0}\sigma ) )= {\omega}\Big(\bigcup\limits_{\ell= j_{k-1}+1}^{j_k-1}\chi_\ell\Big)
        \leqslant {\omega}( B(x_0, 2^{j_{k-1}}2^{i_0}\sigma ) ),
    \end{eqnarray*}
    which, together with the above estimate for the case $j_k=j_{k-1}+1$, yields
    \begin{eqnarray}
    \omega( B(x_0, 2^{j_k}2^{i_0}\sigma ) )
    &\leqslant& C_\omega 2^{{\bf N}+1} \omega( B(x_0, 2^{j_{k-1}}2^{i_0}\sigma ) )\label{molecule e4}
    \end{eqnarray}
    for each integer $k\geq1$.

    We also point out that, by (\ref{molecule e4}), we obtain
    \begin{eqnarray*}
        {\omega}( B(x_0, 2^{j_{k+1}} 2^{i_0}\sigma ) )
        &\leqslant& 2^{\bf N +1}C_\omega \omega( B(x_0, 2^{j_{k}}2^{i_0}\sigma )),
    \end{eqnarray*}
    which together with the following estimates
    \begin{eqnarray*}
        \omega( B(x_0, 2^{j_{k}}2^{i_0}\sigma ))
        &\leqslant&  \omega( B(x_0, 2^{j_{k-1}}2^{i_0}\sigma ))+\omega\Big(\bigcup\limits_{\ell= j_{k-1}+1 }^{j_k}\chi_{\ell}\Big)\\
        &\leqslant&  2\omega\Big(\bigcup\limits_{\ell= j_{k-1}+1 }^{j_k}\chi_{\ell}\Big),
    \end{eqnarray*}
    gives
    \begin{eqnarray}\label{molecule e7}
    \omega( B(x_0, 2^{j_{k+1}} 2^{i_0}\sigma ) )
    &\leqslant& 4C_\omega 2^{\bf N}\omega\Big(\bigcup\limits_{\ell= j_{k-1}+1 }^{j_k}\chi_{\ell}\Big).
    \end{eqnarray}

    We now set
    $$\widetilde{\chi}_0(x):=\chi_0(x),\hskip1cm
    \widetilde{\chi}_{j_k}(x):=\sum\limits_{\ell=j_{k-1}+1}^{j_{k}} \chi_\ell(x)$$
    for integer $k\geq1$, and
    \begin{eqnarray}\label{molecule e2}
    m_k(x):= m(x)\widetilde{\chi}_{j_k}(x)- {1\over \int_{\R^N} \widetilde{\chi}_{j_k}(z)d\omega(z)} \int_{\widetilde{\chi}_{j_k}}m(y)d\omega(y) \widetilde{\chi}_{j_k}(x)
    \end{eqnarray}
    for each integer $k\geq0$.

    Decompose $m$ by
    $$m(x):=\sum\limits_{k=0}^{\infty}m_k(x)+\sum\limits_{k=0}^\infty \overline{m}_k \widetilde{\widetilde{\chi}}_{j_k}(x),  $$
    where $\overline{m}_k=\int_{\R^N} m(x)\widetilde{\chi}_{j_k}(x)d\omega(x)$ and $\widetilde{\widetilde{\chi}}_{j_k}(x)= {\widetilde{\chi}_{j_k}(x) \over \int_{\R^N} \widetilde{\chi}_{j_k}(y)d\omega(y)}$.

    We show that $\sum\limits_{k=0}^{\infty}m_k(x)$ gives an atomic decomposition with $m_k$ are $(p,2)$ atoms due to multiplication of certain constant. Observe  that $m_k$ is supported in
    $\widetilde{\chi}_{j_k}=\bigcup\limits_{\ell=j_{k-1}+1}^{j_{k}} \chi_\ell$, and  $\int_{\R^N} m_k(x)d\omega(x)=0 $. It remains to estimate the $L^2$ norm of $m_k$. First, we have
    \begin{eqnarray*}
        \|m_0\|_{L^2}&\leqslant& \Big( \int_{{\chi}_0} |m(x)|^2 d\omega(x) \Big)^{1/2}+\Big( \int_{{\chi}_0} \Big|{1\over \int_{\R^N} {\chi}_0(z)d\omega(z)} \int_{\chi_0}m(y)d\omega(y) \widetilde{\chi}_0(x)\Big|^2 d\omega(x) \Big)^{1/2}\\
        &\leqslant& 2\Big( \int_{{\chi}_0} |m(x)|^2 d\omega(x) \Big)^{1/2}\\
        &\leqslant& 2\|m\|_{L^2}\\
        &=& 2\sigma^{-{1\over \alpha}}\\
        &\leqslant& 2 \omega( B(x_0,2^{i_0}\sigma) )^{-{1\over \alpha{\bf N}}}\\
        &=& 2\omega({\chi}_0)^{{1\over2}-{1\over p}},
    \end{eqnarray*}
    where in the last inequality we use the fact in (\ref{molecule i0}) with $\alpha={2p\over (2-p) \bf N}$.

    Thus, $ \frac{1}{2}m_0(x) $ is an $(p,2)$ atom. Similarly, for each $k\geq1$,
    \begin{eqnarray*}
        \|m_k\|_{L^2}
        &\leqslant& 2\Big( \int_{\widetilde{\chi}_{j_k}} |m(x)|^2 d\omega(x) \Big)^{1/2}\\
        &=& 2\Big( \int_{\widetilde{\chi}_{j_k}} |m(x)|^2 \omega(B(x_0,\|x-x_0\|))^{1+{2\varepsilon-2\eta\over{\bf N}}} \omega({B(x_0,\|x-x_0\|)})^{-(1+{2\varepsilon-2\eta\over{\bf N}})}  d\omega(x) \Big)^{1/2}\\
        &\leqslant& 2 \omega(B(x_0,2^{j_{k-1}}2^{i_0}\sigma)^{-{1\over2}-{\varepsilon-\eta\over{\bf N}}}  \Big( \int_{\widetilde{\chi}_{j_k}} |m(x)|^2  \omega({B(x_0,\|x-x_0\|})^{1+{2\varepsilon-2\eta\over{\bf N}}} d\omega(x) \Big)^{1/2}\\
        &\leqslant& 2  \omega(B(x_0,2^{j_{k-1}}2^{i_0}\sigma))^{-{1\over2}-{\varepsilon-\eta\over
                {\bf N}}}   \|m\|_{L^2}^{-({{\bf N}+2\varepsilon-2\eta\over {\bf N}}{p\over 2-p}-1)}\\
        &=&    2  \omega(B(x_0,2^{j_{k-1}}2^{i_0}\sigma))^{-{1\over2}-{\varepsilon-\eta\over
                {\bf N}}}  \sigma^{{1\over \alpha} ({{\bf N}+2\varepsilon-2\eta\over {\bf N}}{p\over 2-p}-1) }.
    \end{eqnarray*}
    Applying the estimates given in (\ref{molecule e3}) yields
    \begin{eqnarray*}
        \sigma^{{1\over \alpha} ({{\bf N}+2\varepsilon-2\eta\over {\bf N}}{p\over 2-p}-1)}
        \leq(C_\omega 2^{\bf N}2^{-k}\omega(B(x_0,2^{j_{k}}2^{i_0}\sigma))^{{{1\over {{\bf N}\alpha}}} ({{\bf N}+2\varepsilon-2\eta\over {\bf N}}{p\over 2-p}-1)},
    \end{eqnarray*}
    which, together with the estimates given in (\ref{molecule e4}), namely that $$\omega(B(x_0,2^{j_{k-1}}2^{i_0}\sigma))^{-{1\over2}-{\varepsilon-\eta\over
            {\bf N}}}\leqslant (2^{{\bf N} +1})^{({1\over2}+{\varepsilon-\eta\over{\bf N}})}\omega(B(x_0,2^{j_{k}}2^{i_0}\sigma))^{-{1\over2}-{\varepsilon-\eta\over{\bf N}}}$$ implies
    \begin{equation}\label{molecule e5}
    \begin{aligned}
    \|m_k\|_{L^2}
    &\leqslant  2  \omega(B(x_0,2^{j_{k-1}}2^{i_0}\sigma))^{-{1\over2}-{\varepsilon-\eta\over{\bf N}}}  \sigma^{{1\over \alpha} ({{\bf N}+2\varepsilon-2\eta\over {\bf N}}{p\over 2-p}-1) }\\
    &\leqslant 2\cdot 2^{ -{1\over {\bf N}\alpha} ({{\bf N}+2\varepsilon-2\eta\over {\bf N}}{p\over 2-p}-1)  k}     (C_\omega 2^{{\bf N}+1})^{{1\over2}+{\varepsilon-\eta\over{\bf N}}} (C_\omega 2^{\bf N})^{ {1\over {\bf N}\alpha} ({{\bf N}+2\varepsilon-2\eta\over {\bf N}}{p\over 2-p}-1)  } \\
    &\qquad\qquad\times     \omega(B(x_0,2^{j_{k}}2^{i_0}\sigma))^{({1\over 2}-{1\over p})}.    \end{aligned}
\end{equation}
    Therefore, $2^{-1}\cdot 2^{ {1\over {\bf N}\alpha} ({{\bf N}+2\varepsilon-2\eta\over {\bf N}}{p\over 2-p}-1)  k}     (2^{{\bf N}+1})^{-{1\over2}-{\varepsilon-\eta\over{\bf N}}} (C_\omega2^{\bf N})^{ -{1\over {\bf N}\alpha} ({{\bf N}+2\varepsilon-2\eta\over {\bf N}}{p\over 2-p}-1)  }  m_k   $ are $(p,2)$ atoms.

    Moreover, $\sum\limits_{k=1}^\infty 2^{ -{1\over \bf N\alpha} ({{\bf N}+2\varepsilon-2\eta\over \bf N}{p\over 2-p}-1)kp}<\infty.$
    As a consequence,  $\sum\limits_{k=0}^{\infty}m_k(x)$ gives the desired atomic decomposition and hence, by the Theorem \ref{atom}, belongs to $H_d^p(\R^N)$ with the norm not larger than the constant $C$, which depends only on $p, \bf N, \varepsilon, \eta$ and $C_\omega.$

    It remains to show that $\sum\limits_{k=0}^\infty {\overline{m}}_k \widetilde{\widetilde{\chi}}_{j_k}(x)$ also gives an atomic decomposition. To see this, let $N_{k'}=\sum\limits_{k=k'}^\infty \overline{m}_k$.
    Note that $\sum\limits_{k=0}^\infty \overline{m}_k= \int_{\R^N} m(x)d\omega(x)=0.$ Summing up by parts implies that
    $$
    \sum\limits_{k=0}^\infty \overline{m}_k\widetilde{\widetilde{\chi}}_k(x)=\sum\limits_{k'=0}^\infty (N_{k'}-N_{k'+1})\widetilde{\widetilde{\chi}}_{k'}(x)=\sum\limits_{k'=0}^\infty N_{k'+1}(\widetilde{\widetilde{\chi}}_{k'+1}(x)-\widetilde{\widetilde{\chi}}_{k'}(x)).
    $$
    Observe that the support of $\widetilde{\widetilde{\chi}}_{k'+1}(x)-\widetilde{\widetilde{\chi}}_{k'}(x)$ lies within $ B(x_0,2^{k'+1}\sigma) $
    and $$\int_{\R^N} (\widetilde{\widetilde{\chi}}_{k'+1}(x)-\widetilde{\widetilde{\chi}}_{k'}(x)) d\omega(x)=0$$ since $\int_{\R^N}\widetilde{\widetilde{\chi}}_{k'}(x)d\omega(x)=1 $ for all $k'$. And we also have
    $$
    |\widetilde{\widetilde{\chi}}_{k'+1}(x)-\widetilde{\widetilde{\chi}}_{k'}(x) | \leq
    {1 \over \int_{\R^N} \widetilde{\chi}_{k'+1}(y)d\omega(y)}+{1 \over \int_{\R^N} \widetilde{\chi}_{k'}(y)d\omega(y)}\leqslant {2 \over \int_{\R^N} \widetilde{\chi}_{k'}(y)d\omega(y)}={2\over \omega(\widetilde{\chi}_{k'})}.
    $$
    Now applying (\ref{molecule e7}),  we obtain that
    \begin{eqnarray}\label{molecule e6}
    |\widetilde{\widetilde{\chi}}_{k'+1}(x)-\widetilde{\widetilde{\chi}}_{k'}(x) | \leqslant {8C_\omega\over \omega( B(x_0, 2^{j_{k'+1}}2^{i_0}\sigma ) )}.
    \end{eqnarray}
    Applying the H\"older inequality and the estimates in (\ref{molecule e5}), we obtain that
    \begin{eqnarray*}
        |N_{k'+1}|&\leqslant& \sum\limits_{k=k'+1}^\infty \int_{\R^N} |m(x)\widetilde{\chi}_{j_{k}}(x)|d\omega(x)\\
        &\leqslant& C \sum\limits_{k=k'+1}^\infty \Big(\int_{\widetilde{\chi}_{j_{k}}} |m(x)|^2d\omega(x)\Big)^{1/2}
        \omega(\widetilde{\chi}_{j_{k}})^{1/2}\\
        &\leqslant& C \sum\limits_{k=k'+1}^\infty   2 \cdot 2^{ -{1\over \bf N \alpha} ({{\bf N }+2\varepsilon-2\eta\over {\bf N}}{p\over 2-p}-1)  k}     (2^{\bf N +1}C_\omega)^{{1\over2}+{\varepsilon-\eta\over\bf N}} (2^{\bf N}C_\omega)^{ {1\over {\bf N} \alpha} ({\bf N +2\varepsilon-2\eta\over \bf N}{p\over 2-p}-1)  }      \\
        &&\times \omega(B(x_0,2^{j_{k}}2^{i_0}\sigma))^{({1\over 2}-{1\over p})}\omega( B(x_0, 2^{j_k}2^{i_0}\sigma ))^{1/2}\\
        &\leqslant& C 2 \cdot 2^{ -{1\over \bf N \alpha} ({{\bf N} +2\varepsilon-2\eta\over {\bf N}}{p\over 2-p}-1)  (k'+1)}     (2^{\bf N +1}C_\omega)^{{1\over2}+{\varepsilon-\eta\over\bf N}} (2^{\bf N}C_\omega)^{ {1\over \bf N \alpha} ({\bf N +2\varepsilon-2\eta\over \bf N}{p\over 2-p}-1)  }  \\ &&\times \omega( B(x_0, 2^{j_{k'+1}}2^{i_0}\sigma ) )^{1-{1\over p}}.
    \end{eqnarray*}
    The estimate above and the size estimate of $\widetilde{\widetilde{\chi}}_{k'+1}(x)
    -\widetilde{\widetilde{\chi}}_{k'}(x)$ in (\ref{molecule e6}) imply
    \begin{align*}
        &|N_{k'+1}(\widetilde{\widetilde{\chi}}_{k'+1}(x)
        -\widetilde{\widetilde{\chi}}_{k'}(x))|\\
        &\leqslant C 2\cdot 2^{ -{1\over \bf N \alpha} ({{\bf N} +2\varepsilon-2\eta\over \bf N}{p\over 2-p}-1)  (k'+1)}     (2^{\bf N +1}C_\omega)^{{1\over2}+{\varepsilon-\eta\over\bf N}} (2^{\bf N}C_\omega)^{ {1\over \bf N \alpha} ({\bf N +2\varepsilon-2\eta\over \bf N}{p\over 2-p}-1)  }  \\
        &\quad\times {62^{\bf N}C_\omega\over \omega( B(x_0, 2^{j_{k'+1}}\sigma ) )}     \omega( B(x_0, 2^{j_{k'+1}}2^{i_0}\sigma ) )^{1-{1\over p}}\\
        &\leqslant C  2^{ -{1\over \bf N \alpha} ({{\bf N} +2\varepsilon-2\eta\over {\bf N}}{p\over 2-p}-1)  (k'+1)} \omega( B(x_0, 2^{j_{k'+1}}2^{i_0}\sigma ) )^{-{1\over p}}.
    \end{align*}
    Therefore, we can rewrite $N_{k'+1}(\widetilde{\widetilde{\chi}}_{k'+1}(x)
    -\widetilde{\widetilde{\chi}}_{k'}(x))$ as
    $$ N_{k'+1}(\widetilde{\widetilde{\chi}}_{k'+1}(x)
    -\widetilde{\widetilde{\chi}}_{k'}(x))= \alpha_{k'}\beta_{k'}(x), $$
    where $\alpha_{k'}=C 2^{ -{1\over \bf N \alpha} ({\bf N +2\varepsilon-2\eta\over \bf N}{p\over 2-p}-1)  (k'+1)}$
    and $\beta_{k'}(x)$ are $(p, 2)$ atoms. Hence, by the Theorem \ref{atom},
    $ \sum\limits_{k=0}^\infty \overline{m}_k\widetilde{\widetilde{\chi}}_k(x)$ belongs to $H_d^p(\R^N)$ with the norm does not exceed $C$.

    The proof of Theorem \ref{moleculeinHp} is concluded.
\end{proof}
Applying the molecule theory, we show the Theorem \ref{1.12}.
\begin{proof}[{\bf The proof of Theorem \ref{1.12}}]
Recall that $K(x,y),$ the kernel of $T,$ satisfies the smoothness
condition
    $$|K(x,y)-K(x,y')|\leqslant C\Big(\frac{\|y-y'\|}{\|x-y\|}\Big)^\varepsilon \frac{1}{\omega(B(x,d(x,y)))}\Big(\frac{d(x,y)}{\|x-y\|}\Big)^{M}$$
    with $M>\bf N/2$.

    We first show the sufficient condition for Theorem \ref{1.12}. Note that for ${\bf N\over \bf N +\varepsilon}<p\leqslant 1$, $L^2(\R^N,\omega)\cap H_d^p(\R^N,\omega),$ is dense in $H_d^p(\R^N,\omega).$ As mentioned, if $f\in L^2(\R^N,\omega)\cap H_d^p(\R^N,\omega)$ then $f$ has an atomic decomposition $f=\sum\limits_j \lambda_j a_j$ where the series converges in both $L^2(\R^N,\omega)$ and $H_d^p(\R^N,\omega).$ Therefore,
    to show that $T$ extends to be a bounded operator on $H_d^{p}(\R^N,\omega)$, by Theorem \ref{moleculeinHp}, it suffices to prove that for each $(p,2)$-atom $a$, $ m=T(a) $ is an $(p,2,\varepsilon, \eta)$-molecule with ${\bf N\over \bf N +\varepsilon-\eta}<p\leq1, 0<\eta<\varepsilon,$ up to a multiplication of a constant $C$. To this end, suppose that $a$ is an $(p,2)$ atom with the support $B(x_0,r).$ We write
    \begin{align*}
        &\Big( \int_{\R^N} m(x)^2 d\omega(x) \Big)\Big( \int_{\R^N} m(x)^2 \omega(B(x_0,\|x-x_0\|))^{1+{2\varepsilon-2\eta\over \bf N}} d\omega(x)\Big)^{({\bf N +2\varepsilon-2\eta\over \bf N}{p\over 2-p}-1)^{-1}}\\
        &\leqslant \Big( \int_{\R^N} m(x)^2 d\omega(x) \Big)\Big( \int_{d(x_0,x)\leq2r} m(x)^2\omega(B(x_0,\|x-x_0\|))^{1+{2\varepsilon-2\eta\over \bf N}} d\omega(x)\Big)^{({\bf N +2\varepsilon-2\eta\over \bf N}{p\over 2-p}-1)^{-1}}\\
        & \qquad+\Big( \int_{\R^N} m(x)^2 d\omega(x) \Big)\Big( \int_{d(x_0,x)>2r} m(x)^2 \omega(B(x_0,\|x-x_0\|))^{1+{2\varepsilon-2\eta\over \bf N}} d\omega(x)\Big)^{({\bf N +2\varepsilon-2\eta\over \bf N}{p\over 2-p}-1)^{-1}}\\
        &=:I+I\!I.
    \end{align*}    Observe that, by the $L^2$ boundedness of $T$ and the size condition on $a$, we have  $\|m\|_{L^2}^2\leqslant \omega(B(x_0,r))^{(1-{2\over p})}.$ For $I$, applying the doubling property on $\omega$ implies that
    \begin{align*}
        I&\leqslant C\omega(B(x_0,r))^{(1-{2\over p})}\omega(B(x_0,2r))^{(1+{2\varepsilon-2\eta\over \bf N})({\bf N +2\varepsilon-2\eta\over \bf N}{p\over 2-p}-1)^{-1}}
        \Big( \int_{\R^N} m(x)^2  d\omega(x)\Big)^{({\bf N +2\varepsilon-2\eta\over \bf N}{p\over 2-p}-1)^{-1}}\\
        &\leqslant  C\omega(B(x_0,r))^{(1-{2\over p})}\omega(B(x_0,r))^{(1+{2\varepsilon-2\eta\over \bf N})({\bf N +2\varepsilon-2\eta\over \bf N}{p\over 2-p}-1)^{-1}}
        \omega(B(x_0,r))^{(1-{2\over p})({\bf N +2\varepsilon-2\eta\over \bf N}{p\over 2-p}-1)^{-1}}\\
        &\leqslant C,
    \end{align*}
    where the last inequality follows from the fact that $(1+{2\varepsilon-2\eta\over \bf N}) + (1-{2\over p})=\frac{2-p}{p}\big({\bf N +2\varepsilon-2\eta\over \bf N}{p\over 2-p}-1\big)$ and thus, $(1+{2\varepsilon-2\eta\over \bf N})({\bf N +2\varepsilon-2\eta\over \bf N}{p\over 2-p}-1)^{-1}+ {(1-{2\over p})({\bf N +2\varepsilon-2\eta\over \bf N}{p\over 2-p}-1)^{-1}}=\frac{2-p}{p}.$

    To estimate $I\!I,$ observe that if $d(x_0,x)>2r$, by the support and the cancellation condition on $a$ and the smoothness condition on the kernel $K(x,y)$, we have
    \begin{align*}
        |m(x)|&= \int_{\|y-x_0)\|\leqslant r\leqslant \frac{1}{2}d(x_0,x)}[K(x,y)-K(x,x_0)]a(y)d\omega(y)\\
        &\leqslant C\int_{d(x_0,x)>2r\geqslant 2\|y-x_0\|} {1\over \omega(B(x_0, d(x_0,x)))} \Big( {\|y-x_0\|\over \|x-x_0\|} \Big)^\varepsilon \Big(\frac{d(x_0,x)}{\|x-x_0\|}\Big)^{M}  |a(y)|d\omega(y)\\
        &\leqslant C {1\over \omega(B(x_0, d(x_0,x)))} \Big( {r\over \|x-x_0\|} \Big)^\varepsilon \Big(\frac{d(x_0,x)}{\|x-x_0\|}\Big)^{M} \omega(B(x_0,r))^{1-{1\over p}} .
    \end{align*}
    The estimate of $m(x)$ for $d(x_0,x)>2r$ and the doubling property on $\omega$ give
    \begin{align*}
        I\!I&\leq
        C\omega(B(x_0,r))^{(1-{2\over p})}\Big( \int_{d(x_0,x)>2r} m(x)^2 \omega(B(x_0,\|x-x_0\|))^{1+{2\varepsilon-2\eta\over \bf N}}  d\omega(x)\Big)^{({\bf N +2\varepsilon-2\eta\over \bf N}{p\over 2-p}-1)^{-1}}\\
        &\leqslant C\omega(B(x_0,r))^{(1-{2\over p})}
        \Bigg( \int_{d(x_0,x)>2r}  {1\over \omega(B(x_0,d(x_0,x)))^2}  \Big( {r\over \|x-x_0\|} \Big)^{2\varepsilon}\\
        &\quad\times\Big(\frac{d(x_0,x)}{\|x-x_0\|}\Big)^{2M} \omega(B(x_0,\|x-x_0\|))^{1+{2\varepsilon-2\eta\over \bf N}} \omega(B(x_0,r))^{2-{2\over p}} d\omega(x)\Bigg)^{({\bf N +2\varepsilon-2\eta\over \bf N}{p\over 2-p}-1)^{-1}}.
    \end{align*}
    We now split $\{x: d(x_0,x)>2r\}$ into annuli in terms of the Dunkl and Euclidean metrics as follows.
    \begin{align*}
        I\!I
        &\leqslant C\omega(B(x_0,r))^{(1-{2\over p})}
        \Bigg(\sum\limits_{j,k\geq1}\ \ \int_{\substack{d(x_0,x)\sim 2^jr\\ \|x-x_0\|\sim 2^kd(x_0,x)} }  {1\over \omega(B(x_0, 2^j r))^2}  \Big( {r\over 2^k2^jr} \Big)^{2\varepsilon} \Big(2^{-k}\Big)^{2M}\\
        &\quad\times  \omega(B(x_0,2^k2^j r))^{1+{2\varepsilon-2\eta\over \bf N}} \omega(B(x_0,r))^{2-{2\over p}} d\omega(x)\Bigg)^{({\bf N +2\varepsilon-2\eta\over \bf N}{p\over 2-p}-1)^{-1}}\\
        &\leqslant C\omega(B(x_0,r))^{(1-{2\over p})}
        \Bigg(\sum\limits_{j,k\geq1}\ \ \omega(B(x_0,2^j r))  {1\over \omega(B(x_0, 2^j r))^2}  \Big( {r\over 2^k2^jr} \Big)^{2\varepsilon} \Big(2^{-k}\Big)^{2M}\\
        &\quad\times  \omega(B(x_0,2^k2^j r))^{1+{2\varepsilon-2\eta\over \bf N}} \omega(B(x_0,r))^{2-{2\over p}} \Bigg)^{({\bf N +2\varepsilon-2\eta\over \bf N}{p\over 2-p}-1)^{-1}}\\
        &\leqslant C\omega(B(x_0,r))^{(1-{2\over p})}\\
        &\quad\times
        \Bigg(\sum\limits_{j,k\geq1}  2^{k\cdot \bf N}  2^{-2k\varepsilon}2^{-2j\varepsilon} \Big(2^{-k}\Big)^{2M}
          2^{k(2\varepsilon-2\eta)}2^{j(2\varepsilon-2\eta)} \omega(B(x_0,r))^{{2\varepsilon-2\eta\over \bf N}+2-{2\over p}} \Bigg)^{({\bf N +2\varepsilon-2\eta\over \bf N}{p\over 2-p}-1)^{-1}}\\
        &\leqslant C\omega(B(x_0,r))^{(1-{2\over p})}
        \Bigg( \omega(B(x_0,r))^{{2\varepsilon-2\eta\over \bf N}+2-{2\over p}} \Bigg)^{({\bf N +2\varepsilon-2\eta\over \bf N}{p\over 2-p}-1)^{-1}}\\
        &\leqslant C,
    \end{align*}
    where the third inequality follows from the doubling property of the measure $\omega$, the fourth  inequality follows from the fact that $M>\bf N/2, \eta>0$, and the last inequality follows from the fact that
    $$ 1-{2\over p} +\Big({2\varepsilon-2\eta\over \bf N}+2-{2\over p}\Big)\cdot \Big({\bf N +2\varepsilon-2\eta\over \bf N}{p\over 2-p}-1\Big)^{-1}=0.  $$

    Finally, by the fact that $T^*(1)=0$, we obtain that $ \int_{\R^N} m(x) d\omega(x)=\int_{\R^N} T(a)(x)d\omega(x)=0$ and hence $m$ is the multiple of an $(p,2,\varepsilon,\eta)$ molecule. The proof of the sufficient implication of Theorem \ref{1.12} then follows from Theorem \ref{moleculeinHp}.

We now show the necessary condition of the Theorem \ref{1.12} for
the boundedness on $H_d^p(\R^N,\omega).$ Indeed, we will prove a general
result, that is, if $T$ is bounded on $L^2(\R^N,\omega)$ and on
$H_d^p(\R^N,\omega)$ then $\int_{\R^N} T(f)(x)d\omega(x)=0$ for $f\in
L^2(\R^N,\omega)\cap H_d^p(\R^N,\omega).$ This follows from the following general
result.
\begin{prop} \label{HpinLp}
    If $f\in L^2(\R^N,\omega)\cap H_d^p(\R^N,\omega), \frac{\bf N}{{\bf N}+\varepsilon}<p\leqslant 1,$ then there exists a constant $C$ independent of the $L^2(\R^N,\omega)$ norm of $f$ such that
    \begin{eqnarray}
    \|f\|_p\leqslant C\|f\|_{H_d^p}.
    \end{eqnarray}
\end{prop}
Assuming Proposition \ref{HpinLp} for the moment, if $f\in
L^2(\R^N,\omega)\cap H^p(\R^N,\omega),$ by Proposition \ref{HpinLp},
then $f\in L^p(\R^N,\omega)\cap L^2(\R^N,\omega).$ Hence, by
interpolation, $f\in L^1(\R^N,\omega).$ To see the integral of $f$
is zero, we apply the Calder\'on reproducing formula,  $$f(x)=
\sum\limits_{k\in\mathbb{Z}}\sum\limits_{Q \in
Q^k}\omega(Q)D_k(x,x_Q){\widetilde D}_k(f)(x_Q),$$ where the series
converges in both $L^2(\R^N,\omega)$ and $H_d^p(\R^N,\omega).$ Let
$E_n(k,Q)$ be a finite set of $k\in\mathbb{Z}$ and $Q \in Q^k$ and
$E_n(k,Q)$ tends to the whole set $\lbrace (k,Q): k\in\mathbb{Z},
Q\in Q^k\rbrace.$ Therefore, $\sum\limits_{E_n^c(k,Q)} \omega(Q)
D_k(x,x_Q){\widetilde D}_k(f)(x_Q)$ converges to zero as $n$ tends
to infinity in both $L^2(\R^N,\omega)$ and $H_d^p(\R^N,\omega).$ We
obtain that
\begin{align*}
    &\bigg|\int_{\R^N} f(x)d\omega(x)\bigg|\\
    &\leqslant \Big|\int_{\R^N} \sum\limits_{E_n(k,Q)} \omega(Q) D_k(x,x_Q){\widetilde D}_k(f)(x_Q)d\omega(x)\Big|+  \bigg|\int_{\R^N} \sum\limits_{E_n^c(k,Q)} \omega(Q) D_k(x,x_Q){\widetilde D}_k(f)(x_Q)d\omega(x)\bigg|\\
    &\leq\Big|\int_{\R^N}  \sum\limits_{E_n^c(k,Q)} \omega(Q) D_k(x,x_Q){\widetilde D}_k(f)(x_Q) d\omega(x)\Big|\\
    &\leqslant C\|\sum\limits_{E_n^c(k,Q)}\omega(Q) D_k(x,x_Q){\widetilde D}_k(f)(x_Q) \|_{H_d^p}
    +C\|\sum\limits_{E_n^c(k,Q)} \omega(Q)D_k(x,x_Q){\widetilde D}_k(f)(x_Q)\|_2,
\end{align*}
where the second inequality follows from the fact that
$$\int_{\R^N} \sum\limits_{E_n(k,Q)}\omega(Q) D_k(x,x_Q){\widetilde
D}_k(f)(x_Q)d\omega(x)=0$$ by the cancellation property of
$D_k(x,x_Q).$ Letting $n$ tend to infinity gives the desired result
since the last two terms tend to zero as $n$ tends to infinity.

Now suppose that $T$ is bounded on both $L^2(\R^N,\omega)$ and
$H_d^p(\R^N,\omega)$ and $f\in L^2(\R^N,\omega)\cap H_d^p(\R^N,\omega),$ then $Tf
\in L^2(\R^N,\omega)\cap H_d^p(\R^N,\omega)$ and hence $\int_{\R^N}  Tf(x) d\omega(x)=0.$
The necessary implication of Theorem \ref{1.12} is concluded.

It remains to show Proposition \ref{HpinLp}. The key idea of the
proof is to apply the method of atomic decomposition for subspace
$L^2(\R^N,\omega)\cap H_{cw}^p(\R^N,\omega)$ as in the proof of
Proposition \ref{pr1}. More precisely, if $f\in L^2(\R^N,\omega)\cap
H_{cw}^p(\R^N,\omega),$ we set
\[\Omega_l=\left\{ x\in X: {S_{cw}}(f)(x)>2^l\right\},\]
\[
B_l=\left\{ Q: \omega(Q\cap\Omega_l)>{1\over2}\omega(Q) \text{ and }
\omega(Q \cap\Omega_{l+1})\leqslant {1\over2}\omega(Q)\right\}
\]
and
\[\widetilde{\Omega}_l=\left\{ x\in {\R^N}:  M(\chi_{\Omega_l})(x)>1/2\right\},\]
where $Q$ are Christ's dyadic cubes in space of homogeneous type
$(\R^N, \|\cdot\|, \omega)$ and $M$ is the Hardy--Littlewood maximal
function on $\R^N$ with respect to the measure $\omega$ and hence,
$\omega(\widetilde{\Omega}_l)\leqslant C\omega({\Omega}_l).$

Applying the decomposition of $f$ as in the proof of Proposition
\ref{pr1}, we write
\[
f(x)=\sum\limits_{l=-\infty}^\infty\sum\limits_{ Q\in B_l}
\omega(Q)D_k(x,x_Q) {\widetilde D}_k(f)(x_Q),
\]
where the series converges in both $L^2(\R^N,\omega)$ and
$H_{cw}^p(\R^N,\omega).$ Thus, for $\frac{\bf N}{{\bf
N}+\varepsilon}<p\leqslant 1,$
\[
\|f(x)\|_p^p\leq\sum\limits_{l=-\infty}^\infty\|\sum\limits_{ Q\in
B_l} \omega(Q)D_k(x,x_Q) {\widetilde D}_k(f)(x_Q)\|_p^p.
\]
Note that if $Q\in B_l$ then $Q\subseteq \widetilde{\Omega}_l.$
Therefore, $\sum\limits_{ Q\in B_l}  \omega(Q)D_k(x,x_Q) {\widetilde
D}_k(f)(x_Q)$ is supported in $\widetilde{\Omega}_l.$ Applying
H\"older inequality implies that
\[
\|f(x)\|_p^p\leq\sum\limits_{l=-\infty}^\infty
\mu(\widetilde{\Omega}_l)^{1-\frac{p}{2}}\|\sum\limits_{ Q\in B_l}
\omega(Q)D_k(x,x_Q) {\widetilde D}_k(f)(x_Q)\|_2^p.
\]
As in the proof of Proposition \ref{pr1}, we have
\[
\|\sum\limits_{ Q\in B_l}  \omega(Q)D_k(x,x_Q) {\widetilde
D}_k(f)(x_Q)\|_2\leqslant
C2^l\mu(\widetilde{\Omega}_l)^{\frac{1}{2}},
\]
which gives
\[
\|f(x)\|_p^p\leqslant C\sum\limits_{l=-\infty}^\infty
2^{lp}\mu({\Omega}_l)^{p}\leq C\|S_{cw}(f)\|^p_p\leqslant
C\|f\|^p_{H_{cw}^p}\leqslant C\|f\|_{H_d^p}^p
\]
since $\omega(\widetilde{\Omega}_l)\leqslant C\omega(\Omega_l).$
\end{proof}

\subsection{$T1$ Theorem on Dunkl-Hardy space $H_d^p$}

We show the Theorem \ref{1.13}.

\begin{proof}[{\bf The proof of Theorem \ref{1.13}}]
    By Proposition \ref{HpinLp}, we only need to show that if $T$ is a Dunkl-Calder\'on-Zygmund operator with $T^*(1)=0,$ then $T$ is bounded on the Dunkl-Hardy space $H_d^p(\R^N,\omega), \frac{\bf N}{\bf N +\varepsilon}<p\leqslant 1.$
    Following a similar idea used for the proof of the $T1$ Theorem \ref{1.2}, we consider first that $T$ also satisfies $T(1)=0.$ Since $L^2(\R^N,\omega)\cap H_d^p(\R^N,\omega)$ is dense in $H_d^p(\R^N,\omega),$ it suffices to show $\|T(f)\|_{H_d^p}\leqslant C\|f\|_{H_d^p}$ for $f\in L^2(\R^N,\omega)\cap H_d^p(\R^N,\omega).$ Observe that if $f\in L^2(\R^N,\omega)\cap H_d^p(\R^N,\omega),$ then $f(x)=\sum\limits_{j=-\infty}^\infty\sum\limits_{Q\in Q^j}w(Q)\lambda_{Q}\psi_Q(x,x_{Q})$ with $\|\{\sum\limits_{j=-\infty}^\infty\sum\limits_{Q\in Q^j}|\lambda_{Q}|^2\chi_{Q}\}^{\frac{1}{2}}\|_p\leqslant C\|f\|_{H_d^p}$, where the series converges in both $L^2(\R^N,\omega)$ and $H_d^P(\R^N,\omega).$ We have
    \begin{align*}
        \|T(f)\|_{H_d^p}&=\|S\big(T(f)\big)\|_p\\
        &=\|\{\sum\limits_{j=-\infty}^\infty\sum\limits_{Q\in Q^j}|q_Q\Big(\sum\limits_{k=-\infty}^\infty\sum\limits_{Q'\in Q^k}\omega(Q')\lambda_{Q'}T(\psi_{Q'}(\cdot,x_Q')\Big)(x_Q)|^2\chi_{Q}\}^{\frac{1}{2}}\|_p\\
        &=\|\{\sum\limits_{j=-\infty}^\infty\sum\limits_{Q\in Q^j}|\sum\limits_{k=-\infty}^\infty\sum\limits_{Q'\in Q^k}\omega(Q')\lambda_{Q'}q_QT\psi_{Q'}(x_Q,x_Q')|^2\chi_{Q}\}^{\frac{1}{2}}\|_p.
    \end{align*}
    Applying The Lemma \ref{Lem aoe}, we get
    $$|q_QT\psi_{Q'}(x_Q,x_Q')|\leqslant C r^{-|k-j|\varepsilon'}
    \frac1{V(x,y, r^{-j\vee -k}+d(x,y))}\Big(\frac{r^{-j\vee -k}}{r^{-j\vee -k}+d(x,y)}\Big)^\gamma$$
    and then applying the Lemma \ref{exchange} implies that
    $$\|T(f)\|_{H_d^p}\leqslant C \|\{\sum\limits_{k=-\infty}^\infty\sum\limits_{Q'\in Q^k}|\lambda_{Q'}|^2\chi_{Q'}\}^{\frac{1}{2}}\|_p\leqslant C\|f\|_{H_d^p}.$$

    To remove the condition $T(1)=0,$ set $ {\widetilde T}= T- \Pi_{T1}.$ Then ${\widetilde T}(1)={(\widetilde T)^*}(1)=0,$ so ${\widetilde T}$ is bounded on $H_d^p(\R^N,\omega).$ By a classical result, $\Pi_{T1}$ is bounded on the classical Hardy space $H_{cw}^p$ and thus, $T$ is bounded on $H_d^p(\R^N,\omega).$
    \end{proof}
    Finally, we show the Theorem \ref{1.14}.
\begin{proof}[{\bf The proof of Theorem \ref{1.14}}]
All we need to do is to check that $R_j(x,y),$ the kernel of the
Dunkl-Riesz transforms satisfy the conditions in the Theorem
\ref{1.13}. Indeed, it is known that $R_j, 1\leqslant j\leqslant N$
are bounded on $L^2(\R^N,\omega)$ and $H_d^1(\R^N,\omega),$ see
\cite{ADH}. Then by Proposition \ref{HpinLp} and then the proof of the necessary condition for Theorem \ref{1.13}, $R_j(1)=0$ and hence,
$(R_j)^*(1)=0$ since the Dunk-Riesz transforms are convolution
operators. It remains to see that $R_j(x,y)$ satisfy all kernel
conditions. To this end, observe that
$$\widehat{(R_j(f))}(\xi)=-i\frac{\xi_j}{\|\xi\|} {\widehat f }(\xi),$$
for $j=1,2,\cdots,N.$

Note that
$$ R_j(f)=-T_j(\triangle)^{-1/2}f=-c\int^\infty_0 T_je^{t\triangle}f\frac{dt}{\sqrt t}$$
and
$$e^{t\triangle}f(x)=\int_{\R^N} h_t(x,y)f(y)d\omega(y),$$
where the integral converges in $L^2$ and $h_t(x,y)$ is the heat
kernel. For all $x,y\in \R^N$ and $t>0,$
$$T_jh_t(x,y)=\frac{y_j-x_j}{2t}h_t(x,y).$$

We write the Riesz transforms as follows:
$$R_j f(x)= \int_{\mathbb R^N} R_j(x,y)f(y)d\omega(y).$$

To estimate the kernel $R_j(x,y),$ we recall the following estimates
for the Dunkl-heat kernel given in \cite[Theorem 3.1]{DH2}

\begin{itemize}
    \item[(a)] There are constants $C,c>0$ such that
    $$|h_t(x,y)|\leqslant C\frac1{V(x,y,\sqrt t)}\bigg(1+\frac{\|x-y\|}{\sqrt t}\bigg)^{-2}e^{-cd(x,y)^2/t},$$
    for every $t>0$ and for every $x,y\in \mathbb R^N$.
    \item[(b)] There are constants $C,c>0$ such that
    $$|h_t(x,y)-h(x,y')|\leqslant C\bigg(\frac{\|y-y'\|}{\sqrt t}\bigg)\frac1{V(x,y,\sqrt t)}\bigg(1+\frac{\|x-y\|}{\sqrt t}\bigg)^{-2}e^{-cd(x,y)^2/t},$$
    for every $t>0$ and for every $x,y,y'\in \mathbb R^N$ such that $\|y-y'\|<\sqrt t$.
\end{itemize}
We now estimate the kernel $R_j(x,y)$ as follows.

\begin{align*}
|R_j(x,y)|&\lesssim |y_j-x_j|\int_0^\infty \frac1{V(x,y,\sqrt t)}\frac t{\|x-y\|^2}e^{-cd(x,y)^2/t}\frac{dt}{t\sqrt t} \\
&\leqslant \frac1{\|x-y\|}\bigg(\int_0^{d(x,y)^2} + \int_{d(x,y)^2}^\infty\bigg) \frac1{V(x,y,\sqrt t)}e^{-cd(x,y)^2/t}\frac{dt}{\sqrt t}\\
&=: I_1 +I_2.
\end{align*}
For $t\leqslant d(x,y)^2$, by using the doubling condition we have
that
$$\omega(B(x,d(x,y)))\lesssim \Big(\frac {d(x,y)}{\sqrt
t}\Big)^{\mathbf N}\omega(B(x,\sqrt t))$$ and hence
\begin{equation}\label{eq2.9}
V(x,y,\sqrt t)^{-1}\lesssim \frac1{\omega(B(x,\sqrt t))}\lesssim
\Big(\frac {d(x,y)}{\sqrt t}\Big)^{\mathbf
N}\frac1{\omega(B(x,d(x,y)))}.
\end{equation}
We obtain
\begin{align*}
I_1&\lesssim   \frac1{\|x-y\|}\frac1{\omega(B(x,d(x,y)))} \int_0^{d(x,y)^2}\Big(\frac {d(x,y)}{\sqrt t}\Big)^{\mathbf N}e^{-cd(x,y)^2/t}\frac{dt}{\sqrt t}\\
&\lesssim   \frac1{\|x-y\|}\frac1{\omega(B(x,d(x,y)))} \int_0^{d(x,y)^2}\frac {d(x,y)^{\mathbf N}}{t^{\frac{1+\mathbf N}2}}\Big(\frac{t}{d(x,y)^2}\Big)^{\frac{1+\mathbf N}2}dt\\
&\lesssim  \frac{d(x,y)}{\|x-y\|}\frac1{\omega(B(x,d(x,y)))}.
\end{align*}
It is clear that for $t\geqslant d(x,y)^2$, by using the reversed
doubling condition, $$\Big(\frac {\sqrt
t}{d(x,y)}\Big)^N\omega(B(x,d(x,y)))\lesssim C\omega(B(x,\sqrt
t)),$$ we get
\begin{align*}
I_2&\lesssim \frac1{\|x-y\|}\int_{d(x,y)^2}^\infty \frac1{V(x,y,d(x,y))}\frac {d(x,y)^{N}}{t^{\frac{1+ N}2}}dt\\
&\lesssim  \frac{d(x,y)}{\|x-y\|}\frac1{\omega(B(x,d(x,y)))}.
\end{align*}
To see the smoothness estimates, we write
\begin{align*}
|R_j(x,y)-R_j(x,y')|
&\leqslant c |y_j-y_j'| \int_0^\infty |h_t(x,y)-h_t(x,y')|\frac{dt}{t\sqrt t}  \\
&\leqslant c|y_j-y_j'| \bigg(\int_0^{\|x-y\|^2} +
\int_{\|x-y\|^2}^\infty\bigg)
|h_t(x,y)-h_t(x,y')|\frac{dt}{t\sqrt t} \\
&=: I\!I_1 +I\!I_2.
\end{align*}
Since $\|y-y'\|<\frac 12d(x,y)$, we have $d(x,y')\leqslant
\frac32d(x,y)$
\begin{align*}
I\!I_1&\lesssim \frac{\|y-y'\|}{\|x-y\|^2}\frac1{\omega(B(x,d(x,y)))} \int_0^{\|x-y\|^2}\Big(\frac {d(x,y)}{\sqrt t}\Big)^N e^{-cd(x,y)^2/t}\frac{dt}{\sqrt t}\\
&\lesssim \frac{\|y-y'\|}{\|x-y\|^2}\frac1{\omega(B(x,d(x,y)))} \int_0^{\|x-y\|^2}\frac {d(x,y)^N}{t^{\frac{1+ N}2}}\Big(\frac{t}{d(x,y)^2}\Big)^{\frac{N}2}dt\\
&\lesssim \frac{\|y-y'\|}{\|x-y\|}\frac1{\omega(B(x,d(x,y)))}.
\end{align*}
To estimate $I\!I_2$, we have $\|y-y'\|<\frac 12d(x,y)\le\frac
12\|x-y\|<\sqrt t$ and the above condition (b) gives
\begin{align*}
I\!I_2&\lesssim \|y_j-y_j'\|\int_{\|x-y\|^2}^\infty
\|y-y'\|\frac1{V(x,y,\sqrt t)}
e^{-cd(x,y)^2/t}\frac{dt}{t^2}\\
&\lesssim \|y-y'\|^2 \frac1{\omega(B(x,d(x,y)))}\int_{\|x-y\|^2}^\infty t^{-2}dt \\
&\lesssim \frac{\|y-y'\|}{\|x-y\|}\frac1{\omega(B(x,d(x,y)))}.
\end{align*}
The estimate of the smoothness for $x$ variable is similar.  We
conclude that $R_j, 1\leqslant j\leqslant N,$ satisfy all conditions
in the Theorem \ref{1.13} and hence, the Dunk-Riesz transforms are
bounded on the Dunkl-Hardy space $H_d^p(\R^N,\omega)$ for $\frac{\bf
N}{{\bf N}+1}<p\leq1.$
\end{proof}

\bigskip
\bigskip
\bigskip

\noindent {\bf Acknowledgement}: Chaoqiang Tan is supported by National Natural
	Science Foundation  of China (Grant No. 12071272). Yanchang Han is supported NNSF of China (Grant No. 12071490) and Guangdong Province Natural Science Foundation  (Grant No. 2021A1515010053). M. Lee is supported by MOST 108-2115-M-008-002-MY2. Ji Li is supported by ARC DP 220100285.  Ji Li would like to thank Jorge Betancor for helpful discussions.

\bigskip
\bigskip

\medskip
\vskip 0.5cm
\noindent  Department of Mathematics, Shantou University, Shantou,
515063, R. China.

\noindent {\it E-mail address}: \texttt  { cqtan@stu.edu.cn }

\medskip
\vskip 0.5cm

\noindent School  of Mathematic Sciences, South China Normal University, Guangzhou, 510631, P.R. China.

\noindent {\it E-mail address}: \texttt{20051017@m.scnu.edu.cn}

\medskip
\vskip 0.5cm

\noindent Department of Mathematics, Auburn University, AL
36849-5310, USA.

\noindent {\it E-mail address}: \texttt{hanyong@auburn.edu}

\medskip
\vskip 0.5cm

\noindent Department of Mathematics National Central University Chung-Li 320, Taiwan Republic of China

\noindent {\it E-mail address}: \texttt{mylee@math.ncu.edu.tw}

\medskip
\vskip 0.5cm
\noindent Department of Mathematics, Macquarie University, NSW, 2109, Australia.

\noindent {\it E-mail address}: \texttt{ji.li@mq.edu.au}


\bigskip
\bigskip

\begin{thebibliography}{WWWW}

\bibitem{ARSW} N. Arcozzi, R. Rochberg, E. T. Sawyer and B. D. Wick,  Potential Theory on Trees, Graphs and Ahlfors-regular Metric Spaces,
  Potential Anal., {\bf 41} (2014), 317--366.


\bibitem{AAS} B. Amri, J. Ph. Anker and M. Sifi, Three results in Dunkl theory, Cplloq. Math., {\bf118} (2010), 299--312.

\bibitem{AGS} B. Amri, A. Gasmi and  M. Sifi, Linear and bilinear multiplier operators for Dunkl transform, Mediterr. J. Math., {\bf7} (2010), 503--521.

\bibitem{AH} B. Amri and A. Hammi, Dunkl-Schr\"odinger operators, Complex Anal. Oper. Theory, {\bf 13}  (2019), 1033--1058.

\bibitem{AS} B. Amri and M. Sifi, Riesz transforms for Dunkl transform, Ann. Math. Blaise Pascal, {\bf19}  (2012), 247--262.

\bibitem{A} J.-Ph. Anker, An introduction to Dunkl theory and its analytic aspects, Analytic, Algebraic and Geometric Aspects of Differential Equations, 3-58, Trends Math.,Birkh\"auser, Chem, 2017.

\bibitem{ABDH} J.-Ph. Anker, N. Ben Salem, J. Dziuba\'nski and N. Hamda, The Hardy space $H^1$ in the rational Dunkl setting, Constr. Approx., {\bf42} (2015), 93--128.

\bibitem{ADH} J.-Ph. Anker, J. Dziuba\'nski and A. Hejna, Harmonic functions, conjugate harmonic functions and the Hardy $H^1$ in rational Dunkel setting,
J. Fourier Anal. Appl., {\bf25} (2019), no. 5, 2356--2418.


\bibitem{AuH} P. Auscher and T. Hyt\"onen,  Orthonormal
    bases of regular wavelets in spaces of homogeneous type,
Appl. Comput. Harmon. Anal., {\bf 34} (2013), 266--296.

\bibitem{C} A. P. Calder\'{o}n, Intermediate spaces and interpolation, the complex method, \textit{Studia Math.}, \textbf{24} (1964), 113--190.

\bibitem{CW} R. R. Coifman, G. Weiss, Analyse Harmonique Non-commutative sur Certains Espaces Homog$\grave{\text{e}}$nes, Lecture Notes in Math., \textbf{242}, Springer-Verlag, Berlin (1971).


\bibitem{CW2} R.R. Coifman and G. Weiss,  Extensions
    of Hardy spaces and their use in analysis, Bull. Amer.
Math. Soc., {\bf 83} (1977), 569--645.


\bibitem{DJS} G. David, J. L. Journ\'{e}, S. Semmes, Op\'{e}rateurs de Calder\'{o}n-Zygmund, fonctions paraaccr\'{e}tives et interpolation, \textit{Rev. Mat. Iberoam.}, \textbf{1} (1985), 1--56.


\bibitem{DX} F. Dai and Y. Xu, Analysis on h-harmonics and Dunkl transforms, edited by Sergey Tikhonov,Advanced Courses in Mathematics. CRM Barcelona, Birkh$\ddot{a}$user/Springer, Basel, 2015. MR33009987.


\bibitem{DW} F. Dai and H. Wang, A transference theorem for the Dunkl transform and its applications, J. Funct. Anal., {\bf258} (2010), no. 12, 4052--4074.

\bibitem{DK} L. Deleaval and C. Kriegler, Dunkl spectral multipliers with values in UMD lattices, J. Funct. Anal., {\bf272} (5) (2017), 4052--4074.


\bibitem{DH} D. Deng and Y. Han,  Harmonic analysis on
    spaces of homogeneous type, Lecture Notes in Math., vol.
1966, Springer-Verlag, Berlin, 2009, with a preface by Yves
Meyer.


\bibitem{D1} C.F. Dunkl, Reflection groups and orthogonal polynomials on the sphere, Math. Z., {\bf197} (1) (1988), 33--60.

\bibitem{D2} C.F. Dunkl, Differential-difference operators associated to reflection groups,  Trans. Amer. Math., {\bf 311} (1989), no. 1,
167--183.

\bibitem{D3} C.F. Dunkl, Integral kernels with reflection group invariance. Canad. J. Math., {\bf43} (1991) no. 6 1213--1227.

\bibitem{DH1} J. Dziuba\'nski and A. Hejna, Remark on atomic decompositions for the Hardy space $H^1$ in the rational Dunkl setting,  Studia
Math., {\bf 251}  (2020),  89--110.

\bibitem{DH2} J. Dziuba\'nski and A. Hejna, H\"ormander multiplier theorem for the Dunkl transform, J. Funct. Anal., {\bf277} (2019), 2133--2159.

\bibitem{DH3} J. Dziuba\'nski and A. Hejna,  Singular integrals in the rational Dunkl setting, Revista Matematica
Complutense, (2021) DOI10.1007/s13163-021-00402-1.


\bibitem{FJ} M. Frazier and B. Jawerth,  A discrete
        transform and decomposition of distribution spaces,
        J. Funct. Anal.~\textbf{93} (1990), 34--170.\bibitem{FJ} M. Frazier and B. Jawerth, Decomposition of Besov spaces, Indiana Univ. Math. J., {\bf34} (1985), no. 4, 777--799.

\bibitem{HHL}Ya. Han, Yo. Han and  J. Li,  Criterion of the boundedness of singular integrals on spaces of homogeneous type,
      J. Funct. Anal., {\bf 271} (2016), 3423--3464.

\bibitem{H1} Y. Han, Calder\'{o}n-type reproducing formula and the Tb theorem, Rev. Mat. Iberoam., \textbf{10} (1994), 51--91.

\bibitem{H2} Y. Han, Plancherel-P\'{o}lya type inequality on spaces of homogeneous type and its applications, Proc. Amer. Math. Soc., \textbf{126} (1998), 3315--3327.

\bibitem{HLW} Y. Han, J. Li and L. Ward, Product $H^p, CMO^p, VMO$ and duality via orthonormal bases on spaces of homogeneous type, Appl. Comput. Harmon. Anal., {\bf45} (2018), no. 1, 120--169.


\bibitem{HMY} Y. Han, D. M\"{u}ller and D. Yang, A theory
    of Besov and Triebel-Lizorkin spaces on metric measure
    spaces modeled on Carnot--Carath\'eodory spaces, Abstr.
Appl. Anal., Vol.~2008, Article ID 893409. 250~pages.


\bibitem{HS} Y. Han and E. T. Sawyer, Littlewood-Paley theory on spaces of homogeneous type and classical function spaces, Mem. Amer. Math. Soc., \textbf{110} (1994), no. 530, 1--126.

\bibitem{Jeu} M.F.E. de Jeu, The Dunkl transform, Invent. Math., {\bf113} (1993), 147--162.

\bibitem{LL} Z. Li and J. Liao, Harmonic analysis associated with the one-dimensional Dunkl transform, Constr. Approx., {\bf37} (2) (2013), 233--281.

\bibitem{M} Y. Meyer, Les nouveaux op\'erateurs de Calder\'on-Zygmund, Ast\'erisque, tome {\bf131} (1985), 237--254.

\bibitem{MC} Y. Meyer and R. Coifman, Wavelets Calder\'on-Zygmund and multilinear operators, Translated by David Salinger, Cambridge University Press, 1997.



\bibitem{NS} A. Nagel and E. M. Stein, \emph{The $\overline{\partial}_b$-complex on decoupled boundaries in ${\mathbb C}^n$}, Ann. Math., \textbf{164} (2006), 649--713.


\bibitem{R1} M. R\"osler, Asymototic analysis for the Dunkl transform and its applications, J. Approx. Theory, {\bf119} (1) (2002), 110--126.

\bibitem{R2} M. R\"osler, Positivity of Dunkl intertwining operator, Duke Math. J., {\bf98} (1999) 445--463.

\bibitem{R3} M. R\"osler, Dunkl operators: theory and applications, Orthoganal polynomials and special functions(Leuven, 2002), Lecture Notes in Math. vol. 1817, Springer, Berlin, 2003, pp. 93--135.

\bibitem{R4} M. R\"osler, A positive radial product formula for the Dunkl kernel, Trans. Amer. Math. Soc., {\bf355} (2003), 2413--2438.

\bibitem{R5} M. R\"osler and M. Voit, Markov processes related with Dunkl operators, Adv. Appl. Math., {\bf21} (1998), 575--643.

\bibitem{RV} M. R\"osler and M. Voit, Dunkl theory, convolution algebras, and related Marrkov processes, in: P. Graczyk, M. R\"osler and M. Yor (Eds.),
Harmonic and Stochastic Analysis of Dunkl Processes, in: P. Travaux en Cours, vol. 71, Hermann, Paris, 2008, pp. 1--112.

\bibitem{S} F. Soltani, $L^p$-Fourier multipliers for the Dunkl operator on the real line, J. Funct. Anal., {\bf209} (2004), 16--35.




\bibitem{TX1} S. Thangvelu and Yuan Xu, Convolution operator and maximal function for the Dunkl transform, J. Anal. Math., {\bf97} (2005), 25--55.

\bibitem{TX2} S. Thangvelu and Yuan Xu, Riesz transform and Riesz potentials for Dunkl transform, J. Comput. Appl. Math., {\bf199}
(2007), 181--195.


\bibitem{Tol} X. Tolsa, Littlewood-Paley Theory and the $T(1)$ Theorem with Non-doubling Measures, Adv. in Math., {\bf164} (2001), 57--116.
















































%





















\end{thebibliography}
\end{document}